\documentclass[11pt,american,english]{article}
\usepackage[T1]{fontenc}
\usepackage[latin9]{inputenc}
\usepackage[a4paper]{geometry}
\usepackage{babel}
\usepackage{amsmath}
\usepackage{amsthm}
\usepackage{amssymb}
\usepackage{setspace}
\usepackage{commath}
\usepackage{comment}
\usepackage{mathrsfs}
\usepackage{color,xcolor}
\usepackage[unicode=true,pdfusetitle,
 bookmarks=true,bookmarksnumbered=true,bookmarksopen=true,bookmarksopenlevel=3,
 breaklinks=false,pdfborder={0 0 1},backref=false,colorlinks=false]
 {hyperref}

\makeatletter
\theoremstyle{plain}
\newtheorem{thm}{\protect\theoremname}
  \theoremstyle{plain}
  \newtheorem{lem}[thm]{\protect\lemmaname}
  \theoremstyle{remark}
  \newtheorem{rem}[thm]{\protect\remarkname}
  \theoremstyle{plain}
  \newtheorem{prop}[thm]{\protect\propositionname}
  \theoremstyle{plain}
  \newtheorem{defn}[thm]{Definition}
  \theoremstyle{plain}
  \newtheorem{example}[thm]{Example}

\usepackage{multirow}

\usepackage{arydshln}

\usepackage{lmodern}
\newcommand{\1}{\mbox{1\hspace{-1mm}I}}
\numberwithin{equation}{section}
\numberwithin{thm}{section}

\newcommand{\s}[1]{{\mathcal #1}}
\newcommand{\sr}[1]{{\mathscr #1}}
\newcommand{\bb}[1]{{\mathbb #1}}
\newcommand{\ip}[2]{\left\langle #1,#2 \right\rangle}
\newcommand{\ipt}[1]{\left\langle #1 \right.}

\newcommand{\ipbc}[1]{\left. , #1 \right\rangle}

\newcommand{\h}{\hspace}

\usepackage{calc}
\makeatother

  \addto\captionsamerican{\renewcommand{\lemmaname}{Lemma}}
  \addto\captionsamerican{\renewcommand{\propositionname}{Proposition}}
  \addto\captionsamerican{\renewcommand{\remarkname}{Remark}}
  \addto\captionsamerican{\renewcommand{\theoremname}{Theorem}}
  \addto\captionsenglish{\renewcommand{\lemmaname}{Lemma}}
  \addto\captionsenglish{\renewcommand{\propositionname}{Proposition}}
  \addto\captionsenglish{\renewcommand{\remarkname}{Remark}}
  \addto\captionsenglish{\renewcommand{\theoremname}{Theorem}}
  \providecommand{\lemmaname}{Lemma}
  \providecommand{\propositionname}{Proposition}
  \providecommand{\remarkname}{Remark}
\providecommand{\theoremname}{Theorem}

\begin{document}
\selectlanguage{american}%
\global\long\def\1{\mbox{1\hspace{-1mm}I}}
\selectlanguage{english}%

\title{CONTROL ON HILBERT SPACES AND APPLICATION TO  SOME MEAN FIELD TYPE CONTROL
 PROBLEMS}

\author{Alain Bensoussan\thanks{Alain Bensoussan is also with the School of Data Sciences, City University Hong Kong.
Research supported by the National Science Foundation under grants
 DMS-1905459 and DMS-2204795.}\\
International Center for Decision and Risk Analysis\\
Jindal School of Management, University of Texas at Dallas\\
\\
 Philip  Jameson Graber\thanks{Research supported by the National Science Foundation, grants DMS-1612880 and
DMS-1905459.}\\
Department of Mathematics, Baylor University\\
\\
Sheung Chi Phillip Yam\thanks{Phillip Yam acknowledges the financial supports of HKGRF-14300717 with project title, ``New kinds of forward-backwards stochastic systems with applications.'' and HKGRF-14301321 with project title, ``General Theory for Infinite Dimensional Stochastic Control: Mean Field and Some Classical Problems.''.}\\
Department of Statistics, The Chinese University of Hong Kong
}
\maketitle
\begin{abstract}
We propose a new approach to studying classical solutions of the second order Bellman equation and Master equation for mean field type control problems, using a novel form of the ``lifting'' idea introduced by P.-L.~Lions.
Rather than studying the usual system of Hamilton-Jacobi/Fokker-Planck PDEs using analytic techniques, we instead study a stochastic control problem on a specially constructed Hilbert space, which is reminiscent of a tangent space on the Wasserstein space in optimal transport.
On this Hilbert space we can use classical control theory techniques, despite the fact that it is infinite dimensional.
A consequence of our construction is that the mean field type control problem appears as a special case.
Thus we preserve the advantages
of the lifting procedure, while removing some of the difficulties. 	
Our approach extends previous work by two of the coauthors, which dealt with a deterministic control problem for which the Hilbert space could
be generic \cite{ABY}. 
\end{abstract}

\section{INTRODUCTION }

Mean Field Game/Control theory has made remarkable progress in recent years,
thanks to important contributions, particularly the recent books of
R.~Carmona and F.~Delarue \cite{RCD} and P.~Cardaliaguet, F.~Delarue, J.-M Lasry, and P.-L Lions \cite{CDLL}. Many additional concepts, techniques
and results can be found in the papers of A.~Cosso and H.~Pham \cite{CPH},
H.~Pham and X.~Wei \cite{PHW}, M.F.~Djete, D.~Possamai, and X.~Tan \cite{DPT},
R.~Buckdahn, J.~Li, S.~Peng, and C.~Rainer \cite{BLPR}, R.~Carmona and F.
Delarue \cite{CDF}, C.~Mou and J.~Zhang \cite{MOZ}, and Gangbo and M\'esz\'aros \cite{GM}. All these results
contribute to the rigorous treatment of the Bellman and Master equations of Mean
Field Games and Mean Field Type Control Theory. In this article, we contribute
to this objective with a different vision, extending the theory developed
in the paper of two of the authors, A.~Bensoussan, S.~C.~P.~Yam \cite{ABY}, inspired by the philosophy of the paper of W.~Gangbo and A.~\'{S}wi\k{e}ch \cite{GAS} which handled the first order master equation, while we consider the second order one in this article with analysis based the FBSDE approach. We also refer the reader to \cite{ABY} for the deterministic problem which comes much closer to \cite{GAS}. In \cite{ABY} we have considered an abstract
control problem for a dynamical system whose state space is a Hilbert space.
It was a purely deterministic control problem. The fact that the state
space is infinite dimensional does not prevent the methodology of
control theory to be applicable. We used simple dynamics (since
a major objective was to compare with the approach of W. Gangbo and
A. \'{S}wi\k{e}ch \cite{GAS}) when the Hilbert space is the space
of $L^2$ random variables. Our approach is to implement the interesting idea of ``lifting,'' first introduced by
P.-L.~Lions \cite{PL2} for the purpose of studying derivatives on the space of measures.
This lifting turns out to be quite powerful in obtaining the Bellman equation
and the master equation of mean field games. Conversely, the classical
approach is to use the Wasserstein metric space of probability measures,
since in mean field games and mean field type control problems, the key aspect
is that the payoff functional involves the probability of states. However, a dynamical
system whose state space is not a vector space leads to challenging
difficulties.

Using the Hilbert space of square integrable random
variables simplifies considerably the mathematical technicalities and makes the problem more transparent. The
difficulty is to keep track of the original problem. Is the new control
problem in a Hilbert space of square-integrable random variables equivalent
to the original one? For instance, one has to check whether the dependence
of the value function with respect to the random variable is only
through its probability measure. This, of course, is not needed,
when one solves the problem directly on the Wasserstein space. 

Another difficulty arises when
there are several sources of randomness, one from the state, and
another from a Wiener process of disturbances. So we cannot simply generalize the ``deterministic
case'' dealt with in our previous paper \cite{ABY}. That is, we cannot just
consider a stochastic control problem for a system whose state space
is an arbitrary Hilbert space. There is an interaction between the
randomness of the elements of the Hilbert space and the additional
randomness from the stochastic dynamics.

We have nevertheless tried to implement
the lifting approach, in a previous version of this work, posted
on ArXiv \cite{BGY}, in order to keep most of the advantages of the
deterministic theory, though the Hilbert space cannot be arbitrary. We
have greatly benefited of some quite useful technical results obtained
by R.~Carmona and F.~Delarue \cite{RCD}. We keep the Wasserstein space
of probability measures, but we do not use the concept of Wasserstein
gradient, which seems difficult to extend to the second derivative. We
use the concept of functional derivative, which extends to the second
order. Using the lifting concept of P.-L.~Lions, it is possible to
work with the Hilbert space of square integrable random variables.
The concept of gradient in the Hilbert space is the G\^ateaux derivative.
At this stage an ``almost'' complete equivalence is possible. The
situation does not carry over so nicely to second order derivatives.
So we have a second order G\^ateaux derivative in the Hilbert space
and a second order functional derivative, but not a full equivalence.
On the other hand, when both exist, we have formulas to transform
one concept into the other. In the ArXiv paper \cite{BGY}, we called these formulas
\textit{rules of correspondence}. The advantage of the Hilbert space approach
is that equations can be written in a more synthetic way. From these
equations and the rules of correspondence, the equations with functional
derivatives can be written, but their rigorous study requires a direct
approach and specific assumptions. This is not fully satisfactory,
because our use of the lifting approach was meant precisely to circumvent the direct study of these equations. 

In the present work, we proceed differently. We develop a framework that 
allows us to work in a Hilbert space and use the concept of G\^ateaux derivatives in this Hilbert space. Then we consider a stochastic
control problem in this Hilbert space containing some relevant $L^2$ random variables. We next find that the optimal state and co-state satisfy a system of forward-backward stochastic differential equations in the Hilbert space with the optimal control satisfying the first order condition. By studying the Jacobian flow of the solution to this FBSDE, we can verify  some regularities of the value function, which help us to write the Bellman equation. Finally, the solution to master equation can be established by taking the linear functional derivatives of the Bellman equation and choosing the initial random variable of the FBSDE to be the identity map. Although the estimates in this article are in similar spirit to those in \cite{GAS}, we note that our approach is different from that in \cite{GAS} which comes much closer to \cite{ABY}. In \cite{GAS}, the analysis focuses purely on analytical techniques inherited from the master equation, but not from the FBSDE which is our main emphasis here. We do not claim originality in the result, but in the method. For
instance, in the book by P.~Cardialaguet et al.~\cite{CDLL}, a fully
analytic approach is taken, consisting of starting with the system
of Hamilton-Jacobi-Bellman and Fokker-Planck equations, and the lifting
concept is not used. Our method is new, even with respect to the lifting
concept, since we consider a different Hilbert space, which could be interpreted as a tangent space attached to the initial probability measure $m$; 
indeed, in this approach, we essentially only ``lift'' the random variables generated solely by the Wiener process while we treat the independent initial randomness separately as an isolated functional object.
The advantage
of the new formulation is that, in this way, we can derive all the results using only optimal control theory on a Hilbert space; the original mean field
control problem appears as a particular case. This work aims to provide a new approach based on the Hilbert space argument to solving the mean field type control problems, and resolving  the corresponding Bellman and master equations. We intend to develop the methodology for more general problems, including those with common noise and generic payoff functions, in future work. We refer the readers to the very recent article \cite{BTY} written by us, which deals with the mean field type control problem with generic payoff functions by applying the approach of this current article.

The rest of this article is organized as follows.
In Section \ref{sec:FORMALISM}, we introduce the Wasserstein space of measures and derivatives of functionals defined on this space.
It is here that we introduce a new idea for ``lifting'' functionals to define them on a Hilbert space, which is the foundation for our arguments.
In Section \ref{sec:MEAN-FIELD-TYPE} we introduce an optimal control problem on the Hilbert space introduced in Section \ref{sec:FORMALISM}, with a detailed discussion of the stochastic dynamics.
In Section \ref{sec:value fxn}, we prove several important properties of the value function, which sets up our study of the Bellman equation in Section \ref{sec:bellman}.
Finally, we give an existence result for the Master Equation in Section \ref{sec:master}. Proofs of technical results from Sections \ref{sec:MEAN-FIELD-TYPE}, \ref{sec:value fxn}, and \ref{sec:bellman} are set aside in Appendices.

\section{\label{sec:FORMALISM}FORMALISM }

\subsection{WASSERSTEIN SPACE}

We denote by $\s{P}_2(\bb{R}^n)$ the Wasserstein space of Borel probability measures $m$ on $\bb{R}^n$ such that $\int_{\bb{R}^n} \abs{x}^2 \dif m(x) < \infty$, which is endowed with the metric
\begin{equation} \label{eq:W2 def}
W_2(\mu,\nu) = \sqrt{\inf \cbr{\int \abs{x-y}^2 \dif \pi(x,y) : \pi \in \Pi(\mu,\nu)}},
\end{equation}
where $\Pi(\mu,\nu)$ is the space of all Borel probability measures on $\bb{R}^n \times \bb{R}^n$ whose first and second marginals are $\mu$ and $\nu$, respectively.
we shall often use an alternative definition of $W_2$, given as follows.
Consider an atomless
probability space $(\Omega,\mathcal{A},\bb{P}),$ and on it the space $L^{2}(\Omega,\mathcal{A},\bb{P};\bb{R}^{n})$
of square integrable random variables with values in $\bb{R}^n.$
For $X \in L^{2}(\Omega,\mathcal{A},\bb{P};\bb{R}^{n})$ we denote by $\s{L}_X$ the law of $X$, given by $\s{L}_X(A) = \bb{P}(X \in A)$.
To
any $m$ in $\mathcal{P}_{2}(\bb{R}^n),$ one can find a random variable
$X_{m}$ in $L^{2}(\Omega,\mathcal{A},\bb{P};\bb{R}^{n})$ such that $\mathcal{L}_{X_{m}}=m$. 
We then have 
\begin{equation}
W_{2}^{2}(m,m')=\inf_{\mathcal{L}_{X_{m}} = m, \ 
\mathcal{L}_{X_{m'}} = m'}\bb{E}[|X_{m}-X_{m'}|^{2}]\label{eq:2-100}
\end{equation}
The infimum is attained, so we can find $\hat{X}_{m}$ and $\hat{X}_{m'}$
in $L^{2}(\Omega,\mathcal{A},\bb{P};\bb{R}^{n})$ such that 
\begin{equation}
W_{2}^{2}(m,m')=\bb{E}[|\hat{X}_{m}-\hat{X}_{m'}|^{2}]\label{eq:101}
\end{equation}
A sequence $\{m_{k}\}$ converges to $m$ in $\mathcal{P}_{2}(\bb{R}^n)$ if
and only if it converges in the sense of weak convergence and 
\begin{equation}
\int_{\bb{R}^{n}}|x|^{2}\dif m_{k}(x)\rightarrow\int_{\bb{R}^{n}}|x|^{2}\dif m(x)\label{eq:1010}
\end{equation}
It is also important to indicate the following compactness result
of the space $\mathcal{P}_{2}(\bb{R}^n)$: 

\begin{equation}
\text{A family \ensuremath{\{m_{k}\}} is relatively compact in \ensuremath{\mathcal{P}_{2}(\bb{R}^n)} if \ensuremath{\sup_{k}\bb{E}[|X_{m_{k}}|^{\beta}]<+\infty}}, \text{for some } \beta>2. \label{eq:2-102}
\end{equation}
We refer to Carmona-Delarue \cite{RCD} for details. 

\subsection{FUNCTIONALS AND THEIR DERIVATIVES}

Consider a functional $F(m)$ on $\mathcal{P}_{2}(\bb{R}^{n}).$ Continuity
is clearly defined by the metric. For the concept of derivative in
$\mathcal{P}_{2}(\bb{R}^{n}),$ we use the concept of functional derivative.
\begin{defn} \label{def:functional derivative}
	We say $F$ is \emph{continuously differentiable}  provided there exists a continuous function $\dod{F}{m}:\mathcal{P}_2(\mathbb{R}^n) \times \bb{R}^n \to \bb{R}$ such that, for some $c:\s{P}_2(\bb{R}^{n}) \to \intco{0,\infty}$ that is bounded on bounded subsets, we have
\begin{equation}
\label{eq:derivative moment}
\abs{\dod{F}{m}(m,x)} \leq c(m)\del{1+\abs{x}^2}
\end{equation}
and
\begin{equation}\label{eq:functional derivative}
\lim_{\epsilon \to 0} \frac{F(m+\epsilon(m'-m)) - F(m)}{\epsilon} = \int \od{F}{m}(m,x)\dif \h{1.5pt} (m'-m)(x)
\end{equation}
for any $m' \in \mathcal{P}_2(\mathbb{R}^n)$.
Since $\dod{F}{m}$ is unique only up to a constant, we require the normalization condition
\begin{equation}
\label{eq:normalization}
\int \od{F}{m}(m,x)\dif m(x) = 0,
\end{equation}
which in particular ensures the functional derivative of a constant is $0.$ 
\end{defn}
Thanks to \eqref{eq:derivative moment},
$x \mapsto \dod{F}{m}(m,x)$ is in $L_{m}^{2}(\bb{R}^{n}) := \{\varphi(\cdot)|\int_{\bb{R}^{n}}|\varphi(x)|^{2}\dif m(x)<\infty\},$
and by a slight abuse of notation we shall denote $\dod{F}{m}(m)(x) := \dod{F}{m}(m,x)$.
Note that the definition
\eqref{eq:functional derivative} implies 
\begin{equation}
\dod{}{\theta}F(m+\theta(m'-m))=\int_{\bb{R}^{n}}\dod{F}{m}(m+\theta(m'-m))(x)\dif \h{1.5pt} (m'-m)(x)\label{eq:2-210}
\end{equation}
 and 
\begin{equation}
F(m')-F(m)=\int_{0}^{1}\int_{\bb{R}^n}\dod{F}{m}(m+\theta(m'-m))(x)\dif \h{1.5pt} (m'-m)(x)\dif \theta\label{eq:2-211}
\end{equation}
We prefer the notation $\dod{F}{m}(m)(x)$ to $\dfrac{\delta F}{\delta m}(m)(x)$
used in R.Carmona-F. Delarue \cite{RCD}, because there is no risk
of confusion and it works pretty much like an ordinary G\^{a}teaux derivative.
We can proceed with the second order functional derivative.
\begin{defn}
	\label{def:2nd functional derivative}
We say $F$ is \emph{twice continuously differentiable} provided there exists a continuous function $\dod[2]{F}{m}:\mathcal{P}_2(\mathbb{R}^n) \times \bb{R}^n \times \bb{R}^n \to \bb{R}$ such that, for some $c:\s{P}_2(\bb{R}^{n}) \to \intco{0,\infty}$ that is bounded on bounded subsets,
\begin{equation}
\label{eq:2nd derivative moment}
\abs{\dod[2]{F}{m}(m,x,\tilde{x})} \leq c(m)\del{1+\abs{x}^2 + \abs{\tilde{x}}^2}
\end{equation}
and
\begin{multline} \label{eq:2nd functional derivative}
\lim_{\epsilon \to 0} \frac{1}{\epsilon}\int \del{\dod{F}{m}(m+\epsilon( m'-m),x) - \dod{F}{m}(m,x)}\dif \h{1.5pt} (m'-m)(x)\\ = \iint \dod[2]{F}{m}(m,x,\tilde{x})\dif \h{1.5pt} (m'-m)(x)\dif \h{1.5pt} ( m'-m)(\tilde{x})
\end{multline}
for any $m', m \in \mathcal{P}_2(\mathbb{R}^n)$.
To ensure $\dod[2]{F}{m}(m,x,\tilde{x})$ is uniquely defined, we shall use the normalization convention
\begin{equation}
\int \dod[2]{F}{m}(m,x,\tilde{x})\dif m(\tilde{x}) = 0 \ \forall x, \quad 
\int \dod[2]{F}{m}(m,x,\tilde{x})\dif m(x) = 0 \ \forall \tilde{x}.
\end{equation}
\end{defn}
Again, we shall write $\dod[2]{F}{m}(m,x,\tilde{x}) = \dod[2]{F}{m}(m)(x,\tilde{x})$, where we note that $\dod[2]{F}{m}(m) \in L^2_{m \times m}(\bb{R}^n \times \bb{R}^n ;\bb{R})$.
We have also 
\begin{equation}
\od[2]{}{\theta}F(m+\theta(m'-m))=\int_{\bb{R}^n}\int_{\bb{R}^n}\od[2]{F}{m}(m+\theta(m'-m))(x,\tilde{x})\dif \h{1.5pt} (m'-m)(x)\dif \h{1.5pt} (m'-m)(\tilde{x})\label{eq:2-213}
\end{equation}
and 
\begin{multline}\label{eq:2-214}
F(m')-F(m)=\int_{\bb{R}^n}\dod{F}{m}(m)(x)\dif \h{1.5pt} (m'-m)(x)\\
+\int_{0}^{1}\int_{0}^{1} \int_{\bb{R}^n}\int_{\bb{R}^n} \theta\dod[2]{F}{m}(m+\lambda\theta(m'-m))(x,\tilde{x})\dif \h{1.5pt} (m'-m)(x)\dif \h{1.5pt} (m'-m)(\tilde{x})\dif \lambda \dif \theta.
\end{multline}
If $F$ is twice continuously differentiable, then standard arguments show that $\dod[2]{F}{m}(m)$ is symmetric, i.e.
$$\dod[2]{F}{m}(m)(x,\tilde{x}) = \dod[2]{F}{m}(m)(\tilde{x},x).$$

\begin{rem}
In general, $\dfrac{d^2 F}{dm^2}(m)(x,\widetilde{x})$ is defined up to a constant in form of $c_1(m,x) + c_2(m,\widetilde{x})$ due to its definition. However, if we take the normalizations $\displaystyle\int_{\mathbb{R}^n} \dfrac{d^2 F}{dm^2}(m)(x,\widetilde{x})dm(x)=0$ for all $\widetilde{x} \in \mathbb{R}^n$ and $\displaystyle\int_{\mathbb{R}^n} \dfrac{d^2 F}{dm^2}(m)(x,\widetilde{x})dm(\widetilde{x})=0$ for all $x \in \mathbb{R}^n$, we shall have the symmetry result $ \dfrac{d^2 F}{dm^2}(m)(x,\widetilde{x})= \dfrac{d^2 F}{dm^2}(m)(\widetilde{x},x)$. We refer to \cite{BGY} for the symmetry result in this situation. Therefore, without these normalizations, we shall have Lemma 2.2.4 in the book \cite{CDLL} instead of the symmetry result.
\end{rem}
We conclude this subsection with a remark on the notation of derivatives with respect to variables in $\bb{R}^n$.
If $\bb{R}^n \ni x \mapsto \dod{F}{m}(m)(x)$ is differentiable, we denote its derivative by $D\dod{F}{m}(m)(x)$.
Further if $\bb{R}^n \times \bb{R}^n \ni (x_1,x_2) \mapsto \dod[2]{F}{m}(m)(x_1,x_2)$ is differentiable, we denote by $D_1  \dod[2]{F}{m}(m)(x_1,x_2)$ and $D_2  \dod[2]{F}{m} (m)(x_1,x_2)$ its derivatives with respect to $x_1$ and $x_2$, respectively.
Higher-order derivatives for $\dod{F}{m}(m)(\cdot)$ will simply denoted $D^k$ for $k =1,2,\ldots$, while for $\dod[2]{F}{m}(m)(\cdot)$ we shall use the partial derivatives $D_1^k$ and $D_2^\ell$.

\subsection{\label{subsec:HILBERT-SPACE}HILBERT SPACE}

Let $(\Omega,\s{A},\bb{P})$ be an atomless probability space.
For $m \in \mathcal{P}_2(\mathbb{R}^n)$, let $\s{H}_m := L^2(\Omega,\s{A},\bb{P};L_m^2(\bb{R}^n;\bb{R}^n))$ where $L_m^2(\bb{R}^n;\bb{R}^n)$ is simply the set of all measurable vector fields $\Phi$ such that $\int \abs{\Phi(x)}^2 \dif m(x) < \infty$.
On $\s{H}_m$ we define the inner product
\begin{equation}
\ip{X}{Y}_{\s{H}_m} = \bb{E} \int X(x) \cdot Y(x) \dif m(x) = \int_\Omega \int_{\bb{R}^n} X(\omega,x)\cdot Y(\omega,x)\dif m(x) \dif \bb{P}(\omega).
\end{equation}
The corresponding norm is given by $\enVert{X}_{\s{H}_m} = \sqrt{\ip{X}{X}_{\s{H}_m}}$.
When it is sufficiently clear which inner product we mean, we shall often drop the subscript $\s{H}_m$.

There is a natural isometric isomorphism between $\s{H}_m$ and $L^2(\Omega \times \bb{R}^n, \s{A} \otimes \s{B},\bb{P} \times m;\bb{R}^n)$, where $\s{B}$ is the Borel $\sigma$-algebra on $\bb{R}^n$.
Note that the product measure $\bb{P} \times m$ is a probability measure on $\Omega \times \bb{R}^n$.
If we consider the law of $X \in \s{H}_m$ viewed as an $L^2$ random variable on this product space, we arrive at the following definition:
\begin{defn} \label{def:X otimes m}
Let $m \in \mathcal{P}_2(\mathbb{R}^n)$, $X \in \s{H}_m$.
	We define $X  \otimes  m \in \mathcal{P}_2(\mathbb{R}^n)$ to be $X \otimes m = X \sharp (\mathbb{P} \times m)$ where $\sharp$ is the push-forward operator. That is, for all continuous functions $\phi:\bb{R}^n \to \bb{R}$ such that $x \mapsto \frac{\abs{\phi(x)}}{1+\abs{x}^2}$ is bounded, we have
	\begin{equation} \label{eq:X otimes m}
	\int_{\bb{R}^n} \phi(x)\dif \h{1.5pt} (X  \otimes  m)(x) = \bb{E} \intcc{\int_{\bb{R}^n} \phi\del{X(x)}\dif m(x)} = \int_\Omega \int_{\bb{R}^n} \phi\del{X(\omega,x)}\dif m(x)\dif \bb{P}(\omega).
	\end{equation}
\end{defn}
\begin{rem}
By definition, $X  \otimes  m$ is the law of $X$ when it is viewed as a random variable on the probability space $(\Omega \times \mathbb{R}^n,\mathcal{A} \otimes \mathcal{B}, \mathbb{P} \times m)$. Thus we may view $X$ as a ``lifting'' of the measure $X  \otimes  m$ to a space of random variables. However, notice that, contrary to the usual ``lifting'', this interpretation is completely independent of the measure $m$ itself; that is, $X$ is a lifting of the measure $X  \otimes  m$, provided $X \in \mathcal{H}_m$, for
any $m \in \mathcal{P}_2(\mathbb{R}^n)$ whatsoever. This is because the probability space on which $X$ is taken to be a random variable
depends on $m$. In particular, if $X \in \cap_{\mu \in \mathcal{P}_2(\mathbb{R}^n)} \mathcal{H}_\mu$, it follows that $X$ and $m$ in the expression $X  \otimes  m$ are
independent, which explains the notation of  $X  \otimes  m$ as a tensor product.
\end{rem}
Let us notice that if $X$ is deterministic, i.e.~not $\omega$-dependent, then $X  \otimes  m (E) := m(X^{-1}(E))$ is the push-forward of $m$ through $X$.
In fact, the following lemma helps show how much $ \otimes $ behaves like the push-forward operator:
\begin{lem}
	Let $X\in \s{H}_m$,  and a collection of $Z_1,\, Z_2,\ldots,Z_i,\ldots\in L^2(\Omega,\s{A},\bb{P};\bb{R}^n)$ such that $X$ and all these $Z_i$'s are independent random elements to each other. Also let $f:\mathbb{R}^n\times \mathbb{R}^n\times \cdot\cdot\cdot \times \mathbb{R}^n\times \cdot\cdot\cdot \rightarrow\mathbb{R}^n$ be a bounded measurable function such that $Y:(\omega,y)\mapsto f(y,Z_1(\omega),...,Z_i(\omega),...)\in\s{H}_{X  \otimes  m}$. Now consider $(Y\circ X)(\omega,x)=Y(y,\omega)\big|_{y=X(\omega,x)}=f(X(\omega,x),Z_1(\omega),...,Z_i(\omega),...)$ which clearly belongs to  $\s{H}_m$ and then $(Y \circ X)  \otimes  m = Y  \otimes  (X  \otimes  m)$.
\end{lem}

\begin{proof}
	Let $\phi:\bb{R}^n \to \bb{R}$ be a continuous functions such that $x \mapsto \frac{\abs{\phi(x)}}{1+\abs{x}^2}$ is bounded and $\s{Z}(\omega):=(Z_1(\omega),...,Z_i(\omega),...)$, then we have
	\begin{align}\nonumber
	\int_{\mathbb{R}^n} \phi(z) \dif\del{(Y \circ X)  \otimes  m(z)}
	=& \mathbb{E}\left[\int_{\mathbb{R}^n} \phi\del{Y\left( X(\omega,x),\s{Z}(\omega)\right)} \dif m(x)\right]\\\nonumber
	=&\mathbb{E}\left[ \int_{\mathbb{R}^n} \phi\del{Y\del{y,\s{Z}(\omega)}}\dif \h{1pt}(X  \otimes  m)(y)\right]\\
	=& \int_{\mathbb{R}^n} \phi(z) \dif\del{Y  \otimes  (X  \otimes  m)(z)}.
	\end{align}
\end{proof}
Using properties of the push-forward, we exhibit two useful examples:
\begin{example} \label{ex:two push-forwards}
	(i) If $X(x) = x$ is the identity map, then $X  \otimes  m = m$.
	
	(ii) If $X(x) = a$ is a constant map, then $X  \otimes  m = \delta_a$, the Dirac delta mass concentrated at $a$.
\end{example}

As before, let $F:\mathcal{P}_2(\mathbb{R}^n) \to \bb{R}$.
For every $m \in \mathcal{P}_2(\mathbb{R}^n)$, the map $X \mapsto F(X  \otimes  m)$ is a functional on $\s{H}_m$.
By an abuse of notation, we shall now think of $F(X  \otimes  m)$ as a function of two variables, taking care to remember that $X$ is always attached to $m$ in the sense that $X \in \s{H}_m$.

\begin{lem}
	The map $X \mapsto X  \otimes  m$ is 1-Lipschitz from $\s{H}_m$ to $\mathcal{P}_2(\mathbb{R}^n)$.
	Thus if $F:\mathcal{P}_2(\mathbb{R}^n) \to \bb{R}$ is continuous, then for a fixed $m \in \mathcal{P}_2(\mathbb{R}^n)$, the map $\s{H}_m \to \bb{R}$ given by $X \mapsto F(X  \otimes  m)$ is also continuous.
\end{lem}

\begin{proof}
	Let $X,Y \in \s{H}_m$.
	Let $\pi \in \Pi(X  \otimes  m,Y  \otimes  m)$ be given by
	\begin{equation}
	\int \phi(x,y)\dif \pi(x,y) = \bb{E}\int \phi\del{X(x),Y(x)}\dif m(x).
	\end{equation}
	Since the right-hand side is a bounded, non-negative linear functional on the space of continuous functions, by the Riesz representation theorem this defines a unique measure on $\bb{R}^n \times \bb{R}^n$.
	Moreover, the first marginal of $\pi$ is indeed $X  \otimes  m$ because
	\begin{equation}
	\int \phi(x)\dif \pi(x,y) = \bb{E}\int \phi\del{X(x)}\dif m(x)
	= \int \phi(x)\dif \h{1.5pt} (X  \otimes  m)(x)
	\end{equation}
	and likewise the second marginal of $\pi$ is $Y  \otimes  m$ because
	\begin{equation}
	\int \phi(y)\dif \pi(x,y) = \bb{E}\int \phi\del{Y(x)}\dif m(x)
	= \int \phi(x)\dif \h{1.5pt} (Y  \otimes  m)(x).
	\end{equation}
	It follows that
	\begin{equation}
	W_2^2(X  \otimes  m,Y  \otimes  m) \leq \int \abs{x-y}^2\dif \pi(x,y) = \bb{E}\int \abs{X(x)-Y(x)}^2\dif m(x) = \enVert{X-Y}_{\s{H}_m}^2.
	\end{equation}
\end{proof}

\begin{defn} \label{def:partial X}
	Let $F:\mathcal{P}_2(\mathbb{R}^n) \to \bb{R}$, and consider its extension $F(X  \otimes  m)$.
	For any $m \in \mathcal{P}_2(\mathbb{R}^n)$, we define the ``partial derivative'' of $F$ with respect to $X \in \mathcal{H}_m$ as the unique element $D_X F(X  \otimes  m)$ of $\s{H}_m$, if it exists, such that
	\begin{equation}
	\lim_{\epsilon \to 0} \frac{F\del{(X+\epsilon Y)  \otimes  m} - F(X  \otimes  m)}{\epsilon} = \ip{D_X F(X  \otimes  m)}{Y} \ \forall Y \in \s{H}_m.
	\end{equation}
\end{defn}
Note that the ``partial derivative'' in Definition \ref{def:partial X} is the usual G\^ateaux derivative in a Hilbert space.
It is not a true partial derivative in the sense that it is not independent of the variable $m$; in this sense it is more akin to a tangent vector, where $\s{H}_m$ is likened to a tangent space. 
The following proposition characterizes $D_X F$ in a more elementary way. 
\begin{prop} \label{prop:D_XF}
	Let $F:\mathcal{P}_2(\mathbb{R}^n) \to \bb{R}$ be continuously differentiable and assume $x \mapsto \od{F}{m}(m,x)$ is continuously differentiable in $\bb{R}^n$.
	Assume that its derivative $D \od{F}{m}(m)(x)$ is continuous in both $m$ and $x$ with
	\begin{equation} \label{eq:derivative growth}
	\abs{D \dod{F}{m}(m)(x)} \leq c(m)\del{1+\abs{x}}
	\end{equation}
	for some constant $c(m)$ depending only on $m$.
	Then
	\begin{equation} \label{eq:partial X derivative identity}
	D_XF(X  \otimes  m) = D \od{F}{m}(X  \otimes  m)(X(\cdot)).
	\end{equation}
\end{prop}

Before we move on to its proof, it is worth clarifying a few concepts as follows. First, one can compare this result with Proposition 2.2.3 in \cite{CDLL}; just take $X(x)=x$. Second, we want to make a subtle point about the smoothness of $F(X \otimes m)$.
Suppose $F:\mathcal{P}_2(\mathbb{R}^n) \to \bb{R}$ is continuously differentiable.
Nevertheless, $m\in\mathcal{P}_2(\mathbb{R}^n)\mapsto X  \otimes  m\in\mathcal{P}_2(\mathbb{R}^n)$ is in general not continuous in the Wasserstein space, even if $X\in\cap_{m\in\mathcal{P}_2(\mathbb{R}^n)}\s{H}_m$.
Thus one cannot regard $m\in\mathcal{P}_2(\mathbb{R}^n)\mapsto F(X  \otimes  m)\in\bb{R}$ as a smooth composite function; it may not even be continuous! With these remarks in mind, the new idea in our present article is that one can consider $F(X  \otimes  m)$ as a function $F(X,m)$ that has two separate variables $X$ and $m$, although we do not require $X$ and $m$ to be completely independent. We consider separately its partial derivatives with respect to $X$ and $m$, which are denoted by $D_XF(X  \otimes  m)$ and $\pd{F}{m}(X  \otimes  m)(x)$, respectively. In Definition \ref{def:partial X}, we consider $F(X  \otimes  m)$ as a function $F^m:X\in \s{H}_m\mapsto F(X  \otimes  m)\in \bb{R}$ induced by a fixed $m\in\mathcal{P}_2(\mathbb{R}^n)$, and we take G\^ateaux derivative with respect to $X$ in $\s{H}_m$. Our Proposition \ref{prop:D_XF} shows that, when $F$ is sufficiently smooth, this G\^ateaux derivative in $\s{H}_m$ is indeed the $L$-derivative of $F:\mathcal{P}_2(\mathbb{R}^n) \to \bb{R}$ composed with $X\mapsto X  \otimes  m$.
To see this, we need not refer to the ``lifted'' function $\tilde F(X) = F(\s{L}_X)$, as is the usual approach.
Instead, the measure $X \otimes m$ encodes both measure dependence and, in a partial derivative sense, Hilbert space dependence. In the later Definition \ref{def:partial functional derivative}, we also consider $F(X  \otimes  m)$ as a function $F^X:m\in \mathcal{P}_2(\mathbb{R}^n)\mapsto F(X  \otimes  m)\in \bb{R}$ induced by a fixed $X\in\cap_{m \in \mathcal{P}_2(\mathbb{R}^n)} \s{H}_m$, and we take linear functional derivative as defined in Definition \ref{def:functional derivative} with respect to $m$ in $\mathcal{P}_2(\mathbb{R}^n)$.
Thus with a single formalism, $F(X \otimes m)$, the $L$-derivative and linear functional derivative both appear naturally as ``partial derivatives''.

\begin{proof}[Proof of Proposition \ref{prop:D_XF}]
	Note that, by \eqref{eq:derivative growth}, $D \od{F}{m}(X  \otimes  m,X(\cdot)) \in \s{H}_m$ for any $X \in \s{H}_m$.
	Let $Y \in \s{H}_m$ be arbitrary.
	For $\epsilon \neq 0$, let $\mu = (X+\epsilon Y)  \otimes  m, \nu = X  \otimes  m$, and for $t \in [0,1]$ set $\nu_t = \nu + t(\mu-\nu)$.
	Then we have
	\begin{equation}
	\begin{aligned}
	\frac{1}{\epsilon}\del{F\del{(X+\epsilon Y)  \otimes  m} - F(X  \otimes  m)}
	&= \frac{1}{\epsilon}\int_0^1 \int_{\bb{R}^n}\dod{F}{m}(\nu_t,x)\dif \h{1.5pt} (\mu-\nu)(x)\dif t\\
	&= \frac{1}{\epsilon}\bb{E}\int_0^1 \int_{\bb{R}^n}\del{\dod{F}{m}\del{\nu_t,X(x)+\epsilon Y(x)}-\dod{F}{m}\del{\nu_t,X(x)}}\dif m(x)\dif t\\
	&\to \bb{E} \int_{\bb{R}^n}D\dod{F}{m}\del{X  \otimes  m,X(x)} \cdot Y(x)\dif m(x)\\
	&= \ip{D\dod{F}{m}\del{X  \otimes  m,X(\cdot)}}{Y}_{\s{H}_m}
	\end{aligned}
	\end{equation}
	using the continuity of $D \od{F}{m}$.
\end{proof}
A special case of \eqref{eq:partial X derivative identity} is when $X$ is the identity, i.e.~$X(x) = x$.
In this case $X  \otimes  m = m$ (Example \ref{ex:two push-forwards}), and thus \eqref{eq:partial X derivative identity} implies
\begin{equation}
D_X F(m) = D \od{F}{m}(m)(X(\cdot)).
\end{equation}
This is precisely the $L$-derivative, cf.~\cite{RCD,CDLL}.

We now want to consider $F(X  \otimes  m)$ as $m$ varies but $X$ is fixed.
To do this we should restrict $X$ so that $X \in \cap_{m \in \mathcal{P}_2(\mathbb{R}^n)} \s{H}_m$.
The following observation is useful.
\begin{lem} \label{lem:intersect Hm}
	Let $X:\Omega \times \bb{R}^n \to \bb{R}^n$ be a $(\s{A} \otimes \s{B},\s{B})$ measurable vector field (where $\s{B}$ is the Borel $\sigma$-algebra on $\bb{R}^n$) such that
	\begin{equation} \label{eq:X square growth}
	\bb{E}\abs{X(x)}^2 \leq c(X)\del{1+\abs{x}^2} \ \forall x \in \bb{R}^n,
	\end{equation}
	where $c(X)$ is a constant depending only on $X$.
	Then $X \in \cap_{m \in \mathcal{P}_2(\mathbb{R}^n)} \s{H}_m$.
\end{lem}

\begin{proof}
	Observe that if $m \in \mathcal{P}_2(\mathbb{R}^n)$ then
	\begin{equation}
	\enVert{X}_{\s{H}_m}^2 = \bb{E}\int \abs{X(x)}^2 \dif m(x) = \int \bb{E}\abs{X(x)}^2 \dif m(x)
	\leq \int c(X)\del{1+\abs{x}^2} \dif m(x) < \infty,
	\end{equation}
	and thus $X \in \s{H}_m$ for arbitrary $m \in \mathcal{P}_2(\mathbb{R}^n)$.
\end{proof}

\begin{defn} \label{def:partial functional derivative}
	Let $F:\mathcal{P}_2(\mathbb{R}^n) \to \bb{R}$ and let $X \in \cap_{m \in \mathcal{P}_2(\mathbb{R}^n)} \s{H}_m$.
	We define the partial derivative of $F(X  \otimes  m)$ with respect to $m$, denoted $\pd{F}{m}(X  \otimes  m)(x)$, to be the derivative of $m \mapsto F(X  \otimes  m)$ in the sense of Definition \ref{def:functional derivative}.
\end{defn}

\begin{prop} \label{prop:partial F m}
	Let $F:\mathcal{P}_2(\mathbb{R}^n) \to \bb{R}$ be continuously differentiable and let $X \in \cap_{m \in \mathcal{P}_2(\mathbb{R}^n)} \s{H}_m$.
	Then
	\begin{equation} \label{eq:partial functional der id}
	\pd{F}{m}(X  \otimes  m)(x) = \bb{E}\od{F}{m}(X  \otimes  m)(X(x)).
	\end{equation}
\end{prop}

\begin{proof}
	For $\epsilon \neq 0$ let $\mu = X  \otimes  \del{m + \epsilon(m'-m)}, \nu = X  \otimes  m$, and for $t \in [0,1]$ set $\nu_t = \nu + t(\mu-\nu)$.
	We have, as $\epsilon \to 0$,
	\begin{equation}
	\begin{aligned}
	\frac{1}{\epsilon}\del{F\del{X  \otimes  \del{m + \epsilon(m'-m)}} - F\del{X  \otimes  m}}
	&= \frac{1}{\epsilon}\int_0^1 \int_{{\mathbb R}^n} \od{F}{m}(\nu_t,x)\dif \h{1.5pt} (\mu-\nu)(x)\\
	&= \bb{E}\int_0^1 \int_{{\mathbb R}^n} \dod{F}{m}\del{\nu_t,X(x)}\dif \h{1.5pt} (m'-m)(x)\\
	&\to \bb{E} \int_{{\mathbb R}^n} \dod{F}{m}\del{X  \otimes  m,X(x)}\dif \h{1.5pt} (m'-m)(x),
	\end{aligned}
	\end{equation}
	using the continuity of $\od{F}{m}$.
	The claim follows.
\end{proof}
We conclude this subsection with a remark on notation.
For a functional $F:\mathcal{P}_2(\mathbb{R}^n) \to \bb{R}$, whenever the symbols $D_X F$ and $\pd{F}{m}$ are used, they should be thought of as partial derivatives of $F(X  \otimes  m)$, which can be evaluated at any elements $(X,m)$ such that $m \in \mathcal{P}_2(\mathbb{R}^n)$ and $X \in \cap_{\mu \in \mathcal{P}_2(\mathbb{R}^n)} \s{H}_\mu$.
In particular, $\pd{F}{m}$ should not be confused with $\od{F}{m}$; the relation between the two is clarified by Proposition \ref{prop:partial F m}.
Moreover, standard usage of partial derivative notation applies to compositions of functions.
Thus if $X \mapsto Y_X$ is a map $\s{H}_m \to \s{H}_m$, then the symbol $D_X F(Y_X  \otimes  m)$ should be interpreted as $D_X F$ evaluated at the point $Y_X  \otimes  m$, rather than as the derivative of the composite function $X \mapsto F(Y_X  \otimes  m)$.
When it is necessary to differentiate a composite function, we shall clearly state that we are doing so.

\subsection{MORE ELABORATE DERIVATIVES, CHAIN RULE}
\label{sec:elaborate}

Consider a random vector field $X = X(m,x)$,
with the equivalent of \eqref{eq:X square growth}
\begin{equation} \label{eq:X(m) square growth}
\bb{E}\intcc{|X(m,x)|^{2}}\leq c(X)(1+|x|^{2})
\end{equation}
where $c(X)$ is a constant, not depending on $m,x,$ but only on $X$.
We define the functional derivative of $X$ with respect to $m$ in a way analogous to Definition \ref{def:functional derivative}.
Namely, we say $X$ is continuously differentiable with respect to $m$ if there exists a
random field $\pd{X}{m}:\:\mathcal{P}_{2}(\bb{R}^n)\times \bb{R}^n\times \bb{R}^n$$\rightarrow L^{2}(\Omega,\mathcal{A},\bb{P};\bb{R}^n)$, which is
continuous in all variables, such that 
\begin{equation}
\bb{E}\intcc{\abs{\dpd{X}{m}(m,x,\tilde{x})}^{2}}\leq c(X,m)(1+|x|^{2}+|\tilde{x}|^{2})\label{eq:2-222}
\end{equation}
and
\begin{equation}
\enVert{\dfrac{X(m+\theta(m'-m),\cdot)-X(m,\cdot)}{\theta}-\int_{\bb{R}^n}\dpd{X}{m}(m,\cdot,\tilde{x})\dif \h{1.5pt} (m'-m)(\tilde{x})}_{L^2(\Omega,\s{A},\bb{P};\bb{R}^n)}\rightarrow 0,\:\text{as }\theta\rightarrow0\label{eq:2-223}
\end{equation}
Following our usual convention, we can write $\pd{X}{m}(m,x,\tilde{x}) = \pd{X}{m}(m,x)(\tilde{x})$, where $\pd{X}{m}(m,x)$ is viewed as an element of
$L^2_m\del{L^2(\Omega,\s{A},\bb{P};\bb{R}^n)}$.

We can envisage the functional derivative of $m\mapsto F(X(m,.) \otimes  m).$
The calculation works as for ordinary derivatives, the tensor product acting
as an ordinary product.
 When both sides of the equality exist, we can then write: 
\begin{equation}
\od{}{m}\del{F(X(m,.) \otimes  m)}(x)=\dfrac{\partial F}{\partial m}(X(m,\cdot) \otimes  m))(x)+\ip{D_{X}F(X(m,\cdot) \otimes  m)}{\dpd{X}{m}(m,.)(x)}_{\s{H}_m}\label{eq:2-224}
\end{equation}
which we can make explicit as follows
\begin{multline}
\dod{}{m}\del{F(X(m,.) \otimes  m)}(x)=\bb{E}\dod{F}{m}(X(m,.) \otimes  m)(X(m,x))\\
+\bb{E}\int_{\bb{R}^n}D\dod{F}{m}(X(m,.) \otimes  m)(X(m,\xi))\cdot \dpd{X}{m}(m,\xi)(x)\dif m(\xi)
\end{multline}
We have, by \eqref{eq:derivative moment}, \eqref{eq:derivative growth}, \eqref{eq:X(m) square growth} and \eqref{eq:2-222},
\begin{equation}
\abs{\dod{}{m}F(X(m,.) \otimes  m)(x)}\leq c(X,m)(1+|x|^{2}).\label{eq:2-226}
\end{equation}

\subsection{SECOND ORDER G\^ATEAUX DERIVATIVE IN THE HILBERT SPACE}

The functional $F(X \otimes  m)$ has a second order G\^{a}teaux derivative
in $\mathcal{H}_{m}$, denoted $D_{X}^{2}F(X \otimes  m)\in\mathcal{L}(\mathcal{H}_{m};\mathcal{H}_{m})$,
if 
\begin{equation}
\dfrac{\ip{D_{X}F((X+\epsilon Y) \otimes  m)-D_{X}F(X \otimes  m)}{Y}_{\s{H}_m}}{\epsilon}\rightarrow \ip{D_{X}^{2}F(X \otimes  m)(Y)}{Y}_{\s{H}_m}\:\text{as}\; \epsilon\rightarrow0,\forall Y\in\mathcal{H}_{m},\label{eq:2-227}
\end{equation}
and we can define $D_{X}^{2}F(X \otimes  m)(Z)$ using the parallelogram law:
\begin{multline}
\ip{D_{X}^{2}F(X \otimes  m)(Z)}{W}_{\s{H}_m}
\\=\cfrac{1}{4}\left(\ip{D_{X}^{2}F(X \otimes  m)(Z+W)}{Z+W}_{\s{H}_m}-\ip{D_{X}^{2}F(X \otimes  m)(Z-W)}{Z-W}_{\s{H}_m}\right).\label{eq:2-1}
\end{multline}
Note that $D_{X}^{2}F(X \otimes  m)$ is self-adjoint. It is convenient
to define the symmetric bilinear form on $\mathcal{H}_{m}$, which, by a slight abuse of notation, we may denote
\begin{equation}
D_{X}^{2}F(X \otimes  m)(Z,W)=\ip{D_{X}^{2}F(X \otimes  m)(Z)}{W}.\label{eq:2-2}
\end{equation}
Observe that 
\begin{equation}
\od{}{\theta}\ip{D_{X}F((X+\theta Z) \otimes  m)}{W}=\ip{D_{X}^{2}F((X+\theta Z) \otimes  m)(Z)}{W}.\label{eq:2-228}
\end{equation}
Since 
\[
\od{}{\theta}F((X+\theta Y) \otimes  m)= \ip{D_{X}F((X+\theta Y) \otimes  m)}{Y}
\]
we have also 
\begin{equation}
\od[2]{}{\theta}F((X+\theta Y) \otimes  m)=\ip{D_{X}^{2}F((X+\theta Y) \otimes  m)(Y)}{Y}\label{eq:2-229}
\end{equation}
Hence 
\begin{equation}
F((X+Y) \otimes  m))=F(X \otimes  m)+\ip{D_{X}F(X)}{Y}+\int_{0}^{1}\int_{0}^{1}\theta\ip{D_{X}^{2}F((X+\theta\lambda Y) \otimes  m)(Y)}{Y}\dif \theta \dif \lambda\label{eq:2-230}
\end{equation}
From Proposition \ref{prop:partial F m} we have $\ip{D_{X}F(X \otimes  m)}{W}= \bb{E}\int_{\bb{R}^n}D\od{F}{m}(X \otimes  m)(X(x))\cdot W(x)\dif m(x)$.
It follows that
\begin{multline*}
\dfrac{\ip{D_{X}F((X+\epsilon Z) \otimes  m)}{W}-\ip{D_{X}F(X \otimes  m)}{W}}{\epsilon}
\\
= \dfrac{\bb{E}\int_{\bb{R}^n}\left(D\od{F}{m}((X+\epsilon Z) \otimes  m)(X(x)+\epsilon Z(x))-D\od{F}{m}(X \otimes  m)(X(x))\right)\cdot W(x)\dif m(x)}{\epsilon},
\end{multline*}
so, assuming the existence and continuity of $\dod[2]{F}{m}(m)(x,\tilde{x})$
and its derivatives $D_1D_2\dod[2]{F}{m}(m)(x,\tilde{x}),$ as well as the existence and continuity of $D^{2}\dod{F}{m}(m)(x)$, we deduce that the above limit is 
\begin{multline}
\ip{D_{X}^{2}F(X \otimes  m)(Z)}{W}=\bb{E}\int_{\bb{R}^n}D^{2}\dod{F}{m}(X \otimes  m)(X(x))Z(x)\cdot W(x)\dif m(x)\\ \label{eq:2-231}
+\bb{E}\tilde{\bb{E}}\int_{\bb{R}^n}\int_{\bb{R}^n}D_1D_2\dod[2]{F}{m}(X \otimes  m)(\tilde{X}(\tilde{x}),X(x))\tilde{Z(}\tilde{x})\cdot W(x)\dif m(\tilde{x})\dif m(x),
\end{multline}
 in which $\tilde{X}(\tilde{x}),\tilde{Z(}\tilde{x})$ are independent
copies of $X(x),Z(x).$ 
Consequently, we can write
\begin{equation}
D_{X}^{2}F(X \otimes  m)(Z)(x)=D^{2}\od{F}{m}(X \otimes  m)(X(x))Z(x)+\tilde{\bb{E}}\int_{\bb{R}^n}D_1D_2\od[2]{F}{m}(X \otimes  m)(\tilde{X}(\tilde{x}),X(x))\tilde{Z(}\tilde{x})\dif m(\tilde{x})\label{eq:2-232}
\end{equation}
in which the expectation $\tilde{\bb{E}}$ is independent of $X(x).$
The above equation shows the link between the second order G\^ateaux derivative $D^2_X$ and the partial order Lions derivatives $D$ (in our context, while they use $\partial_x$ in \cite{CDLL,RCD}), $D\frac{d}{dm}$ (in our context, or $\partial_m$ in \cite{CDLL,RCD} being expressed in Wasserstein space of probability measures).
If we take $X(x)=x,$ we obtain (recall Example \ref{ex:two push-forwards})
\begin{equation}
D_{X}^{2}F(m)(Z)(x)=D^{2}\dod{F}{m}(m)(x)Z(x)+\int_{\bb{R}^n}D_2D_1\od[2]{F}{m}(m)(\tilde{x},x)\tilde{\bb{E}}\tilde{Z(}\tilde{x})\dif m(\tilde{x})\label{eq:2-233}
\end{equation}
It follows immediately that if $Z$ is independent of $X$ and $\bb{E}[Z]=0,$ then 
\begin{equation}
D_{X}^{2}F(X \otimes  m)(Z)(x)=D^{2}\od{F}{m}(X \otimes  m)(X(x))Z(x)\label{eq:2-234}
\end{equation}
In order to get $D_{X}^{2}F(X \otimes  m)\in \mathcal{L}(\mathcal{H}_{m};\mathcal{H}_{m})$, it will suffice
to assume 
\begin{equation}
\abs{D^2\od{F}{m}(m)(x)} \leq c(m),\: \abs{D_2D_1\od[2]{F}{m}(m)(\tilde{x},x)}\leq c(m)\label{eq:2-235}
\end{equation}
where $|\cdot|$ in (\ref{eq:2-235}) is the matrix norm. 

\section{\label{sec:MEAN-FIELD-TYPE} CONTROL PROBLEM WITH STATE VARIABLE
IN $\mathcal{H}_{m}$}

For this entire section we shall fix a measure $m \in \s{P}_2(\bb{R}^n)$. We shall define an optimal control problem on the Hilbert space $\s{H}_m$ attached to $m$.
Then we shall discuss its solution.
In the following section, we shall study properties of the value function.

\subsection{PRELIMINARIES} \label{sec:ctrl prelim}

Let $(\Omega,\s{A},\bb{P})$ be a probability space sufficiently large to contain a standard Wiener process in $\bb{R}^n$, denoted $w(t)$, with filtration $\s{W}_t = \{\s{W}_t^s\}_{s \geq t}$ where $\s{W}_t^s = \sigma\del{(w(\tau)-w(t)): t \leq \tau \leq s}$.
We also assume $(\Omega,\s{A},\bb{P})$ is rich enough to support random variables that are independent of the entire Wiener process.
For example, we could take $(\Omega,\s{A},\bb{P}) = (\Omega_0 \times \Omega_1,\s{A}_0 \otimes \s{A}_1,\bb{P}_0 \times \bb{P}_1)$ where $\s{W}_t^s \subset \s{A}_0$ and $(\Omega_1,\s{A}_1,\bb{P}_1)$ is itself a sufficiently rich probability space.

For a given $t \geq 0$, we denote by $\s{H}_{m,t}$ the space of all $X = X_t \in \s{H}_m$ such that $X$ is independent of $\s{W}_t$ (say, $\sigma(X) \subset \s{A}_1$).
For $X \in \s{H}_{m,t}$ we define $\sigma$-algebras $\s{W}_{Xt}^s = \sigma(X) \vee \s{W}_t^s$, and the filtration generated by these will be denoted $\s{W}_{Xt}$.

For the remainder of this section we shall fix $t \geq 0$ and $X \in \s{H}_{m,t}$.
For the control problem stated below in Section \ref{sec:ctrl statement}, the space of controls will be $L^2_{\s{W}_{Xt}}(t,T;\s{H}_m)$, the set of all processes in $L^2(t,T;\s{H}_m)$ that are adapted to $\s{W}_{Xt}$.
Because $X$ is independent of $\s{W}_t$, we have an important observation:
\begin{lem} \label{lem:isometry}
	Let $X \in \s{H}_{m,t}$.
	Then there exists a natural linear isometry between $L^2_{\s{W}_{Xt}}(t,T;\s{H}_m)$ and $L^2_{\s{W}_{t}}(t,T;\s{H}_{X  \otimes  m})$, obtained by inserting the random variable $X$ in place of the argument $x$ in the vector field $v$.
\end{lem}

The proof can be found in Appendix \ref{ap:A}. Now for a given $v \in L^2_{\s{W}_{Xt}}(t,T;\s{H}_m)$, which we also denote $v_{Xt}$ to emphasis the measurability constraint, we consider the SDE
\begin{equation} \label{eq:SDE}
X(s) = X + \int_t^s v(\tau)\dif \tau + \eta(w(s)-w(t)),
\end{equation}
where $\eta$ is a fixed, deterministic $n \times n$ matrix, which is symmetric and positive definite.
Equation \eqref{eq:SDE} defines a process $X(\cdot) \in L^2_{\s{W}_{Xt}}(t,T;\s{H}_m)$, which we shall denote $X_{Xt}(s) = X_{Xt}(s;v_{Xt}(\cdot))$.
Indeed,
\begin{equation} \label{eq:EX(s)^2 bound}
\begin{aligned}
\bb{E}\enVert{X(s)}_{\s{H}_m}^2 &\leq 3\enVert{X}_{\s{H}_m}^2 + 3(s-t)\bb{E}\int_t^s \enVert{v(\tau)}_{\s{H}_m}^2 \dif \tau
+ 3\abs{\eta}^2\bb{E}\abs{w(s)-w(t)}^2\\
&\leq 3\enVert{X}_{\s{H}_m}^2 + 3(s-t)\enVert{v}_{L^2_{\s{W}_{Xt}}(t,T;\s{H}_m)}^2 + 3\abs{\eta}^2(s-t),
\end{aligned}
\end{equation}
where $\abs{\eta}$ denotes the matrix norm of $\eta$.
\begin{rem} \label{rem:X(s) in H_m,s}
	Let $X(\cdot)$ be the solution of \eqref{eq:SDE}.
	It is critical to observe that $X(s) \in \s{H}_{m,\tau}$ whenever $\tau \geq s \geq t$, i.e.~$X(s)$ is independent of $\s{W}_{\tau}$.
	To see this, notice that $X(s)$ is $\s{W}_{Xt}^s$-measurable, where $\s{W}_{Xt}^s = \sigma(X) \vee \s{W}_t^s$ by definition.
	As $\s{W}_t^s$ is independent of $\s{W}_\tau$ by independent increments and $\sigma(X)$ is independent of $\s{W}_t$ by assumption, we conclude that $X(s)$ is indeed independent of $\s{W}_{\tau}$.
\end{rem}

We can also interpret \eqref{eq:SDE} as a finite-dimensional SDE.
Let $\tilde v \in L^2_{\s{W}_t}(t,;\s{H}_{X  \otimes  m})$ be the representative of $v$ given by Lemma \ref{lem:isometry}.
For $m$-a.e.~$x$, consider
\begin{equation}
\label{eq:ordinary SDE}
x_t(s) = x + \int_t^s \tilde v(\tau,x)\dif \tau + \eta(w(s)-w(t)).
\end{equation}
This defines a unique solution $x(\cdot) \in L^2_{\s{W}_t}(t,T;\bb{R}^n)$, which we denote $x(s;x,\tilde v(\cdot,x))$.
By viewing each term in \eqref{eq:SDE} as an element in $L^2_m(\bb{R}^n;\bb{R}^n)$ and evaluating at $x$, we have the relation
\begin{equation}
X_{Xt}\left(s;v_{Xt}(\cdot)\right)(x) = x\left(s;X(x),\tilde v\left(\cdot,X(x)\right)\right),
\end{equation}
for $m$-a.e.~$x$.
More precisely, using the decomposition $\Omega = \Omega_0 \times \Omega_1$ as above, we can write for $\omega = (\omega_0,\omega_1)$
\begin{equation}
X_{Xt}\left(\omega,s;v_{Xt}(\cdot)\right)(x) = x\left(\omega_0,s;X(\omega_1,x),\tilde v\left(\cdot,X(\omega_1,x)\right)\right)
\end{equation}
for $m$-a.e.~$x$.
\begin{lem} \label{lem:law of X_Xt}
	The law of $X_{Xt}(s;v_{Xt}(\cdot))$, considered as a random variable on the product space $\Omega \times m$, is $x(s;\cdot,\tilde v(\cdot,\cdot))  \otimes  (X  \otimes  m)$.
\end{lem}
The proof can be found in Appendix \ref{ap:A}. As a result of Lemma \ref{lem:law of X_Xt}, we write
\begin{equation} \label{eq:X_Xt otimes m}
X_{Xt}(s;v_{Xt}(\cdot))  \otimes  m := x(s;\cdot,\tilde v(\cdot,\cdot))  \otimes  (X  \otimes  m).
\end{equation}

To conclude this subsection, we introduce the following conventions.
\begin{itemize}
	\item The symbol $v_{\xi t}(\cdot)$ (or possibly $v_{xt}(\cdot)$), with a lower-case letter as its first subscript, will actually refer to the vector field $\tilde v \in L^2_{\s{W}_t}(t,T;\s{H}_{X  \otimes  m})$.
	Meanwhile $v$ itself will always be denoted by inserting the argument $X$, i.e.~$v_{Xt}(\cdot)$.
	This accords with Lemma \ref{lem:isometry}.
	Note that $v_{Xt}(\cdot)$ here refers to an element of $L^2_{\s{W}_{Xt}}(t,T;\s{H}_{m})$, and we may write $v_{Xt}(\cdot)$ to emphasize this point.
	\item In a similar spirit, the symbol $X_{\xi t}(\cdot)$ (or any other lower-case letter in place of $\xi$) will actually refer to $x(\cdot;\xi,\tilde v(\cdot,\xi))$, i.e.~an element of $L^2_{\s{W}_t}(t,T;\bb{R}^n)$.
	If we plug in the random variable $X$, we recover $X_{Xt}(\cdot)$, the trajectory in $L^2_{\s{W}_{Xt}}(t,T;\s{H}_m)$ driven by the control $v_{Xt}(\cdot)$.
	Thus \eqref{eq:X_Xt otimes m} becomes
	\begin{equation} \label{eq:X_Xt otimes m v2}
	X_{Xt}(s)  \otimes  m \del{= X_{Xt}(s;v_{Xt}(\cdot))  \otimes  m} = X_{\cdot t}(s;v_{\cdot}(\cdot))  \otimes  (X_{\cdot t}  \otimes  m).
	\end{equation}
\end{itemize}
Although we risk some confusion in using these conventions, which are technically an abuse of notation, we nevertheless believe that their use in the following arguments are sufficiently clear.
They are also evocative, in that $v_{Xt}$ and $v_{\xi t}$ are, by Lemma \ref{lem:isometry}, not essentially distinct objects, but merely the same object expressed in different spaces (the same remark applies to $X_{Xt}$ and $X_{\xi t}$).

\subsection{CONTROL PROBLEM} \label{sec:ctrl statement}

Recall that $t \geq 0$ and $X \in \s{H}_{m,t}$ are fixed.
Consider a state process $X_{Xt}(s)=X_{Xt}(s;v_{Xt}(\cdot))$ associated to a
control $v_{Xt}(\cdot)$.
Define the cost functional $J_{X \otimes  m,t}:L^2_{\s{W}_t}(t,T;\s{H}_{X \otimes  m}) \to \bb{R}$ by
\begin{multline}
J_{X \otimes  m,t}(v_{\cdot t}(\cdot))=\dfrac{\lambda}{2}\int_{t}^{T}\int_{\bb{R}^n}\bb{E}|v_{\xi t}(s)|^{2}\dif \h{1.5pt} (X \otimes  m)(\xi)\dif s\\
+\int_{t}^{T}F(X_{\cdot t}(s;v_{\cdot t}(\cdot )) \otimes (X_{\cdot t} \otimes  m))\dif s+F_{T}(X_{\cdot t}(T;v_{\cdot t}(\cdot )) \otimes (X_{\cdot t} \otimes  m)),\label{eq:J_Xotimesm}
\end{multline}
where $\lambda>0$. An equivalent and more condensed version of $J_{X \otimes  m,t}$ is the functional $J_{Xt}:L_{\mathcal{W}_{Xt}}^{2}(t,T;\mathcal{H}_{m}) \to \bb{R}$ given by
\begin{equation}
J_{Xt}(v_{Xt}(\cdot ))=\dfrac{\lambda}{2}\int_{t}^{T}||v_{Xt}(s)||^{2}\dif s+\int_{t}^{T}F(X_{Xt}(s;v_{Xt}(\cdot )) \otimes  m)\dif s+F_{T}(X_{Xt}(T;v_{Xt}(\cdot )) \otimes  m)\label{eq:J_X}
\end{equation}
We shall make precise assumptions to guarantee the strict convexity
of the functional $J_{Xt}(v_{Xt}(\cdot )),$ its coeciveness, hence existence
and uniqueness of an optimal minimum, for which we shall write the
necessary and sufficient optimality conditions.

\begin{rem}
For the control problem with a dynamic depending on the law of the state process, that is
\begin{equation*} 
X_{Xt}(s;v(\cdot))= X(s) = X + \int_t^s G\big(v(\tau),X(\tau),\mathbb{L}_{X(\tau)}\big)\dif \tau + \eta(w(s)-w(t)),
\end{equation*}
given that $G$ has at most linear growth in $x$ and $v$, the variational techniques cannot be used directly due to the presence of the measure argument. To this end, we have to use the FBSDE approach. This will be illustrated in our upcoming paper, where we shall demonstrate how the alternative approach developed in the present article plays a crucial role in the mean field games setting. 
\end{rem}

\subsection{ASSUMPTIONS ON COST FUNCTIONAL} 
\label{sec:asms cost}
We describe assumptions on the functionals $m \mapsto F(m)$
and $m \mapsto F_{T}(m)$ on $\mathcal{P}_2(\mathbb{R}^n).$ Throughout these assumptions, $c,c_T,c',$ and $c_T'$ are fixed positive constants. We first assume a growth bound:
\begin{equation}
|F(  m)|\leq c\left(1+\int_{\mathbb{R}^n} |x|^2dm(x)\right),\:|F_{T}(m)|\leq c_{T}\left(1+\int_{\mathbb{R}^n} |x|^2dm(x)\right).\label{eq:3-1, no lift}
\end{equation}
For any $x,\widetilde{x} \in \mathbb{R}^n$ and $m \in \mathcal{P}_2(\mathbb{R}^2)$, we also assume the functionals have derivatives satisfying 
\begin{gather}
\left|D \dod{F}{m}(m)(x)\right|\leq \dfrac{c}{\sqrt{2}}(1+|x|)
\h{1pt},\h{10pt}
\left|D \dod{F_T}{m}(m)(x)\right|\leq \dfrac{c_T}{\sqrt{2}}(1+|x|); \label{eq:3-2, no lift}\\
\abs{D^{2}\dod{F}{m}(m)(x)}\leq \dfrac{c}{2},\: \abs{D_2D_1\dod[2]{F}{m}(m)(\tilde{x},x)}\leq \dfrac{c}{2}, \abs{D^{2}\dod{F_T}{m}(m)(x)}\leq \dfrac{c_T}{2},\: \abs{D_2D_1\dod[2]{F_T}{m}(m)(\tilde{x},x)}\leq \dfrac{c_T}{2}\label{eq:3-3, no lift}
\end{gather}
as well as the following monotonicity conditions:
\begin{multline}
\bb{E}\int_{\bb{R}^n}D^{2}\dod{F}{m}(X \otimes  m)(X(x))Y(x)\cdot Y(x)\dif m(x)\\
+\bb{E}\tilde{\bb{E}}\int_{\bb{R}^n}\int_{\bb{R}^n}D_2D_1\dod[2]{F}{m}(X \otimes  m)(\tilde{X}(\tilde{x}),X(x))\tilde{Y}(\tilde{x})\cdot Y(x)\dif m(\tilde{x})\dif m(x)\geq
-c'\bb{E}\int_{\bb{R}^n}|Y(x)|^{2}\dif m(x), \label{eq:3-33}
\end{multline}
\begin{multline}
\bb{E}\int_{\bb{R}^n}D^{2}\dod{F_T}{m}(X \otimes  m)(X(x))Y(x)\cdot Y(x)\dif m(x)\\
+\bb{E}\tilde{\bb{E}}\int_{\bb{R}^n}\int_{\bb{R}^n}D_2D_1\dod[2]{F_T}{m}(X \otimes  m)(\tilde{X}(\tilde{x}),X(x))\tilde{Y}(\tilde{x})\cdot Y(x)\dif m(\tilde{x})\dif m(x)\geq-c'_{T}\bb{E}\int_{\bb{R}^n}|Y(x)|^{2}\dif m(x).\label{eq:5-34}
\end{multline}

\begin{rem}
Suppose $c' = c_T' = 0$. The monotonicity conditions \eqref{eq:3-33} and \eqref{eq:5-34} are equivalent to the \emph{displacement monotonicity} of $\od{F}{m}$ and $\od{F_T}{m}$, as defined in \cite{GMMZ}.
If, in addition, the respective first terms on the left hand side of \eqref{eq:3-33} and \eqref{eq:5-34} are dropped, then the resulting inequality conditions are equivalent to Lasry-Lions monotonicity assumption.
Cf.~(2.5) in \cite{CDLL}. Here $c'$ and $c'_T$ can be positive, which means that we may need to restrict the time horizon in order to ensure the uniqueness of optimal trajectories; below we quantify the precise relationship between $T,c'$, and $c'_T$ that is needed.
See Remark \ref{rem6-5} below for a representative example that satisfies our assumptions.
\end{rem}

\begin{rem}
We provide the corresponding properties on the functionals $X \mapsto F (X \otimes m) $ and $X \mapsto F_T (X \otimes m)$ on $\mathcal{H}_m$ based on assumptions \eqref{eq:3-1, no lift}-\eqref{eq:5-34}. Note that the assumptions below are necessary but
not sufficient for assumptions \eqref{eq:3-1, no lift}-\eqref{eq:5-34} to hold.
We first have a growth bound:
\begin{equation}
|F(X \otimes  m)|\leq c(1+\enVert{X}^{2}),\:|F_{T}(X \otimes  m)|\leq c_{T}(1+\enVert{X}^{2}).\label{eq:3-1}
\end{equation}
Their G\^ateaux derivatives, which are also Fr\'{e}chet derivatives, satisfying 
\begin{gather}
||D_{X}F(X \otimes  m)||\leq c(1+\enVert{X}),\:||D_{X}F_{T}(X \otimes  m)||\leq c_{T}(1+\enVert{X})\label{eq:3-2}\\
||D_{X}F(X_{1} \otimes  m)-D_{X}F(X_{2} \otimes  m)||\leq c||X_{1}-X_{2}||,\;||D_{X}F_{T}(X_{1} \otimes  m)-D_{X}F_{T}(X_{2} \otimes  m)||\leq c_T||X_{1}-X_{2}||\label{eq:3-3}
\end{gather}
as well as the following monotonicity conditions:
\begin{gather}
\ip{D_{X}F(X_{1} \otimes  m)-D_{X}F(X_{2} \otimes  m)}{X_{1}-X_{2}}\geq-c'||X_{1}-X_{2}||^{2}\label{eq:3-4}\\
\ip{D_{X}F_T(X_{1} \otimes  m)-D_{X}F_T(X_{2} \otimes  m)}{X_{1}-X_{2}}\geq-c'_{T}||X_{1}-X_{2}||^{2}.\label{eq:3-5}
\end{gather}
If the second order derivatives $D_{X}^{2}F(X \otimes  m)(Z)$ and $D_{X}^{2}F_T(X \otimes  m)(Z)$ exist, we have
\begin{gather}
||D_{X}^{2}F(X \otimes  m)(Z)||\leq c||Z||,\;||D_{X}^{2}F_{T}(X \otimes  m)(Z)||\leq c_{T}||Z||,\;\forall X,Z\in\mathcal{H}_{m}\label{eq:3-30}\\
D_{X}^{2}F(X \otimes  m)(Y,Y)\geq-c'||Y||^{2},\:D_{X}^{2}F_{T}(X \otimes  m)(Y,Y)\geq-c'_{T}||Y||^{2}\label{eq:3-31}
\end{gather}
\end{rem}

Assumptions \eqref{eq:3-1}-\eqref{eq:3-5} are in the same spirit as those in Sections 2.1 and 7.2 of our former work \cite{ABY}. It is worth noting that the mean-field type control problem is generally related to yet far different from any generic mean field games settings; also see \cite{BFY,BFY-0,BFY-1,RCD,RCDFLA} on the comparison of these two theories. Therefore, one cannot directly compare the assumptions adopted in the mean-field type control problem with those purely set for the mean field games. In the literature on the mean field games, such as \cite{CDLL,RCD,GAS,GMMZ}, most of them studied the topic through the analysis in the Wasserrstein space without a comprehensive use of the lifting procedure, and these are fundamentally different from our method developed in this article. For instance, in the most recent article \cite{GMMZ}, the authors established the global well-posedness of master equations of some general mean field games with non-separable Hamiltonians based on some \emph{a priori} uniform $W_2$-Lipschitz estimates which were obtained by analyzing the propagation property of the displacement monotonicity, and this later condition is a generalization of the Lasry-Lions monotonicity assumption (2.5) on page 36 in \cite{CDLL}. Nevertheless, by the time of more thorough discussion on the master equations for mean field type problems in Section \ref{sec:master}, we shall compare our proposed assumptions with those commonly used for mean field games, more explanation on the connections between these two sets of assumptions will be indicated, e.g. the mentioned Lasry-Lions monotonicity assumption is actually related to \eqref{convexF}-\eqref{convexFT} after taking a function derivative with respect to measure. On the other hand, within the framework of mean field type control problem, our assumptions \eqref{eq:3-1}-\eqref{eq:3-2} are essentially consistent with those assumptions (H2) in \cite{PHW}, in which the authors derived the dynamic programming principle with the corresponding Hamilton-Jacobi-Bellman equations, and they also proved the viscosity nature and uniqueness of the solution. However, our assumptions \eqref{eq:3-3}-\eqref{eq:3-5}, whose counterpart could not be found in \cite{PHW}, are only used for enhancing the regularity of the solution, that is the value function, of the master equation; further elaboration and derivations on this claim are put in Section \ref{sec:value fxn}. Hence, our overall obtained results and the new approach developed here are totally novel in the discipline of mean-field type control theory. 

Considering the map $v_{Xt}(\cdot )\rightarrow J_{Xt}(v_{Xt}(\cdot))$ as a
functional on the Hilbert space $L_{\mathcal{W}_{Xt}}^{2}(t,T;\mathcal{H}_{m})$, we get the 
\begin{lem}
\label{lem3-1}Under the assumptions (\ref{eq:3-1}),(\ref{eq:3-2})
the functional $J_{Xt}(v_{Xt}(\cdot))$ has a G\^ateaux derivative, given
by 
\begin{multline}
D_{v}J_{Xt}(v_{Xt}(\cdot))(s)=\lambda v_{Xt}(s)\\
+\bb{E}\left[\left.\int_{s}^{T}D_{X}F(X_{Xt}(\tau;v_{Xt}(\cdot)) \otimes  m)\dif \tau
+D_{X}F_T(X_{Xt}(T;v_{Xt}(\cdot)) \otimes  m)\right|\mathcal{W}_{Xt}^{s}\right].\label{eq:DvJ}
\end{multline}
\end{lem}
The proof can be found in Appendix \ref{ap:A}.
\begin{rem}
\label{rem3-1}We can replace in $(\ref{eq:DvJ})$ the conditional
expectation with respect to $\mathcal{W}_{Xt}^{s}=\sigma(X)\vee\mathcal{W}_{t}^{s}$
by the conditional expectation with respect to $\mathcal{B}\vee\mathcal{W}_{t}^{s},$ for some $\mathcal{B}$ independent of $\mathcal{W}_{s}$ such that
$\sigma(X)\subset\mathcal{B}$. 
Indeed, the random variable $\int_{s}^{T}D_{X}F(X_{Xt}(\tau;v_{Xt}(\cdot)) \otimes  m)\dif \tau+D_{X}F(X_{Xt}(T;v_{Xt}(\cdot)) \otimes  m)$
is $\sigma(X)\vee\mathcal{W}_{t}^{T}$ measurable, and $\sigma(X)\vee\mathcal{W}_{t}^{T}=\sigma(X)\vee\mathcal{W}_{t}^{s}\vee\mathcal{W}_{s}^{T}$.
But $\sigma(X)\vee\mathcal{W}_{t}^{s}\subset\mathcal{B}\vee\mathcal{W}_{t}^{s}$, which is independent of $\mathcal{W}_{s}^{T}.$ It follows that
the conditional expextation with respect to $\mathcal{B}\vee\mathcal{W}_{t}^{s}$
is the same as the conditional expectation with respect to $\sigma(X)\vee\mathcal{W}_{t}^{s}.$
This remark will be very useful for comparison purposes. 
\end{rem}

\subsection{CONVEXITY OF THE OBJECTIVE FUNCTIONAL}

The following result gives conditions that imply the existence of a unique solution to the optimal control problem.
From now on we shall assume that these conditions hold.
\begin{prop}
\label{prop:convexity}Assume \eqref{eq:3-1},\eqref{eq:3-2},\eqref{eq:3-4}, \eqref{eq:3-5}
and 
\begin{equation}
\lambda-T\left(c'_{T}+\dfrac{c'T}{2}\right)>0,\label{eq:3-6}
\end{equation}
where $\lambda$, $c'$ and $c'_T$ are given in \eqref{eq:J_Xotimesm}, \eqref{eq:3-4} and \eqref{eq:3-5} respectively. Then the functional $J_{Xt}(v_{Xt}(\cdot))$ is strictly convex in the control $v_{Xt}(\cdot)$. It is also
coercive, i.e.~$J_{Xt}(v_{Xt}(\cdot))\rightarrow+\infty,$ as $\int_{t}^{T}||v_{Xt}(s)||^{2}\dif s\rightarrow+\infty.$
Consequently, there exists one and only one minimizer of $J_{Xt}(v_{Xt}(\cdot)).$
\end{prop}

The proof can be found in Appendix \ref{ap:A}. 
\begin{rem}
The continuity, in the control $v_{Xt}(\cdot)$, of the functional $J_{Xt}(v_{Xt}(\cdot))$ can be also derived from the continuity of $F$, $F_T$ and that of $X_{Xt}(s;v_{Xt}(\cdot ))$ in $v_{Xt}(\cdot)$.
\end{rem}

\begin{rem}
The discussion on the global-in-time existence of solutions over intervals of arbitrary length for more generic convex payoff functions is contained in another article \cite{BTY}. In both articles, we use variational techniques to obtain the global-in-time existence result of classical solutions over intervals of arbitrary length. Even for the dynamics with drift function like $A(X)+v$, where $X$ is the state and $v$ is the control, the variational techniques still work. However, for more general dynamics (for example, with drift function involving the control, the state and the law of the state non-separably), variational techniques cannot be used to obtain the existence due to the presence of the measure argument. To this end, we have to use the Hilbert-space-valued forward backward stochastic differential equations to establish the global existence of the classical solution.
\end{rem}

\subsection{NECESSARY AND SUFFICIENT CONDITION FOR OPTIMALITY}

Here and from now on, assumptions of Section \ref{sec:asms cost} and Proposition \ref{prop:convexity} are in force.
According to Proposition \ref{prop:convexity}, there exists one and only
one optimal control $\hat{v}_{Xt}(s)$. It must satisfy the necessary
and sufficient condition $D_{v}J_{Xt}(\hat{v}_{Xt}(\cdot))(s)=0$, which further implies the existence and uniqueness of the solution pair $\del{Y_{Xt}(s),Z_{Xt}(s)}$ to \eqref{eq:3-7}-\eqref{eq:3-8}. Calling
$Y_{Xt}$$(s)$ the corresponding optimal state and $Z_{Xt}(s)=-\lambda\hat{v}_{Xt}(s),$
the pair $\del{Y_{Xt}(s),Z_{Xt}(s)}$ is the unique solution of the
system 
\begin{align}
Y_{Xt}(s) &=X-\dfrac{1}{\lambda}\int_{t}^{s}Z_{Xt}(\tau)\dif \tau+\eta(w(s)-w(t)),\label{eq:3-7}\\
Z_{Xt}(s) &=\bb{E}\left[\left.\int_{s}^{T}D_{X}F(Y_{Xt}(\tau) \otimes  m)\dif \tau+D_{X}F_{T}(Y_{Xt}(T) \otimes  m)\right|\mathcal{W}_{Xt}^{s}\right].\label{eq:3-8}
\end{align}
Moreover, since $L_{\mathcal{W}_{Xt}}^{2}(t,T;\mathcal{H}_{m})$ is
isometric to $L_{\mathcal{W}_{t}}^{2}(t,T;\mathcal{H}_{X \otimes  m})$ by Lemma \ref{lem:isometry}, using the convention outline in Section \ref{sec:ctrl prelim}
there exists $Y_{\xi t}(s),$$Z_{\xi t}(s)$ belonging to $L_{\mathcal{W}_{t}}^{2}(t,T;\mathcal{H}_{X \otimes  m})$
such that $Y_{Xt}(s)=Y_{\xi t}(s)|_{\xi=X},\:Z_{Xt}(s)=Z_{\xi t}(s)|_{\xi=X}.$
The pair of random fields $\del{Y_{\cdot t}(s),Z_{\cdot t}(s)}$ is the solution of 
\begin{align}
Y_{\xi t}(s) &=\xi-\dfrac{1}{\lambda}\int_{t}^{s}Z_{\xi t}(\tau)\dif \tau+\eta(w(s)-w(t)),\label{eq:3-9}\\
Z_{\xi t}(s) &=\bb{E}\left[\left.\int_{s}^{T}D \dod{F}{m}(Y_{\cdot t}(\tau) \otimes (X \otimes  m))(Y_{\xi t}(\tau))\dif \tau+D \dod{F_T}{m}(Y_{\cdot t}(T) \otimes (X \otimes  m))(Y_{\xi t}(T))\right|\mathcal{W}_{t}^{s}\right]\cdot \label{eq:3-10}
\end{align}
Indeed, by Proposition \ref{prop:D_XF} we have $D \dod{F}{m}(Y_{\cdot t}(\tau) \otimes (X \otimes  m))(Y_{\xi t}(\tau))|_{\xi=X}=D_{X}F(Y_{Xt}(\tau) \otimes  m),$ and similarly for $F_T$.
We notice that $Y_{\xi t}(s)$ and $Z_{\xi t}(s)$ depend on $m$ only
through $X \otimes  m,$ so we can write them $Y_{\xi,X \otimes  m,t}(s)$ and $Z_{\xi,X \otimes  m,t}(s)$, respectively.
\begin{rem}\label{rem:wellposed fwd bckwd}
	The well-posedness of forward-backward systems such as \eqref{eq:3-7}-\eqref{eq:3-8} or \eqref{eq:3-9}-\eqref{eq:3-10} follows from the estimates assumed on $F$ and $F_T$ using standard arguments, because they are necessary conditions for a strictly convex minimization problem.
\end{rem}

We can express the value function as
\begin{equation}
V(X,t) := J_{Xt}(\hat{v}_{Xt}(\cdot ))=\dfrac{1}{2\lambda}\int_{t}^{T}||Z_{Xt}(s)||^{2}\dif s+\int_{t}^{T}F(Y_{Xt}(s) \otimes  m)\dif s+F_{T}(Y_{Xt}(T) \otimes  m),\label{eq:V(X,t) def}
\end{equation}
which depends only on the probability measure $X \otimes  m$
and $t.$
Equivalently, by a slight abuse of notation, the value function can be written as follows:
\begin{multline}
V(X \otimes  m,t):= J_{X \otimes  m,t}(\hat v_{\cdot t}(\cdot)) = \dfrac{1}{2\lambda}\int_{t}^{T}\bb{E}\int_{\bb{R}^n}|Z_{\xi,X \otimes  m,t}(s)|^{2}\dif \h{1.5pt} (X \otimes  m)(\xi)\dif s\\
+\int_{t}^{T}F(Y_{\cdot ,X \otimes  m,t}(s) \otimes (X \otimes  m))\dif s+F_{T}(Y_{\cdot ,X \otimes  m,t}(T) \otimes (X \otimes  m)).\label{eq:V(Xotimes m,t) def}
\end{multline}
Cf.~Section \ref{sec:ctrl statement}.

\subsection{DYNAMIC OPTIMALITY PRINCIPLE}

For a fixed $h > 0$, consider $Y_{Xt}(t+h),$ which is an element of $\mathcal{H}_{m,t+h}$ (i.e.~independent
of $\mathcal{W}_{t+h}$--see Section \ref{sec:ctrl prelim}) that is also $\mathcal{W}_{Xt}^{t+h}$ measurable.
We can then consider a control problem starting at $t+h$ instead
of $t$ with initial value $Y_{Xt}(t+h)$, or simply $Y(t+h)$ if there is no danger of ambiguity. 
To compute the optimal trajectory, we find the unique solution $\del{Y_{Y(t+h),t+h}(s),Z_{Y(t+h),t+h}(s)}, s \in [t+h,T]$ to a forward-backward system, similar to (\ref{eq:3-7})-
(\ref{eq:3-8}) (see Remark \ref{rem:wellposed fwd bckwd}):
\begin{align}
Y_{Y(t+h),t+h}(s) &=Y(t+h)-\dfrac{1}{\lambda}\int_{t+h}^{s}Z_{Y(t+h),t+h}(\tau)\dif \tau+\eta(w(s)-w(t+h))\label{eq:3-13}\\
Z_{Y(t+h),t+h}(s) &=\bb{E}\left[\left.\int_{s}^{T}D_{X}F(Y_{Y(t+h),t+h}(\tau) \otimes  m)\dif \tau+D_{X}F(Y_{Y(t+h),t+h}(T) \otimes  m)\right|\mathcal{W}_{Y(t+h),t+h}^{s}\right]
\end{align}
where $\mathcal{W}_{Y(t+h),t+h}^{s}=\sigma(Y(t+h))\vee\mathcal{W}_{t+h}^{s}$. 
Using Remark \ref{rem3-1}, we can replace $\mathcal{W}_{Y(t+h),t+h}^{s}$
with $\mathcal{W}_{Xt}^{s}$ in the conditional expectation, and we see that $\del{Y_{Xt}(s),Z_{Xt}(s)}$
is still a solution.
By the uniqueness of solution we deduce
\begin{equation}
Y_{Y(t+h),t+h}(s)=Y_{Xt}(s)\; \text{and} \; Z_{Y(t+h),t+h}(s)=Z_{Xt}(s)\: \forall s>t+h.\label{eq:3-14}
\end{equation}
Consequently, we also have
\begin{equation}
Y_{\cdot \,Y(t+h) \otimes  m,t+h}(s) \otimes (Y(t+h) \otimes  m)=Y_{\cdot \,X \otimes  m,t}(s) \otimes (X \otimes  m)\label{eq:3-15}
\end{equation}
Therefore, from \eqref{eq:V(Xotimes m,t) def}, we get 
\begin{multline*}
V(X \otimes  m,t)=\dfrac{1}{2\lambda}\int_{t}^{t+h}||Z_{Xt}(s)||^{2}\dif s+\int_{t}^{t+h}F(Y_{Xt}(s) \otimes  m)\dif s\\
+\dfrac{1}{2\lambda}\int_{t+h}^{T}\bb{E}\int_{\bb{R}^n}|Z_{\xi\,Y(t+h) \otimes  m\,t+h}(s)|^{2}\dif \del{Y(t+h) \otimes  m}(\xi)+\int_{t+h}^{T}F(Y_{\cdot \,Y(t+h) \otimes  m,t+h}(s) \otimes (Y(t+h) \otimes  m))\dif s\\
+F_{T}(Y_{\cdot \,Y(t+h) \otimes  m,t+h}(s) \otimes (Y(t+h) \otimes  m)),
\end{multline*}
and by substituting $(Y(t+h),t+h)$ for $(X,t)$ in \eqref{eq:V(Xotimes m,t) def} and applying \eqref{eq:3-15}, we finally deduce
\begin{equation}
V(X \otimes  m,t)=\dfrac{1}{2\lambda}\int_{t}^{t+h}||Z_{Xt}(s)||^{2}\dif s+\int_{t}^{t+h}F(Y_{Xt}(s) \otimes  m)\dif s+V(Y_{Xt}(t+h) \otimes  m,t+h),\label{eq:3-16}
\end{equation}
which is the dynamic optimality principle. 

\section{PROPERTIES OF THE VALUE FUNCTION}
\label{sec:value fxn}

In this section we systematically study the regularity of $V(X \otimes  m,t)$ defined in \eqref{eq:V(Xotimes m,t) def}, beginning with pointwise estimates, then proceeding to derivatives with respect to $X$ and $m$, and finishing with continuity in time.
Recall that $V(X \otimes  m,t)$ is well-defined for any $m \in \mathcal{P}_2(\mathbb{R}^n), t \geq 0$, and $X \in \s{H}_{m,t}$, where $\s{H}_{m,t}$ is the closed subspace of $\mathcal{H}_{m}$
of random fields $X_{\cdot t}$ independent of $\mathcal{W}_{t}$ (see Section \ref{sec:ctrl prelim}).

\subsection{BOUNDS}

We begin with the following estimates, which express the growth rate of the value function and optimal trajectory with respect to $\enVert{X}$.
\begin{prop}
\label{prop4-1}
Assume (\ref{eq:3-1}), (\ref{eq:3-2}), (\ref{eq:3-3}), (\ref{eq:3-4}), and (\ref{eq:3-6}).
Let $m \in \s{P}_2(\bb{R}^n), t \geq 0,$ and $X \in \s{H}_{m,t}$, and let $\del{Y_{Xt}, Z_{Xt}}$ be the solution of \eqref{eq:3-7}-\eqref{eq:3-8}.
Then we have 
\begin{gather}
||Y_{Xt}(s)||,\:||Z_{Xt}(s)||\leq C_{T}(1+\enVert{X}), \; \forall s\in(t,T),\label{eq:4-1}\\
|V(X \otimes  m,t)|\leq C_{T}(1+\enVert{X}^{2}) \; \forall t \geq 0,\label{eq:4-2}
\end{gather}
where $C_{T}$ is a constant depending only on the data and $T$, independent of $X,m,s$, and $t$.
\end{prop}

The proof can be found in Appendix \ref{ap:B}.

\subsection{REGULARITY OF $V(X \otimes  m,t)$ WITH RESPECT TO $X$}

\begin{prop}
\label{prop4-2} Assume (\ref{eq:3-1}), (\ref{eq:3-2}), (\ref{eq:3-3}), (\ref{eq:3-4}), and (\ref{eq:3-6}). 
Let $m \in \s{P}_2(\bb{R}^n)$ and $t \geq 0$.
Then the functional $\s{H}_{m,t} \ni X\mapsto V(X \otimes  m,t)$ is G\^ateaux differentiable
and 
\begin{equation}
D_{X}V(X \otimes  m,t)=Z_{Xt}(t).\label{eq:4-3}
\end{equation}
We also have the Lipschitz property: for all $X^{1},X^{2}\in\mathcal{H}_{m,t},$ we have
\begin{equation}
||D_{X}V(X^{1} \otimes  m,t)-D_{X}V(X^{2} \otimes  m,t)||\leq C_{T}||X^{1}-X^{2}||\label{eq:4-4}
\end{equation}
where $C_T$ is a constant depending only on the data and $T$, independent of $X^1,X^2,m$, and $t$.
\end{prop}

The proof can be found in Appendix \ref{ap:B}.

\subsection{FUNCTIONAL DERIVATIVE OF $V(m,t)$}

If we take $X_{xt}=x,$ recalling that $X \otimes  m=m,$ we get $D_{X}V(m,t)=Z_{xt}(t)=Z_{xmt}(t)$, where the pair $\del{Y_{xmt}(s),Z_{xmt}(s)}$ is the unique solution of
\begin{align}
Y_{xmt}(s) &=x-\dfrac{1}{\lambda}\int_{t}^{s}Z_{xmt}(\tau)\dif \tau+\eta(w(s)-w(t)),\label{eq:Yxmt}\\
Z_{xmt}(s) &=\bb{E}\left[\left.\int_{s}^{T} D \dod{F}{m}(Y_{\cdot mt}(\tau) \otimes  m)(Y_{xmt}(\tau))\dif \tau+D \dod{F_T}{m}(Y_{\cdot mt}(T) \otimes  m)(Y_{xmt}(T))\right|\mathcal{W}_{t}^{s}\right],\label{eq:Zxmt}
\end{align}
cf.~(\ref{eq:3-9})-(\ref{eq:3-10}) and Remark \ref{rem:wellposed fwd bckwd}.
Equations \eqref{eq:Yxmt} and \eqref{eq:Zxmt} form the system of optimality conditions for the following control problem, obtained by appropriately specifying the dynamics \eqref{eq:SDE} and objective functional \eqref{eq:J_Xotimesm}:
\begin{align}
X_{xt}(s) &=x+\int_{t}^{s}v_{xt}(\tau)\dif \tau+\eta(w(s)-w(t)),\quad\label{eq:4-7}\\
J_{mt}(v_{.t}(\cdot)) &=\dfrac{\lambda}{2}\int_{t}^{T}||v_{.t}(s)||_{\mathcal{H}_{m}}^{2}\dif s+\int_{t}^{T}F(X_{.t}(s;v_{.t}(\cdot)) \otimes  m)\dif s+F_{T}(X_{.t}(T;v_{.t}(\cdot)) \otimes  m)\label{eq:4-8}
\end{align}
with $v_{.t}(\cdot)$ $\in L_{\mathcal{W}_{t}}^{2}(t,T;\mathcal{H}_{m})$. 
From (\ref{eq:V(Xotimes m,t) def}) it follows that
\begin{align}\nonumber
V(m,t)=&\inf_{v_{\cdot t}(\cdot )}J_{m,t}(v_{\cdot t}(\cdot ))\\
=&\dfrac{1}{2\lambda}\int_{t}^{T}\bb{E}\int_{\bb{R}^n}|Z_{xmt}(s)|^{2}\dif m(x) \dif s +\int_{t}^{T}F(Y_{\cdot mt}(s) \otimes  m)\dif s+F_{T}(Y_{\cdot mt}(T) \otimes  m).\label{eq:4-9}
\end{align}
Applying Proposition \ref{prop4-1} (under the assumptions stated there), we have
\begin{equation}
\bb{E}|Y_{xmt}(s)|^{2}\leq C_{T}(1+|x|^{2}),\:\bb{E}|Z_{xmt}(s)|^{2}\leq C_{T}(1+|x|^{2}).\label{eq:4-90}
\end{equation}
Equation \eqref{eq:4-90} implies that the optimal control is an element of the set
\begin{equation}
\label{eq:restricted controls}
\sr{V} := \cbr{v_{\cdot t}(\cdot) : x \mapsto \dfrac{v_{xt}(\cdot )}{1+|x|^{2}}\in L^{\infty}(\bb{R}^n;L_{\mathcal{W}_{t}}^{2}(t,T;\bb{R}^n))},
\end{equation}
and therefore the value function $V(m,t)$ remains unchanged if we restrict the domain of $J_{m,t}$ to $\sr{V}$.
A crucial fact is that
\begin{equation}
\sr{V} \subset \bigcap_{\mu \in \mathcal{P}_2(\mathbb{R}^n)} L^2_{\s{W}_t}(t,T;\s{H}_\mu),
\end{equation}
which is proved in the same way as Lemma \ref{lem:intersect Hm}.
Thus, for a given $v_{\cdot t}(\cdot) \in \sr{V}$, $J_{m,t}(v_{\cdot t}(\cdot))$ is defined for \emph{all} $m \in \mathcal{P}_2(\mathbb{R}^n)$.
This will be used to prove the following:
\begin{prop}
\label{prop4-3}
Assume (\ref{eq:3-1}), (\ref{eq:3-2}), (\ref{eq:3-3}), (\ref{eq:3-4}), and (\ref{eq:3-6}). 
Then the value function $V(m,t)$ has a functional derivative $\dod{}{m}V(m,t)(x),$
given by 
\begin{multline}
\dod{}{m}V(m,t)(x)\\=\dfrac{1}{2\lambda}\int_{t}^{T}\bb{E}|Z_{xmt}(s)|^{2}\dif s+\int_{t}^{T}\bb{E}\dod{F}{m}(Y_{\cdot mt}(s) \otimes  m)(Y_{xmt}(s))\dif s+\bb{E}\dod{F_T}{m}(Y_{\cdot mt}(T) \otimes  m)(Y_{xmt}(T)).\label{eq:4-10}
\end{multline}
Moreover, for $m\in\mathcal{P}_2(\mathbb{R}^n)$ and $t\in[0,T]$, $\dod{}{m}V(m,t)(x)$ is continuously differentiable in $x$ and
\begin{equation}
D\dod{}{m}V(m,t)(x)=Z_{xmt}(t) \;\text{and}\; D_{X}V(X \otimes  m,t)=D\dod{}{m}V(X \otimes  m,t)(X).\label{eq:4-11}
\end{equation}
\end{prop}
The proof can be found in Appendix \ref{ap:B}. We can also extend the result (\ref{eq:4-11}) as follows 
\begin{prop}
\label{prop4-4} 
Assume (\ref{eq:3-1}), (\ref{eq:3-2}), (\ref{eq:3-3}), (\ref{eq:3-4}), and (\ref{eq:3-6}). 
Then for any $x \in \bb{R}^n$, we have
\begin{equation}
Z_{xmt}(s)=D\dod{V}{m}(Y_{\cdot mt}(s) \otimes  m,s)(Y_{xmt}(s)) \quad \forall s \in [t,T].\label{eq:4-17}
\end{equation}
\end{prop}
The proof can be found in Appendix \ref{ap:B}. By \eqref{eq:partial X derivative identity} we can also write (\ref{eq:4-17}) as 
\begin{equation}
Z_{\cdot mt}(s)=D_{X}V(Y_{\cdot mt}(s) \otimes  m,s).\label{eq:4-24}
\end{equation}
In fact, by a similar argument, we can derive the identity
\begin{equation}
Z_{Xt}(s)=D_{X}V(Y_{Xt}(s) \otimes  m,s) \quad \forall s \in [t,T], \ \forall X \in \s{H}_m. \label{eq:4-25}
\end{equation}
The proof is quite similar to that of Proposition \ref{prop4-4}.
Indeed, one considers the solution $\del{Y_{Xt}(s),Z_{Xt}(s)}$ of
the system (\ref{eq:3-7})-(\ref{eq:3-8}) and then repeats the same reasoning as above; we omit the details.
It follows that $Y_{Xt}(s)$ can be expressed as the solution of a
stochastic differential equation in $\mathcal{H}_{m},$ namely
\begin{equation}
Y_{Xt}(s)=X-\dfrac{1}{\lambda}\int_{t}^{s}D_{X}V(Y_{Xt}(\tau) \otimes  m,\tau)\dif \tau+\eta(w(s)-w(t)).
\label{eq:4-250}
\end{equation}

\subsection{REGULARITY IN TIME}

\begin{prop}
\label{prop4-5}
Assume (\ref{eq:3-1}), (\ref{eq:3-2}), (\ref{eq:3-3}), (\ref{eq:3-4}), and (\ref{eq:3-6}). 
Let $m \in \s{P}_2(\bb{R}^n)$ and $t \geq 0$.
Then for any $X=X_{xt} \in \s{H}_{m,t}$ and $h > 0$ we have
\begin{gather}
|V(X \otimes  m,t+h)-V(X \otimes  m,t)|\leq C_{T}h(1+\enVert{X}^{2}),\label{eq:4-26}\\
\enVert{D_{X}V(X \otimes  m,t+h)-D_{X}V(X \otimes  m,t)}\leq C_{T}(h^{\frac{1}{2}} + h)(1 + \enVert{X}),\label{eq:4-27}
\end{gather}
where $C_{T}$ is a constant depending only on the data and $T$, independent of $X,m,t$, and $h$.
\end{prop}
The proof can be found in Appendix \ref{ap:B}.

\section{BELLMAN EQUATION}
\label{sec:bellman}

\subsection{FURTHER REGULARITY ASSUMPTIONS} \label{sec:further regularity}

In order to show that the value function defined in \eqref{eq:V(Xotimes m,t) def} is a classical solution to the Bellman equation (see Equation \eqref{eq:Bellman} below), we shall need additional regularity assumptions on the data.
we shall assume that the functionals $X\mapsto F(X \otimes  m)$ and $X\mapsto F_{T}(X \otimes  m)$
have second order G\^ateaux derivatives in $\mathcal{H}_{m}$ (which take values in $\mathcal{L}(\mathcal{H}_{m};\mathcal{H}_{m})$, denoted
$D_{X}^{2}F(X \otimes  m)$ and $D_{X}^{2}F_{T}(X \otimes  m)$, respectively. From formula
(\ref{eq:2-232}) we have
\begin{equation}
D_{X}^{2}F(X \otimes  m)(Z)=\Phi(X)Z+\Psi_{XZ}(X)\label{eq:5-1}
\end{equation}
where $\Phi:\bb{R}^n\rightarrow\mathcal{L}(\bb{R}^n;\bb{R}^n)$
and $\Psi_{XZ}:\bb{R}^n\rightarrow \bb{R}^n$ are measurable functions.
The 
map $\s{H}_{m} \ni Z\rightarrow\Psi_{XZ}(x)$
is linear for each $X \in \s{H}_m$ and $x \in \bb{R}^n$.

We also assume the following estimates:
\begin{gather}
||D_{X}^{2}F(X \otimes  m)||\leq c,\:||D_{X}^{2}F_{T}(X \otimes  m)||\leq c_{T},\label{eq:5-2}\\
\ip{D_{X}^{2}F(X \otimes  m)(Z)}{Z} + c'||Z||^{2}\geq 0,\: \ip{D_{X}^{2}F_T(X \otimes  m)(Z)}{Z} + c'_{T}||Z||^{2}\geq 0, \; \forall Z\in\mathcal{H}_{m},\label{eq:5-3}\\
\begin{aligned}
\bb{E}\int_{\bb{R}^n}|D_{X}^{2}F(X \otimes  m)(Z)|\dif m(x)&\leq cE\int_{\bb{R}^n}|Z(x)|\dif m(x),\:\forall Z\in L^{1}(\Omega,\mathcal{A},P;L_{m}^{1}(\bb{R}^n;\bb{R}^n))\\
\bb{E}\int_{\bb{R}^n}|D_{X}^{2}F_{T}(X \otimes  m)(Z)|\dif m(x)&\leq c_{T}\bb{E}\int_{\bb{R}^n}|Z(x)|\dif m(x),\:\forall Z\in L^{1}(\Omega,\mathcal{A},P;L_{m}^{1}(\bb{R}^n;\bb{R}^n))
\end{aligned}\label{eq:5-50}
\end{gather}
We also assume the following continuity: 
\begin{equation}
\begin{aligned}
X\mapsto D_{X}^{2}F(X \otimes  m)(Z)\:\text{is continuous from}\:\mathcal{H}_{m}\:\text{to }\mathcal{H}_{m},\text{for each} \:Z \in \s{H}_m,\\
X\mapsto D_{X}^{2}F_{T}(X \otimes  m)(Z)\:\text{is continuous from}\:\mathcal{H}_{m}\:\text{to }\mathcal{H}_{m},\text{for each} \:Z \in \s{H}_m,
\end{aligned}\label{eq:5-4}
\end{equation}
and
\begin{multline}
\forall\epsilon,M > 0, \ \forall X,Z \in \s{H}_m\:\text{s.t.}\:\enVert{Z}\leq M,\exists\delta_{X}(\epsilon,M) > 0\;\text{s.t.} \ \enVert{X_{k}-X}\leq\delta_{X}(\epsilon,M)\\
	\Longrightarrow
\begin{cases}
\bb{E}\int_{\bb{R}^n}|D_{X}^{2}F(X_{k} \otimes  m)(Z)-D_{X}^{2}F(X \otimes  m)(Z)|\dif m(x)\leq\epsilon,\\
\bb{E}\int_{\bb{R}^n}|D_{X}^{2}F_{T}(X_{k} \otimes  m)(Z)-D_{X}^{2}F_{T}(X \otimes  m)(Z)|\dif m(x)\leq\epsilon.
\end{cases}\label{eq:5-5}
\end{multline}

\begin{rem}
Note that the assumption of the existence of second-order Fr\'{e}chet derivative $D_X^2F(X \otimes  m)$ is a stronger assumption than $F\in C^{2,1}_b(\mathcal{P}_2(\mathbb{R}^n))$ since there is a counterexample in \cite{BLPR} showing that $F\in C^{2,1}_b(\mathcal{P}_2(\mathbb{R}^n))$ cannot warrant the existence of second-order Fr\'{e}chet derivative of the lifted version of $F$ in $\mathcal{H}_{m}$. However, our assumptions are not severe since many common models such as the linear-quadratic one satisfy our assumption which can be checked directly.
\end{rem}
We have seen in (\ref{eq:3-3, no lift}),(\ref{eq:3-33}),(\ref{eq:5-34})
how to achieve (\ref{eq:5-2}),(\ref{eq:5-3}). We have seen in (\ref{eq:2-235})
that to get (\ref{eq:5-2}) it suffices to assume 
\begin{equation}
\begin{aligned}
&\abs{D^{2}\dod{F}{m}(m)(x)}\leq c,\:\abs{D_1D_2\dod[2]{F}{m}(m)(x,\tilde{x})}\leq c,\\
&\abs{D^2\dfrac{dF_{T}}{dm}(m)(x)} \leq c_{T},\:\abs{D_2D_1\dfrac{d^{2}F_{T}}{dm^{2}}(m)(x,\tilde{x})} \leq c_{T},\forall x,\tilde{x}\in \bb{R}^n,
\end{aligned}\label{eq:5-6}
\end{equation}
where $|\cdot|$ denotes the matrix norm. We may
fulfil (\ref{eq:5-3}) by assuming that 
\begin{equation}
\begin{aligned}
 D^{2}\dod{F}{m}(m)(x)\xi \cdot \xi
+ D_2D_1\dod[2]{F}{m}(m)(x,\tilde{x})\xi \cdot \tilde{\xi} \geq-c'|\xi|(|\xi|+|\tilde{\xi}|),\\ 
D^{2}\dod{F_T}{m}(m)(x)\xi \cdot \xi
+ D_2D_1\dod[2]{F_T}{m}(m)(x,\tilde{x})\xi \cdot \tilde{\xi} \geq-c'_T|\xi|(|\xi|+|\tilde{\xi}|),\forall x,\tilde{x},\xi,\tilde{\xi}.
\end{aligned}\label{eq:5-7}
\end{equation}
The bounds (\ref{eq:5-6}) also
suffice to imply (\ref{eq:5-50}).
Finally, to get the continuity properties in \eqref{eq:5-4} and (\ref{eq:5-5}), it is sufficient to assume 
\begin{multline}
(m,x)\mapsto D^{2}\dod{F}{m}(m)(x),\:(m,x,\tilde{x})\mapsto D_2D_1\dod[2]{F}{m}(m)(x,\tilde{x})\:\text{are continuous from}\:\mathcal{P}_{2}(\bb{R}^n)\times \bb{R}^n\\ 
\text{and}\:\mathcal{P}_{2}(\bb{R}^n)\times \bb{R}^n\times \bb{R}^n\rightarrow\mathcal{L}(\bb{R}^n;\bb{R}^n), \; \text{respectively,}
\label{eq:5-8}
\end{multline}
and likewise for $F_{T}.$
To see this, we use the identity \eqref{eq:2-232} and apply standard arguments.
Let us provide the details for the most difficult step, leaving the rest of the argument to the reader.
we shall show that if
$X_{k}\rightarrow X$ in $\mathcal{H}_{m}$ and $||Z||\leq M,$ $\forall\epsilon,$
$\exists\delta_{X}(\epsilon,M)$ such that $||X_{k}-X||\leq\delta_{X}(\epsilon,M)$
implies 
\begin{multline}
I_{k} :=\\
\bb{E}\int_{\bb{R}^n}\left|\tilde{\bb{E}}\int_{\bb{R}^n}\bigg(D_2D_1\dod[2]{F}{m}(X_{k} \otimes  m)(X_{k}(x),\tilde{X}_{k}(\tilde{x}))-D_2D_1\dod[2]{F}{m}(X \otimes  m)(X(x),\tilde{X}(\tilde{x}))\tilde{Z}(\tilde{x})\bigg)\dif m(\tilde{x})\right|\dif m(x)\\
\leq\epsilon.\label{eq:5-9}
\end{multline}
By the Cauchy-Scwhartz inequality, we have $I_k \leq C_M J_k$ where
\[
J_k := \sqrt{\bb{E}\tilde{\bb{E}}\int_{\bb{R}^n}\int_{\bb{R}^n}\abs{D_2D_1\dod[2]{F}{m}(X_{k} \otimes  m)(X_{k}(x),\tilde{X}_{k}(\tilde{x}))-D_2D_1\dod[2]{F}{m}(X \otimes  m)(X(x),\tilde{X}(\tilde{x}))}^{2}\dif m(x)\dif m(\tilde{x})}.
\]
We claim $J_k \to 0$ as $k \to \infty$.
Indeed, we can assume, by extracting a subsequence,
that $X_{k}(x)\rightarrow X(x)$ a.s.~a.e. Then by the continuity
assumption \eqref{eq:5-8} and the uniform bound \eqref{eq:5-6}, $J_{k}$ tends to $0.$ This is true
for any converging subsequence. Hence the full sequence $J_{k}\to 0.$
Thus for any $\eta > 0,$ there exists $\beta_{X}(\eta) > 0$ such that $\enVert{X_{k}-X}\leq \beta_{X}(\eta)\:\Rightarrow J_{k}\leq\eta.$
Setting
$\delta_{X}(\epsilon,M)=\beta_{X}\del{\dfrac{\epsilon}{C_{M}}}$ the result
(\ref{eq:5-9}) is obtained. 

\subsection{EXISTENCE OF SECOND DERIVATIVE OF THE VALUE FUNCTION }

The following crucial result provides the existence and continuity of $D_X^2 V$.
\begin{prop}
\label{prop5-1}
Assume (\ref{eq:3-1}), (\ref{eq:3-2}), (\ref{eq:3-3}), (\ref{eq:3-4}), (\ref{eq:3-6}), 
 (\ref{eq:5-2}), (\ref{eq:5-3}), (\ref{eq:5-4}), and (\ref{eq:5-5}).

(i) The value function $V(X \otimes  m,t)$, defined in \eqref{eq:V(Xotimes m,t) def}, has a second order G\^ateaux derivative with respect to $X \in \s{H}_{m,t}$, which we denote 
$D_{X}^{2}V(X \otimes  m,t)\in\mathcal{L}(\mathcal{H}_m,t;\mathcal{H}_m,t)$.
It satisfies 
\begin{equation}
||D_{X}^{2}V(X \otimes  m,t)(\mathcal{X})||\leq C_{T}||\mathcal{X}||\label{eq:5-1000}
\end{equation}
where $C_{T}$ is a constant not depending on $X$ or $t.$

(ii) We have the following continuity property.
Let $t_{k}\downarrow t$,
and let $X_{k},\mathcal{X}_{k} \in \s{H}_{m,t_k}$ 
converge in $\mathcal{H}_{m}$ to $X,\mathcal{X}$ respectively.
Then 
\begin{equation}
D_{X}^{2}V(X_{k} \otimes  m,t_{k})(\mathcal{X}_{k})\rightarrow D_{X}^{2}V(X \otimes  m,t)(\mathcal{X})\;\text{in}\;\mathcal{H}_{m}.\label{eq:5-10}
\end{equation}
The limits $X,\mathcal{X}$ are independent of $\mathcal{W}_{t}$, i.e.~$X,\s{X} \in \s{H}_{m,t}$.

(iii) We also have a formula for the second order G\^ateaux derivative.
Let $\del{Y_{Xt}(s),Z_{Xt}(s)}$ be the solution of the system (\ref{eq:3-7}),
(\ref{eq:3-8}). We define $\del{\mathcal{Y}_{X\mathcal{X}t}(s), \mathcal{Z}_{X\mathcal{X}t}(s)}$ to be the
unique solution of the system 
\begin{equation}
\begin{aligned}
\mathcal{Y}_{X\mathcal{X}t}(s) &=\mathcal{X}-\dfrac{1}{\lambda}\int_{t}^{s}\mathcal{Z}_{X\mathcal{X}t}(\tau)\dif \tau\\
\mathcal{Z}_{X\mathcal{X}t}(s) &=\bb{E}\left[\left.\int_{s}^{T}D_{X}^{2}F(Y_{Xt}(\tau) \otimes  m)(\mathcal{Y}_{X\mathcal{X}t}(\tau))\dif \tau+D_{X}^{2}F_{T}(Y_{Xt}(T) \otimes  m)(\mathcal{Y}_{X\mathcal{X}t}(T))\right|\mathcal{W}_{X\mathcal{X}t}^{s}\right]
\end{aligned}
\label{eq:5-11}
\end{equation}
where $\mathcal{W}_{X\mathcal{X}t}^{s}=\sigma(X,\mathcal{X})\vee\mathcal{W}_{t}^{s}.$
Then 
\begin{equation}
D_{X}^{2}V(X \otimes  m,t)(\mathcal{X})=\mathcal{Z}_{X\mathcal{X}t}(t)\label{eq:5-12}
\end{equation}
\end{prop}
The proof can be found in Appendix \ref{ap:C}.

\subsection{BELLMAN EQUATION } 

The Bellman equation for the optimal control problem stated in Section \ref{sec:ctrl statement} is given by 
\begin{equation}\label{eq:Bellman}
\begin{aligned}
&\dfrac{\partial V}{\partial t}(X \otimes  m,t)+\dfrac{1}{2}\ip{D_{X}^{2}V(X \otimes  m,t)(\eta N)}{\eta N}-\dfrac{1}{2\lambda}||D_{X}V(X \otimes  m,t)||^{2}+F(X \otimes  m) =0,\\
&V(X \otimes  m,T)=F_{T}(X \otimes  m).
\end{aligned}
\end{equation}
Let us now define what we mean by solutions to \eqref{eq:Bellman}.
\begin{defn}\label{def5-3}
	Let $V:\s{P}_2(\bb{R}^n) \times [0,T] \to \bb{R}$ be any function satisfying the following regularity properties:	
	\begin{itemize}
		\item $V$ is continuous and satisfies the estimates \eqref{eq:4-2} and \eqref{eq:4-26};
		\item $D_X V(X  \otimes  m,t)$ exists, and  for each $m\in\mathcal{P}_2(\mathbb{R}^n)$, it is separately (i.e. marginally only but not jointly) continuous in $X$ and in $t$ in the sense of \eqref{eq:4-4} and \eqref{eq:4-27};
		\item $D_X^2 V(X  \otimes  m,t)$ exists, and for each $m\in\mathcal{P}_2(\mathbb{R}^n)$, it is sequentially continuous in the order pair $(X,t)$ in the sense of (\ref{eq:5-10}), and $D_X^2 V(X  \otimes  m,t)$ also satisfies the property \eqref{eq:5-1000};
		\item the following continuity property is satisfied:
if $t_{k}\downarrow t$, and $X_{k},\mathcal{X}_{k} \in \s{H}_{\color{black}m,t_k}$ so that $X_{k}\rightarrow X$
in $\mathcal{H}_{m}$ and $\displaystyle\sup_k\|\mathcal{X}_{k}\|_{\mathcal{H}_{m}}<\infty$,
then
\begin{align}
\bb{E}\int_{\bb{R}^n}\abs{D_{X}^{2}V(X_{k} \otimes  m,t_{k})(\mathcal{X}_{k}) - D_{X}^{2}V(X \otimes  m,t_{k})(\mathcal{X}_{k})}\dif m(x)\;\rightarrow0;
\end{align}
\item {\color{black}for each $m\in \s{P}_2(\bb{R}^n)$ and each $X\in \mathcal{H}_{m}$}, $t \mapsto V(X  \otimes  m,t)$ is right-differentiable.
	\end{itemize}
We say that $V$ is a \emph{classical solution} to the Bellman equation \eqref{eq:Bellman} provided that,
for any $t \geq 0, m \in \s{P}_2(\bb{R}^n),$ and $X \in \mathcal{H}_{m,t}$, and for any  standard Gaussian variable $N$ in $\bb{R}^n$ that is independent
of the filtration $\mathcal{W}_{0}$ and of $X$, Equation \eqref{eq:Bellman} holds, where $\dfrac{\partial V}{\partial t}(X \otimes  m,t)$
stands for the right-hand derivative. 
\end{defn}
{\color{black} \begin{rem}
A standard Gaussian variable $N$ valued in $\bb{R}^n$, which is independent of both the filtration $\mathcal{W}_{0}$ and $X$ always exists, for instance, simply take $N=\dfrac{w(t+h)-w(t)}{\sqrt{h}}\sim\mathcal{N}[0,1]$, see also the proof of Lemma \ref{lem8-1}.
\end{rem}}
\begin{thm}
\label{theo5-1} 
Assume (\ref{eq:3-1}), (\ref{eq:3-2}), (\ref{eq:3-3}), (\ref{eq:3-4}), (\ref{eq:3-6}), 
(\ref{eq:5-2}), (\ref{eq:5-3}), (\ref{eq:5-4}), (\ref{eq:5-5}), 
and (\ref{eq:5-50}). We also assume 
\begin{equation}
\lambda-T\left(c_{T}+c\dfrac{T}{2}\right)>0\label{eq:5-120}
\end{equation}
where $c,c_{T}$ are the constants appearing in (\ref{eq:5-50}). Let $V$ be the value function defined in \eqref{eq:V(Xotimes m,t) def}.
Then $V$ is the unique classical solution to the Bellman equation \eqref{eq:Bellman}.
\end{thm}

The proof is given in Appendix \ref{ap:C}. 

\subsection{CASE $X_{xt}=x$}

By inserting $X_{xt}=x$ into the Bellman equation \eqref{eq:Bellman}, we can reduce it to a PDE on the space of measures.
Recall that $X \otimes  m=m$ (Example \ref{ex:two push-forwards}).
By (\ref{eq:4-11}) we have
$D_{X}V(m,t)=D\dod{}{m}V(m,t)(x).$ Therefore 
\begin{equation}
||D_{X}V(X \otimes  m,t)||^{2}=\int_{\bb{R}^n}\abs{D\dod{}{m}V(m,t)(x)}^{2}\dif m(x).\label{eq:5-14}
\end{equation}
We need next to interpret $\ip{D_{X}^{2}V(m,t)(\eta N)}{\eta N}.$
Consider the system
\begin{equation}
\begin{aligned}
\mathcal{Y}(s) &=\mathcal{\eta}N-\dfrac{1}{\lambda}\int_{t}^{s}\mathcal{Z}(\tau)\dif \tau,\\
\mathcal{Z}(s) &=\bb{E}\left[\left.\int_{s}^{T}D_{X}^{2}F(Y_{\cdot \,t}(\tau) \otimes  m)(\mathcal{Y}(\tau))\dif \tau+D_{X}^{2}F_{T}(Y_{\cdot \,t}(T) \otimes  m)(\mathcal{Y}(T))\right|\mathcal{W}_{Nt}^{s}\right],
\end{aligned}
\label{eq:5-15}
\end{equation}
which we derive by taking $\s{X} = \eta N$ in (\ref{eq:5-11}).
Applying Proposition \ref{prop5-1}, we deduce
\begin{equation}
\mathcal{Z}(t)=D_{X}^{2}V(m,t)(\eta N)\label{eq:5-16}
\end{equation}
Recall the solution $\del{Y_{xmt}(s),\:Z_{xmt}(s)}$ of system (\ref{eq:Yxmt})-(\ref{eq:Zxmt}). 
From the regularity of $F$ and $F_{T}$ we can
differentiate in $x.$ Define 
\begin{equation}
\mathcal{Y}_{xmt}(s)= D Y_{xmt}(s),\;\mathcal{Z}_{xmt}(s)=D Z_{xmt}(s).\label{eq:5-17}
\end{equation}
Then from (\ref{eq:Yxmt})-(\ref{eq:Zxmt}) it follows 
\begin{align}
\mathcal{Y}_{xmt}(s) &=I-\dfrac{1}{\lambda}\int_{t}^{s}\mathcal{Z}_{xmt}(\tau)\dif \tau \label{eq:calYxmt}\\
\mathcal{Z}_{xmt}(s) &=\bb{E}\left[\left.\int_{s}^{T}D^{2}\dod{F}{m}(Y_{\cdot mt}(\tau) \otimes  m)(Y_{xmt}(\tau))\mathcal{Y}_{xmt}(\tau)\dif \tau+D^{2}\dod{F_{T}}{m}(Y_{\cdot mt}(T) \otimes  m)(Y_{xmt}(T))\mathcal{Y}_{xmt}(T)\right|\mathcal{W}_{t}^{s}\right]. \label{eq:calZxmt}
\end{align}

We claim that 
\begin{equation}
\mathcal{Y}(s)=\mathcal{Y}_{xmt}(s)\eta N,\;\mathcal{Z}(s)=\mathcal{Z}_{xmt}(s)\eta N\label{eq:5-19}
\end{equation}
is the solution of (\ref{eq:5-15}). It suffices to show that
\begin{equation}
D_{X}^{2}F(Y_{\cdot \,t}(\tau) \otimes  m)(\mathcal{Y}(\tau))=D^{2}\dod{}{m}F(Y_{\cdot mt}(\tau) \otimes  m)(Y_{xmt}(\tau))\mathcal{Y}_{xmt}(\tau)\label{eq:5-20}
\end{equation}
Now since $\widetilde{N}$ is independent of $\widetilde{Y}_{xmt}(\tau)$
and $\widetilde{\mathcal{Y}}_{xmt}(s)$ and has mean $0$, we have
\begin{equation}
\widetilde{\bb{E}}\int_{\bb{R}^n}{\color{black}D_1D_2\dod[2]{F}{m}}(Y_{\cdot mt}(\tau) \otimes  m)(Y_{xmt}(\tau),\widetilde{Y}_{xmt}(\tau))\widetilde{\mathcal{Y}}_{xmt}(s)\eta\widetilde{N}{\color{black}dm(x)}=0,\label{eq:5-21}
\end{equation}
and \eqref{eq:5-20} follows from \eqref{eq:5-21} plugged into formula (\ref{eq:2-232}).
An analogous statement holds for $F_{T}$, and so the claim follows.

Combining (\ref{eq:5-16}) and \eqref{eq:5-19} we deduce
\begin{equation}
D_{X}^{2}V(m,t)(\eta N)=\mathcal{Z}_{xmt}(t)\eta N
=D^{2}\dod{}{m}V(m,t)(x)\eta N\label{eq:5-22}
\end{equation}
and thus 
\begin{equation}\label{eq:D_X^2 V sigma N sigma N}
\begin{aligned}
\ip{D_{X}^{2}V(m,t)(\eta N)}{\eta N} &=\bb{E} \int_{{\mathbb R}^n}D^{2}\dod{}{m}V(m,t)(x)\eta N \cdot (\eta N) \dif m(x)\\ &=\int_{\bb{R}^n}\text{\text{tr }}\del{\eta{\color{black}\eta^*}D^{2}\dod{}{m}V(m,t)(x)}\dif m(x).
\end{aligned}
\end{equation}
We now plug \eqref{eq:5-14} and \eqref{eq:D_X^2 V sigma N sigma N} into the Bellman equation (\ref{eq:Bellman}) to obtain a PDE for $V(m,t)$.
Introducing the second order differential operator 
\begin{equation}
A_{x}\varphi(x)=-\dfrac{1}{2}\text{\text{tr }}\del{\eta{\color{black}\eta^*}D^{2}\varphi(x)}\label{eq:5-23}
\end{equation}
we obtain 
\begin{equation}\label{eq:Bellman reduced}
\begin{aligned}
&-\dfrac{\partial V}{\partial t}(m,t)+\int_{\bb{R}^n}A_{x}\dod{}{m}V(m,t)(x)\dif m(x)+\dfrac{1}{2\lambda}\int_{\bb{R}^n}\abs{D\dod{}{m}V(m,t)(x)}^{2}\dif m(x)=F(m),\\
&V(m,T)=F_{T}(m).
\end{aligned}
\end{equation}

\section{MASTER EQUATION}
\label{sec:master}

\subsection{THE EQUATION}

The Master Equation is given by
\begin{equation}\label{eq:master equation}
\begin{aligned}\begin{cases}
&-\dfrac{\partial U}{\partial t}(x,m,t)+A_{x}U(x,m,t)+\int_{\bb{R}^n}A_{\xi}\dod{}{m}U(\xi,m,t)(x)\dif m(\xi)
+\dfrac{1}{2\lambda}|D U(x,m,t)|^{2}\\
&\quad +\dfrac{1}{\lambda}\int_{\bb{R}^n}D_{\xi}U(\xi,m,t)\cdot D_{\xi}\dod{}{m}U(\xi,m,t)(x)\dif m(\xi)=\dod{}{m}F(m)(x),\\
&U(x,m,T)=\dod{}{m}F_{T}(m)(x).
\end{cases}\end{aligned}
\end{equation}
Equation \eqref{eq:master equation} can be derived from \eqref{eq:Bellman reduced} by taking a functional derivative.
Indeed, let us define $U$ to be the functional derivative of $V$, i.e.
\begin{equation}
U(x,m,t)=\dod{}{m}V(m,t)(x).\label{eq:U def}
\end{equation}
Formally, one differentiates Equation \eqref{eq:Bellman reduced} with respect to $m$ to see that \eqref{eq:master equation} is satisfied.
we shall justify this calculation below in Section \ref{sec:existence}.

The goal of this section is to establish that the Master Equation \eqref{eq:master equation} has a solution, which is given by \eqref{eq:U def}.
We use some additional regularity on the data, following the definition below.
\begin{defn}
	Let $c > 0$ be a fixed constant.
	We say that a functional $F:\s{P}_2(\bb{R}^n) \to \bb{R}$ is of class $\s{S}_c$ provided that $D\dod{}{m}F$, $D^{2}\dod{}{m}F$, $D_1\dod[2]{}{m}F$, $D_2D_1\dod[2]{}{m}F$, $D^{3}\dod{}{m}F$, $D_1^{2}\dod[2]{}{m}F,$ and $D_1^{2}D_2\dod[2]{}{m}F$ all exist, are continuous, and satisfy the following estimates for all $m \in \s{P}_2(\bb{R}^n)$ and all $x,\tilde x \in \bb{R}^n$:
	\begin{equation}\label{eq:6-15}
	\begin{aligned}
	\abs{D\dod{}{m}F(m)(x)}\leq c(1+|x|), \quad
	\abs{D^2\dod{}{m}F(m)(x)}\leq c,
	\quad 
	\abs{D^{3}\dod{}{m}F(m)(x)}\leq c,\\
	\abs{D_1\dod[2]{}{m}F(m)(x,\widetilde{x})}\leq c(1+|\widetilde{x}|), \quad
	\abs{D_2D_1\dod[2]{}{m}F(m)(x,\widetilde{x})}\leq c, \\
	\abs{D_1^{2}\dod[2]{}{m}F(m)(x,\widetilde{x})}\leq c(1+|\widetilde{x}|), \quad
	\abs{D_1^{2}D_2\dod[2]{}{m}F(m)(x,\widetilde{x})}\leq c.
	\end{aligned}
	\end{equation}
\end{defn}

we shall now state our main result.
\begin{thm} \label{thm:master equation}
	Assume (\ref{eq:3-1}), (\ref{eq:3-2}), (\ref{eq:3-3}), (\ref{eq:3-4}), (\ref{eq:3-6}), 
	(\ref{eq:5-2}), (\ref{eq:5-3}), (\ref{eq:5-4}), (\ref{eq:5-5}), 
	and (\ref{eq:5-50}).
	Assume also that $F$ is of class $\s{S}_c$ and $F_T$ is of class $\s{S}_{c_T}$ for some constants $c,c_T > 0$.
	Then there exists a $\lambda_T$ large enough, depending on $c,c_T,$ and $T$, such that if $\lambda \geq \lambda_T$, the following assertion holds.
	
	Let $V$ be the value function defined in \eqref{eq:V(Xotimes m,t) def}, and let $U$ be given by \eqref{eq:U def}.
	Then $U$ is a solution of Equation \eqref{eq:master equation} in a pointwise sense.
	The derivative $\dpd{U}{t}$ is a right-hand derivative, while all the other derivatives appearing in the equation exist.  
\end{thm}

{\color{black}
\begin{thm} \label{thm:continuity}
Under the assumptions specified in Theorem \ref{thm:master equation}, we further assume the following:\\
$(i)$ The functions $D \dod{F}{m}(m)(x)$, $D \dod{F_T}{m}(m)(x)$, $D^{2}\dod{F}{m}(m)(x)$, $D^{2}\dod{F_T}{m}(m)(x)$, $D^{3}\dod{F}{m}(m)(x)$ and $D^{3}\dod{F_T}{m}(m)(x)$ are jointly Lipschitz continuous in $(m,x)\in \mathcal{P}_2(\mathbb{R}^n)\times\mathbb{R}^n$. \footnote{For instance, there exists a positive constant $C$ such that 
\begin{align*}
\left|D \dod{F}{m}(m')(x')-D \dod{F}{m}(m)(x)\right|\leq& C\left(|x-x'|+W_2(m',m)\right).
\end{align*}}\\
$(ii)$ The functions $D_1\dod[2]{F}{m}(m)(x,\widetilde{x})$, $D_1\dod[2]{F_T}{m}(m)(x,\widetilde{x})$, $D_2D_1\dod[2]{}{m}F(m)(x,\widetilde{x})$, $D_2D_1\dod[2]{}{m}F_T(m)(x,\widetilde{x})$, $D_1^2\dod[2]{F}{m}(m)(x,\widetilde{x})$, $D_1^2\dod[2]{F_T}{m}(m)(x,\widetilde{x})$, $D_2D_1^2\dod[2]{}{m}F(m)(x,\widetilde{x})$ and $D_2D_1^2\dod[2]{}{m}F_T(m)(x,\widetilde{x})$ are jointly Lipschitz continuous in $(m,x,\widetilde{x})\in \mathcal{P}_2(\mathbb{R}^n)\times\mathbb{R}^n\times\mathbb{R}^n$. \footnote{For instance, there exists a positive constant $C$ such that
\begin{align*}
\left|D_2D_1\dod[2]{}{m}F(m')(x',\widetilde{x}')-D_2D_1\dod[2]{}{m}F(m)(x,\widetilde{x})\right|\leq& C\left(|x-x'|+|\tilde{x}'-\tilde{x}|+W_2(m',m)\right).
\end{align*}}
Then there exists a $\lambda'_T$ large enough, depending on $c,\,c_T,\,T$ and the Lipschitz constants in the above assumptions $(i)$ and $(ii)$, such that if $\lambda \geq \lambda'_T$, the solution $U$ to Equation \eqref{eq:master equation}  given in Theorem \ref{thm:master equation} possesses the following regularity:\\
(a) for $t\in[0,T]$, $D_x U(x,m,t)$ and $D_x^2 U(x,m,t)$ are jointly Lipschitz continuous in $(m,x)\in \mathcal{P}_2(\mathbb{R}^n)\times\mathbb{R}^n$, and  this joint continuity is uniformly held in $t$;\\
(b) for $t\in[0,T]$ and each $\xi\in\bb{R}^n$, $D_x \dfrac{d}{dm}U(x,m,t)(\xi)$ and $D_x^2 \dfrac{d}{dm}U(x,m,t)(\xi)$ are jointly Lipschitz continuous in $(m,x)\in \mathcal{P}_2(\mathbb{R}^n)\times\mathbb{R}^n$, and  this joint continuity is uniformly held in $t$;\\
(c) for $t\in[0,T]$ and $(m,x)\in \mathcal{P}_2(\mathbb{R}^n)\times\mathbb{R}^n$, $D_x \dfrac{d}{dm}U(x,m,t)(\xi)$ and $D_x^2 \dfrac{d}{dm}U(x,m,t)(\xi)$ are Lipschitz continuous in $\xi\in  \mathbb{R}^n$, and  this joint continuity is uniformly held in $t$, $m$ and $x$;\\ 
(d) for any compact set $K\subset \bb{R}^n$ and $t\in[0,T]$, $D_x \dfrac{d}{dm}U(x,m,t)(\xi)$ and $D_x^2 \dfrac{d}{dm}U(x,m,t)(\xi)$ are jointly Lipschitz continuous in $(m,x,\xi)\in \mathcal{P}_2(\mathbb{R}^n)\times\mathbb{R}^n\times K$, and  this joint continuity is uniformly held in $t$.
\end{thm}

\begin{rem}
In \cite{CDLL}, Mean Field Game problem was studied in $n$-dimensional torus $\bb{T}^n$ which is always compact. They obtained the regularity of their solutions to that Master equations under their Mean Field Game setting as follows: for $t\in[0,T]$ and $m\in\s{P}(\bb{T}^n)$, $U(\cdot,m,t)$ and $\dfrac{d}{dm}U(\cdot,m,t)(\cdot)$ are in the H\"{o}lder spaces $C^{d+2}(\bb{T}^n)$ and $C^{d+2}(\bb{T}^n)\times C^{d+1}(\bb{T}^n)$ respectively, independently of $(t,m)$, and $m\in\s{P}(\bb{T}^n)\mapsto U(\cdot,m,t)\in C^{d+2}$ and $m\in\s{P}(\bb{T}^n)\mapsto\dfrac{d}{dm}U(\cdot,m,t)(\cdot)\in C^{d+2}\times C^{d+1}$ are Lipschitz continuous functions uniformly in time. They used $W_1$-norm for $m\in\s{P}(\bb{T}^n)$ and studied H\"{o}lder continuity for the state variables $x\in\bb{T}^n$ and $\xi\in\bb{T}^n$, while here in our article we use $W_2$-norm for $m\in\s{P}_2(\bb{R}^n)$ and study Lipschitz continuity for the state variables $x\in\bb{R}^n$ and $\xi\in\bb{R}^n$.
\end{rem}

\begin{rem}\label{rem6-4}
Assume further that:\\
There exists a positive constant $\lambda_F$ such that, for any $m\in\mathcal{P}_2(\mathbb{R}^n)$ and any $X,Y\in\mathcal{H}_m$,
\begin{multline}
\bb{E}\int_{\bb{R}^n}D^{2}\dod{F}{m}(X \otimes  m)(X(x))Y(x)\cdot Y(x)\dif m(x)\\
+\bb{E}\tilde{\bb{E}}\int_{\bb{R}^n}\int_{\bb{R}^n}D_2D_1\dod[2]{F}{m}(X \otimes  m)(\tilde{X}(\tilde{x}),X(x))\tilde{Y}(\tilde{x})\cdot Y(x)\dif m(\tilde{x})\dif m(x)\geq
\lambda_F\bb{E}\int_{\bb{R}^n}|Y(x)|^{2}\dif m(x), \label{convexF}
\end{multline}
\begin{multline}
\bb{E}\int_{\bb{R}^n}D^{2}\dod{F_T}{m}(X \otimes  m)(X(x))Y(x)\cdot Y(x)\dif m(x)\\
+\bb{E}\tilde{\bb{E}}\int_{\bb{R}^n}\int_{\bb{R}^n}D_2D_1\dod[2]{F_T}{m}(X \otimes  m)(\tilde{X}(\tilde{x}),X(x))\tilde{Y}(\tilde{x})\cdot Y(x)\dif m(\tilde{x})\dif m(x)\geq \lambda_F\bb{E}\int_{\bb{R}^n}|Y(x)|^{2}\dif m(x),\label{convexFT}
\end{multline}
The proof can be found in Appendix \ref{ap:D}. Remark \ref{rem6-4} are together a little stronger than those of \eqref{eq:3-33}-\eqref{eq:5-34}, where $\lambda_F$ can be allowed to be negative.
Then we have an alternative approach of deriving the individual continuity either (but not both) in $x$ or in $\xi$ but not in $m$ yet with a smaller value of $\lambda'_T$; see Step $1'$ and $3'$ fo the following proof for more details. 
\end{rem}}
{\color{black}
\begin{rem}\label{rem6-5}
As a representative example, $F(m)=\int_{\mathbb{R}^n}f(x,m)dm(x)$ with a function $f:\bb{R}^n\times\s{P}_2(\mathbb{R}^n)\to\bb{R}$ possessing derivatives $D_xf(x,m)$ and $D^2_xf(x,m)$, and linear functional derivative $\dod{f}{m}(x,m)(\tilde{x})$ such that its derivatives $D_x \dod{f}{m}(x,m)(\tilde{x})$, $D_{\tilde{x}} \dod{f}{m}(x,m)(\tilde{x})$ are all continuous in $x$ for each fixed $m\in\s{P}_2(\mathbb{R}^n)$. 
Define $\phi(x,m):=\dod{F}{m}(m)(x)=f(x,m)+\int_{\mathbb{R}^n}\dod{f}{m}(\widetilde{x},m)(x)dm(\widetilde{x})$. Simple calculus gives
\begin{align*}
\left(D^{2}\dod{F}{m}(X \otimes  m)(X(x))Y(x)\right)\cdot Y(x)=Y(x)^\top D^2_x\phi(X(x),X \otimes  m)Y(x),
\end{align*}
and 
\begin{align*}
\dod[2]{F}{m}(X \otimes  m)(\tilde{X}(\tilde{x}),X(x))=\dod{\phi}{m}(X(x),X \otimes  m)(\tilde{X}(\tilde{x})).
\end{align*} 
Therefore, \eqref{convexF} is equivalent to
\begin{align}\label{convexFnew}
&\bb{E}\int_{\bb{R}^n}\h{-3pt}Y(x)^\top \h{-1.5pt} D^2_x\phi(X(x),X \otimes  m)Y(x)dm(x)
\h{-1.5pt}+\h{-1.5pt}\bb{E}\tilde{\bb{E}}\int_{\bb{R}^n}
\h{-1.5pt}\int_{\bb{R}^n} \h{-3pt} Y(x)^\top
\h{-1.5pt}D_{\tilde{x}}D_x\dod{\phi}{m}(X(x),X \otimes  m)(\tilde{X}(\tilde{x}))\tilde{Y}(\tilde{x})dm(\tilde{x})dm(x)\\\nonumber
&\geq\lambda_F\bb{E}\int_{\bb{R}^n}|Y(x)|^{2}\dif m(x).
\end{align}

\noindent In particular, as long as $\phi(x,m)$ is convex in $x$ for each $m\in\s{P}_2(\mathbb{R}^n)$, i.e., $D^2_x\phi(x,m)$ is positive definite, so that there exists a positive constant $\lambda_\phi$ such that $\bb{E}\int_{\bb{R}^n}Y(x)^\top D^2_x\phi(X(x),X \otimes  m)Y(x)dm(x)\geq \lambda_\phi\bb{E}\int_{\bb{R}^n}|Y(x)|^{2}\dif m(x)$. Therefore, one can take $\lambda_F=\lambda_\phi$ if 
\begin{align*}
\bb{E}\tilde{\bb{E}}\int_{\bb{R}^n}\int_{\bb{R}^n} Y(x)^\top D_{\tilde{x}}D_x\dod{\phi}{m}(X(x),X \otimes  m)(\tilde{X}(\tilde{x}))\tilde{Y}(\tilde{x})dm(\tilde{x})dm(x)\geq 0,
\end{align*} 
which is equivalent to Lasry-Lions monotonicity assumption (2.5) in \cite{CDLL} now applying to $\phi(x,m)$.
As a particular case of this formulation, we further consider the well-known linear-quadratic one, also recall \cite{BSY},
\begin{align*}
f(x,m)=\frac{1}{2}x^\top Q x+\frac{1}{2}y^\top Q y,
\end{align*} 
where $y:=\int_{\mathbb{R}^n}S\widetilde{x}\,dm(\widetilde{x})$ and $Q$, $\bar{Q}$ and $S$ are constant $n\times n$ matrices.
Then $\phi(x,m):=\frac{1}{2}x^\top Q x+\frac{1}{2}y^\top \bar{Q} y+x^\top S^\top\bar{Q}y$. 
The \eqref{convexFnew} can now be further rewritten as:
\begin{align}
&\bb{E}\int_{\bb{R}^n}Y(x)^\top Q Y(x)dm(x)+\left(\bb{E}\int_{\bb{R}^n} SY(x) dm(x)\right)^\top \bar{Q}  \left(\tilde{\bb{E}}\int_{\bb{R}^n}S\tilde{Y}(\tilde{x})dm(\tilde{x})\right)\geq\lambda_F\bb{E}\int_{\bb{R}^n}|Y(x)|^{2}\dif m(x).
\end{align}
As long as $Q$ is positive definite and $\bar{Q}$ is non-negative definite, both \eqref{convexF} and \eqref{convexFnew} can be satisfied.
\end{rem}
}

 \begin{rem}
The Lipschitz continuity of the derivatives in the assumptions of this Theorem aims to ensure that the classical solution to the master equation has Lipschitz continuous derivatives. The assumptions can be weaker in order to obtain the classical existence of the solution to the master equation, however, we do not expect the solution to have Lipschitz continuous derivatives under such weak assumptions. We refer the reader to Proposition 5.4 of \cite{BTY} with the analysis under weaker assumptions, namely, only the continuity of the derivatives, not Lipschitzness, can ensure the continuity of various derivatives of the solution to the master equation.
 \end{rem}

The master equation \eqref{eq:master equation} is interpreted in mean field game theory as the limiting equation for a Nash system \cite{CDLL}.
Here we cannot interpret it in terms of Nash equilibrium, unless the corresponding Nash game is \emph{potential} \cite{RCD}.
Instead we focus on the importance of $U$ defined in \eqref{eq:U def} as a decoupling field for the system of necessary and sufficient conditions (\ref{eq:Yxmt})-(\ref{eq:Zxmt}).
To see this, let $\del{Y_{xmt}(s), Z_{xmt}(s)}$ be the solution of System \eqref{eq:Yxmt}-\eqref{eq:Zxmt}.
Combine Propositions \ref{prop4-3} and \ref{prop4-4} to see that $Y_{xmt}(s)$ is the solution of a stochastic
differential equation
\begin{equation}
Y_{xmt}(s)=x-\dfrac{1}{\lambda}\int_{t}^{s}DU(Y_{xmt}(\tau),Y_{\cdot mt}(\tau) \otimes  m,\tau)\dif \tau+\eta(w(s)-w(t)).\label{eq:6-2}
\end{equation}
In other words, with the function $U = \od{V}{m}$ in hand, we can find the optimal trajectory by solving \eqref{eq:6-2} for $Y_{xmt}(s)$, after which the adjoint state $Z_{xmt}(s)$ is given by Equation \eqref{eq:4-17} from Proposition \ref{prop4-4}.

Before proving Theorem \ref{thm:master equation}, we make a remark about uniqueness.
Under certain assumptions, one can show that $U$ is unique using the Lasry-Lions monotonicity argument \cite{LasLio07}, which goes as follows.
Let $U_1,U_2$ be two solutions, and for $i = 1,2$ let $m_i(t)$ be the weak solution of the Fokker-Planck equation
\begin{equation}
\label{eq:fp}
\dpd{m_i}{t} + A_x m_i - \frac{1}{\lambda}\mathrm{div}\del{DU_i(x,m_i(t),t)m_i} = 0, \ m_i(\tau) = m \in \s{P}_2(\bb{R}^n)
\end{equation}
for some $\tau \in [0,T]$.
One checks that
\begin{multline}
\label{eq:energy}
\int_{{\mathbb R}^n} \del{\dod{F_{T}}{m}(m_1(T))(x)-\dod{F_{T}}{m}(m_2(T))(x)}\dif \h{1.5pt} (m_1(T)-m_2(T))(x)\\
 + \int_\tau^T \int_{{\mathbb R}^n} \del{\dod{F}{m}(m_1(t))(x)-\dod{F_{T}}{m}(m_2(t))(x)}\dif \h{1.5pt} (m_1(t)-m_2(t))(x)\dif t\\
 +  \frac{1}{2\lambda}\int_\tau^T \int_{{\mathbb R}^n} \abs{DU_1(x,m_1(t),t) - DU_2(x,m_2(t),t)}^2\dif \h{1.5pt} (m_1(t)+m_2(t))\dif t = 0.
\end{multline}
Formally, \eqref{eq:energy} is derived by using $DU_j(x,m_j(t),t)$, $j=1,2$ as a test function in \eqref{eq:fp}, $i=1,2$ and then subtracting.
A typical assumption would be that $F$ and $F_T$ are monotone in the sense that
\begin{equation}
\int_{{\mathbb R}^n} \del{\dod{F}{m}(m_1)(x)-\dod{F_{T}}{m}(m_2)(x)}\dif \h{1.5pt} (m_1-m_2)(x) \geq 0 \quad \forall m_1,m_2 \in \s{P}_2(\bb{R}^n).
\end{equation}
In this case, \eqref{eq:energy} immediately implies $DU_1(x,m,\tau) = DU_2(x,m,\tau)$ (at least on the support of $m$).
One can exploit this to deduce that $U_1 = U_2$, provided both solutions are sufficiently regular.

In our framework, sufficient regularity to prove uniqueness in this way is a delicate issue.
Although uniqueness for the Bellman equation \eqref{eq:Bellman} is obtained using convexity of the underlying control problem, we do not find such a straightforward path to uniqueness for the Master equation \eqref{eq:master equation}.
We leave this issue for future study.

\subsection{EXISTENCE OF A SOLUTION TO THE MASTER EQUATION}
\label{sec:existence}

In this subsection we prove 
Theorem \ref{thm:master equation}.
The whole argument consists in differentiating the Bellman equation \eqref{eq:Bellman reduced} with respect to $m$.
We just need to show that this step is justified. 
Recall that $\del{Y_{xmt}(s), Z_{xmt}(s)}$ is the solution of System \eqref{eq:Yxmt}-\eqref{eq:Zxmt}.
Consider also the solution $\del{\mathcal{Y}_{xmt}(s), \mathcal{Z}_{xmt}(s)}$ of System \eqref{eq:calYxmt}-\eqref{eq:calZxmt}, which is in fact the derivative of $\del{Y_{xmt}(s), Z_{xmt}(s)}$ with respect to $x$ (see Equation \eqref{eq:5-17}).
By Proposition \ref{prop4-4} we have 
\begin{equation}
Z_{xmt}(t)=D U(x,m,t),\;\mathcal{Z}_{xmt}(t)=D^{2}U(x,m,t).\label{eq:6-4}
\end{equation}
By (\ref{eq:6-4}), we can write Bellman
equation \eqref{eq:Bellman reduced} as
\begin{equation}
\begin{aligned}
-\dfrac{\partial V}{\partial t}(m,t)-\dfrac{1}{2}\int_{\bb{R}^n}\text{tr}\left(\eta{\color{black}\eta^*}\mathcal{Z}_{\xi mt}(t)\right)\dif m(\xi)+\dfrac{1}{2\lambda}\int_{\bb{R}^n}|Z_{\xi mt}(t)|^{2}\dif m(\xi)=F(m),\\
V(m,T)=F_{T}(m).
\end{aligned}\label{eq:6-9}
\end{equation}
The key step now is to check that
$Z_{\xi mt}(t)$ and $\mathcal{Z}_{\xi mt}(t)$ have functional derivatives
with respect to $m$. In fact we shall show the differentiability with respect to
$m$ of $Y_{\xi mt}(s),Z_{\xi mt}(s),\mathcal{Y}_{\xi mt}(s),$ and $\mathcal{Z}_{\xi mt}(s)$, for any $s>t$. We shall label their derivatives $\bar{Y}_{mt}(s,\xi,x),\:\bar{Z}_{mt}(s,\xi,x),\:\bar{\mathcal{Y}}_{mt}(s,\xi,x),$ and $\bar{\mathcal{Z}}_{mt}(s,\xi,x)$, respectively.
We obtain them by taking the derivative in $m$
of the systems \eqref{eq:Yxmt}-\eqref{eq:Zxmt} and \eqref{eq:calYxmt}-\eqref{eq:calZxmt}, being careful to replace $x$ with $\xi$ in these expressions. 
See Section \ref{sec:linear systems} below.

Given the existence of these derivatives, it is straightforward to differentiate \eqref{eq:6-9} with respect to $m$, from which we obtain
\begin{multline}\label{eq:6-10}
-\dfrac{\partial U}{\partial t}(x,m,t)-\dfrac{1}{2}\text{tr}\left(\eta{\color{black}\eta^*}\mathcal{Z}_{xmt}(t)\right)-\dfrac{1}{2}\int_{\bb{R}^n}\text{tr}\left(\eta{\color{black}\eta^*}\bar{\mathcal{Z}}_{mt}(t,\xi,x)\right)\dif m(\xi)\\
+\dfrac{1}{2\lambda}\abs{Z_{xmt}(t)}^{2}+\dfrac{1}{\lambda}\int_{\bb{R}^n}Z_{\xi mt}(t)\cdot \bar{Z}_{mt}(t,\xi,x)\dif m(\xi)=\dod{}{m}F(m)(x).
\end{multline}
Substituting \eqref{eq:6-4} into \eqref{eq:6-10} and using the definition of $\bar{Z}_{mt}$ and $\bar{\s{Z}}_{mt}$, we see that the Master Equation \eqref{eq:master equation} holds.

\subsubsection{LINEAR SYSTEMS FOR THE DERIVATIVES} \label{sec:linear systems}

We define the pair $\del{\bar{Y}_{mt}(s,\xi,x),\:\bar{Z}_{mt}(s,\xi,x)}$ to be the solution of the linear system 
\begingroup
\allowdisplaybreaks
\begin{align}
\bar{Y}_{mt}(s,\xi,x) &=-\dfrac{1}{\lambda}\int_{t}^{s}\bar{Z}_{mt}(\tau,\xi,x)\dif \tau,\label{eq:6-11}\\
\label{eq:6-12}
\bar{Z}_{mt}(s,\xi,x) &=\bb{E}\left[ \int_{s}^{T} D^{2}\dod{F}{m}(Y_{\cdot mt}(\tau) \otimes  m)(Y_{\xi mt}(\tau))\bar{Y}_{mt}(\tau,\xi,x)\dif \tau \right.\\
\nonumber
&\hspace{-1.5cm}
+ D^{2}\dod{}{m}F_{T}(Y_{\cdot mt}(T) \otimes  m)(Y_{\xi mt}(T))\bar{Y}_{mt}(T,\xi,x)\\
\nonumber
&\hspace{-1.5cm}+\widetilde{\bb{E}}\int_{s}^{T}\int_{\bb{R}^n}D_2D_1\dod[2]{F}{m}(Y_{\cdot mt}(\tau) \otimes  m)(Y_{\xi mt}(\tau),\widetilde{Y}_{\zeta mt}(\tau))\widetilde{\bar{Y}}_{mt}(\tau,\zeta,x)\dif m(\zeta) \dif \tau\\
\nonumber
&\hspace{-1.5cm} +\widetilde{\bb{E}}\int_{\bb{R}^n}D_2D_1\dod[2]{F_{T}}{m}(Y_{\cdot mt}(T) \otimes  m)(Y_{\xi mt}(T),\widetilde{Y}_{\zeta mt}(T))\widetilde{\bar{Y}}_{mt}(T,\zeta,x)\dif m(\zeta)\\
\nonumber
&\hspace{-1.5cm}+ \left.\left.\int_{s}^{T}\widetilde{\bb{E}}D_1\dod[2]{F}{m}(Y_{\cdot mt}(\tau) \otimes  m)(Y_{\xi mt}(\tau),\widetilde{Y}_{xmt}(\tau))\dif \tau+\widetilde{\bb{E}}D_1\dod[2]{F_T}{m}(Y_{\cdot mt}(T) \otimes  m)(Y_{\xi mt}(T),\widetilde{Y}_{xmt}(T))\right|\mathcal{W}_{t}^{s}\right],
\end{align}
\endgroup
and the pair $\del{\bar{\mathcal{Y}}_{mt}(s,\xi,x), \bar{\mathcal{Z}}_{mt}(s,\xi,x)}$
is the solution of 
\begin{align}
\bar{\mathcal{Y}}_{mt}(s,\xi,x) &=-\dfrac{1}{\lambda}\int_{t}^{s}\bar{\mathcal{Z}}_{mt}(\tau,\xi,x)\dif \tau,\label{eq:6-13}\\
\label{eq:6-14}
\bar{\mathcal{Z}}_{mt}(s,\xi,x) &
 =\bb{E}\left[\int_{s}^{T}D^{2}\dod{}{m}F(Y_{\cdot mt}(\tau) \otimes  m)(Y_{\xi mt}(\tau))\bar{\mathcal{Y}}_{mt}(\tau,\xi,x)\dif \tau \right.\\
 \nonumber 
 &\hspace{1cm} \left. +\left.D^{2}\dod{}{m}F_{T}(Y_{\cdot mt}(T) \otimes  m)(Y_{\xi mt}(T))\bar{\mathcal{Y}}_{mt}(T,\xi,x)\right|\mathcal{W}_{t}^{s}\right]\\
\nonumber
& \hspace{-1cm}+\bb{E}\left[ \int_{s}^{T}\left(D^{3}\dod{}{m}F(Y_{\cdot mt}(\tau) \otimes  m)(Y_{\xi mt}(\tau))\bar{Y}_{mt}(\tau,\xi,x)+\widetilde{\bb{E}}D_1^{2}\dod[2]{}{m}F(Y_{\cdot mt}(\tau) \otimes  m)(Y_{\xi mt}(\tau),\widetilde{Y}_{xmt}(\tau))\right.\right.\\
\nonumber
& \hspace{-1cm}+\left.\widetilde{\bb{E}}\int_{\bb{R}^n}D_1^{2}D_2\dod[2]{}{m}F(Y_{\cdot mt}(\tau) \otimes  m)(Y_{\xi mt}(\tau),\widetilde{Y}_{\zeta mt}(\tau))\widetilde{\bar{Y}}_{mt}(\tau,\zeta,x)\dif m(\zeta)\right)\mathcal{Y}_{\xi mt}(\tau)\dif \tau\\
\nonumber
& \hspace{-1cm}
+ \left(D^{3}\dod{}{m}F_{T}(Y_{\cdot mt}(T) \otimes  m)(Y_{\xi mt}(T))\bar{Y}_{mt}(\tau,\xi,x)+\widetilde{\bb{E}}D_1^{2}\dod[2]{}{m}F_{T}(Y_{\cdot mt}(T) \otimes  m)(Y_{\xi mt}(T),\widetilde{Y}_{xmt}(T))\right.\\
\nonumber
& \hspace{-1cm}
\left.\left.+\left.\widetilde{\bb{E}}\int_{\bb{R}^n}D_1^{2}D_2\dod[2]{}{m}F_{T}(Y_{\cdot mt}(T) \otimes  m)(Y_{\xi mt}(T),\widetilde{Y}_{\zeta mt}(T))\widetilde{\bar{Y}}_{mt}(T,\zeta,x)\dif m(\zeta)\right)\mathcal{Y}_{\xi mt}(T)\right|\mathcal{W}_{t}^{s}\right].
\end{align}
We now provide a result on the well-posedness of these systems.
\begin{prop}
\label{prop6-1} Assume that $F$ is of class $\s{S}_c$ and $F_T$ is of class $\s{S}_{c_T}$ for some constants $c,c_T > 0$.
Then there exists a $\lambda_T$ large enough, depending on $c,c_T,$ and $T$, such that if $\lambda \geq \lambda_T$, each of the systems
(\ref{eq:6-11})-(\ref{eq:6-12}) and (\ref{eq:6-13})-(\ref{eq:6-14})
has a unique solution, satisfying the estimates
\begin{gather}
\bb{E}\int_{\bb{R}^n}|\bar{Y}_{mt}(s,\xi,x)|^{2}\dif m(\xi)\leq C_{T}(1+|x|^{2}),\:\bb{E}\int_{\bb{R}^n}|\bar{Z}_{mt}(s,\xi,x)|^{2}\dif m(\xi)\leq C_{T}(1+|x|^{2}),\label{eq:6-16}\\
\bb{E}\int_{\bb{R}^n}|\bar{\mathcal{Y}}_{mt}(s,\xi,x)|^{2}\dif m(\xi)\leq C_{T}(1+|x|^{2}),\:\bb{E}\int_{\bb{R}^n}|\bar{\mathcal{Z}}_{mt}(s,\xi,x)|^{2}\dif m(\xi)\leq C_{T}(1+|x|^{2}).\label{eq:6-17}
\end{gather}
\end{prop}
The proof is given in Appendix D.

\begin{rem}
	\label{rem:a priori estimates without m}
	In fact, the a priori estimates \eqref{eq:6-16} and \eqref{eq:6-17} can be improved.
	We get, for $\lambda$ large enough,
	\begin{gather}
	\bb{E}|\bar{Y}_{mt}(s,\xi,x)|^{2}\leq C_{T}(1+|x|^{2}),\:\bb{E}|\bar{Z}_{mt}(s,\xi,x)|^{2}\leq C_{T}(1+|x|^{2}),\label{eq:a priori1}\\
	\bb{E}|\bar{\mathcal{Y}}_{mt}(s,\xi,x)|^{2}\leq C_{T}(1+|x|^{2}),\:\bb{E}|\bar{\mathcal{Z}}_{mt}(s,\xi,x)|^{2}\leq C_{T}(1+|x|^{2}).\label{eq:a priori2}
	\end{gather}
	To see this, we again use \eqref{eq:4-90} and the estimates \eqref{eq:6-15} applied to both $F$ and $F_T$ in (\ref{eq:6-12}), but this time also appealing to \eqref{eq:6-16}, already proved.
	By using the same argument as in the proof of Proposition \ref{prop6-1}, we derive \eqref{eq:a priori1}.
	The proof of \eqref{eq:a priori2} is similar.
\end{rem}

\subsubsection{Differentiating $Y_{\xi mt}(s)$, $Z_{\xi mt}(s)$, $\mathcal{Y}_{\xi mt}(s),$ and $\mathcal{Z}_{\xi mt}(s)$ with respect to $m$}
Using Proposition \ref{prop6-1}, it is now possible to verify that $\bar{Y}_{mt}(s,\xi,x)$,  $\bar{Z}_{mt}(s,\xi,x)$,  $\bar{\mathcal{Y}}_{mt}(s,\xi,x)$, and $\bar{\mathcal{Z}}_{mt}(s,\xi,x)$
are indeed the functional derivatives of $Y_{\xi mt}(s)$, $Z_{\xi mt}(s)$, $\mathcal{Y}_{\xi mt}(s),$ and $\mathcal{Z}_{\xi mt}(s)$, respectively.
First, we observe that these functionals are continuous with respect to $(m,x)$.
Indeed, if we take $(m_n,x_n) \to (m,x)$ and consider differences, e.g.~$\bar Y_{m_n t}(s,\xi,x_n) - \bar Y_{mt}(s,\xi,x)$, and consider the resulting system of equations satisfied by these differences, it is straightforward (but tedious) to show that these differences converge to zero.
(We also use Remark \ref{rem:a priori estimates without m}, which mean that our estimates will not depend on $m$.)

{\color{black}
\begin{proof}[Proof of Theorem \ref{thm:continuity}]
In the following lemma, we first prove the jointly Lipschitz continuity of $D_x U(x,m,t)$ in $(m,x)\in \mathcal{P}_2(\mathbb{R}^n)\times\mathbb{R}^n$ for each $t\in[0,T]$.
\begin{lem}
Suppose that, $D\dod{F}{m}(m)(x)$ and $D\dod{F_T}{m}(m)(x)$ are jointly Lipschitz continuous in $(m,x)\in \mathcal{P}_2(\mathbb{R}^n)\times\mathbb{R}^n$, and $F$ is of class $\s{S}_c$ and $F_T$ is of class $\s{S}_{c_T}$ for some constants $c,c_T > 0$, respectively. Then $(Y_{x mt}(s),Z_{x mt}(s))$ defined by \eqref{eq:Yxmt}-\eqref{eq:Zxmt} is\\
(i) Lipschitz continuous in $x\in \mathbb{R}^n$ for each $m\in\s{P}_2(\bb{R}^n)$ in the following sense: 
\begin{align}\label{LipsY}
\sup_{s\in[t,T]} \left|  Y_{x'mt}(s)-Y_{xmt}(s)\right| +\sup_{s\in[t,T]} \left|  Z_{x'mt}(s)-Z_{xmt}(s)\right|\leq \frac{\lambda (1+c_T+cT)}{\lambda-T(c_T+cT)}|x'-x|,
\end{align}
if $\lambda>T(c_T+cT)$ (a bit stronger than \eqref{eq:5-120} of Theorem \ref{theo5-1}); 
and it is also\\
(ii) Lipschitz continuous in $m\in \s{P}_2(\bb{R}^n)$ in the following sense: 
\begin{align}\label{iieq3}
\sup_{s\in[t,T]}\bb{E}\left|Y_{xm't}(s)-Y_{xmt}(s)\right|^2 \leq & \dfrac{2T^2C_{DdF}(1+T^2)}{\lambda^2-3T^2C_{DdF}(1+T^2)}  \left(\frac{\lambda}{\lambda-T(c_T+cT)}\right)^2 \cdot W_2^2(m,m'),\\\label{iieq4}
\sup_{s\in[t,T]}\bb{E}\left|Z_{xm't}(s)-Z_{xmt}(s)\right|^2\leq &\dfrac{2\lambda^2C_{DdF}(1+T^2)}{\lambda^2-3T^2C_{DdF}(1+T^2)}  \left(\frac{\lambda}{\lambda-T(c_T+cT)}\right)^2 \cdot W_2^2(m,m'),
\end{align}  
if $\lambda>T\sqrt{3C_{DdF}(1+T^2)}$ where the positive constant $C_{DdF}$ was defined in \eqref{CDdF}.
In particular, since $Z_{xmt}(t)=D_x U(x,m,t)$ is deterministic, $D_x U(x,m,t)$ is also jointly Lipschitz continuous in $(m,x)\in \mathcal{P}_2(\mathbb{R}^n)\times\mathbb{R}^n$ by \eqref{LipsY} and \eqref{iieq4}.
\label{lem 6.9}
\end{lem}
The proof is given in Appendix D.
\begin{rem}\label{rem6-8}
If we assume further that \eqref{convexF} and \eqref{convexFT} are valid, then we can show that, without invoking $\lambda>T\sqrt{3C_{DdF}(1+T^2)}$ but just only $\lambda>T(c_T+cT)$, 
\begin{align}\label{Lipsm}
\int_t^T \bb{E} \int_{\bb{R}^n}\left|Z_{xm't}(s)-Z_{xmt}(s)\right|^2 dm(x) ds+\int_t^T \bb{E}\int_{\bb{R}^n} \left|Y_{xm't}(s)-Y_{xmt}(s)\right|^2 dm(x) ds&\\\nonumber
+\bb{E}\int_{\bb{R}^n}\left|Y_{xm't}(T)-Y_{xmt}(T)\right|^2dm(x)\leq C(c,c_T,\lambda,\lambda_F,T)W_2^2(m,m'),&
\end{align}
\begin{align}
\label{LipsYm}
\int_t^T \bb{E}W_2^2(Z_{\cdot m't}(s) \otimes  m',Z_{\cdot mt}(s) \otimes  m)ds+\int_t^T \bb{E}W_2^2(Y_{\cdot m't}(s) \otimes  m',Y_{\cdot mt}(s) \otimes  m)ds&\\\nonumber
+\bb{E}W_2^2(Y_{\cdot m't}(T) \otimes  m',Y_{\cdot mt}(T) \otimes  m)
\leq  C(c,c_T,\lambda,\lambda_F,T)W_2^2(m,m').&
\end{align}
Note that, in all the above cases, $C(c,c_T,\lambda,\lambda_F,T)$ is a positive constant depending only on $c$, $c_T$, $\lambda$, $\lambda_F$, $T$ but independent of $m$ and $m'$.
\end{rem}
The proof of the claim in Remark \ref{rem6-8} is given in Appendix D.

\normalsize
The overall proof of the Lipschitz continuity of $(\bar{Y}_{mt}(s,\xi,x),\bar{Z}_{mt}(s,\xi,x))$ in $x$, $\xi$ and $m$ are divided into three steps, respectively, in the following under the assumption $\lambda>\max\big\{12T(c_T+cT),T\sqrt{30(c_T^2+c^2T^2)}\big\}$.

Step $1$. For any fixed $p\geq 1$, recall the condition \eqref{eq:3-6}, with the additional condition of larger value of $\lambda>6T \left(2(c_T^p+c^pT^p)\right)^{1/p}$; for instant, when $p=1$, just take $\lambda>6T \left(2(c_T^p+c^pT^p)\right)^{1/p}=12T(c_T+cT)$; we prove the Lipschitz continuity of $(\bar{Y}_{mt}(s,\xi,x),\bar{Z}_{mt}(s,\xi,x))$ in $x\in\bb{R}^n$ for each $(s,\xi,m)\in[t,T]\times\bb{R}^n\times\s{P}_2(\bb{R}^n)$ in the following.
By \eqref{eq:6-11}-\eqref{eq:6-12}, one obtains:
\small\begin{align}
\bar{Y}_{mt}(s,\xi,x')-\bar{Y}_{mt}(s,\xi,x) 
=-\dfrac{1}{\lambda}\int_{t}^{s}\left(\bar{Z}_{mt}(\tau,\xi,x')-\bar{Z}_{mt}(\tau,\xi,x)\right)\dif \tau,\label{eq6-41}
\end{align}
and \\
\begin{align}\nonumber
&\bar{Z}_{mt}(s,\xi,x')-\bar{Z}_{mt}(s,\xi,x) \\\nonumber
=&\bb{E}\left[ \int_{s}^{T} D^{2}\dod{F}{m}(Y_{\cdot m t}(\tau) \otimes  m)(Y_{\xi mt}(\tau))\left(\bar{Y}_{mt}(\tau,\xi,x')-\bar{Y}_{mt}(\tau,\xi,x)\right)\dif \tau \right.\\
\nonumber
&+ D^{2}\dod{}{m}F_{T}(Y_{\cdot mt}(T) \otimes  m)(Y_{\xi mt}(T))\left(\bar{Y}_{mt}(T,\xi,x')-\bar{Y}_{mt}(T,\xi,x)\right)\\
\nonumber
&+\widetilde{\bb{E}}\int_{s}^{T}\int_{\bb{R}^n}D_2D_1\dod[2]{F}{m}(Y_{\cdot mt}(\tau) \otimes  m)(Y_{\xi mt}(\tau),\widetilde{Y}_{\zeta mt}(\tau))\left(\widetilde{\bar{Y}}_{mt}(\tau,\zeta,x')-\widetilde{\bar{Y}}_{mt}(\tau,\zeta,x)\right)\dif m(\zeta) \dif \tau\\
\nonumber
&+\widetilde{\bb{E}}\int_{\bb{R}^n}D_2D_1\dod[2]{F_{T}}{m}(Y_{\cdot mt}(T) \otimes  m)(Y_{\xi mt}(T),\widetilde{Y}_{\zeta mt}(T))\left(\widetilde{\bar{Y}}_{mt}(T,\zeta,x')-\widetilde{\bar{Y}}_{mt}(T,\zeta,x)\right)\dif m(\zeta)\\\nonumber
&+\int_{s}^{T}\widetilde{\bb{E}}D_1\dod[2]{F}{m}(Y_{\cdot mt}(\tau) \otimes  m)(Y_{\xi mt}(\tau),\widetilde{Y}_{x'mt}(\tau))\dif \tau+\widetilde{\bb{E}}D_1\dod[2]{F_T}{m}(Y_{\cdot mt}(T) \otimes  m)(Y_{\xi mt}(T),\widetilde{Y}_{x'mt}(T))\\\label{eq6-42}
&- \left.\left.\int_{s}^{T}\widetilde{\bb{E}}D_1\dod[2]{F}{m}(Y_{\cdot mt}(\tau) \otimes  m)(Y_{\xi mt}(\tau),\widetilde{Y}_{xmt}(\tau))\dif \tau-\widetilde{\bb{E}}D_1\dod[2]{F_T}{m}(Y_{\cdot mt}(T) \otimes  m)(Y_{\xi mt}(T),\widetilde{Y}_{xmt}(T))\right|\mathcal{W}_{t}^{s}\right].
\end{align}
Note that, by \eqref{eq:6-15} and \eqref{xLipCon}, we see that,
\begin{align}\nonumber
&\left| \widetilde{\bb{E}}D_1\dod[2]{F}{m}(Y_{\cdot mt}(s) \otimes  m)(Y_{\xi mt}(s),\widetilde{Y}_{x'mt}(s))-\widetilde{\bb{E}}D_1\dod[2]{F}{m}(Y_{\cdot mt}(s) \otimes  m)(Y_{\xi mt}(s),\widetilde{Y}_{xmt}(s)) \right|\\\nonumber
=&\left|\int_0^1 \widetilde{\bb{E}}D_2 D_1\dod[2]{F}{m}(Y_{\cdot mt}(s) \otimes  m)\left(Y_{\xi mt}(s),\theta \widetilde{Y}_{x'mt}(s)+(1-\theta)\widetilde{Y}_{xmt}(s)\right) \cdot \left(\widetilde{Y}_{x'mt}(s)-\widetilde{Y}_{xmt}(s)\right) d\theta \right|\\\label{eq6-43}
\leq & \frac{c\lambda}{\lambda-T(c_T+cT)}\cdot \left|x'-x\right|,
\end{align}
and, 
\begin{align}\nonumber
&\left| \widetilde{\bb{E}}D_1\dod[2]{F_T}{m}(Y_{\cdot mt}(T) \otimes  m)(Y_{\xi mt}(T),\widetilde{Y}_{x'mt}(T))-\widetilde{\bb{E}}D_1\dod[2]{F_T}{m}(Y_{\cdot mt}(T) \otimes  m)(Y_{\xi mt}(T),\widetilde{Y}_{xmt}(T)) \right|\\\nonumber
=&\left|\int_0^1 \widetilde{\bb{E}}D_2 D_1\dod[2]{F_T}{m}(Y_{\cdot mt}(T) \otimes  m)\left(Y_{\xi mt}(T),\theta \widetilde{Y}_{x'mt}(T)+(1-\theta)\widetilde{Y}_{xmt}(T)\right) \cdot \left(\widetilde{Y}_{x'mt}(T)-\widetilde{Y}_{xmt}(T)\right) d\theta \right|\\\label{eq6-44}
\leq &  \frac{c_T\lambda}{\lambda-T(c_T+cT)}\cdot \left|x'-x\right|.
\end{align}
First, by the simple algebra, we have, 
\begin{align*}
\sup_{s\in[t,T]}\sup_{\xi\in\bb{R}^n}\bb{E} \left|\bar{Y}_{mt}(s,\xi,x')-\bar{Y}_{mt}(s,\xi,x)\right|^p \leq \dfrac{|T-t|^p}{\lambda^p}\sup_{s\in[t,T]}\sup_{\xi\in\bb{R}^n}\bb{E} \left|\bar{Z}_{mt}(s,\xi,x')-\bar{Z}_{mt}(s,\xi,x)\right|^p ,
\end{align*}
and then, by \eqref{eq:6-15}, \eqref{eq6-43} for the intertemporal term and \eqref{eq6-44} for the terminal term, we can also obtain:
\begingroup
\allowdisplaybreaks
\begin{align*}\nonumber
&\bb{E}\left[\left|\bar{Z}_{mt}(s,\xi,x')-\bar{Z}_{mt}(s,\xi,x)\right|^p\right] \\\nonumber
\leq &3^p\bb{E}\left[ \left|\int_{s}^{T} D^{2}\dod{F}{m}(Y_{\cdot m t}(\tau) \otimes  m)(Y_{\xi mt}(\tau))\left(\bar{Y}_{mt}(\tau,\xi,x')-\bar{Y}_{mt}(\tau,\xi,x)\right)\dif \tau \right.\right.\\
\nonumber
&\left.\ \ \ \ \ + D^{2}\dod{}{m}F_{T}(Y_{\cdot mt}(T) \otimes  m)(Y_{\xi mt}(T))\left(\bar{Y}_{mt}(T,\xi,x')-\bar{Y}_{mt}(T,\xi,x)\right)\right|^p\\
\nonumber
&\ \ \ \ \ +\left|\widetilde{\bb{E}}\int_{s}^{T}\int_{\bb{R}^n}D_2D_1\dod[2]{F}{m}(Y_{\cdot mt}(\tau) \otimes  m)(Y_{\xi mt}(\tau),\widetilde{Y}_{\zeta mt}(\tau))\left(\widetilde{\bar{Y}}_{mt}(\tau,\zeta,x')-\widetilde{\bar{Y}}_{mt}(\tau,\zeta,x)\right)\dif m(\zeta) \dif \tau\right.\\
&\ \ \ \ \ \left.+\widetilde{\bb{E}}\int_{\bb{R}^n}D_2D_1\dod[2]{F_{T}}{m}(Y_{\cdot mt}(T) \otimes  m)(Y_{\xi mt}(T),\widetilde{Y}_{\zeta mt}(T))\left(\widetilde{\bar{Y}}_{mt}(T,\zeta,x')-\widetilde{\bar{Y}}_{mt}(T,\zeta,x)\right)\dif m(\zeta)\right|^p\\\nonumber
&\ \ \ \ \ +\left|\int_{s}^{T}\widetilde{\bb{E}}D_1\dod[2]{F}{m}(Y_{\cdot mt}(\tau) \otimes  m)(Y_{\xi mt}(\tau),\widetilde{Y}_{x'mt}(\tau))\dif \tau+\widetilde{\bb{E}}D_1\dod[2]{F_T}{m}(Y_{\cdot mt}(T) \otimes  m)(Y_{\xi mt}(T),\widetilde{Y}_{x'mt}(T))\right.\\
&\ \ \ \ \ - \left.\left.\int_{s}^{T}\widetilde{\bb{E}}D_1\dod[2]{F}{m}(Y_{\cdot mt}(\tau) \otimes  m)(Y_{\xi mt}(\tau),\widetilde{Y}_{xmt}(\tau))\dif \tau-\widetilde{\bb{E}}D_1\dod[2]{F_T}{m}(Y_{\cdot mt}(T) \otimes  m)(Y_{\xi mt}(T),\widetilde{Y}_{xmt}(T))\right|^p\right]\\
\leq&3^p\bb{E}\bigg[\left|\int_{s}^{T} c\left|\bar{Y}_{mt}(\tau,\xi,x')-\bar{Y}_{mt}(\tau,\xi,x)\right|\dif \tau+ c_T\left|\bar{Y}_{mt}(T,\xi,x')-\bar{Y}_{mt}(T,\xi,x)\right|\right|^p\\
\nonumber
&\ \ \ \ \ +\left|\widetilde{\bb{E}}\int_{s}^{T}\int_{\bb{R}^n}c\left|\widetilde{\bar{Y}}_{mt}(\tau,\zeta,x')-\widetilde{\bar{Y}}_{mt}(\tau,\zeta,x)\right|\dif m(\zeta) \dif \tau+\widetilde{\bb{E}}\int_{\bb{R}^n}c_T\left|\widetilde{\bar{Y}}_{mt}(T,\zeta,x')-\widetilde{\bar{Y}}_{mt}(T,\zeta,x)\right|\dif m(\zeta)\right|^p\\\nonumber
&\ \ \ \ \  +\frac{(c_T+c|T-t|)^p\lambda^p}{\left(\lambda-T(c_T+cT)\right)^p}\cdot \left|x'-x\right|^p\bigg]\\
\leq&3^p\bb{E}\bigg[\left|\int_{s}^{T} c\left|\bar{Y}_{mt}(\tau,\xi,x')-\bar{Y}_{mt}(\tau,\xi,x)\right|\dif \tau+ c_T\left|\bar{Y}_{mt}(T,\xi,x')-\bar{Y}_{mt}(T,\xi,x)\right|\right|^p\\
&\ \ \ \ \  +\left|(c_T+c|T-t|)\cdot\left(\sup_{s\in[t,T]}\widetilde{\bb{E}}\int_{\bb{R}^n}\left|\widetilde{\bar{Y}}_{mt}(s,\zeta,x')-\widetilde{\bar{Y}}_{mt}(s,\zeta,x)\right|\dif m(\zeta)\right) \right|^p
+\frac{(c_T+c|T-t|)^p\lambda^p}{\left(\lambda-T(c_T+cT)\right)^p}\cdot \left|x'-x\right|^p\bigg]\\
\leq&3^p\bb{E}\bigg[\left|\int_{s}^{T} c\left|\bar{Y}_{mt}(\tau,\xi,x')-\bar{Y}_{mt}(\tau,\xi,x)\right|\dif \tau+ c_T\left|\bar{Y}_{mt}(T,\xi,x')-\bar{Y}_{mt}(T,\xi,x)\right|\right|^p\\
\nonumber
&\ \ \ \ \  + (c_T+c|T-t|)^p\cdot\left(\sup_{s\in[t,T]} \bb{E} \int_{\bb{R}^n}\left| \bar{Y} _{mt}(s,\zeta,x')- \bar{Y} _{mt}(s,\zeta,x)\right|^p\dif m(\zeta) \right)
+\frac{(c_T+c|T-t|)^p\lambda^p}{\left(\lambda-T(c_T+cT)\right)^p}\cdot \left|x'-x\right|^p\bigg]\\
\leq&3^p\bb{E}\bigg[\left|\int_{s}^{T} c\left|\bar{Y}_{mt}(\tau,\xi,x')-\bar{Y}_{mt}(\tau,\xi,x)\right|\dif \tau+ c_T\left|\bar{Y}_{mt}(T,\xi,x')-\bar{Y}_{mt}(T,\xi,x)\right|\right|^p\\
\nonumber
&\ \ \ \ \  + (c_T+c|T-t|)^p\cdot\left(\sup_{s\in[t,T]} \sup_{\zeta\in\bb{R}^n}\bb{E} \left| \bar{Y} _{mt}(s,\zeta,x')- \bar{Y} _{mt}(s,\zeta,x)\right|^p\right)
+\frac{(c_T+c|T-t|)^p\lambda^p}{\left(\lambda-T(c_T+cT)\right)^p}\cdot \left|x'-x\right|^p\bigg].
\end{align*}\endgroup
Thus,
\begingroup
\allowdisplaybreaks
\small
\begin{align*}\nonumber
&\sup_{\xi\in\bb{R}^n}\bb{E}\left|\bar{Z}_{mt}(s,\xi,x')-\bar{Z}_{mt}(s,\xi,x)\right|^p  \\\nonumber
\leq &3^p\bigg(\sup_{\xi\in\bb{R}^n}\bb{E} \left|\int_{s}^{T} c\left|\bar{Y}_{mt}(\tau,\xi,x')-\bar{Y}_{mt}(\tau,\xi,x)\right|\dif \tau+ c_T\left|\bar{Y}_{mt}(T,\xi,x')-\bar{Y}_{mt}(T,\xi,x)\right|\right|^p \\
\nonumber
&\ \ \ \ \ \ + (c_T+c|T-t|)^p\cdot\left(\sup_{s\in[t,T]} \sup_{\zeta\in\bb{R}^n}\bb{E} \left| \bar{Y} _{mt}(s,\zeta,x')- \bar{Y} _{mt}(s,\zeta,x)\right|^p \right) +\frac{(c_T+c|T-t|)^p\lambda^p}{\left(\lambda-T(c_T+cT)\right)^p}\cdot \left|x'-x\right|^p\bigg)\\
\leq &3^p\bigg(2^p c^p|T-s|^{p-1}\sup_{\xi\in\bb{R}^n} \bb{E} \int_{s}^{T} \left|\bar{Y}_{mt}(\tau,\xi,x')-\bar{Y}_{mt}(\tau,\xi,x)\right|^p\dif \tau  + 2^pc_T^p\sup_{\xi\in\bb{R}^n}\bb{E} \left|\bar{Y}_{mt}(T,\xi,x')-\bar{Y}_{mt}(T,\xi,x)\right|^p \\
\nonumber
&\ \ \ \ \ \ + (c_T+c|T-t|)^p\cdot\left(\sup_{s\in[t,T]} \sup_{\zeta\in\bb{R}^n} \bb{E}  \left| \bar{Y} _{mt}(s,\zeta,x')- \bar{Y} _{mt}(s,\zeta,x)\right|^p  \right)+\frac{(c_T+c|T-t|)^p\lambda^p}{\left(\lambda-T(c_T+cT)\right)^p}\cdot \left|x'-x\right|^p\bigg)\end{align*}
\begin{align*}
\leq &3^p\bigg(2^p(c_T^p+c^p|T-t|^p)\sup_{s\in[t,T]}\sup_{\xi\in\bb{R}^n}\bb{E}  \left|\bar{Y}_{mt}(s,\xi,x')-\bar{Y}_{mt}(s,\xi,x)\right|^p \ \ \ \ \ \ \ \ \ \ \ \ \ \ \ \ \ \ \ \ \ \ \ \ \ \ \ \ \ \ \ \ \ \ \ \ \ \ \ \ \\
\nonumber
&\ \ \ \ \ \ + (c_T+c|T-t|)^p\cdot\left(\sup_{s\in[t,T]} \sup_{\zeta\in\bb{R}^n} \bb{E}  \left| \bar{Y} _{mt}(s,\zeta,x')- \bar{Y} _{mt}(s,\zeta,x)\right|^p \right)+\frac{(c_T+c|T-t|)^p\lambda^p}{\left(\lambda-T(c_T+cT)\right)^p}\cdot \left|x'-x\right|^p\bigg),
\end{align*}\normalsize
\endgroup
and hence
\begingroup
\allowdisplaybreaks
\begin{align*}
&\sup_{s\in[t,T]}\sup_{\xi\in\bb{R}^n}\bb{E} \left|\bar{Z}_{mt}(s,\xi,x')-\bar{Z}_{mt}(s,\xi,x)\right|^p\\
\leq &3^p\bigg(2^{p+1}(c_T^p+c^p|T-t|^p)\cdot\left(\sup_{s\in[t,T]}\sup_{\xi\in\bb{R}^n}\bb{E} \left|\bar{Y}_{mt}(s,\xi,x')-\bar{Y}_{mt}(s,\xi,x)\right|^p \right)\ \ \ \ \ \ \ \ \ \ \ \ \ \ \ \ \ \ \\
&\ \ \ \ \ \ +(c_T+c|T-t|)^p\left(\frac{\lambda}{\lambda-T(c_T+cT)}\right)^p\cdot \left|x'-x\right|^p\bigg).
\end{align*}\endgroup
Therefore, if $\lambda>6T \left(2(c_T^p+c^pT^p)\right)^{1/p}$,
\begin{align}\label{LipsYx}
\sup_{s\in[t,T]} \sup_{\xi\in\bb{R}^n} \bb{E}\left|\bar{Y}_{mt}(s,\xi,x')-\bar{Y}_{mt}(s,\xi,x)\right|^p  \leq& \frac{3^p(c_T+cT)^pT^p}{ \lambda^p-2\cdot 6^p\cdot(c_T^p+c^pT^p)\cdot T^p }\left(\frac{\lambda}{\lambda-T(c_T+cT)}\right)^p  |x'-x|^p,
\end{align}
\begin{align}\label{LipsZx}
\sup_{s\in[t,T]}\sup_{\xi\in\bb{R}^n}\bb{E} \left|\bar{Z}_{mt}(s,\xi,x')-\bar{Z}_{mt}(s,\xi,x)\right|^p \leq& \frac{3^p\lambda^p(c_T+cT)^p}{ \lambda^p-2\cdot 6^p\cdot(c_T^p+c^pT^p)\cdot T^p }\left(\frac{\lambda}{\lambda-T(c_T+cT)}\right)^p |x'-x|^p.
\end{align}

Particularly, for the choice of $p=1$, since $\bar{Z}_{mt}(t,\xi,x)=D_\xi \frac{d}{dm}U(\xi,m,t)(x)$ is deterministic, therefore $D_\xi \frac{d}{dm}U(\xi,m,t)(x)$ is Lipschitz continuous in $x\in\bb{R}^n$ uniformly in $\xi$, $m$ and $t$ since the Lipschitz constant $\frac{3\lambda(c_T+cT)}{ \lambda-12\cdot(c_T+cT)\cdot T }\cdot\frac{\lambda}{\lambda-T(c_T+cT)}$ is independent of $\xi$, $m$ and $t$, and we have the following estimate:
\begin{align}\label{LipsDdUx}
\left|D_\xi \frac{d}{dm}U(\xi,m,t)(x')-D_\xi \frac{d}{dm}U(\xi,m,t)(x)\right| \leq& \frac{3\lambda(c_T+cT)}{ \lambda-12\cdot(c_T+cT)\cdot T }\cdot\frac{\lambda}{\lambda-T(c_T+cT)}\cdot|x'-x|.
\end{align}

Step $2$. We next show the continuity of $(\bar{Y}_{mt}(s,\xi,x),\bar{Z}_{mt}(s,\xi,x))$ in $\xi\in\bb{R}^n$ for each $m\in\s{P}_2(\bb{R}^n)$ and $x\in\bb{R}^n$ under the additional condition of $\lambda>T\sqrt{8(c_T^2+c^2T^2)}$.
Note that, by \eqref{eq:4-90}, \eqref{eq:6-15} and \eqref{xLipCon}, we have:
\begin{align}\nonumber
&\left|\widetilde{\bb{E}}D_1\dod[2]{F}{m}(Y_{\cdot mt}(\tau) \otimes  m)(Y_{\xi' mt}(\tau),\widetilde{Y}_{xmt}(\tau))-\widetilde{\bb{E}}D_1\dod[2]{F}{m}(Y_{\cdot mt}(\tau) \otimes  m)(Y_{\xi mt}(\tau),\widetilde{Y}_{xmt}(\tau))\right|\\\label{6-48}
\leq &\sqrt{C_T(1+|x|^2)}\cdot\frac{c\lambda}{\lambda-T(c_T+cT)}|\xi'-\xi|,
\end{align}
and
\begingroup
\allowdisplaybreaks
\begin{align}\nonumber
&\left|\widetilde{\bb{E}}D_1\dod[2]{F_T}{m}(Y_{\cdot mt}(T) \otimes  m)(Y_{\xi' mt}(T),\widetilde{Y}_{xmt}(T))-\widetilde{\bb{E}}D_1\dod[2]{F_T}{m}(Y_{\cdot mt}(T) \otimes  m)(Y_{\xi mt}(T),\widetilde{Y}_{xmt}(T))\right|\\\label{eq6-47}
\leq &\sqrt{C_T(1+|x|^2)}\cdot\frac{c_T\lambda}{\lambda-T(c_T+cT)}|\xi'-\xi|.
\end{align}\endgroup
Further, by \eqref{eq:6-11}-\eqref{eq:6-12}, one has\small
\begingroup
\allowdisplaybreaks\begin{align}\nonumber
&\bar{Y}_{mt}(s,\xi',x)-\bar{Y}_{mt}(s,\xi,x) =-\dfrac{1}{\lambda}\int_{t}^{s}\left(\bar{Z}_{mt}(\tau,\xi',x)-\bar{Z}_{mt}(\tau,\xi,x)\right)\dif \tau,\\
\nonumber
&\bar{Z}_{mt}(s,\xi',x)-\bar{Z}_{mt}(s,\xi,x) \\\nonumber
=&\bb{E}\left[ \int_{s}^{T} D^{2}\dod{F}{m}(Y_{\cdot m t}(\tau) \otimes  m)(Y_{\xi' mt}(\tau))\left(\bar{Y}_{mt}(\tau,\xi',x)-\bar{Y}_{mt}(\tau,\xi,x)\right)\dif \tau \right.\\\nonumber
&\ \ \ \ +\int_{s}^{T} \left(D^{2}\dod{F}{m}(Y_{\cdot m t}(\tau) \otimes  m)(Y_{\xi' mt}(\tau))-D^{2}\dod{F}{m}(Y_{\cdot m t}(\tau) \otimes  m)(Y_{\xi mt}(\tau))\right)\bar{Y}_{mt}(\tau,\xi,x)\dif \tau
\\
\nonumber
&\ \ \ \ + D^{2}\dod{}{m}F_{T}(Y_{\cdot mt}(T) \otimes  m)(Y_{\xi' mt}(T))\left(\bar{Y}_{mt}(T,\xi',x)-\bar{Y}_{mt}(T,\xi,x)\right)\\
\nonumber
&\ \ \ \ + \left(D^{2}\dod{}{m}F_{T}(Y_{\cdot mt}(T) \otimes  m)(Y_{\xi' mt}(T))-D^{2}\dod{}{m}F_{T}(Y_{\cdot mt}(T) \otimes  m)(Y_{\xi mt}(T))\right)\bar{Y}_{mt}(T,\xi,x)\\
\nonumber
&\ \ \ \ +\widetilde{\bb{E}}\int_{s}^{T}\int_{\bb{R}^n}\left(D_2D_1\dod[2]{F}{m}(Y_{\cdot mt}(\tau) \otimes  m)(Y_{\xi' mt}(\tau),\widetilde{Y}_{\zeta mt}(\tau))-D_2D_1\dod[2]{F}{m}(Y_{\cdot mt}(\tau) \otimes  m)(Y_{\xi mt}(\tau),\widetilde{Y}_{\zeta mt}(\tau))\right)\\\nonumber
&\ \ \ \ \ \ \ \ \ \ \ \ \ \ \ \ \ \ \ \ \ \cdot \widetilde{\bar{Y}}_{mt}(\tau,\zeta,x) \dif m(\zeta) \dif \tau\\
\nonumber
&\ \ \ \ +\widetilde{\bb{E}}\int_{\bb{R}^n}\left(D_2D_1\dod[2]{F_{T}}{m}(Y_{\cdot mt}(T) \otimes  m)(Y_{\xi' mt}(T),\widetilde{Y}_{\zeta mt}(T)) -D_2D_1\dod[2]{F_{T}}{m}(Y_{\cdot mt}(T) \otimes  m)(Y_{\xi mt}(T),\widetilde{Y}_{\zeta mt}(T))\right)\\\nonumber
&\ \ \ \ \ \ \ \ \ \ \ \ \ \ \ \ \ \ \ \ \left.\left.\cdot \widetilde{\bar{Y}}_{mt}(T,\zeta,x) \dif m(\zeta)\right|\mathcal{W}_{t}^{s}\right]\\\nonumber
&+\bb{E}\left[\int_{s}^{T}\widetilde{\bb{E}}D_1\dod[2]{F}{m}(Y_{\cdot mt}(\tau) \otimes  m)(Y_{\xi' mt}(\tau),\widetilde{Y}_{xmt}(\tau))\dif \tau+\widetilde{\bb{E}}D_1\dod[2]{F_T}{m}(Y_{\cdot mt}(T) \otimes  m)(Y_{\xi' mt}(T),\widetilde{Y}_{xmt}(T))\right.\\
&\ \ \ \ \ \ \ \ - \left.\left.\int_{s}^{T}\widetilde{\bb{E}}D_1\dod[2]{F}{m}(Y_{\cdot mt}(\tau) \otimes  m)(Y_{\xi mt}(\tau),\widetilde{Y}_{xmt}(\tau))\dif \tau-\widetilde{\bb{E}}D_1\dod[2]{F_T}{m}(Y_{\cdot mt}(T) \otimes  m)(Y_{\xi mt}(T),\widetilde{Y}_{xmt}(T))\right|\mathcal{W}_{t}^{s}\right].
\end{align}\endgroup\normalsize
First, by the simple algebra, we also have, 
\begin{align}\nonumber
&\sup_{s\in[t,T]} \bb{E} \left|\bar{Y}_{mt}(s,\xi',x)-\bar{Y}_{mt}(s,\xi,x)\right|^2 \leq \dfrac{|T-t|^2}{\lambda^2} \sup_{s\in[t,T]}\bb{E} \left|\bar{Z}_{mt}(s,\xi',x)-\bar{Z}_{mt}(s,\xi,x)\right|^2,
\end{align}
and then, by \eqref{eq:6-15} and \eqref{6-48} for the intertemporal term, and \eqref{eq6-47} for the terminal term, we can also obtain:
\begingroup
\allowdisplaybreaks
\begin{align}\nonumber
&\h{-10pt}\bb{E}\left[\left|\bar{Z}_{mt}(s,\xi',x)-\bar{Z}_{mt}(s,\xi,x)\right|^2\right] \\\nonumber
\leq&4\bb{E}\Bigg[ \bigg|\int_{s}^{T} D^{2}\dod{F}{m}(Y_{\cdot m t}(\tau) \otimes  m)(Y_{\xi' mt}(\tau))\left(\bar{Y}_{mt}(\tau,\xi',x)-\bar{Y}_{mt}(\tau,\xi,x)\right)\dif \tau \\\nonumber
&\ \ \ \ \ \ + D^{2}\dod{}{m}F_{T}(Y_{\cdot mt}(T) \otimes  m)(Y_{\xi' mt}(T))\left(\bar{Y}_{mt}(T,\xi',x)-\bar{Y}_{mt}(T,\xi,x)\right)\bigg|^2\\
\nonumber
&\ \ \ \ \ \ + \bigg|\int_{s}^{T} \left(D^{2}\dod{F}{m}(Y_{\cdot m t}(\tau) \otimes  m)(Y_{\xi' mt}(\tau))-D^{2}\dod{F}{m}(Y_{\cdot m t}(\tau) \otimes  m)(Y_{\xi mt}(\tau))\right)\bar{Y}_{mt}(\tau,\xi,x)\dif \tau
\\
\nonumber
&\ \ \ \ \ \ + \left(D^{2}\dod{}{m}F_{T}(Y_{\cdot mt}(T) \otimes  m)(Y_{\xi' mt}(T))-D^{2}\dod{}{m}F_{T}(Y_{\cdot mt}(T) \otimes  m)(Y_{\xi mt}(T))\right)\bar{Y}_{mt}(T,\xi,x)\bigg|^2
\end{align}\endgroup\small
\begingroup
\allowdisplaybreaks
\begin{align}\nonumber
&+\bigg|\widetilde{\bb{E}}\int_{s}^{T}\int_{\bb{R}^n}\left(D_2D_1\dod[2]{F}{m}(Y_{\cdot mt}(\tau) \otimes  m)(Y_{\xi' mt}(\tau),\widetilde{Y}_{\zeta mt}(\tau))-D_2D_1\dod[2]{F}{m}(Y_{\cdot mt}(\tau) \otimes  m)(Y_{\xi mt}(\tau),\widetilde{Y}_{\zeta mt}(\tau))\right)\\\nonumber
&\ \ \ \ \ \ \ \ \ \ \ \ \ \ \ \ \ \ \cdot \widetilde{\bar{Y}}_{mt}(\tau,\zeta,x) \dif m(\zeta) \dif \tau\\
\nonumber
&\ \ \ \ +\widetilde{\bb{E}}\int_{\bb{R}^n}\left(D_2D_1\dod[2]{F_{T}}{m}(Y_{\cdot mt}(T) \otimes  m)(Y_{\xi' mt}(T),\widetilde{Y}_{\zeta mt}(T)) -D_2D_1\dod[2]{F_{T}}{m}(Y_{\cdot mt}(T) \otimes  m)(Y_{\xi mt}(T),\widetilde{Y}_{\zeta mt}(T))\right)\\\nonumber
&\ \ \ \ \ \ \ \ \ \ \ \ \ \ \ \ \ \ \cdot \widetilde{\bar{Y}}_{mt}(T,\zeta,x) \dif m(\zeta)\bigg|^2\\\nonumber
&+\bigg|\int_{s}^{T}\widetilde{\bb{E}}D_1\dod[2]{F}{m}(Y_{\cdot mt}(\tau) \otimes  m)(Y_{\xi' mt}(\tau),\widetilde{Y}_{xmt}(\tau))\dif \tau+\widetilde{\bb{E}}D_1\dod[2]{F_T}{m}(Y_{\cdot mt}(T) \otimes  m)(Y_{\xi' mt}(T),\widetilde{Y}_{xmt}(T))\\\nonumber
&\ \ \ \ -\int_{s}^{T}\widetilde{\bb{E}}D_1\dod[2]{F}{m}(Y_{\cdot mt}(\tau) \otimes  m)(Y_{\xi mt}(\tau),\widetilde{Y}_{xmt}(\tau))\dif \tau-\widetilde{\bb{E}}D_1\dod[2]{F_T}{m}(Y_{\cdot mt}(T) \otimes  m)(Y_{\xi mt}(T),\widetilde{Y}_{xmt}(T))\bigg|^2\Bigg]\\\nonumber
\leq&4 \bb{E}\Bigg[ \bigg|\int_{s}^{T} c\left|\bar{Y}_{mt}(\tau,\xi',x)-\bar{Y}_{mt}(\tau,\xi,x)\right|\dif \tau + c_T\left|\bar{Y}_{mt}(T,\xi',x)-\bar{Y}_{mt}(T,\xi,x)\right|\bigg|^2\\
\nonumber
&\ \ \ \ +\bigg|\int_{s}^{T} c\left|Y_{\xi' mt}(\tau)-Y_{\xi mt}(\tau)\right|\cdot\left|\bar{Y}_{mt}(\tau,\xi,x)\right|\dif \tau+ c_T \left| Y_{\xi' mt}(T) -Y_{\xi mt}(T)\right|\cdot\left|\bar{Y}_{mt}(T,\xi,x)\right|\bigg|^2\\\nonumber
&\ \ \ \ +\bigg|\widetilde{\bb{E}}\int_{s}^{T}\int_{\bb{R}^n}c\left| Y_{\xi' mt}(\tau)-Y_{\xi mt}(\tau)\right|\cdot\left| \widetilde{\bar{Y}}_{mt}(\tau,\zeta,x)\right| \dif m(\zeta) \dif \tau+\widetilde{\bb{E}}\int_{\bb{R}^n}c_T\left|Y_{\xi' mt}(T)-Y_{\xi mt}(T)\right|\cdot\left| \widetilde{\bar{Y}}_{mt}(T,\zeta,x)\right| \dif m(\zeta)\bigg|^2\\\nonumber
&\ \ \ \ + C_T(1+|x|^2) \frac{(c_T+c|T-t|)^2\lambda^2}{\left(\lambda-T(c_T+cT)\right)^2}|\xi'-\xi|^2\Bigg]\\
\leq&4\bb{E} \Bigg[ 2|T-t|c^2\cdot\int_{s}^{T} \left|\bar{Y}_{mt}(\tau,\xi',x)-\bar{Y}_{mt}(\tau,\xi,x)\right|^2\dif \tau + 2c_T^2\left|\bar{Y}_{mt}(T,\xi',x)-\bar{Y}_{mt}(T,\xi,x)\right|^2\\
\nonumber
&\ \ \ \ +2|T-t|c^2\cdot\int_{s}^{T} \left|Y_{\xi' mt}(\tau)-Y_{\xi mt}(\tau)\right|^2\left|\bar{Y}_{mt}(\tau,\xi,x)\right|^2\dif \tau+ 2c_T^2 \left| Y_{\xi' mt}(T) -Y_{\xi mt}(T)\right|^2\cdot\left|\bar{Y}_{mt}(T,\xi,x)\right|^2\\\nonumber
&\ \ \ \ +2|T-t|c^2\int_{s}^{T}\bigg(\left| Y_{\xi' mt}(\tau)-Y_{\xi mt}(\tau)\right|^2\cdot\widetilde{\bb{E}}\int_{\bb{R}^n}\left| \widetilde{\bar{Y}}_{mt}(\tau,\zeta,x)\right|^2 \dif m(\zeta)\bigg) \dif \tau\\\nonumber
&\ \ \ \ +2c_T^2\left|Y_{\xi' mt}(T)-Y_{\xi mt}(T)\right|^2\cdot\widetilde{\bb{E}}\int_{\bb{R}^n}\left| \widetilde{\bar{Y}}_{mt}(T,\zeta,x)\right|^2 \dif m(\zeta)
+ C_T(1+|x|^2)\cdot \frac{(c_T+c|T-t|)^2\lambda^2}{\left(\lambda-T(c_T+cT)\right)^2}\cdot|\xi'-\xi|^2\Bigg]\\
\leq&4 \bb{E}\Bigg[ 2|T-t|c^2\cdot\int_{s}^{T} \left|\bar{Y}_{mt}(\tau,\xi',x)-\bar{Y}_{mt}(\tau,\xi,x)\right|^2\dif \tau + 2c_T^2\left|\bar{Y}_{mt}(T,\xi',x)-\bar{Y}_{mt}(T,\xi,x)\right|^2\\
\nonumber
&\ \ \ \ +2|T-t|\frac{c^2\lambda^2}{\left(\lambda-T(c_T+cT)\right)^2}|\xi'-\xi|^2\cdot\int_{s}^{T} \left|\bar{Y}_{mt}(\tau,\xi,x)\right|^2\dif \tau+ 2 \frac{c_T^2\lambda^2}{\left(\lambda-T(c_T+cT)\right)^2}|\xi'-\xi|^2\left|\bar{Y}_{mt}(T,\xi,x)\right|^2\\\nonumber
&\ \ \ \ +2|T-t|\frac{c^2\lambda^2}{\left(\lambda-T(c_T+cT)\right)^2}|\xi'-\xi|^2\cdot \int_{s}^{T}\bigg(\widetilde{\bb{E}}\int_{\bb{R}^n}\left| \widetilde{\bar{Y}}_{mt}(\tau,\zeta,x)\right|^2 \dif m(\zeta)\bigg) \dif \tau\\\nonumber
&\ \ \ \ +2\frac{c_T^2\lambda^2}{\left(\lambda-T(c_T+cT)\right)^2}|\xi'-\xi|^2\widetilde{\bb{E}}\int_{\bb{R}^n}\left| \widetilde{\bar{Y}}_{mt}(T,\zeta,x)\right|^2 \dif m(\zeta)+ C_T(1+|x|^2) \frac{(c_T+c|T-t|)^2\lambda^2}{\left(\lambda-T(c_T+cT)\right)^2}|\xi'-\xi|^2\Bigg]\\
\leq&4 \bb{E}\Bigg[ 2|T-t|c^2\cdot\int_{s}^{T} \left|\bar{Y}_{mt}(\tau,\xi',x)-\bar{Y}_{mt}(\tau,\xi,x)\right|^2\dif \tau + 2c_T^2\left|\bar{Y}_{mt}(T,\xi',x)-\bar{Y}_{mt}(T,\xi,x)\right|^2\\
\nonumber
&\ \ \ \ +2\cdot|T-t|\frac{c^2\lambda^2}{\left(\lambda-T(c_T+cT)\right)^2}|\xi'-\xi|^2\cdot\int_{s}^{T} \left|\bar{Y}_{mt}(\tau,\xi,x)\right|^2\dif \tau+ \frac{2c_T^2\lambda^2}{\left(\lambda-T(c_T+cT)\right)^2}|\xi'-\xi|^2\cdot\left|\bar{Y}_{mt}(T,\xi,x)\right|^2\\\nonumber
&\ \ \ \ +2\frac{\left(c_T^2+c^2|T-t|^2\right)\lambda^2}{\left(\lambda-T(c_T+cT)\right)^2}\cdot|\xi'-\xi|^2\cdot C_T(1+|x|^2)+ C_T(1+|x|^2) \cdot\frac{(c_T+c|T-t|)^2\lambda^2}{\left(\lambda-T(c_T+cT)\right)^2}\cdot|\xi'-\xi|^2\Bigg],
\end{align}\endgroup\normalsize
where we use \eqref{xLipCon} in the second last inequality, and use \eqref{eq:6-16} in the very last inequality.
Thus,

\begingroup
\allowdisplaybreaks
\small
\begin{align}\nonumber
&\sup_{s\in[t,T]}\bb{E}\left|\bar{Z}_{mt}(s,\xi',x)-\bar{Z}_{mt}(s,\xi,x)\right|^2\\\nonumber
\leq&4 \Bigg( 2\left(c_T^2+|T-t|^2c^2\right)\cdot\sup_{s\in[t,T]}\bb{E} \left|\bar{Y}_{mt}(s,\xi',x)-\bar{Y}_{mt}(s,\xi,x)\right|^2+2\frac{\left(c_T^2+|T-t|^2c^2\right)\lambda^2}{\left(\lambda-T(c_T+cT)\right)^2}|\xi'-\xi|^2\cdot\sup_{s\in[t,T]}\bb{E}\left|\bar{Y}_{mt}(s,\xi,x)\right|^2\\\nonumber
&\ \ \ \ +4\frac{\left(c_T^2+c^2|T-t|^2\right)\lambda^2}{\left(\lambda-T(c_T+cT)\right)^2}\cdot|\xi'-\xi|^2\cdot C_T(1+|x|^2)\Bigg)\\
\leq&4 \Bigg( 2\left(c_T^2+|T-t|^2c^2\right)\cdot\sup_{s\in[t,T]}\bb{E} \left|\bar{Y}_{mt}(s,\xi',x)-\bar{Y}_{mt}(s,\xi,x)\right|^2+6\frac{\left(c_T^2+c^2|T-t|^2\right)\lambda^2}{\left(\lambda-T(c_T+cT)\right)^2}\cdot|\xi'-\xi|^2\cdot C_T(1+|x|^2)\Bigg),
\end{align}\normalsize
\endgroup
where we use \eqref{eq:a priori1} in the last inequality.
Therefore, under the additional condition of $\lambda>T\sqrt{8(c_T^2+c^2T^2)}$,
\begin{align}\label{LipsYxi}
\sup_{s\in[t,T]}\bb{E}\left|\bar{Y}_{mt}(s,\xi',x)-\bar{Y}_{mt}(s,\xi,x)\right|^2 \leq& \frac{24(c_T^2+c^2T^2)T^2}{ \lambda^2-8(c_T^2+c^2T^2)T^2 }\left(\frac{\lambda}{\lambda-T(c_T+cT)}\right)^2\cdot C_T(1+|x|^2) |\xi'-\xi|^2,
\end{align}
\begin{align}\label{LipsZxi}
\sup_{s\in[t,T]}\bb{E}\left|\bar{Z}_{mt}(s,\xi',x)-\bar{Z}_{mt}(s,\xi,x)\right|^2 \leq& \frac{24\lambda^2(c_T^2+c^2T^2)}{ \lambda^2-8(c_T^2+c^2T^2)T^2 }\left(\frac{\lambda}{\lambda-T(c_T+cT)}\right)^2\cdot C_T(1+|x|^2) |\xi'-\xi|^2.
\end{align}
Since $\bar{Z}_{mt}(t,\xi,x)=D_\xi \frac{d}{dm}U(\xi,m,t)(x)$ is deterministic, therefore, for each $x\in\bb{R}^n$, $D_\xi \frac{d}{dm}U(\xi,m,t)(x)$ is Lipschitz continuous in $\xi\in\bb{R}^n$ uniformly in $m$ and $t$ since the Lipschitz constant\\ $\left(\frac{24\lambda^2(c_T^2+c^2T^2)}{ \lambda^2-8(c_T^2+c^2T^2)T^2 }\left(\frac{\lambda}{\lambda-T(c_T+cT)}\right)^2\cdot C_T(1+|x|^2)\right)^{1/2}$ is independent of $m$ and $t$, and we have the following estimate:
\begin{align}\nonumber
&\left|D_\xi \frac{d}{dm}U(\xi',m,t)(x)-D_\xi \frac{d}{dm}U(\xi,m,t)(x)\right| \\\label{LipsDdUxi}
\leq& \left(\frac{24\lambda^2(c_T^2+c^2T^2)}{ \lambda^2-8(c_T^2+c^2T^2)T^2 }\left(\frac{\lambda}{\lambda-T(c_T+cT)}\right)^2\cdot C_T(1+|x|^2)\right)^{1/2}\cdot|\xi'-\xi|.
\end{align}

Step $3$. We finally show the continuity of $(\bar{Y}_{mt}(s,\xi,x),\bar{Z}_{mt}(s,\xi,x))$ in $m\in\s{P}_2(\bb{R}^n)$ for each $(s,\xi,x)\in[t,T]\times\bb{R}^n\times\bb{R}^n$ under the additional condition of $\lambda>T\sqrt{30(c_T^2+c^2T^2)}>T\sqrt{8(c_T^2+c^2T^2)}$, which is the lower bound we specified at the very beginning of Step 2. First, we have
\begingroup
\allowdisplaybreaks
\begin{align*}\nonumber
&\bar{Y}_{m't}(s,\xi,x)-\bar{Y}_{mt}(s,\xi,x) =-\dfrac{1}{\lambda}\int_{t}^{s}\left(\bar{Z}_{m't}(\tau,\xi,x)-\bar{Z}_{mt}(\tau,\xi,x)\right)\dif \tau,\\
\nonumber
&\bar{Z}_{m' t}(s,\xi,x)-\bar{Z}_{mt}(s,\xi,x) \\\nonumber
=&\bb{E}\left[ \int_{s}^{T} \left(D^{2}\dod{F}{m}(Y_{\cdot m' t}(\tau) \otimes  m')(Y_{\xi m't}(\tau))\bar{Y}_{m't}(\tau,\xi,x)-D^{2}\dod{F}{m}(Y_{\cdot mt}(\tau) \otimes  m)(Y_{\xi mt}(\tau))\bar{Y}_{mt}(\tau,\xi,x)\right)\dif \tau \right.\\
\nonumber
&\ \ \ \ + D^{2}\dod{}{m}F_{T}(Y_{\cdot m't}(T) \otimes  m')(Y_{\xi m't}(T))\bar{Y}_{m't}(T,\xi,x)-D^{2}\dod{}{m}F_{T}(Y_{\cdot mt}(T) \otimes  m)(Y_{\xi mt}(T))\bar{Y}_{mt}(T,\xi,x)\\
\nonumber
&\ \ \ \ +\widetilde{\bb{E}}\int_{s}^{T}\int_{\bb{R}^n}D_2D_1\dod[2]{F}{m}(Y_{\cdot m't}(\tau) \otimes  m')(Y_{\xi m't}(\tau),\widetilde{Y}_{\zeta m't}(\tau))\widetilde{\bar{Y}}_{m't}(\tau,\zeta,x)\dif m'(\zeta) \dif \tau
\end{align*}
\begin{align*}\nonumber
&\ \ \ \ -\widetilde{\bb{E}}\int_{s}^{T}\int_{\bb{R}^n}D_2D_1\dod[2]{F}{m}(Y_{\cdot mt}(\tau) \otimes  m)(Y_{\xi mt}(\tau),\widetilde{Y}_{\zeta mt}(\tau))\widetilde{\bar{Y}}_{mt}(\tau,\zeta,x)\dif m(\zeta) \dif \tau\\
\nonumber
&\ \ \ \ +\widetilde{\bb{E}}\int_{\bb{R}^n}D_2D_1\dod[2]{F_{T}}{m}(Y_{\cdot m't}(T) \otimes  m')(Y_{\xi m't}(T),\widetilde{Y}_{\zeta m't}(T))\widetilde{\bar{Y}}_{m't}(T,\eta,x)\dif m'(\zeta)\\
\nonumber
&\ \ \ \ -\widetilde{\bb{E}}\int_{\bb{R}^n}D_2D_1\dod[2]{F_{T}}{m}(Y_{\cdot mt}(T) \otimes  m)(Y_{\xi mt}(T),\widetilde{Y}_{\zeta mt}(T))\widetilde{\bar{Y}}_{mt}(T,\zeta,x)\dif m(\zeta)\\\nonumber
&\ \ \ \ +\int_{s}^{T}\widetilde{\bb{E}}D_1\dod[2]{F}{m}(Y_{\cdot m't}(\tau) \otimes  m')(Y_{\xi m't}(\tau),\widetilde{Y}_{xm't}(\tau))\dif \tau+\widetilde{\bb{E}}D_1\dod[2]{F_T}{m}(Y_{\cdot m't}(T) \otimes  m')(Y_{\xi m't}(T),\widetilde{Y}_{xm't}(T))\\
&\ \ \ \ - \left.\left.\int_{s}^{T}\widetilde{\bb{E}}D_1\dod[2]{F}{m}(Y_{\cdot mt}(\tau) \otimes  m)(Y_{\xi mt}(\tau),\widetilde{Y}_{xmt}(\tau))\dif \tau-\widetilde{\bb{E}}D_1\dod[2]{F_T}{m}(Y_{\cdot mt}(T) \otimes  m)(Y_{\xi mt}(T),\widetilde{Y}_{xmt}(T))\right|\mathcal{W}_{t}^{s}\right],
\end{align*}\endgroup
and
\begin{align}\nonumber
&D^{2}\dod{F}{m}(Y_{\cdot m' t}(\tau) \otimes  m')(Y_{\xi m't}(\tau))\bar{Y}_{m't}(\tau,\xi,x)-D^{2}\dod{F}{m}(Y_{\cdot mt}(\tau) \otimes  m)(Y_{\xi mt}(\tau))\bar{Y}_{mt}(\tau,\xi,x)\\\nonumber
=&D^{2}\dod{F}{m}(Y_{\cdot m' t}(\tau) \otimes  m')(Y_{\xi m't}(\tau))\bar{Y}_{m't}(\tau,\xi,x)-D^{2}\dod{F}{m}(Y_{\cdot m't}(\tau) \otimes  m')(Y_{\xi m't}(\tau))\bar{Y}_{mt}(\tau,\xi,x)\\\label{3primeeq1}
&+D^{2}\dod{F}{m}(Y_{\cdot m' t}(\tau) \otimes  m')(Y_{\xi m't}(\tau))\bar{Y}_{mt}(\tau,\xi,x)-D^{2}\dod{F}{m}(Y_{\cdot mt}(\tau) \otimes  m)(Y_{\xi mt}(\tau))\bar{Y}_{mt}(\tau,\xi,x)
\end{align}
and
\begingroup
\allowdisplaybreaks\begin{align}\nonumber
&\int_{\bb{R}^n}D_2D_1\dod[2]{F}{m}(Y_{\cdot m't}(\tau) \otimes  m')(Y_{\xi m't}(\tau),\widetilde{Y}_{\zeta m't}(\tau))\widetilde{\bar{Y}}_{m't}(\tau,\zeta,x)\dif m'(\zeta)\\\nonumber
&-\int_{\bb{R}^n}D_2D_1\dod[2]{F}{m}(Y_{\cdot mt}(\tau) \otimes  m)(Y_{\xi mt}(\tau),\widetilde{Y}_{\zeta mt}(\tau))\widetilde{\bar{Y}}_{mt}(\tau,\zeta,x)\dif m(\zeta)\\\nonumber
=&\int_{\bb{R}^n}D_2D_1\dod[2]{F}{m}(Y_{\cdot m't}(\tau) \otimes  m')(Y_{\xi m't}(\tau),\widetilde{Y}_{\zeta m't}(\tau))\widetilde{\bar{Y}}_{m't}(\tau,\zeta,x)\dif m'(\zeta)\\\nonumber
&-\int_{\bb{R}^n}D_2D_1\dod[2]{F}{m}(Y_{\cdot m't}(\tau) \otimes  m')(Y_{\xi m't}(\tau),\widetilde{Y}_{\zeta m't}(\tau))\widetilde{\bar{Y}}_{m't}(\tau,\zeta,x)\dif m(\zeta)\\\nonumber
&+\int_{\bb{R}^n}D_2D_1\dod[2]{F}{m}(Y_{\cdot m't}(\tau) \otimes  m')(Y_{\xi m't}(\tau),\widetilde{Y}_{\zeta m't}(\tau))\widetilde{\bar{Y}}_{m't}(\tau,\zeta,x)\dif m(\zeta)\\\nonumber
&-\int_{\bb{R}^n}D_2D_1\dod[2]{F}{m}(Y_{\cdot m't}(\tau) \otimes  m')(Y_{\xi m't}(\tau),\widetilde{Y}_{\zeta m't}(\tau))\widetilde{\bar{Y}}_{mt}(\tau,\zeta,x)\dif m(\zeta)\\\nonumber
&+\int_{\bb{R}^n}D_2D_1\dod[2]{F}{m}(Y_{\cdot m't}(\tau) \otimes  m')(Y_{\xi m't}(\tau),\widetilde{Y}_{\zeta m't}(\tau))\widetilde{\bar{Y}}_{mt}(\tau,\zeta,x)\dif m(\zeta)\\\label{3primeeq2}
&-\int_{\bb{R}^n}D_2D_1\dod[2]{F}{m}(Y_{\cdot mt}(\tau) \otimes  m)(Y_{\xi mt}(\tau),\widetilde{Y}_{\zeta mt}(\tau))\widetilde{\bar{Y}}_{mt}(\tau,\zeta,x)\dif m(\zeta).
\end{align}\endgroup
Note also that, by \eqref{eq:6-15}, \eqref{eq:a priori1}, \eqref{xLipCon} and \eqref{LipsYxi}, \small
\begin{align}\nonumber
&\left|\widetilde{\bb{E}} \int_{\bb{R}^n}D_2D_1\dod[2]{F}{m}(Y_{\cdot m't}(s) \otimes  m')(Y_{\xi m't}(s),\widetilde{Y}_{\zeta m't}(s))\widetilde{\bar{Y}}_{m't}(s,\zeta,x)\dif\, (m'-m)(\zeta)\right|^2\\\nonumber
=&\left|\widetilde{\bb{E}} \int_{\widehat{\Omega}}D_2D_1\dod[2]{F}{m}(Y_{\cdot m't}(s) \otimes  m')(Y_{\xi m't}(s),\widetilde{Y}_{\hat{X}_{m'}(\widehat{\omega}) m't}(s))\widetilde{\bar{Y}}_{m't}(s,\hat{X}_{m'}(\widehat{\omega}),x)\dif \widehat{\bb{P}}(\widehat{\omega})\right.\\\nonumber
&\left.-\widetilde{\bb{E}} \int_{\widehat{\Omega}}D_2D_1\dod[2]{F}{m}(Y_{\cdot m't}(s) \otimes  m')(Y_{\xi m't}(s),\widetilde{Y}_{\hat{X}_{m}(\widehat{\omega}) m't}(s))\widetilde{\bar{Y}}_{m't}(s,\hat{X}_{m}(\widehat{\omega}),x)\dif \widehat{\bb{P}}(\widehat{\omega})\right|^2
\end{align}
\begin{align}\nonumber
=&\left|\widetilde{\bb{E}} \int_{\widehat{\Omega}}D_2D_1\dod[2]{F}{m}(Y_{\cdot m't}(s) \otimes  m')(Y_{\xi m't}(s),\widetilde{Y}_{\hat{X}_{m'}(\widehat{\omega}) m't}(s))\widetilde{\bar{Y}}_{m't}(s,\hat{X}_{m'}(\widehat{\omega}),x)\dif \widehat{\bb{P}}(\widehat{\omega})\right.\\\nonumber
& -\widetilde{\bb{E}} \int_{\widehat{\Omega}}D_2D_1\dod[2]{F}{m}(Y_{\cdot m't}(s) \otimes  m')(Y_{\xi m't}(s),\widetilde{Y}_{\hat{X}_{m'}(\widehat{\omega}) m't}(s))\widetilde{\bar{Y}}_{m't}(s,\hat{X}_{m}(\widehat{\omega}),x)\dif \widehat{\bb{P}}(\widehat{\omega})\\\nonumber
&+\widetilde{\bb{E}} \int_{\widehat{\Omega}}D_2D_1\dod[2]{F}{m}(Y_{\cdot m't}(s) \otimes  m')(Y_{\xi m't}(s),\widetilde{Y}_{\hat{X}_{m'}(\widehat{\omega}) m't}(s))\widetilde{\bar{Y}}_{m't}(s,\hat{X}_{m}(\widehat{\omega}),x)\dif \widehat{\bb{P}}(\widehat{\omega})\\\nonumber
&\left.-\widetilde{\bb{E}} \int_{\widehat{\Omega}}D_2D_1\dod[2]{F}{m}(Y_{\cdot m't}(s) \otimes  m')(Y_{\xi m't}(s),\widetilde{Y}_{\hat{X}_{m}(\widehat{\omega}) m't}(s))\widetilde{\bar{Y}}_{m't}(s,\hat{X}_{m}(\widehat{\omega}),x)\dif \widehat{\bb{P}}(\widehat{\omega})\right|^2\\\nonumber
\leq& 2c^2\widetilde{\bb{E}}\int_{\widehat{\Omega}}\left|\widetilde{\bar{Y}}_{m't}(s,\hat{X}_{m'}(\widehat{\omega}),x)-\widetilde{\bar{Y}}_{m't}(s,\hat{X}_{m}(\widehat{\omega}),x)\right|^2d\widehat{\bb{P}}(\widehat{\omega})\\\nonumber
&+2c^2 C_T(1+|x|^2) \cdot\widetilde{\bb{E}}\int_{\widehat{\Omega}}\left|\widetilde{Y}_{\hat{X}_{m'}(\widehat{\omega}) m't}(s)-\widetilde{Y}_{\hat{X}_{m}(\widehat{\omega}) m't}(s)\right|^2d\widehat{\bb{P}}(\widehat{\omega})\\\nonumber
\leq& \left(\frac{24(c_T^2+c^2T^2)T^2}{ \lambda^2-8(c_T^2+c^2T^2)T^2 }\left(\frac{\lambda}{\lambda-T(c_T+cT)}\right)^2+\frac{\lambda}{\lambda-T(c_T+cT)}\right)\cdot   2c^2C_T(1+|x|^2) \cdot\int_{\widehat{\Omega}}\left|\hat{X}_{m'}(\widehat{\omega})-\hat{X}_{m}(\widehat{\omega})\right|^2d\widehat{\bb{P}}(\widehat{\omega})\\\label{3primeeq3}
=& C_6(c,c_T,\lambda,T) \cdot C_T(1+|x|^2) \cdot W_2^2(m,m'),
\end{align}\normalsize
where $\hat{X}_{m}$ and $\hat{X}_{m'}$ are random variables in $L^{2}(\widehat{\Omega},\widehat{\mathcal{A}},\widehat{\bb{P}};\bb{R}^{n})$ such that 
\begin{equation}
W_{2}^{2}(m,m')=\bb{E}[|\hat{X}_{m}-\hat{X}_{m'}|^{2}],
\end{equation}
and $C_6(c,c_T,\lambda,T):=\left(\frac{24(c_T^2+c^2T^2)T^2}{ \lambda^2-8(c_T^2+c^2T^2)T^2 }\left(\frac{\lambda}{\lambda-T(c_T+cT)}\right)^2+\frac{\lambda}{\lambda-T(c_T+cT)}\right)\cdot   2c^2$.

Since $D^2\dod{F}{m}(m)(x)$ and $D^2\dod{F_T}{m}(m)(x)$ are assumed to be jointly Lipschitz continuous in $(m,x)\in \mathcal{P}_2(\mathbb{R}^n)\times\mathbb{R}^n$, and $D_1\dod[2]{F}{m}(m)(x,\widetilde{x})$, $D_1\dod[2]{F_T}{m}(m)(x,\widetilde{x})$, $D_2D_1\dod[2]{F}{m}(m)(x,\tilde{x})$ and $D_2D_1\dod[2]{F_T}{m}(m)(x,\tilde{x})$ are assumed to be jointly Lipschitz continuous in $(m,x,\tilde{x})\in \mathcal{P}_2(\mathbb{R}^n)\times\mathbb{R}^n\times\mathbb{R}^n$, there exist positive constants $C_{D2dF}$, $C_{Dd2F}$ and $C_{D2d2F}$ such that
\begin{align*}
&\bb{E}\left|D^{2}\dod{F}{m}(Y_{\cdot m' t}(\tau) \otimes  m')(Y_{\xi m't}(\tau)) -D^{2}\dod{F}{m}(Y_{\cdot mt}(\tau) \otimes  m)(Y_{\xi mt}(\tau)) \right|^2\\
\leq& C_{D2dF}\left(\bb{E}W_2^2(Y_{\cdot m' t}(\tau) \otimes  m',Y_{\cdot m' t}(\tau) \otimes  m')+\bb{E}\left|Y_{\xi m't}(\tau)-Y_{\xi mt}(\tau)\right|^2\right), \\
&\bb{E}\left|D^{2}\dod{F_T}{m}(Y_{\cdot m' t}(\tau) \otimes  m')(Y_{\xi m't}(\tau)) -D^{2}\dod{F_T}{m}(Y_{\cdot mt}(\tau) \otimes  m)(Y_{\xi mt}(\tau)) \right|^2\\
\leq& C_{D2dF}\left(\bb{E}W_2^2(Y_{\cdot m' t}(\tau) \otimes  m',Y_{\cdot m' t}(\tau) \otimes  m')+\bb{E}\left|Y_{\xi m't}(\tau)-Y_{\xi mt}(\tau)\right|^2\right), 
\end{align*}
\begin{align*}
&\bb{E}\left|D_1\dod[2]{F}{m}(Y_{\cdot m't}(\tau) \otimes  m')(Y_{\xi m't}(\tau),\widetilde{Y}_{x m't}(\tau))-D_1\dod[2]{F}{m}(Y_{\cdot mt}(\tau) \otimes  m)(Y_{\xi mt}(\tau),\widetilde{Y}_{x mt}(\tau))\right|^2\\
\leq& C_{Dd2F}\left(\bb{E}W_2^2(Y_{\cdot m' t}(\tau) \otimes  m',Y_{\cdot m' t}(\tau) \otimes  m')+\bb{E}\left|Y_{\xi m't}(\tau)-Y_{\xi mt}(\tau)\right|^2+\bb{E}\left|\widetilde{Y}_{x m't}(\tau)-\widetilde{Y}_{x mt}(\tau)\right|^2\right),
\end{align*}
\begin{align*}
&\bb{E}\left|D_1\dod[2]{F_T}{m}(Y_{\cdot m't}(\tau) \otimes  m')(Y_{\xi m't}(\tau),\widetilde{Y}_{x m't}(\tau))-D_1\dod[2]{F_T}{m}(Y_{\cdot mt}(\tau) \otimes  m)(Y_{\xi mt}(\tau),\widetilde{Y}_{x mt}(\tau))\right|^2\\
\leq& C_{Dd2F}\left(\bb{E}W_2^2(Y_{\cdot m' t}(\tau) \otimes  m',Y_{\cdot m' t}(\tau) \otimes  m')+\bb{E}\left|Y_{\xi m't}(\tau)-Y_{\xi mt}(\tau)\right|^2+\bb{E}\left|\widetilde{Y}_{x m't}(\tau)-\widetilde{Y}_{x mt}(\tau)\right|^2\right),
\end{align*}
and
\begin{align*}
&\bb{E}\left|D_2D_1\dod[2]{F}{m}(Y_{\cdot m't}(\tau) \otimes  m')(Y_{\xi m't}(\tau),\widetilde{Y}_{\zeta m't}(\tau))-D_2D_1\dod[2]{F}{m}(Y_{\cdot mt}(\tau) \otimes  m)(Y_{\xi mt}(\tau),\widetilde{Y}_{\zeta mt}(\tau))\right|^2\\
\leq& C_{D2d2F}\left(\bb{E}W_2^2(Y_{\cdot m' t}(\tau) \otimes  m',Y_{\cdot m' t}(\tau) \otimes  m')+\bb{E}\left|Y_{\xi m't}(\tau)-Y_{\xi mt}(\tau)\right|^2+\bb{E}\left|\widetilde{Y}_{\zeta m't}(\tau)-\widetilde{Y}_{\zeta mt}(\tau)\right|^2\right),
\end{align*}
\begin{align*}
&\bb{E}\left|D_2D_1\dod[2]{F_T}{m}(Y_{\cdot m't}(\tau) \otimes  m')(Y_{\xi m't}(\tau),\widetilde{Y}_{\zeta m't}(\tau))-D_2D_1\dod[2]{F_T}{m}(Y_{\cdot mt}(\tau) \otimes  m)(Y_{\xi mt}(\tau),\widetilde{Y}_{\zeta mt}(\tau))\right|^2\\
\leq& C_{D2d2F}\left(\bb{E}W_2^2(Y_{\cdot m' t}(\tau) \otimes  m',Y_{\cdot m' t}(\tau) \otimes  m')+\bb{E}\left|Y_{\xi m't}(\tau)-Y_{\xi mt}(\tau)\right|^2+\bb{E}\left|\widetilde{Y}_{\zeta m't}(\tau)-\widetilde{Y}_{\zeta mt}(\tau)\right|^2\right),
\end{align*}
based on which we can combine with \eqref{eq:6-15}, \eqref{eq:a priori1}, \eqref{iieq3}-\eqref{iieq4}, \eqref{W2Y}, \eqref{LipsYxi}, \eqref{3primeeq1}-\eqref{3primeeq3} to get

\small
\begingroup
\allowdisplaybreaks
\begin{align*}
&\sup_{s\in[t,T]}\bb{E}\left|\bar{Y}_{m't}(s,\xi,x)-\bar{Y}_{mt}(s,\xi,x)\right|^2 
\leq \dfrac{|T-t|^2}{\lambda^2} \sup_{s\in[t,T]}\bb{E}\left|\bar{Z}_{m't}(s,\xi,x)-\bar{Z}_{mt}(s,\xi,x)\right|^2,\\
\nonumber
&\bb{E}\left|\bar{Z}_{m' t}(s,\xi,x)-\bar{Z}_{mt}(s,\xi,x)\right|^2 \\\nonumber
\leq&6\bb{E}\Bigg[ \bigg|\int_{s}^{T} \left(D^{2}\dod{F}{m}(Y_{\cdot m' t}(\tau) \otimes  m')(Y_{\xi m't}(\tau))\bar{Y}_{m't}(\tau,\xi,x)-D^{2}\dod{F}{m}(Y_{\cdot mt}(\tau) \otimes  m)(Y_{\xi mt}(\tau))\bar{Y}_{mt}(\tau,\xi,x)\right)\dif \tau\bigg|^2  .\\
\nonumber
&\ \ \ \ + \bigg|D^{2}\dod{}{m}F_{T}(Y_{\cdot m't}(T) \otimes  m')(Y_{\xi m't}(T))\bar{Y}_{m't}(T,\xi,x)-D^{2}\dod{}{m}F_{T}(Y_{\cdot mt}(T) \otimes  m)(Y_{\xi mt}(T))\bar{Y}_{mt}(T,\xi,x)\bigg|^2\\
\nonumber
&\ \ \ \ +\bigg|\widetilde{\bb{E}}\int_{s}^{T}\int_{\bb{R}^n}D_2D_1\dod[2]{F}{m}(Y_{\cdot m't}(\tau) \otimes  m')(Y_{\xi m't}(\tau),\widetilde{Y}_{\zeta m't}(\tau))\widetilde{\bar{Y}}_{m't}(\tau,\zeta,x)\dif m'(\zeta) \dif \tau\\
\nonumber
&\ \ \ \ \ \ \ \ -\widetilde{\bb{E}}\int_{s}^{T}\int_{\bb{R}^n}D_2D_1\dod[2]{F}{m}(Y_{\cdot mt}(\tau) \otimes  m)(Y_{\xi mt}(\tau),\widetilde{Y}_{\zeta mt}(\tau))\widetilde{\bar{Y}}_{mt}(\tau,\zeta,x)\dif m(\zeta) \dif \tau\bigg|^2\\
\nonumber
&\ \ \ \ +\bigg|\widetilde{\bb{E}}\int_{\bb{R}^n}D_2D_1\dod[2]{F_{T}}{m}(Y_{\cdot m't}(T) \otimes  m')(Y_{\xi m't}(T),\widetilde{Y}_{\zeta m't}(T))\widetilde{\bar{Y}}_{m't}(T,\eta,x)\dif m'(\zeta)\\
\nonumber
&\ \ \ \ \ \ \ \ -\widetilde{\bb{E}}\int_{\bb{R}^n}D_2D_1\dod[2]{F_{T}}{m}(Y_{\cdot mt}(T) \otimes  m)(Y_{\xi mt}(T),\widetilde{Y}_{\zeta mt}(T))\widetilde{\bar{Y}}_{mt}(T,\zeta,x)\dif m(\zeta)\bigg|^2\\
&\ \ \ \ +\bigg|\int_{s}^{T}\widetilde{\bb{E}}D_1\dod[2]{F}{m}(Y_{\cdot m't}(\tau) \otimes  m')(Y_{\xi m't}(\tau),\widetilde{Y}_{xm't}(\tau))\dif \tau- \int_{s}^{T}\widetilde{\bb{E}}D_1\dod[2]{F}{m}(Y_{\cdot mt}(\tau) \otimes  m)(Y_{\xi mt}(\tau),\widetilde{Y}_{xmt}(\tau))\dif \tau\bigg|^2\\
&\ \ \ \ +\bigg|\widetilde{\bb{E}}D_1\dod[2]{F_T}{m}(Y_{\cdot m't}(T) \otimes  m')(Y_{\xi m't}(T),\widetilde{Y}_{xm't}(T))-\widetilde{\bb{E}}D_1\dod[2]{F_T}{m}(Y_{\cdot mt}(T) \otimes  m)(Y_{\xi mt}(T),\widetilde{Y}_{xmt}(T))\bigg|^2\Bigg]\\
\leq&12\left(c_T^2+c^2|T-t|^2\right)\cdot\sup_{s\in[t,T]} \bb{E}\left|\bar{Y}_{m't}(s,\xi,x)-\bar{Y}_{mt}(s,\xi,x)\right|^2\\
&+12C_T(1+|x|^2)(1+|T-t|^2)\cdot C_{D2dF}\cdot\sup_{s\in[t,T]}\left(\bb{E}W_2^2(Y_{\cdot m' t}(s) \otimes  m',Y_{\cdot m' t}(s) \otimes  m')+\bb{E}\left|Y_{\xi m't}(s)-Y_{\xi mt}(s)\right|^2\right)\\
&+18C_6(c,c_T,\lambda,T) C_T(1+|x|^2)(1+|T-t|^2)\cdot W_2^2(m,m')\\
&+18\left(c_T^2+c^2|T-t|^2\right)\cdot\sup_{s\in[t,T]} \bb{E}\int_{\bb{R}^n}\left|\bar{Y}_{m't}(s,\zeta,x)-\bar{Y}_{mt}(s,\zeta,x)\right|^2dm(\zeta)\\
&+18C_T(1+|x|^2)(1+|T-t|^2)\cdot C_{D2d2F}\cdot\sup_{s\in[t,T]}\bigg(\bb{E}W_2^2(Y_{\cdot m' t}(s) \otimes  m',Y_{\cdot m' t}(s) \otimes  m')+\bb{E}\left|Y_{\xi m't}(s)-Y_{\xi mt}(s)\right|^2\\
&\ \ \ \ \ \ \ \ +\bb{E}\int_{\bb{R}^n}\left| Y_{\zeta m't}(s)- Y_{\zeta mt}(s)\right|^2dm(\zeta)\bigg)\\
&+ 6(1+|T-t|^2)\cdot C_{Dd2F}\cdot\sup_{s\in[t,T]}\bigg(\bb{E}W_2^2(Y_{\cdot m' t}(s) \otimes  m',Y_{\cdot m' t}(s) \otimes  m')+\bb{E}\left|Y_{\xi m't}(s)-Y_{\xi mt}(s)\right|^2+\bb{E}\left| Y_{x m't}(s)- Y_{x mt}(s)\right|^2\bigg)\\
\leq&12\left(c_T^2+c^2|T-t|^2\right)\cdot\sup_{s\in[t,T]} \bb{E}\left|\bar{Y}_{m't}(s,\xi,x)-\bar{Y}_{mt}(s,\xi,x)\right|^2\\
&+12C_T(1+|x|^2)(1+|T-t|^2)\cdot C_{D2dF}\cdot \left(2 \left(\frac{\lambda}{\lambda-T(c_T+cT)}\right)^2\cdot\dfrac{\lambda^2-T^2C_{DdF}(1+T^2)}{\lambda^2-3T^2C_{DdF}(1+T^2)}\cdot W_2^2(m,m')\right.\\
&\ \ \ \ \ \ \ \ \ \ \ \ \ \ \ \ \left.+\dfrac{2T^2C_{DdF}(1+T^2)}{\lambda^2-3T^2C_{DdF}(1+T^2)}  \left(\frac{\lambda}{\lambda-T(c_T+cT)}\right)^2 \cdot W_2^2(m,m')\right)\\
&+18C_6(c,c_T,\lambda,T)\cdot C_T(1+|x|^2)\cdot(1+|T-t|^2)\cdot W_2^2(m,m')\\
&+18\left(c_T^2+c^2|T-t|^2\right)\sup_{s\in[t,T]} \bb{E}\int_{\bb{R}^n}\left|\bar{Y}_{m't}(s,\zeta,x)-\bar{Y}_{mt}(s,\zeta,x)\right|^2dm(\zeta)\\
&+\left(18C_T(1+|x|^2)\cdot(1+|T-t|^2)\cdot C_{D2d2F}+6(1+|T-t|^2)\cdot C_{Dd2F}\right)\cdot\\
&\ \ \ \ \ \ \ \ \left(2 \left(\frac{\lambda}{\lambda-T(c_T+cT)}\right)^2\cdot\dfrac{\lambda^2-T^2C_{DdF}(1+T^2)}{\lambda^2-3T^2C_{DdF}(1+T^2)}+\dfrac{4T^2C_{DdF}(1+T^2)}{\lambda^2-3T^2C_{DdF}(1+T^2)}  \left(\frac{\lambda}{\lambda-T(c_T+cT)}\right)^2\right)W_2^2(m,m').
\end{align*}\endgroup\normalsize
Therefore, we arrive with
\begin{align}\label{EY2}
&\sup_{s\in[t,T]}\bb{E}\left|\bar{Y}_{m't}(s,\xi,x)-\bar{Y}_{mt}(s,\xi,x)\right|^2 
\leq \dfrac{|T-t|^2}{\lambda^2} \sup_{s\in[t,T]}\bb{E}\left|\bar{Z}_{m't}(s,\xi,x)-\bar{Z}_{mt}(s,\xi,x)\right|^2,\\
\nonumber
&\bb{E}\left|\bar{Z}_{m' t}(s,\xi,x)-\bar{Z}_{mt}(s,\xi,x)\right|^2 \\\nonumber
\leq&12\left(c_T^2+c^2|T-t|^2\right)\sup_{s\in[t,T]} \bb{E}\left|\bar{Y}_{m't}(s,\xi,x)-\bar{Y}_{mt}(s,\xi,x)\right|^2\\\label{EZ2}
&+18\left(c_T^2+c^2|T-t|^2\right)\sup_{s\in[t,T]} \bb{E}\int_{\bb{R}^n}\left|\bar{Y}_{m't}(s,\zeta,x)-\bar{Y}_{mt}(s,\zeta,x)\right|^2dm(\zeta)+C_7\cdot W_2^2(m,m'),
\end{align}
where 
\begingroup
\allowdisplaybreaks
\begin{align*}
C_7:=&12C_T(1+|x|^2)\cdot(1+|T-t|^2)\cdot C_{D2dF}\cdot \left(2 \left(\frac{\lambda}{\lambda-T(c_T+cT)}\right)^2\cdot\dfrac{\lambda^2-T^2C_{DdF}(1+T^2)}{\lambda^2-3T^2C_{DdF}(1+T^2)} \right.\\
&\ \ \ \ \ \ \ \ \ \ \ \ \ \ \ \ \left.+\dfrac{2T^2C_{DdF}(1+T^2)}{\lambda^2-3T^2C_{DdF}(1+T^2)}  \left(\frac{\lambda}{\lambda-T(c_T+cT)}\right)^2  \right)\\
&+18C_6(c,c_T,\lambda,T) C_T(1+|x|^2)(1+|T-t|^2) \\
&+\left(18C_T(1+|x|^2)(1+|T-t|^2)\cdot C_{D2d2F}+6(1+|T-t|^2)\cdot C_{Dd2F}\right)\cdot\\
&\ \ \ \ \ \ \ \ \left(2 \left(\frac{\lambda}{\lambda-T(c_T+cT)}\right)^2\cdot\dfrac{\lambda^2-T^2C_{DdF}(1+T^2)}{\lambda^2-3T^2C_{DdF}(1+T^2)}+\dfrac{4T^2C_{DdF}(1+T^2)}{\lambda^2-3T^2C_{DdF}(1+T^2)}  \left(\frac{\lambda}{\lambda-T(c_T+cT)}\right)^2  \right).
\end{align*}\endgroup
Then,
\begin{align*}
&\sup_{s\in[t,T]}\bb{E}\int_{\bb{R}^n}\left|\bar{Y}_{m't}(s,\xi,x)-\bar{Y}_{mt}(s,\xi,x)\right|^2 dm(\xi)
\leq \dfrac{|T-t|^2}{\lambda^2} \sup_{s\in[t,T]}\bb{E}\int_{\bb{R}^n}\left|\bar{Z}_{m't}(s,\xi,x)-\bar{Z}_{mt}(s,\xi,x)\right|^2dm(\xi),\\
\nonumber
&\sup_{s\in[t,T]}\bb{E}\int_{\bb{R}^n}\left|\bar{Z}_{m' t}(s,\xi,x)-\bar{Z}_{mt}(s,\xi,x)\right|^2dm(\xi) \\\nonumber
\leq&30\left(c_T^2+c^2|T-t|^2\right)\sup_{s\in[t,T]} \bb{E}\int_{\bb{R}^n}\left|\bar{Y}_{m't}(s,\xi,x)-\bar{Y}_{mt}(s,\xi,x)\right|^2dm(\xi)+C_7\cdot W_2^2(m,m'),
\end{align*}
which implies that, whenever $\lambda>T\sqrt{30(c_T^2+c^2T^2)}$,
\begin{align*}
&\sup_{s\in[t,T]}\bb{E}\int_{\bb{R}^n}\left|\bar{Y}_{m't}(s,\xi,x)-\bar{Y}_{mt}(s,\xi,x)\right|^2 dm(\xi)
\leq \dfrac{C_7 T^2}{\lambda^2-30(c_T^2+c^2T^2)T^2}\cdot W_2^2(m,m'),\\
\nonumber
&\sup_{s\in[t,T]}\bb{E}\int_{\bb{R}^n}\left|\bar{Z}_{m' t}(s,\xi,x)-\bar{Z}_{mt}(s,\xi,x)\right|^2dm(\xi)\leq \dfrac{C_7 \lambda^2}{\lambda^2-30(c_T^2+c^2T^2)T^2}\cdot W_2^2(m,m').
\end{align*}
Put this back to \eqref{EZ2}, we have
\begin{align*}
&\sup_{s\in[t,T]}\bb{E}\left|\bar{Y}_{m't}(s,\xi,x)-\bar{Y}_{mt}(s,\xi,x)\right|^2 
\leq \dfrac{|T-t|^2}{\lambda^2} \sup_{s\in[t,T]}\bb{E}\left|\bar{Z}_{m't}(s,\xi,x)-\bar{Z}_{mt}(s,\xi,x)\right|^2,\\
\nonumber
&\sup_{s\in[t,T]}\bb{E}\left|\bar{Z}_{m' t}(s,\xi,x)-\bar{Z}_{mt}(s,\xi,x)\right|^2 \\\nonumber
\leq&12\left(c_T^2+c^2|T-t|^2\right)\sup_{s\in[t,T]} \bb{E}\left|\bar{Y}_{m't}(s,\xi,x)-\bar{Y}_{mt}(s,\xi,x)\right|^2+C_8\cdot W_2^2(m,m'),
\end{align*}
where $C_8:=\dfrac{  \lambda^2-12(c_T^2+c^2T^2)T^2}{\lambda^2-30(c_T^2+c^2T^2)T^2}\cdot C_7$.
Therefore, we finally arrive with:
\begin{align*}
&\sup_{s\in[t,T]}\bb{E}\left|\bar{Y}_{m't}(s,\xi,x)-\bar{Y}_{mt}(s,\xi,x)\right|^2 
\leq \dfrac{C_8 T^2}{\lambda^2-12(c_T^2+c^2T^2)T^2}\cdot W_2^2(m,m') ,\\
\nonumber
&\sup_{s\in[t,T]}\bb{E}\left|\bar{Z}_{m' t}(s,\xi,x)-\bar{Z}_{mt}(s,\xi,x)\right|^2 \leq \dfrac{C_8 \lambda^2}{\lambda^2-12(c_T^2+c^2T^2)T^2}\cdot W_2^2(m,m').
\end{align*}

Since $\bar{Z}_{mt}(t,\xi,x)=D_\xi \frac{d}{dm}U(\xi,m,t)(x)$ is deterministic, therefore, for each $x\in\bb{R}^n$, $D_\xi \frac{d}{dm}U(\xi,m,t)(x)$ is Lipschitz continuous in $m\in\s{P}_2(\bb{R}^n)$ uniformly in $\xi$ and $t$ since the Lipschitz constant $\left(\dfrac{C_8 \lambda^2}{\lambda^2-12(c_T^2+c^2T^2)T^2}\right)^{1/2}$ is independent of $\xi$ and $t$,
and we have the following estimate:
\begin{align}\label{LipsDdUm}
\left|D_\xi \frac{d}{dm}U(\xi,m',t)(x)-D_\xi \frac{d}{dm}U(\xi,m,t)(x)\right| \leq& \left(\dfrac{C_8 \lambda^2}{\lambda^2-12(c_T^2+c^2T^2)T^2}\right)^{1/2}\cdot W_2(m',m).
\end{align}
By \eqref{LipsDdUx}, \eqref{LipsDdUxi} and \eqref{LipsDdUm}, we have
\begin{align*}
&\left|D_\xi \frac{d}{dm}U(\xi',m',t)(x')-D_\xi \frac{d}{dm}U(\xi,m,t)(x)\right|\\
\leq& \left|D_\xi \frac{d}{dm}U(\xi',m',t)(x')-D_\xi \frac{d}{dm}U(\xi',m',t)(x)\right|+\left|D_\xi \frac{d}{dm}U(\xi',m',t)(x)-D_\xi \frac{d}{dm}U(\xi,m',t)(x)\right|\\
&+\left|D_\xi \frac{d}{dm}U(\xi,m',t)(x)-D_\xi \frac{d}{dm}U(\xi,m,t)(x)\right|\\
\leq& \frac{3\lambda(c_T+cT)}{ \lambda-12\cdot(c_T+cT)\cdot T }\cdot\frac{\lambda}{\lambda-T(c_T+cT)}\cdot|x'-x|\\
&+ \left(\frac{24\lambda^2(c_T^2+c^2T^2)}{ \lambda^2-8(c_T^2+c^2T^2)T^2 }\left(\frac{\lambda}{\lambda-T(c_T+cT)}\right)^2\cdot C_T(1+|x|^2)\right)^{1/2}\cdot|\xi'-\xi|\\
&+\left(\dfrac{C_8 \lambda^2}{\lambda^2-12(c_T^2+c^2T^2)T^2}\right)^{1/2}\cdot W_2(m',m),
\end{align*}
which means that, for any compact set $K\subset \bb{R}^n$ and $t\in[0,T]$, $D_x \dfrac{d}{dm}U(\xi,m,t)(x)$ is jointly Lipschitz continuous in $(m,\xi,x)\in \mathcal{P}_2(\mathbb{R}^n)\times\mathbb{R}^n\times K$, and  this joint continuity is uniformly held in $t$.

Using the same arguments of Steps $1$ to $3$, as above, one can further establish the joint continuity of $D^2_\xi \frac{d}{dm}U(\xi,m,t)(x)$ in $(m,x,\xi)\in\s{P}_2(\bb{R}^n)\times\bb{R}^{n}\times\bb{R}^{n}$ which is stated in Theorem \ref{thm:continuity} under the following additional assumptions: (i) $D^2 \frac{dF}{dm}(m)(x)$, $D^2 \frac{dF_T}{dm}(m)(x)$, $D^3\frac{d}{dm}F(m)(x)$ and $D^3\frac{d}{dm}F_T(m)(x)$ are jointly Lipschitz continuous in $(m,x)\in\s{P}_2(\bb{R}^n)\times\bb{R}^n$; and (ii) $D_1^2D_2\frac{d^2}{dm^2}F(m)(x,\widetilde{x})$, $D_1^2D_2\frac{d^2}{dm^2}F_T(m)(x,\widetilde{x})$, $D_1^2\frac{d^2}{dm^2}F(m)(x,\widetilde{x})$ and $D_1^2\frac{d^2}{dm^2}F_T(m)(x,\widetilde{x})$ are jointly Lipschitz continuous in $(m,x,\widetilde{x})\in\s{P}_2(\bb{R}^n)\times\bb{R}^n\times\bb{R}^n$.\\
\end{proof}

}

Now we take $\tilde m,m \in \s{P}_2(\bb{R}^n)$ and $\epsilon > 0$ arbitrary, then define
\begin{align*}
m_\epsilon &:= m + \epsilon(\tilde m - m),\\
Y_\xi^\epsilon(s) &= \frac{1}{\epsilon}\del{Y_{\xi m_\epsilon t}(s) - Y_{\xi mt}(s)} - \int_{{\mathbb R}^n} \bar{Y}_{mt}(s,\xi,x)\dif \h{1pt}(\tilde m - m)(x),\\
Z_\xi^\epsilon(s) &= \frac{1}{\epsilon}\del{Z_{\xi m_\epsilon t}(s) - Z_{\xi mt}(s)} - \int_{{\mathbb R}^n} \bar{Z}_{mt}(s,\xi,x)\dif \h{1pt}(\tilde m - m)(x),\\
\s{Y}_\xi^\epsilon(s) &= \frac{1}{\epsilon}\del{\s{Y}_{\xi m_\epsilon t}(s) - \s{Y}_{\xi mt}(s)} - \int_{{\mathbb R}^n} \bar{\s{Y}}_{mt}(s,\xi,x)\dif \h{1pt}(\tilde m - m)(x),\\
\s{Z}_\xi^\epsilon(s) &= \frac{1}{\epsilon}\del{\s{Z}_{\xi m_\epsilon t}(s) - \s{Z}_{\xi mt}(s)} - \int_{{\mathbb R}^n} \bar{\s{Z}}_{mt}(s,\xi,x)\dif \h{1pt}(\tilde m - m)(x).
\end{align*}
Our goal is to show that $Y^\epsilon_{\cdot}$,$Z^\epsilon_{\cdot}$, $\s{Y}^\epsilon_{\cdot}$, and $\s{Z}^\epsilon_{\cdot}$ all converge to zero in $\s{H}_m$ as $\epsilon \to 0$.
we shall focus on $Y^\epsilon$ and $Z^\epsilon$, the proof for $\s{Y}^\epsilon$, and $\s{Z}^\epsilon$ being very similar.
First, observe that $Y_\xi^\epsilon(s) = -\frac{1}{\lambda}\int_t^s Z_\xi^\epsilon(\tau)\dif \tau$.
Next, we shall further divide $Z^\epsilon$ into two parts, using the following definitions:
\begin{align*}
Z_{\xi m t}^1(s) &:= \bb{E}\left[\left.\int_{s}^{T} D \dod{F}{m}(Y_{\cdot mt}(\tau) \otimes  m)(Y_{\xi mt}(\tau))\dif \tau\right|\mathcal{W}_{t}^{s}\right],\\
Z_{\xi m t}^2(s) &:= \bb{E}\left[\left.D \dod{F_T}{m}(Y_{\cdot mt}(T) \otimes  m)(Y_{\xi mt}(T))\right|\mathcal{W}_{t}^{s}\right],
\end{align*}
\begingroup
\allowdisplaybreaks
\begin{multline*}
\bar{Z}_{mt}^1(s,\xi,x) := 
\bb{E}\left[ \int_{s}^{T} D^{2}\dod{F}{m}(Y_{\cdot mt}(\tau) \otimes  m)(Y_{\xi mt}(\tau))\bar{Y}_{mt}(\tau,\xi,x)\dif \tau \right.\\
+\widetilde{\bb{E}}\int_{s}^{T}\int_{\bb{R}^n}D_2D_1\dod[2]{F}{m}(Y_{\cdot mt}(\tau) \otimes  m)(Y_{\xi mt}(\tau),\widetilde{Y}_{\zeta mt}(\tau))\widetilde{\bar{Y}}_{mt}(\tau,\zeta,x)\dif m(\zeta) \dif \tau\\
+ \left.\left.\int_{s}^{T}\widetilde{\bb{E}}D_1\dod[2]{F}{m}(Y_{\cdot mt}(\tau) \otimes  m)(Y_{\xi mt}(\tau),\widetilde{Y}_{xmt}(\tau))\dif \tau
\right|\mathcal{W}_{t}^{s}\right],
\end{multline*}\endgroup
\begin{multline*}
\bar{Z}_{mt}^2(s,\xi,x) := 
\bb{E}\left[ 
D^{2}\dod{}{m}F_{T}(Y_{\cdot mt}(T) \otimes  m)(Y_{\xi mt}(T))\bar{Y}_{mt}(T,\xi,x) \right.\\ +\widetilde{\bb{E}}\int_{\bb{R}^n}D_2D_1\dod[2]{F_{T}}{m}(Y_{\cdot mt}(T) \otimes  m)(Y_{\xi mt}(T),\widetilde{Y}_{\zeta mt}(T))\widetilde{\bar{Y}}_{mt}(T,\zeta,x)\dif m(\zeta)\\
+ \left.\left.\widetilde{\bb{E}}D_1\dod[2]{F_T}{m}(Y_{\cdot mt}(T) \otimes  m)(Y_{\xi mt}(T),\widetilde{Y}_{xmt}(T))\right|\mathcal{W}_{t}^{s}\right].
\end{multline*}
Then we can write $Z_\xi^\epsilon = Z_\xi^{\epsilon,1} + Z_\xi^{\epsilon,2}$ where for $i = 1,2$ we define
\begin{equation}
Z^{\epsilon,i}(s) := \frac{1}{\epsilon}\del{Z^i_{\xi m_\epsilon t}(s) - Z^i_{\xi mt}(s)} - \int_{{\mathbb R}^n} \bar{Z}^i_{mt}(s,\xi,x)\dif \h{1pt}(\tilde m - m)(x).
\end{equation}
Using the Fundamental Theorem of Calculus and rules for differentiation from Section \ref{sec:FORMALISM}, we can rewrite
\begin{multline*}
Z^1_{\xi m_\epsilon t}(s) - Z^1_{\xi mt}(s)
= \bb{E}\left[ \int_{s}^{T} \int_0^1 D^{2}\dod{F}{m}(\mu^\theta_\epsilon)(\chi_{\xi}^{\epsilon,\theta}(\tau))\del{Y_{\xi m_\epsilon t}(\tau) - Y_{\xi m_\epsilon t}(\tau)}\dif \theta \dif \tau \right.\\
+\widetilde{\bb{E}}\int_{s}^{T}\int_{\bb{R}^n} \int_0^1 \int_0^1 D_2D_1\dod[2]{F}{m}(\mu^\theta_\epsilon)(\chi_{\xi}^{\epsilon,\theta}(\tau),\widetilde{\chi}^{\epsilon,\theta'}_{\eta}(\tau))\del{\tilde{Y}_{\eta m_\epsilon t}(\tau) - \tilde{Y}_{\eta m t}(\tau)}\dif \theta'\dif \theta \dif m(\zeta) \dif \tau\\
+ \left.\left.\epsilon\int_{s}^{T} \widetilde{\bb{E}}\int_{{\mathbb R}^n}\int_0^1 D_1\dod[2]{F}{m}(\mu^\theta_\epsilon)(\chi_{\xi}^{\epsilon,\theta}(\tau),\widetilde{Y}_{xm_\epsilon t}(\tau))\dif \theta \dif \h{1pt}(\tilde m - m)(x)\dif \tau
\right|\mathcal{W}_{t}^{s}\right],
\end{multline*}
where
$\chi_{\xi}^{\epsilon,\theta}(\tau) := \theta Y_{\xi m_\epsilon t}(\tau) + (1-\theta)Y_{\xi m t}(\tau)$ and $
\mu_\epsilon^\theta := \theta Y_{\cdot m_\epsilon t}(\tau)  \otimes  m_\epsilon + (1-\theta)Y_{\cdot m t}(\tau)  \otimes  m$.
We now define
{\small
\begin{align*}
\Phi_1^\epsilon(\xi,\tau) &:= \int_0^1 \del{D^{2}\dod{F}{m}(\mu^\theta_\epsilon)(\chi_{\xi}^{\epsilon,\theta}(\tau)) - D^{2}\dod{F}{m}(Y_{\cdot mt}(\tau) \otimes  m)(Y_{\xi mt}(\tau))}\dif \theta,\\
\Phi_2^\epsilon(\xi,\eta,\tau) &:= \int_0^1 \int_0^1 \del{D_2D_1\dod[2]{F}{m}(\mu^\theta_\epsilon)(\chi_{\xi}^{\epsilon,\theta}(\tau),\widetilde{\chi}^{\epsilon,\theta'}_{\eta}(\tau)) - D_2D_1\dod[2]{F}{m}(Y_{\cdot mt}(\tau) \otimes  m)(Y_{\xi mt}(\tau),\widetilde{Y}_{\zeta mt}(\tau))}\dif \theta'\dif \theta,\\
\Phi_3^\epsilon(\xi,\tau) &:= \widetilde{\bb{E}}\int_{{\mathbb R}^n}\int_0^1 \del{D_1\dod[2]{F}{m}(\mu^\theta_\epsilon)(\chi_{\xi}^{\epsilon,\theta}(\tau),\widetilde{Y}_{xm_\epsilon t}(\tau)) - D_1\dod[2]{F}{m}(Y_{\cdot mt}(\tau) \otimes  m)(Y_{\xi mt}(\tau),\widetilde{Y}_{xmt}(\tau))}\dif \theta \dif \h{1pt}(\tilde m - m)(x).
\end{align*}}
By the a priori estimates from Proposition \ref{prop6-1}, arguing as in Section \ref{sec:further regularity}, we see that as $\epsilon \to 0$, $\Phi_1^\epsilon(\cdot,\tau),\Phi_3^\epsilon(\cdot,\tau) \to 0$ in $\s{H}_m$ uniformly in $\tau$, and similarly $\Phi_2^\epsilon(\cdot,\cdot,\tau) \to 0$ in $\s{H}_{m \times m}$, uniformly in $\tau$.

Now $Z_\xi^{\epsilon,1}$ can be written in the form
\begin{multline*}
Z_\xi^{\epsilon,1}(s) =  \bb{E}\left[ \int_{s}^{T} \int_0^1 D^{2}\dod{F}{m}(\mu^\theta_\epsilon)(\chi_{\xi}^{\epsilon,\theta}(\tau))Y^\epsilon(\tau)\dif \theta \dif \tau \right.\\
+\widetilde{\bb{E}}\int_{s}^{T}\int_{\bb{R}^n} \int_0^1 \int_0^1 D_2D_1\dod[2]{F}{m}(\mu^\theta_\epsilon)(\chi_{\xi}^{\epsilon,\theta}(\tau),\widetilde{\chi}^{\epsilon,\theta'}_{\eta}(\tau))Y^\epsilon(\tau)\dif \theta'\dif \theta \dif m(\zeta) \dif \tau\\
+ \int_{s}^{T} \int_{{\mathbb R}^n}\Phi_1^\epsilon(\xi,\tau)\bar{Y}_{mt}(\tau,\xi,x)\dif \h{1.5pt} (\tilde m - m)(x)\dif \tau \\
+\widetilde{\bb{E}}\int_{s}^{T}\int_{\bb{R}^n}\int_{\bb{R}^n}\Phi_2^\epsilon(\xi,\eta,\tau)\widetilde{\bar{Y}}_{mt}(\tau,\zeta,x)\dif m(\zeta)\dif \h{1.5pt} (\tilde m - m)(x) \dif \tau
+ \left.\left.\int_{s}^{T}\Phi_3^\epsilon(\xi,\tau)\dif \tau
\right|\mathcal{W}_{t}^{s}\right].
\end{multline*}
By the same argument, $Z^{\epsilon,2}$ can be written in an analogous way, with $F$ replaced by $F_T$ and integrals in time replaced by evaluation at $T$.
Indeed, we can write
\begin{multline*}
Z_\xi^{\epsilon}(s) =  \bb{E}\left[ \int_{s}^{T} \int_0^1 D^{2}\dod{F}{m}(\mu^\theta_\epsilon)(\chi_{\xi}^{\epsilon,\theta}(\tau))Y^\epsilon(\tau)\dif \theta \dif \tau
+ \int_0^1 D^{2}\dod{F_T}{m}(\mu^\theta_\epsilon)(\chi_{\xi}^{\epsilon,\theta}(T))Y^\epsilon(T)\dif \theta \right.\\
+\widetilde{\bb{E}}\int_{s}^{T}\int_{\bb{R}^n} \int_0^1 \int_0^1 D_2D_1\dod[2]{F}{m}(\mu^\theta_\epsilon)(\chi_{\xi}^{\epsilon,\theta}(\tau),\widetilde{\chi}^{\epsilon,\theta'}_{\eta}(\tau))Y^\epsilon(\tau)\dif \theta'\dif \theta \dif m(\zeta) \dif \tau\\
\left.\left.\widetilde{\bb{E}}\int_{\bb{R}^n} \int_0^1 \int_0^1 D_2D_1\dod[2]{F_T}{m}(\mu^\theta_\epsilon)(\chi_{\xi}^{\epsilon,\theta}(T),\widetilde{\chi}^{\epsilon,\theta'}_{\eta}(T))Y^\epsilon(T)\dif \theta'\dif \theta \dif m(\zeta)
+ \Phi^\epsilon(\xi)
\right|\mathcal{W}_{t}^{s}\right],
\end{multline*}
where $\Phi^\epsilon \to 0$ in $\s{H}_m$.
We use the fact that $Y_\xi^\epsilon(s) = -\frac{1}{\lambda}\int_t^s Z_\xi^\epsilon(\tau)\dif \tau$ and $\lambda$ is sufficiently large, as in the proof of Proposition \ref{prop6-1}, to deduce that $Z^\epsilon_{\cdot}(s) \to 0$ in $\s{H}_m$, uniformly in $s$.
Therefore the same holds for $Y^\epsilon$.
This concludes the proof.

{\color{black}

}

\appendix
\section{PROOFS FROM SECTION \ref{sec:MEAN-FIELD-TYPE}} \label{ap:A}

\subsection{PROOF OF LEMMA \ref{lem:isometry}}

	Without loss of generality we shall assume $(\Omega,\s{A},\bb{P}) = (\Omega_0 \times \Omega_1,\s{A}_0 \otimes \s{A}_1,\bb{P}_0 \times \bb{P}_1)$ as above.
	Let $v \in L^2_{\s{W}_{Xt}}(t,T;\s{H}_m)$.
	we shall show there exists a unique $\tilde v \in L^2_{\s{W}_{t}}(t,T;\s{H}_{X  \otimes  m})$ such that for any $\omega = (\omega_0,\omega_1) \in \Omega$ and any $s \in [t,T]$, we have $v(\omega,s,x) = \tilde v(\omega_0,s,X(\omega_1,x))$.
	First, note that because $(\omega_0,\omega_1,x) \mapsto v(\omega_0,\omega_1,s,x)$ is $(\s{W}_t^s \otimes \sigma(X) \otimes \s{B},\s{B})$ measurable, it follows that there exists a $(\s{W}_t^s \otimes \s{B},\s{B})$-measurable function $(\omega_0,x) \mapsto \tilde v (\omega_0,s,x)$ such that $v(\omega,s,x) = \tilde v(\omega_0,s,X(\omega_1,x))$.
	Thus $\tilde v$ is adapted to $\s{W}_t$.	
	Now observe that
	\begin{equation} \label{eq:isometry}
	\begin{aligned}
	\enVert{v}_{L^2_{\s{W}_{Xt}}(t,T;\s{H}_m)}^2 &=
	\int_t^T \int_{\Omega} \int_{\bb{R}^n} \abs{v(\omega,s,x)}^2 \dif m(x) \dif \bb{P}(\omega)\dif s\\
	&=
	\int_t^T \int_{\Omega_0}\int_{\Omega_1} \int_{\bb{R}^n} \abs{\tilde v(\omega_0,s,X(\omega_1,x))}^2 \dif m(x) \dif \bb{P}_1(\omega_1)\dif \bb{P}_0(\omega_0)\dif s\\
	&=
	\int_t^T \int_{\Omega_0} \int_{\bb{R}^n} \abs{\tilde v(\omega_0,s,x)}^2 \dif \h{1pt}(X  \otimes  m)(x) \dif \bb{P}_0(\omega_0)\dif s\\
	&= \enVert{\tilde v}_{L^2_{\s{W}_{t}}(t,T;\s{H}_{X  \otimes  m})}^2,
	\end{aligned}
	\end{equation}
	which proves $\tilde v \in L^2_{\s{W}_t}(t,T;\s{H}_{X  \otimes  m})$ and also that $v \mapsto \tilde v$ is an isometry.
	To see that $\tilde v$ is unique, observe that if $\tilde v'$ is another element of $L^2_{\s{W}_t}(t,T;\s{H}_{X  \otimes  m})$ such that $v(\omega,s,x) = \tilde v'(\omega_0,s,X(\omega_1,x))$, then for any other $u \in L^2_{\s{W}_t}(t,T;\s{H}_{X  \otimes  m})$, we have
	\begin{equation} 
	\begin{aligned}
	\int_t^T \int_{\Omega_0} \int_{\bb{R}^n} &u(\omega_0,s,x)\cdot \tilde v'(\omega_0,s,x) \dif \h{1pt}(X  \otimes  m)(x) \dif \bb{P}_0(\omega_0)\dif s\\
	&=\int_t^T \int_{\Omega_0}\int_{\Omega_1} \int_{\bb{R}^n} u(\omega_0,s,X(\omega_1,x)) \cdot \tilde v'(\omega_0,s,X(\omega_1,x)) \dif m(x) \dif \bb{P}_1(\omega_1)\dif \bb{P}_0(\omega_0)\dif s
	\\
	&=\int_t^T \int_{\Omega_0} \int_{\bb{R}^n} u(\omega_0,s,x)\cdot \tilde v(\omega_0,s,x) \dif \h{1pt}(X  \otimes  m)(x) \dif \bb{P}_0(\omega_0)\dif s
	\end{aligned}
	\end{equation}
	where we have used the fact that $f \mapsto f \circ X$ is an isometry from $L^2_{X  \otimes  m}$ to $L^2_{\bb{P} \times m}$ to see that $u(\omega,s,X(\omega_1{\color{black},x}))$ defines an element in $L^2_{W_{Xt}}(t,T;\s{H}_m)$.
	It follows that $\tilde v' = \tilde v$, so $v \mapsto \tilde v$ is a well-defined isometry.
	Linearity is easily checked. $\blacksquare$

\subsection{PROOF OF LEMMA \ref{lem:law of X_Xt}}
	First, notice from \eqref{eq:ordinary SDE} that
	\begin{equation}
	\bb{E}\int_{\bb{R}^n} \abs{x(s;x,\tilde v(\cdot,x))}^2 \dif m(x) \leq 3\int_{\bb{R}^n} \abs{x}^2 \dif m(x) + 3(s-t)\enVert{\tilde v}^2_{L^2_{\s{W}_t}(t,T;\s{H}_{X  \otimes  m})} + 3\abs{\eta}^2 (s-t),
	\end{equation}
	and thus, $\bb{P}_0$-a.s., $x(s;\cdot,\tilde v(\cdot,\cdot)) \in L^2_m(\bb{R}^n;\bb{R}^n)$.
	Now let $\phi:\bb{R}^n \to \bb{R}^n$ be a continuous function such that $x \mapsto \frac{\phi(x)}{1+\abs{x}^2}$ is bounded.
	Then, using Fubini's Theorem, we have
	\begin{equation}
	\begin{aligned}
	\bb{E}\int_{\bb{R}^n} \phi\del{X_{Xt}(s;v)(x)}\dif m(x)
	&= \int_\Omega \int_{\bb{R}^n} \phi\del{x(\omega_0,s;X(\omega_1,x),\tilde v(\cdot,X(\omega_1,x)))}\dif m(x)\dif \bb{P}_0(\omega_0,\omega_1)\\
	&= \int_{\Omega_0} \int_{\bb{R}^n} \phi\del{x(\omega_0,s;x,\tilde v(\cdot,x))}\dif \h{1pt}(X  \otimes  m)(x)\dif \bb{P}_0(\omega_0).
	\end{aligned}
	\end{equation}
	In particular, using the fact that the left-hand side is finite by \eqref{eq:EX(s)^2 bound}, we deduce that $x(s;\cdot,\tilde v(\cdot,\cdot)) \in L^2_{X  \otimes  m}(\bb{R}^n;\bb{R}^n)$.
	Therefore, we deduce
	\begin{equation}
	\bb{E}\int_{\bb{R}^n} \phi\del{X_{Xt}(s;v)(x)}\dif m(x)
	= \int_{\bb{R}^n} \phi(x)\dif \del{x(s;\cdot,\tilde v(\cdot,\cdot))  \otimes  (X  \otimes  m)}(x),
	\end{equation}
	which completes the proof. $\blacksquare$

\subsection{PROOF OF LEMMA \ref{lem3-1}}

Consider a control $v_{Xt}(s)+\epsilon\tilde{v}_{Xt}(s)$.
The corresponding
state is $X_{Xt}(s;v_{Xt}(\cdot))+\epsilon\int_{t}^{s}\tilde{v}_{Xt}(\tau)\dif \tau.$
Therefore 
\begin{multline*}
J_{Xt}(v_{Xt}(\cdot)+\epsilon\tilde{v}_{Xt}(\cdot))-J_{Xt}(v_{Xt}(\cdot))=\lambda\epsilon\int_{t}^{T}\ip{v_{Xt}(s)}{\tilde{v}_{Xt}(s)}\dif s+\epsilon^{2}\dfrac{\lambda}{2}\int_{t}^{T}||\tilde{v}_{Xt}(s)||^{2}\dif s\\
+ \int_{t}^{T}\del{F\del{\del{X_{Xt}(s;v_{Xt}(\cdot))+\epsilon\int_{t}^{s}\tilde{v}_{Xt}(\tau)\dif \tau} \otimes  m)}-F(X_{Xt}(s;v_{Xt}(\cdot)) \otimes  m)}\dif s\\
+ F_{T}\del{\del{X_{Xt}(T;v_{Xt}(\cdot))+\epsilon\int_{t}^{T}\tilde{v}_{Xt}(\tau)\dif \tau} \otimes  m}-F_T(X_{Xt}(T;v_{Xt}(\cdot)) \otimes  m)\\
=\lambda\epsilon\int_{t}^{T}\ip{v_{Xt}(s)}{\tilde{v}_{Xt}(s)}\dif s
+ \epsilon\int_{t}^{T}\int_{0}^{1}\ip{D_{X}F\del{\del{X_{Xt}(s;v_{Xt}(\cdot))+\theta\epsilon\int_{t}^{s}\tilde{v}_{Xt}(\tau)\dif \tau} \otimes  m}}{\int_{t}^{s}\tilde{v}_{Xt}(\tau)\dif \tau}\dif s\dif \theta\\
+ \epsilon\int_{0}^{1}\ip{D_{X}F_{T}\del{\del{X_{Xt}(T;v_{Xt}(\cdot))+\theta\epsilon\int_{t}^{T}\tilde{v}_{Xt}(\tau)\dif \tau} \otimes  m}}{\int_{t}^{T}\tilde{v}_{Xt}(\tau)\dif \tau}\dif \theta+o(\epsilon).
\end{multline*}
From the continuity assumptions \eqref{eq:3-3} and the membership $v_{Xt},\tilde v_{Xt} \in L^2_{\s{W}_{Xt}}(t,T;\s{H}_m)$, we obtain
\begin{multline*}
\dfrac{J_{Xt}(v_{Xt}(\cdot)+\epsilon\tilde{v}_{Xt}(\cdot))-J_{Xt}(v_{Xt}(\cdot))}{\epsilon}\rightarrow\lambda\int_{t}^{T}\ip{v_{Xt}(s)}{\tilde{v}_{Xt}(s)}\dif s\\
 +  \int_{t}^{T}\ip{D_{X}F(X_{Xt}(s;v_{Xt}(\cdot)) \otimes  m)}{\int_{t}^{s}\tilde{v}_{Xt}(\tau)\dif \tau}\dif s +  \ip{D_{X}F_{T}(X_{Xt}(T;v_{Xt}(\cdot)) \otimes  m)}{\int_{t}^{T}\tilde{v}_{Xt}(\tau)\dif \tau}\\
=\lambda\int_{t}^{T}\ip{v_{Xt}(s)}{\tilde{v}_{Xt}(s)}\dif s \\
+ \int_{t}^{T}\ip{\int_{s}^{T}D_{X}F(X_{Xt}(\tau;v_{Xt}(\cdot)) \otimes  m)\dif \tau + D_{X}F_{T}(X_{Xt}(T;v_{Xt}(\cdot)) \otimes  m)}{\tilde{v}_{Xt}(s)}\dif s.
\end{multline*}
 Using the fact that $\tilde{v}_{Xt}(s)$ is arbitrary and $\mathcal{W}_{Xt}^{s}$
measurable, we immediately obtain  formula \ref{eq:DvJ}. $\blacksquare$ 

\subsection{PROOF OF PROPOSITION \ref{prop:convexity}}

We take two controls $v_{Xt}^{1}$ and $v_{Xt}^{2}$. We are going
to check that 
\begin{multline}
\int_{t}^{T}\ip{D_{v}J_{Xt}(v_{Xt}^{1}(\cdot))(s)-D_{v}J_{Xt}(v_{Xt}^{2}(\cdot))(s)}{v_{Xt}^{1}(s)-v_{Xt}^{2}(s)}\dif s\\
\geq \del{\lambda-T\del{c'_{T} + \dfrac{c'T}{2}}}\int_{t}^{T}||v_{Xt}^{1}(s)-v_{Xt}^{2}(s)||^{2}\dif s.\label{eq:Ap1}
\end{multline}
Then from the assumption (\ref{eq:3-6}) the result will follow immediately.
To simplify notation, we set $v^{1}(s)=v_{Xt}^{1}(s),\:v^{2}(s)=v_{Xt}^{2}(s)$
and 
\[
X^{1}(s)=X_{Xt}(s;v_{Xt}^{1}(\cdot)),\:X^{2}(s)=X_{Xt}(s;v_{Xt}^{2}(\cdot))
\]
From formula (\ref{eq:DvJ}) we have 
\begin{multline*}
\int_{t}^{T}\ip{D_{v}J_{Xt}(v_{Xt}^{1}(\cdot))(s)-D_{v}J_{Xt}(v_{Xt}^{2}(\cdot))(s)}{v_{Xt}^{1}(s)-v_{Xt}^{2}(s)}\dif s=\lambda\int_{t}^{T}||v^{1}(s)-v^{2}(s)||^{2}\dif s \\
 + \int_{t}^{T}\ip{\int_{s}^{T}(D_{X}F(X^{1}(\tau) \otimes  m)-D_{X}F(X^{2}(\tau) \otimes  m))\dif \tau}{v^{1}(s)-v^{2}(s)}\dif s\\
 + \int_{t}^{T}\ip{D_{X}F_T(X^{1}(T) \otimes  m)-D_{X}F_T(X^{2}(T) \otimes  m)}{v^{1}(s)-v^{2}(s)}\dif s
\end{multline*}
using the fact that $v^{1}(s)-v^{2}(s)$ is $\mathcal{W}_{Xt}^{s}$
measurable. Next since $v^{1}(s)-v^{2}(s)=\dod{}{s}(X^{1}(s)-X^{2}(s)$
and $X^{1}(t)-X^{2}(t)=0,$ we have 
\begin{multline*}
\int_{t}^{T}\ip{\int_{s}^{T}(D_{X}F(X^{1}(\tau) \otimes  m)-D_{X}F(X^{2}(\tau) \otimes  m))\dif \tau}{v^{1}(s)-v^{2}(s)}\dif s\\
+ \int_{t}^{T}\ip{D_{X}F_T(X^{1}(T) \otimes  m)-D_{X}F_T(X^{2}(T) \otimes  m)}{v^{1}(s)-v^{2}(s)}\dif s\\
= \int_{t}^{T}\ip{D_{X}F(X^{1}(s) \otimes  m)-D_{X}F(X^{2}(s) \otimes  m)}{X^{1}(s)-X^{2}(s)}\dif s\\
 + \ip{D_{X}F_T(X^{1}(T) \otimes  m)-D_{X}F_T(X^{2}(T) \otimes  m)}{X^{1}(T)-X^{2}(T)} \\
 \geq
 -c'\int_{t}^{T}||X^{1}(s)-X^{2}(s)||^{2}\dif s-c'_{T}||X^{1}(T)-X^{2}(T)||^{2}
\end{multline*}
by the monotonicity conditions \eqref{eq:3-4} and \eqref{eq:3-5}.
 We next use 
\[
X^{1}(s)-X^{2}(s)=\int_{t}^{s}(v^{1}(\tau)-v^{2}(\tau))\dif \tau
\]
 to deduce
\[
||X^{1}(T)-X^{2}(T)||^{2}\leq T\int_{t}^{T}||v^{1}(s)-v^{2}(s)||^{2}\dif s
\]
and
\[
\int_{t}^{T}||X^{1}(s)-X^{2}(s)||^{2}\dif s\leq\dfrac{T^{2}}{2}\int_{t}^{T}||v^{1}(s)-v^{2}(s)||^{2}\dif s.
\]
 Collecting results, we obtain (\ref{eq:Ap1}), as desired.
 Equation \ref{eq:Ap1} also implies
\begin{equation}
\int_{t}^{T}\ip{D_{v}J_{Xt}(v_{Xt}(\cdot))(s)-D_{v}J_{Xt}(0)(s)}{v_{Xt}(s)}\dif s \geq c_{0}\int_{t}^{T}||v_{Xt}(s)||^{2}\dif s\label{eq:Ap2}
\end{equation}
 But 
\[
J_{Xt}(v_{Xt}(\cdot))-J_{Xt}(0)=\int_{0}^{1}\ip{D_{v}J_{Xt}(\theta v_{Xt}(\cdot))(s)}{v_{Xt}(s)}\dif s\dif \theta,
\]
which, when combined with (\ref{eq:Ap2}), implies
\[
J_{Xt}(v_{Xt}(\cdot))-J_{Xt}(0)\geq\int_{t}^{T}{D_{v}J_{Xt}(0)(s)}{v_{Xt}(s)}\dif s + \dfrac{c_{0}}{2}\int_{t}^{T}||v_{Xt}(s)||^{2}\dif s.
\]
This implies that $J_{Xt}$ is both strictly convex and coercive, from which we deduce the existence and uniqueness of
a minimizer of $J_{Xt}(v_{Xt}(\cdot))$ by using Theorem 7.2.12. of \cite{DM07}. This completes the proof. $\blacksquare$

\section{PROOFS FROM SECTION \ref{sec:value fxn}} \label{ap:B}

\subsection{PROOF OF PROPOSITION \ref{prop4-1}}

To simplify notation, we omit the indices $Xt$ in $Y_{Xt}(s),Z_{Xt}(s).$
Recall that the inner product $\ip{Y(s)}{Z(s)}$ is defined as an expected value.
Using the tower property of iterated expectation, \eqref{eq:3-7}-\eqref{eq:3-8} together imply
\begin{multline*}
\ip{Y(s + \epsilon)}{Z(s + \epsilon)}-\ip{Y(s)}{Z(s)} = \ip{Y(s + \epsilon)}{Z(s + \epsilon)-Z(s)} \\
 + \ip{Y(s + \epsilon)-Y(s)}{Z(s)} = -\ip{Y(s + \epsilon)}{\int_{s}^{s + \epsilon}D_{X}F(Y_{Xt}(\tau) \otimes  m)\dif \tau}-\dfrac{1}{\lambda}\ip{\int_{s}^{s + \epsilon}Z(\tau)\dif \tau}{Z(s)}.
 \end{multline*}
Divide by $\epsilon$ and let $\epsilon$ tend to 0.
As the necessary continuity to pass to the limit is easily checked, we obtain
\[
\dod{}{s}\ip{Y(s)}{Z(s)} = -\dfrac{1}{\lambda}\enVert{Z(s)}^{2} - \ip{Y(s)}{D_{X}F(Y(s) \otimes  m})
\]
Integrating between $t$ and $T,$ we obtain 
\begin{equation}
\ip{X}{Z(t)} = \dfrac{1}{\lambda}\int_{t}^{T}\enVert{Z(s)}^{2}\dif s + \int_{t}^{T}\ip{D_{X}F(Y(s) \otimes  m)}{Y(s)}\dif s + \ip{D_{X}F_{T}(Y(T) \otimes  m)}{Y(T)}.\label{eq:ApB2}
\end{equation}
We also have 
\begin{equation}
\ip{X}{Z(t)} = \ip{X}{\int_{t}^{T}D_{X}F(Y(s) \otimes  m)\dif s + D_{X}F_{T}(Y(T) \otimes  m)}\label{eq:ApB3}
\end{equation}
simply by taking the inner product of $X$ with \eqref{eq:3-8}.
Combining \eqref{eq:ApB2} and \eqref{eq:ApB3}, we get
\begin{multline} \label{eq:ApB7}
\dfrac{1}{\lambda}\int_{t}^{T}\enVert{Z(s)}^{2}\dif s + \int_{t}^{T}\ip{D_{X}F(Y(s) \otimes  m) - D_{X}F(\delta)}{Y(s)}\dif s + \ip{D_{X}F_{T}(Y(T) \otimes  m) - D_{X}F_{T}(\delta)}{Y(T)} \\
= \ip{X}{\int_{t}^{T}(D_{X}F(Y(s) \otimes  m) - D_{X}F(\delta))\dif s + D_{X}F_{T}(Y(T) \otimes  m) - D_{X}F_{T}(\delta)} \\
+ \ip{X - Y(T)}{D_{X}F_{T}(\delta)} + \int_{t}^{T}\ip{X - Y(s)}{D_{X}F(\delta)}\dif s,
\end{multline}
where $\delta$ is the Dirac measure concentrated at the origin.
We proceed to estimate the right-hand side of \eqref{eq:ApB7}.
Using the fact that $D_{X}F(\delta)$ and $D_{X}F_{T}(\delta)$ are deterministic, we use \eqref{eq:3-7} to obtain
\begin{multline*}
\ip{X - Y(T)}{D_{X}F_{T}(\delta)} + \int_{t}^{T}\ip{X - Y(s)}{D_{X}F(\delta)} \\
= \dfrac{1}{\lambda}\ip{D_{X}F_{T}(\delta)}{\int_{t}^{T}Z(\tau)\dif \tau} + \dfrac{1}{\lambda}\int_{t}^{T}\ip{D_{X}F(\delta)}{\int_{t}^{s}Z(\tau)\dif \tau}\dif s,
\end{multline*}
and then by applying Cauchy-Schwartz we get 
\begin{multline}
\abs{\ip{X - Y(T)}{D_{X}F_{T}(\delta)} + \int_{t}^{T}\ip{X - Y(s)}{D_{X}F(\delta)}}\\
 \leq \dfrac{1}{\lambda}\sqrt{T}\del{||D_{X}F_{T}(\delta)|| + \dfrac{2}{3}T||D_{X}F(\delta)||} \sqrt{\int_{t}^{T}\enVert{Z(s)}^{2}\dif s}. \label{eq:ApB4}
\end{multline}
On the other hand, using the Lipschitz property \eqref{eq:3-3} and writing $\delta = 0  \otimes  m$ (Example \ref{ex:two push-forwards}), we have 
\begin{multline}
\abs{\ip{X}{\int_{t}^{T}(D_{X}F(Y(s) \otimes  m) - D_{X}F(\delta))\dif s + D_{X}F_{T}(Y(T) \otimes  m) - D_{X}F_{T}(\delta)}}\\
 \leq \enVert{X}\del{c_{T}||Y(T)|| + c\int_{t}^{T}\enVert{Y(s)}\dif s}. \label{eq:ApB5}
\end{multline}
Applying Cauchy-Schwartz 
 directly to Equation (\ref{eq:3-7}) we have 
\begin{equation}
\enVert{Y(s)} \leq \enVert{X} + \dfrac{1}{\lambda}\sqrt{s - t}\sqrt{\int_{t}^{T}\enVert{Z(\tau)}^{2}\dif \tau} + \enVert{\eta}\sqrt{s - t}, \label{eq:ApB5-2}
\end{equation}
which is plugged into \eqref{eq:ApB5} to get, after some simple estimates,
\begin{multline}
\abs{\ip{X}{\int_{t}^{T}(D_{X}F(Y(s) \otimes  m) - D_{X}F(\delta))\dif s + D_{X}F_{T}(Y(T) \otimes  m) - D_{X}F_{T}(\delta)}} \\
\leq (c_{T} + cT)\enVert{X}^{2} 
 + \sqrt{T}\del{c_{T} + \dfrac{2}{3}cT}\del{\enVert{\eta} + \dfrac{1}{\lambda}\sqrt{\int_{t}^{T}\enVert{Z(s)}^{2}\dif s}}\enVert{X}. \label{eq:ApB6}
\end{multline}
Combining \eqref{eq:ApB3} with inequalities (\ref{eq:ApB4}) and (\ref{eq:ApB6}),
then using assumption \eqref{eq:3-4}, we obtain 
\begin{multline}
\dfrac{1}{\lambda}\int_{t}^{T}\enVert{Z(s)}^{2}\dif s - c'_{T}||Y(T)||^{2} - c'\int_{t}^{T}\enVert{Y(s)}^{2}\dif s \leq (c_{T} + cT)\enVert{X}^{2} + \enVert{\eta}\sqrt{T}\del{c_{T} + \dfrac{2}{3}cT} \\
 + \frac{1}{\lambda}\sqrt{T}\del{\del{c_{T} + \dfrac{2}{3}cT}\enVert{X} + ||D_{X}F_{T}(\delta)|| + \dfrac{2}{3}T||D_{X}F(\delta)||}\sqrt{\int_{t}^{T}\enVert{Z(s)}^{2}\dif s}\label{eq:ApB8}
\end{multline}
Now using a weighted Young's inequality in (\ref{eq:ApB5-2}) we derive, for arbitrary $\epsilon > 0$,
\[
\enVert{Y(s)}^{2} \leq \dfrac{1}{\lambda^{2}}(s - t)\del{1 + \dfrac{1}{\epsilon}}\int_{t}^{T}\enVert{Z(s)}^{2}\dif s + (\enVert{X} + \enVert{\eta}\sqrt{s - t})^{2}\del{1 + \dfrac{1}{\epsilon}}
\]
 and therefore, (\ref{eq:ApB8}) implies 
\begin{multline*}
\dfrac{1}{\lambda}\int_{t}^{T}\enVert{Z(s)}^{2}\dif s
\leq   (c_{T} + cT)\enVert{X}^{2} + \enVert{\eta}\sqrt{T}\del{c_{T} + \dfrac{2}{3}cT} + (\enVert{X} + \enVert{\eta}\sqrt{s - t})^{2}\del{1 + \dfrac{1}{\epsilon}}(c'_{T} + c'T)\\
+ \frac{1}{\lambda}\sqrt{T}\del{\del{c_{T} + \dfrac{2}{3}cT}\enVert{X} + ||D_{X}F_{T}(\delta)|| + \dfrac{2}{3}T||D_{X}F(\delta)||}\sqrt{\int_{t}^{T}\enVert{Z(s)}^{2}\dif s}\\
+ \dfrac{1}{\lambda^{2}}T(1 + \epsilon)\del{c'_{T} + c'\dfrac{T}{2}}\int_{t}^{T}\enVert{Z(s)}^{2}\dif s.
\end{multline*}
Then from the assumption (\ref{eq:3-6}) taking $\epsilon$ small enough, it follows that
\begin{equation}
\dfrac{1}{\lambda}\int_{t}^{T}\enVert{Z(s)}^{2}\dif s \leq  C_{T}(1 + \enVert{X}^{2})\label{eq:ApB9}
\end{equation}
Plugging \eqref{eq:ApB9} back into (\ref{eq:ApB5}) we get $\enVert{Y(s)} \leq  C_{T}(1 + \enVert{X}).$
From formula (\ref{eq:3-8}) and assumption (\ref{eq:3-2}) we get
also $\enVert{Z(s)} \leq  C_{T}(1 + \enVert{X}).$ From (\ref{eq:V(X,t) def}) and assumption
(\ref{eq:3-1}) we get immediately (\ref{eq:4-2}). The proof is complete.
$\blacksquare$ 

\subsection{PROOF OF PROPOSITION \ref{prop4-2} }

Let $X^{1},X^{2}\in\mathcal{H}_{m}$, independent of $\mathcal{W}_{t}.$
We consider the corresponding systems 
\begin{align}
Y_{X^{1}t}(s) &= X^{1} - \dfrac{1}{\lambda}\int_{t}^{s}Z_{X^{1}t}(\tau)\dif \tau + \eta(w(s) - w(t)), \label{eq:ApB10}\\
Z_{X^{1}t}(s) &= \bb{E}\left[\left.\int_{s}^{T}D_{X}F(Y_{X^{1}t}(\tau) \otimes  m)\dif \tau + D_{X}F(Y_{X^{1}t}(T) \otimes  m)\right|\mathcal{W}_{X^{1}t}^{s}\right], \label{eq:ApB11}\\
Y_{X^{2}t}(s) &= X^{2} - \dfrac{1}{\lambda}\int_{t}^{s}Z_{X^{2}t}(\tau)\dif \tau + \eta(w(s) - w(t)), \label{eq:ApB12}\\
Z_{X^{2}t}(s) &= \bb{E}\left[\left.\int_{s}^{T}D_{X}F(Y_{X^{2}t}(\tau) \otimes  m)\dif \tau + D_{X}F(Y_{X^{2}t}(T) \otimes  m)\right|\mathcal{W}_{X^{2}t}^{s}\right]. \label{eq:ApB13}
\end{align}
To simplify notation we write $Y^{1}(s) = Y_{X^{1}t}(s),\:Z^{1}(s) = Z_{X^{1}t}(s)$
and $Y^{2}(s) = Y_{X^{2}t}(s),\:Z^{2}(s) = Z_{X^{2}t}(s).$ According
to Remark \ref{rem3-1}, we can replace both $\mathcal{W}_{X^{1}t}^{s}$ in \eqref{eq:ApB11}
and $\mathcal{W}_{X^{2}t}^{s}$ in \eqref{eq:ApB13} by $\mathcal{W}_{X^{1}X^{2}t}^{s} := \sigma(X^{1},X^{2})\vee\mathcal{W}_{t}^{s}$.
Similarly, we denote by $\mathcal{W}_{X^{1}X^{2}t}$ the filtration generated
by the family of $\sigma - $algebras $\mathcal{W}_{X^{1}X^{2}t}^{s}.$
Then for $i = 1,2$, $Y^{i}(s)$ can be interpreted
as the optimal trajectory and $Z^{i}(s)$ the corresponding adjoint state for the following optimal control problem:
\begin{equation}
\begin{aligned}
X_{X^{1}X^{2}t}^{i}(s) &= X^{i} + \int_{t}^{s}v_{X^{1}X^{2}t}(\tau)\dif \tau + \eta(w(s) - w(t)),\:s>t,\\
J^i_{X^{1}X^{2}t}(v_{X^{1}X^{2}t}(\cdot)) &= \dfrac{\lambda}{2}\int_{t}^{T}||v_{X^{1}X^{2}t}(s)||^{2}\dif s + \int_{t}^{T}F(X_{X^{1}X^{2}t}^{i}(s;v_{X^{1}X^{2}t}(\cdot)) \otimes  m)\dif s \\
&\quad \quad + F_{T}(X_{X^{1}X^{2}t}^{i}(T;v_{X^{1}X^{2}t}(\cdot)) \otimes  m), \quad \forall v_{X^{1}X^{2}t}(\cdot) \in L_{\mathcal{W}_{X^{1}X^{2}t}}^{2}(t,T;\mathcal{H}_{m}),
\end{aligned}
\end{equation}
where as usual $J_{X^1 X^2 t}$ denotes the cost to be minimized.
Note that $X_{X^{1}X^{2}t}^{1}(s)$ and $X_{X^{1}X^{2}t}^{2}(s)$ belong to $\mathcal{H}_{m}$, so the probability measures
$X_{X^{1}X^{2}t}^{1}(s;v_{X^{1}X^{2}t}(\cdot)) \otimes  m$ and $X_{X^{1}X^{2}t}^{2}(s;v_{X^{1}X^{2}t}(\cdot)) \otimes  m$
are well defined. 
The optimal controls are respectively $ - \dfrac{1}{\lambda}Z^{1}(s)$ and
$ - \dfrac{1}{\lambda}Z^{2}(s)$.
Thanks to the optimality in the control
space $L_{\mathcal{W}_{X^{1}X^{2}t}}^{2}(t,T;\mathcal{H}_{m}),$
we have that $ - \dfrac{1}{\lambda}Z^{2}(s)$ is an admissible control
for the problem starting with initial condition $X^{1}$ and thus
sub-optimal. The trajectory corresponding to this control is 
\[
X^{1} - \dfrac{1}{\lambda}\int_{t}^{s}Z^{2}(\tau)\dif \tau + \eta(w(s) - w(t)) = Y^{2}(s) + X^{1} - X^{2}
\]
and the sub-optimality allows to write the inequality 
\begin{multline}\label{eq:ApB14}
V(X^{1} \otimes  m,t) - V(X^{2} \otimes  m,t) \leq \int_{t}^{T}\del{F((Y^{2}(s) + X^{1} - X^{2}) \otimes  m) - F(Y^{2}(s) \otimes  m)}\dif s \\
  + F_{T}((Y^{2}(T) + X^{1} - X^{2}) \otimes  m) - F(Y^{2}(T) \otimes  m) \\
 = \int_{0}^{1}\int_{t}^{T}\ip{D_{X}F((Y^{2}(s) + \theta(X^{1} - X^{2})) \otimes  m)}{X^{1} - X^{2}}\dif s\dif \theta \\
 + \int_{0}^{1}\ip{D_{X}F_{T}((Y^{2}(T) + \theta(X^{1} - X^{2})) \otimes  m)}{X^{1} - X^{2}}\dif \theta \\
 \leq \ip{\int_{t}^{T}D_{X}F(Y^{2}(s) \otimes  m)\dif s + D_{X}F(Y^{2}(T) \otimes  m)}{X^{1} - X^{2}} + \dfrac{1}{2}(c_{T} + cT)\enVert{X^{1} - X^{2}}^{2},
\end{multline}
where the last inequality comes from the assumption (\ref{eq:3-3}).
Combining \eqref{eq:ApB13} with \eqref{eq:ApB14}, we obtain
\begin{equation}
V(X^{1} \otimes  m,t) - V(X^{2} \otimes  m,t) \leq \ip{Z^{2}(t)}{X^{1} - X^{2}} + \dfrac{1}{2}(c_{T} + cT)\enVert{X^{1} - X^{2}}^{2}\label{eq:ApB15}
\end{equation}

To obtain a lower bound, we exchange the roles of $X^{1}$ and $X^{2}$ to get 
\begin{equation}
V(X^{1} \otimes  m,t) - V(X^{2} \otimes  m,t)\geq \ip{Z^{2}(t)}{X^{1} - X^{2}} - \dfrac{1}{2}(c_{T} + cT)\enVert{X^{1} - X^{2}}^{2} + \ip{Z^{1}(t) - Z^{2}(t)}{X^{1} - X^{2}}.\label{eq:ApB16}
\end{equation}
We now need to bound $\ip{Z^{1}(t) - Z^{2}(t)}{X^{1} - X^{2}}$ from below.
Arguing as in the proof of Proposition \ref{prop4-1}, we have
\begin{multline} \label{eq:Z1-Z2Y1-Y2}
\ip{Z^{1}(s) - Z^{2}(s)}{Y^{1}(s) - Y^{2}(s)}
= \ip{\int_{s}^{T}(D_{X}F(Y^{1}(\tau) \otimes  m) - D_{X}F(Y^{2}(\tau) \otimes  m))\dif \tau}{Y^{1}(s) - Y^{2}(s)} \\
 + \ip{D_{X}F_{T}(Y^{1}(T) \otimes  m) - D_{X}F_{T}(Y^{2}(T) \otimes  m)}{Y^{1}(s) - Y^{2}(s)}
\end{multline}
Differentiate this formula in $s$, substitute the identity $\od{}{s}\del{Y^1(s)-Y^2(s)} = -\frac{1}{\lambda}\del{Z^1(s)-Z^2(s)}$, and then reintegrate the resulting equation from $t$ to $T$.
Using \eqref{eq:ApB11} and \eqref{eq:ApB13}, we obtain
\begin{multline} \label{eq:ApB17}
\ip{Z^{1}(t) - Z^{2}(t)}{X^{1} - X^{2}} = \dfrac{1}{\lambda}\int_{t}^{T}\enVert{Z^{1}(s) - Z^{2}(s)}^{2}\dif s \\
 + \int_{t}^{T}\ip{D_{X}F(Y^{1}(s) \otimes  m) - D_{X}F(Y^{2}(s) \otimes  m)}{Y^{1}(s) - Y^{2}(s)}\dif s \\
  + \ip{D_{X}F_{T}(Y^{1}(T) \otimes  m) - D_{X}F(Y^{2}(T) \otimes  m)}{Y^{1}(T) - Y^{2}(T)}
  \\
  \geq\dfrac{1}{\lambda}\int_{t}^{T}\enVert{Z^{1}(s) - Z^{2}(s)}^{2}\dif s - c'\int_{t}^{T}\enVert{Y^{1}(s) - Y^{2}(s)}^{2}\dif s - c'_{T}\enVert{Y^{1}(T) - Y^{2}(T)}^{2}.
\end{multline}
Using estimates similar to those in the proof of Proposition \ref{prop4-1},
we get 
\begin{equation*}
\begin{aligned}
\int_{t}^{T}\enVert{Y^{1}(s) - Y^{2}(s)}^{2}\dif s &\leq \dfrac{1}{\lambda^{2}}(1 + \epsilon)\dfrac{T^{2}}{2}\int_{t}^{T}\enVert{Z^{1}(s) - Z^{2}(s)}^{2}\dif s + T\del{1 + \dfrac{1}{\epsilon}}\enVert{X^{1} - X^{2}}^{2},\\
\enVert{Y^{1}(T) - Y^{2}(T)}^{2} &\leq \dfrac{1}{\lambda^{2}}(1 + \epsilon)T\int_{t}^{T}\enVert{Z^{1}(s) - Z^{2}(s)}^{2}\dif s + \del{1 + \dfrac{1}{\epsilon}}\enVert{X^{1} - X^{2}}^{2},
\end{aligned}
\end{equation*}
which we plug into \eqref{eq:ApB17} to obtain
\begin{multline}\label{eq:ApB18}
\ip{Z^{1}(t) - Z^{2}(t)}{X^{1} - X^{2}}  \geq  \dfrac{1}{\lambda}\del{1 - \dfrac{1}{\lambda}T(1 + \epsilon)(c'\dfrac{T}{2} + c'_{T})}\int_{t}^{T}\enVert{Z^{1}(s) - Z^{2}(s)}^{2}\dif s\\
- \del{1 + \dfrac{1}{\epsilon}}(c'T + c'_{T})\enVert{X^{1} - X^{2}}^{2}
\end{multline}
From the assumption (\ref{eq:3-6}) and choosing $\epsilon$ sufficiently
small, the first term in the right hand side is positive. 

Combining \eqref{eq:ApB18}
with (\ref{eq:ApB16}) and taking (\ref{eq:ApB15}) into account, it
follows that 
\begin{equation}
|V(X^{1} \otimes  m,t) - V(X^{2} \otimes  m,t) - \ip{Z^{2}(t)}{X^{1} - X^{2}}| \leq  C_{T}\enVert{X^{1} - X^{2}}^{2},\label{eq:ApB19}
\end{equation}
which implies that $X \mapsto V(X  \otimes  m,t)$ is G\^ateaux differentiable and that Equation (\ref{eq:4-3}) holds. Next, we use \eqref{eq:Z1-Z2Y1-Y2} with $s = t$ and the Lipschitz estimates \eqref{eq:3-3} to deduce
\[
\ip{Z^{1}(t) - Z^{2}(t)}{X^{1} - X^{2}}  \leq \enVert{X^{1} - X^{2}}\del{c\int_{t}^{T}\enVert{Y^{1}(s) - Y^{2}(s)}\dif s + c_{T}\enVert{Y^{1}(T) - Y^{2}(T)}}.
\]
Using estimates as before, cf.~\eqref{eq:ApB5-2} and \eqref{eq:ApB6}, we derive 
\[
\ip{Z^{1}(t) - Z^{2}(t)}{X^{1} - X^{2}} \leq (c_{T} + cT)\enVert{X^{1} - X^{2}}^{2} + \dfrac{\sqrt{T}}{\lambda}(c_{T} + \dfrac{2}{3}cT)\enVert{X^{1} - X^{2}}\sqrt{\int_{t}^{T}\enVert{Z^{1}(s) - Z^{2}(s)}^{2}\dif s},
\]
which we plug into (\ref{eq:ApB18}) to obtain 
\begin{multline}\label{eq:ApB20}
\dfrac{1}{\lambda}\del{1 - \dfrac{1}{\lambda}T(1 + \epsilon)(c'\dfrac{T}{2} + c'_{T})}\int_{t}^{T}\enVert{Z^{1}(s) - Z^{2}(s)}^{2}\dif s \leq (c_{T} + cT + \del{1 + \dfrac{1}{\epsilon}}(c'T + c'_{T}))\enVert{X^{1} - X^{2}}^{2} \\
 + \dfrac{\sqrt{T}}{\lambda}(c_{T} + \dfrac{2}{3}cT)\enVert{X^{1} - X^{2}}\sqrt{\int_{t}^{T}\enVert{Z^{1}(s) - Z^{2}(s)}^{2}\dif s}.
\end{multline}
From \eqref{eq:ApB20} we deduce
\begin{equation}
\int_{t}^{T}\enVert{Z^{1}(s) - Z^{2}(s)}^{2}\dif s \leq  C_{T}\enVert{X^{1} - X^{2}}^{2}.\label{eq:ApB21}
\end{equation}
Returning to Equations \eqref{eq:ApB10}-\eqref{eq:ApB13}, applying \eqref{eq:ApB21} and using the Lipschitz estimates \eqref{eq:3-3}, we obtain
\[
\enVert{Y^{1}(s) - Y^{2}(s)} \leq  C_{T}\enVert{X^{1} - X^{2}},\;\enVert{Z^{1}(s) - Z^{2}(s)} \leq  C_{T}\enVert{X^{1} - X^{2}}
\]
and thus (\ref{eq:4-4}) has been proven, which completes the proof.
$\blacksquare$

\subsection{PROOF OF PROPOSITION \ref{prop4-3}}

Consider $J_{m,t}(v_{\cdot t}(\cdot ))$, defined in (\ref{eq:4-8}), for an arbitrary $v_{\cdot t}(\cdot ) \in \sr{V}$.
It is straightforward to see that $m \mapsto J_{m,t}(v_{\cdot t}(\cdot ))$ is continuously differentiable and has a functional derivative given by the expression 
\begin{multline}
\dod{}{m}J_{m,t}(v_{\cdot t}(\cdot ))(x)=\dfrac{\lambda}{2}\int_{t}^{T}\bb{E}|v_{xt}(s)|^{2}\dif s+\int_{t}^{T}\bb{E}\dod{F}{m}(X_{\cdot t}(s;v_{\cdot t}(\cdot )) \otimes  m)(X_{xt}(s;v_{xt}(\cdot )))\dif s\\
+\bb{E}\dod{F_{T}}{m}(X_{\cdot t}(T;v_{\cdot t}(\cdot )) \otimes  m)(X_{xt}(s;v_{xt}(\cdot ))).\label{eq:4-12}
\end{multline}
Thus (\ref{eq:4-10}) follows by a simple application of the Envelope
theorem.
%
%
{\color{black}More precisely, let $\widehat{v}_{xmt}(\cdot)$ be the optimal control, which can be solved explicitly, $\widehat{v}_{xmt}(\cdot)=-\frac{1}{\lambda}Z_{xmt}(\cdot)$, by the first order condition in our settings, then $V(m,t)=J_{m,t}(\widehat{v}_{xmt}(\cdot ))$; by Lemma 3.5, $D_v J_{m,t}(v_{\cdot}(\cdot ))$ exists and by Proposition 3.7, $D_v J_{m,t}(\widehat{v}_{xmt}(\cdot ))=0$; since $F(m)$ and $F_T(m)$ have a functional derivative, $m \mapsto J_{m,t}(v_{\cdot t}(\cdot ))$ also has a functional derivative given by (4.15); therefore, $\dod{}{m}V(m,t)(x)=\dod{}{m}J_{m,t}(\widehat{v}_{xmt}(\cdot ))(x)=\dod{}{m}J_{m,t}(-\frac{1}{\lambda}Z_{xmt}(\cdot))(x)$, which does not involve the functional differentiability of $Z_{xmt}(\cdot)$ (to be proved later in Section 6.2.1) with respect to $m$, is indeed a functional derivative of $V(m,t)$; the idea behind this proof essentially follows that for the Envelope theorem found in, for example, Appendix Section 5 in \cite{OYEM} or on page 160 in \cite{TAAT}; see also \cite{FWHRRW}}
Next we consider the following ordinary stochastic control problem depending parametrically
on the function $s\mapsto Y_{\cdot mt}(s) \otimes  m$ with values in
$\mathcal{P}_{2}(\bb{R}^n).$
We take
controls in $L_{\mathcal{W}_{t}}^{2}(t,T;\bb{R}^n)$ and the state is
defined by
\begin{equation}
x_{xt}(s)=x+\int_{t}^{s}v(\tau)\dif \tau+\eta(w(s)-w(t))\label{eq:4-13}
\end{equation}
The cost to minimize is given by 
\begin{equation}
K_{xt}(v(\cdot ))=\dfrac{\lambda}{2}\bb{E}\int_{t}^{T}|v(s)|^{2}\dif s+\int_{t}^{T}\bb{E}\dod{F}{m}(Y_{\cdot mt}(s) \otimes  m)(x_{xt}(s))\dif s+\bb{E}\dod{F_T}{m}(Y_{\cdot mt}(T) \otimes  m)(x_{xt}(T)).\label{eq:4-14}
\end{equation}
It is straightforward to
check that the optimal control coincides with $-\dfrac{1}{\lambda}Z_{xmt}(s)$
and the optimal state is $Y_{xmt}(s)$;
indeed, one writes the necessary conditions of optimality and observes that they coincide with system \eqref{eq:Yxmt}-\eqref{eq:Zxmt}, whose solution is unique.
Therefore, \eqref{eq:4-10} implies
\begin{equation}
\dod{}{m}V(m,t)(x)=\inf_{v(\cdot )}K_{xt}(v(\cdot )).\label{eq:4-15}
\end{equation}
But $x\mapsto K_{xt}(v(\cdot ))$ is differentiable, with 
\begin{equation}
D_1K_{xt}(v(\cdot ))=\int_{t}^{T}\bb{E}\,D\dod{F}{m}(Y_{\cdot mt}(s) \otimes  m)(x_{xt}(s))\dif s+\bb{E}\,D\dod{F_T}{m}(Y_{\cdot mt}(T) \otimes  m)(x_{xt}(T))\label{eq:4-16}
\end{equation}
Applying the Envelope theorem again, we deduce 
\[
D\dod{}{m}V(m,t)(x)=\int_{t}^{T}\bb{E}\,D\dod{F}{m}(Y_{\cdot mt}(s) \otimes  m)(Y_{xmt}(s))\dif s+\bb{E}\,D\dod{F_T}{m}(Y_{\cdot mt}(T) \otimes  m)(Y_{xmt}(T)),
\]
which proves the first part of (\ref{eq:4-11}). The second part follows immediately from Proposition \ref{prop:D_XF}. This completes the proof. $\blacksquare$ 

\subsection{PROOF OF PROPOSITION \ref{prop4-4}}

We consider first the processes $Y_{\xi,Y_{\cdot mt}(s) \otimes  m,s}(\tau),Z_{\xi,Y_{\cdot mt}(s) \otimes  m,s}(\tau)$, defined as the solution
of the system
\begin{align}
Y_{\xi,Y_{\cdot mt}(s) \otimes  m,s}(\tau) &=\xi-\dfrac{1}{\lambda}\int_{s}^{\tau}Z_{\xi,Y_{\cdot mt}(s) \otimes  m,s}(\theta)\dif \theta+\eta(w(\tau)-w(s)),\label{eq:4-18}\\
Z_{\xi,Y_{\cdot mt}(s) \otimes  m,s}(\tau) &=\bb{E}\left[\int_{\tau}^{T}D \dod{F}{m}(Y_{\cdot ,Y_{\cdot mt}(s) \otimes  m,s}(\theta) \otimes (Y_{\cdot mt}(s) \otimes  m))(Y_{\xi,Y_{\cdot mt}(s) \otimes  m,s}(\theta))\dif \theta\right.\\
&\quad \left.+\left. D \dod{F_T}{m}(Y_{\cdot ,Y_{\cdot mt}(s) \otimes  m,s}(T) \otimes (Y_{\cdot mt}(s) \otimes  m))(Y_{\xi,Y_{\cdot mt}(s) \otimes  m,s}(T))\right|\mathcal{W}_{s}^{\tau}\right],
\nonumber
\end{align}
cf.~(\ref{eq:Yxmt})-(\ref{eq:Zxmt})
Then, applying the first part of (\ref{eq:4-11}) we have 
\begin{equation}
Z_{\xi,Y_{\cdot mt}(s) \otimes  m,s}(s)=D\dod{V}{m}(Y_{\cdot mt}(s) \otimes  m,s)(\xi)\label{eq:4-20}
\end{equation}
Next, define 

\begin{equation}
\tilde{Y}_{xmt}(s,\tau)=Y_{Y_{xmt}(s),Y_{\cdot mt}(s) \otimes  m,s}(\tau),\;\tilde{Z}_{xmt}(s,\tau)=Z_{Y_{xmt}(s),Y_{\cdot mt}(s) \otimes  m,s}(\tau)\label{eq:4-21}
\end{equation}
then, from (\ref{eq:4-20}) we can write

\begin{equation}
\tilde{Z}_{xmt}(s,s)=D\dod{V}{m}(Y_{\cdot mt}(s) \otimes  m,s)(\tilde{Y}_{xmt}(s,s))\label{eq:4-22}
\end{equation}
We are going to show that 
\begin{equation}
\tilde{Y}_{xmt}(s,\tau)=Y_{xmt}(\tau),\;\tilde{Z}_{xmt}(s,\tau)=Z_{xmt}(\tau) \quad \forall t \leq s \leq \tau \leq T. \label{eq:4-23}
\end{equation}
Equation \eqref{eq:4-23} implies $\tilde{Y}_{xmt}(s,s)=Y_{xmt}(s),\:\tilde{Z}_{xmt}(s,s)=Z_{xmt}(s)$,
and so the result (\ref{eq:4-17}) will follow immediately from (\ref{eq:4-22}).

It remains to check (\ref{eq:4-23}).
We notice that $Y_{\cdot ,Y_{\cdot mt}(s) \otimes  m,s}(\theta) \otimes (Y_{\cdot mt}(s) \otimes  m)=\tilde{Y}_{\cdot mt}(s,\theta) \otimes  m$ (see Equation \eqref{eq:X_Xt otimes m v2})
and 
\begin{multline*}
\bb{E}\left[\left.\int_{\tau}^{T}D \dod{F}{m}(\tilde{Y}_{\cdot mt}(s,\theta) \otimes  m)(\tilde{Y}_{xmt}(s,\theta))\dif \theta+D \dod{F_T}{m}(\tilde{Y}_{\cdot mt}(s,T) \otimes  m)(\tilde{Y}_{xmt}(s,T))\right|\mathcal{W}_{t}^{\tau}\right]\\
=
\bb{E}\left[\int_{\tau}^{T}D \dod{F}{m}(\tilde{Y}_{\cdot mt}(s,\theta) \otimes  m)(Y_{\xi,Y_{\cdot mt}(s) \otimes  m,s}(\theta))\dif \theta+\right.\\
+\left.\left.D \dod{F_T}{m}(\tilde{Y}_{\cdot mt}(s,T) \otimes  m)(Y_{\xi,Y_{\cdot mt}(s) \otimes  m,s}(T))\right|\mathcal{W}_{s}^{\tau}\right]_{|\xi=Y_{xmt}(s)}.
\end{multline*}
Therefore, the pair $\del{\tilde{Y}_{xmt}(s,\tau),\tilde{Z}_{xmt}(s,\tau)}$
is the solution to the system 
\begin{align*}
\tilde{Y}_{xmt}(s,\tau) &=x-\dfrac{1}{\lambda}\int_{t}^{s}Z_{xmt}(\theta)\dif \theta-\dfrac{1}{\lambda}\int_{s}^{\tau}\tilde{Z}_{xmt}(s,\theta)\dif \theta+\eta(w(\tau)-w(t))\\
\tilde{Z}_{xmt}(s,\tau) &=\bb{E}\left[\left.\int_{\tau}^{T}D \dod{F}{m}(\tilde{Y}_{\cdot mt}(s,\theta) \otimes  m)(\tilde{Y}_{xmt}(s,\theta))\dif \theta+D \dod{F_T}{m}(\tilde{Y}_{\cdot mt}(s,T) \otimes  m)(\tilde{Y}_{xmt}(s,T))\right|\mathcal{W}_{t}^{\tau}\right].
\end{align*}
By taking $s = \tau$ in (\ref{eq:Yxmt})- (\ref{eq:Zxmt}), we see that $\del{Y_{xmt}(\tau),Z_{zmt}(\tau)}$ is another solution, and so by uniqueness we deduce (\ref{eq:4-23}). The proof is complete. $\blacksquare$ 

\subsection{PROOF OF PROPOSITION \ref{prop4-5}}

We begin with (\ref{eq:4-26}). From the optimality principle
(\ref{eq:3-16}) we have
\begin{multline} \label{eq:ApB22}
V(X \otimes  m,t) - V(X \otimes  m,t + h) = \dfrac{1}{2\lambda}\int_{t}^{t + h}||Z_{Xt}(s)||^{2}\dif s + \int_{t}^{t + h}F(Y_{Xt}(s) \otimes  m)\dif s \\
+ V(Y_{Xt}(t + h) \otimes  m,t + h) - V(X \otimes  m,t + h),
\end{multline}
and from (\ref{eq:4-4}) we obtain 
\begin{equation}
|V(Y_{Xt}(t + h) \otimes  m,t + h) - V(X \otimes  m,t + h) - \ip{D_{X}V(X \otimes  m,t + h)}{Y_{Xt}(t + h) - X}| \leq  C_{T}||Y_{Xt}(t + h) - X||^{2}.\label{eq:ApB23}
\end{equation}
Since $D_{X}V(X \otimes  m,t + h)$ is $\sigma(X)$ measurable, while $X$ is
independent of $\mathcal{W}_{t},$ we can multiply \eqref{eq:3-7} by $D_{X}V(X \otimes  m,t + h)$ and integrate to get
\[
	\ip{D_{X}V(X \otimes  m,t + h)}{Y_{Xt}(t + h) - X} =  - \dfrac{1}{\lambda}\ip{D_{X}V(X \otimes  m,t + h)}{\int_{t}^{t + h}Z_{Xt}(s)\dif s)},
\]
which implies, using Propositions \ref{prop4-1} and \ref{prop4-2},
\begin{equation} \label{eq:ipDXVYXt}
\abs{\ip{D_{X}V(X \otimes  m,t + h)}{Y_{Xt}(t + h) - X}} \leq C_Th\del{1+\enVert{X}^2}.
\end{equation}
Finally, using Equation (\ref{eq:3-7}) as when we derived \eqref{eq:ApB5-2}, we get
\begin{equation}
\enVert{Y_{Xt}(t+h)-X} \leq \dfrac{1}{\lambda}\sqrt{h}\sqrt{\int_{t}^{T}\enVert{Z(\tau)}^{2}\dif \tau} + \enVert{\eta}\sqrt{h}
\leq C_T \sqrt{h}\del{1+\enVert{X}}, \label{eq:YXt-X}
\end{equation}
where the second inequality follows from Proposition \ref{prop4-1}.
Combine inequalities \eqref{eq:YXt-X}, \eqref{eq:ipDXVYXt}, and (\ref{eq:ApB23}) with (\ref{eq:ApB22}) to
conclude (\ref{eq:4-26}). 

We turn to (\ref{eq:4-27}),
which by \eqref{eq:4-3} is equivalent to 
\begin{equation}
\enVert{Z_{X,t + h}(t + h) - Z_{Xt}(t)} \leq  C_{T}(h^{\frac{1}{2}} + h)(1 + \enVert{X}).\label{eq:ApB24}
\end{equation}
From (\ref{eq:3-8}) we have
\begin{equation} \label{eq:ZXt conditioned}
Z_{Xt}(t) = \bb{E}\intcc{\left.\int_{t}^{t + h}D_{X}F(Y_{Xt}(s) \otimes  m)\dif s\right|\sigma(X)} + \bb{E}\intcc{Z_{Xt}(t + h)|\sigma(X)}.
\end{equation}
We use the assumption \eqref{eq:3-2}, the estimate \eqref{eq:4-1} from Proposition \ref{prop4-1}, and the tower property to get
\begin{equation} \label{eq:int DXFYXt estimate}
\enVert{\bb{E}\intcc{\left.\int_{t}^{t + h}D_{X}F(Y_{Xt}(s) \otimes  m)\dif s\right|\sigma(X)}} \leq C_{T}h(1 + \enVert{X}).
\end{equation}
Subtract \eqref{eq:ZXt conditioned} from $Z_{X,t + h}(t + h)$ and use \eqref{eq:int DXFYXt estimate} to get
\[
\enVert{Z_{X,t + h}(t + h) - Z_{Xt}(t)} \leq  C_{T}h(1 + \enVert{X}) + \enVert{Z_{X,t + h}(t + h) - \bb{E}[Z_{Xt}(t + h)|\sigma(X)]},
\]
and since $Z_{X,t + h}(t + h)$ is $\sigma(X)$ measurable, the tower property implies
\begin{equation}
\enVert{Z_{X,t + h}(t + h) - Z_{Xt}(t)} \leq  C_{T}h(1 + \enVert{X}) + \enVert{Z_{X,t + h}(t + h) - Z_{Xt}(t + h)}\label{eq:ApB25}
\end{equation}
Note that $Z_{Xt}(t + h) = Z_{Y_{Xt}(t + h),t + h}(t + h)$, so by the
proof of Proposition \ref{prop4-2} combined with \eqref{eq:YXt-X},
\begin{equation} \label{eq:ZXt(t+h)-ZXt(t)}
||Z_{Xt}(t + h)(s) - Z_{X,t + h}(s)|| \leq  C_{T}||Y_{Xt}(t + h) - X|| \leq C_T \sqrt{h}(1+\enVert{X}), \ \forall s\in[t + h,T].
\end{equation}
Combining \eqref{eq:ZXt(t+h)-ZXt(t)} and (\ref{eq:ApB25})
 we  obtain \eqref{eq:ApB24}, which implies (\ref{eq:4-27}). $\blacksquare$

\section{PROOFS FROM SECTION \ref{sec:bellman}} \label{ap:C}

\subsection{PROOF OF PROPOSITION \ref{prop5-1}. }

We connect the system (\ref{eq:5-11}) to a control problem. The space
of controls is $L_{\mathcal{W}_{X\mathcal{X}t}}^{2}(t,T;\mathcal{H}_{m})$
where $\mathcal{W}_{X\mathcal{X}t}$ is the filtration generated by
the $\sigma$-algebras $\mathcal{W}_{X\mathcal{X}t}^{s}$. If $\mathcal{V}_{X\mathcal{X}t}(s)$
is a control, the state is defined by 
\begin{equation}
\mathcal{X}_{X\mathcal{X}t}(s) = \mathcal{X} + \int_{t}^{s}\mathcal{V}_{X\mathcal{X}t}(\tau)\dif \tau\label{eq:8-1}
\end{equation}
and the payoff is 
\begin{multline} \label{eq:8-2}
\mathcal{J}_{X\mathcal{X}t}(\mathcal{V}_{X\mathcal{X}t}(\cdot)) = \dfrac{\lambda}{2}\int_{t}^{T}||\mathcal{V}_{X\mathcal{X}t}(s)||^{2}\dif s + \dfrac{1}{2}\int_{t}^{T}\ip{D^{2}F(Y_{Xt}(s) \otimes  m)(\mathcal{X}_{X\mathcal{X}t}(s))}{\mathcal{X}_{X\mathcal{X}t}(s)}\dif s \\
 + \dfrac{1}{2}\ip{D^{2}F(Y_{Xt}(T) \otimes  m)(\mathcal{X}_{X\mathcal{X}t}(T))}{\mathcal{X}_{X\mathcal{X}t}(T)}.
\end{multline}
Thanks to the assumption (\ref{eq:3-6}) this is a strictly convex linear quadratic problem, which has a unique optimal control. The system
(\ref{eq:5-11}) has a unique solution and the optimal control is
$\mathcal{\hat{V}}_{X\mathcal{X}t}(s) =  - \dfrac{1}{\lambda}\mathcal{Z}_{X\mathcal{X}t}(s).$
The
optimal state is $\mathcal{\mathcal{Y}}_{X\mathcal{X}t}(s)$. Moreover,
we have
\begin{multline} \label{eq:8-3}
\inf_{\mathcal{V}_{X\mathcal{X}t}(\cdot)}\dfrac{1}{2}\ip{\mathcal{Z}_{X\mathcal{X}t}(t)}{\s{X}} = \dfrac{1}{2\lambda}\int_{t}^{T}||\mathcal{Z}_{X\mathcal{X}t}(s)||^{2}\dif s + \dfrac{1}{2}\int_{t}^{T}\ip{D^{2}F(Y_{Xt}(s) \otimes  m)(\mathcal{Y}_{X\mathcal{X}t}(s))}{\mathcal{Y}_{X\mathcal{X}t}(s)}\dif s \\
 + \dfrac{1}{2}\ip{D^{2}F(Y_{Xt}(T) \otimes  m)(\mathcal{V}_{X\mathcal{X}t}(T))}{\mathcal{V}_{X\mathcal{X}t}(T)}
\end{multline}
 where in (\ref{eq:8-3}) $\del{\mathcal{Y}_{X\mathcal{X}t}(s), \mathcal{Z}_{X\mathcal{X}t}(s)}$
is the solution of (\ref{eq:5-11}). Thanks to (\ref{eq:3-6}) we
check easily that 
\begin{equation}
||\mathcal{Y}_{X\mathcal{X}t}(s)||,\:||\mathcal{Z}_{X\mathcal{X}t}(s)|| \leq  C_{T}||\mathcal{X}||.\label{eq:8-4}
\end{equation}
 Now $\mathcal{X}\mapsto\mathcal{Y}_{X\mathcal{X}t}(s)$ and $\mathcal{X}\mapsto\mathcal{Z}_{X\mathcal{X}t}(s)$
are linear. Indeed, the conditional expectation in the definition
of $\mathcal{Z}_{X\mathcal{X}t}(s)$ does not introduce nonlinearities,
since taking an initial condition $\alpha\mathcal{X}_{1} + \beta\mathcal{X}_{2}$,
one can extend the conditioning $\sigma$-algebra to contain both
$\mathcal{X}_{1},\mathcal{X}_{2}$ and the linearity follows easily.
Therefore the maps $\mathcal{X}\mapsto\mathcal{Y}_{X\mathcal{X}t}(s)$
and $\mathcal{X}\mapsto\mathcal{Z}_{X\mathcal{X}t}(s)$ belong
to $\mathcal{L}(\mathcal{H}_{\color{black}m,t},\mathcal{H}_{\color{black}m,t}).$

The next important
step is to check the convergence 
\begin{equation}
\dfrac{Y_{X + \epsilon\mathcal{X},t}(s) - Y_{Xt}(s)}{\epsilon}\rightarrow\mathcal{Y}_{X\mathcal{X}t}(s),
\ \text{and} \
\dfrac{Z_{X + \epsilon\mathcal{X},t}(s) - Z_{Xt}(s)}{\epsilon}\rightarrow\mathcal{Z}_{X\mathcal{X}t}(s) 
\ \text{as} \ 
\epsilon \to 0, 
\ \forall s\in[t,T].\label{eq:8-5}
\end{equation}
Define 
\[
Y_{X\mathcal{X}t}^{\epsilon}(s) = \dfrac{Y_{X + \epsilon\mathcal{X},t}(s) - Y_{Xt}(s)}{\epsilon},\:Z_{X\mathcal{X}t}^{\epsilon}(s) = \dfrac{Z_{X + \epsilon\mathcal{X},t}(s) - Z_{Xt}(s)}{\epsilon}.
\]
Then the following relations follow from \eqref{eq:3-7} and \eqref{eq:3-8}:
\begin{equation}\label{eq:8-6}
\begin{aligned}
Y_{X\mathcal{X}t}^{\epsilon}(s) &= \mathcal{X} - \dfrac{1}{\lambda}\int_{t}^{s}Z_{X\mathcal{X}t}^{\epsilon}(\tau)\dif \tau,\\
Z_{X\mathcal{X}t}^{\epsilon}(s) &= \bb{E}\left[\int_{s}^{T}\dfrac{D_{X}F(Y_{X + \epsilon\mathcal{X},t}(\tau) \otimes  m) - D_{X}F(Y_{Xt}(\tau) \otimes  m)}{\epsilon}\dif \tau \right. \\
&\quad\quad \quad \left.\left. + \dfrac{D_{X}F_{T}(Y_{X + \epsilon\mathcal{X},t}(T) \otimes  m) - D_{X}F_{T}(Y_{Xt}(T) \otimes  m)}{\epsilon}\right|\mathcal{W}_{X,\mathcal{X},t}^{s}\right].
\end{aligned}
\end{equation}
The second relation in (\ref{eq:8-6}) can be written 
\begin{multline}\label{eq:8-7}
Z_{X\mathcal{X}t}^{\epsilon}(s) = \bb{E}\left[\int_{s}^{T}\int_{0}^{1}D_{X}^{2}F((Y_{Xt}(\tau) + \theta\epsilon Y_{X\mathcal{X}t}^{\epsilon}(\tau)) \otimes  m)(Y_{X\mathcal{X}t}^{\epsilon}(\tau))\dif \theta \dif \tau \right. \\
\left.\left. + \int_{0}^{1}D_{X}^{2}F_{T}((Y_{Xt}(T) + \theta\epsilon Y_{X\mathcal{X}t}^{\epsilon}(T)) \otimes  m)(Y_{X\mathcal{X}t}^{\epsilon}(T))\dif \theta\right|\mathcal{W}_{X,\mathcal{X},t}^{s}\right].
\end{multline}
Using an argument analogous to the proof of Proposition \ref{prop4-1}, we conclude that $Y_{X\mathcal{X}t}^{\epsilon}(s),Z_{X\mathcal{X}t}^{\epsilon}(s)$
are bounded in $L_{\mathcal{W}_{X\mathcal{X}t}}^{\infty}(t,T;\mathcal{H}_{m}).$
Passing to a subsequence, we can take $Z_{X\mathcal{X}t}^{\epsilon}(s)$ 
to converge weakly in $L_{\mathcal{W}_{X\mathcal{X}t}}^{2}(t,T;\mathcal{H}_{m})$
to $\mathcal{Z}_{X\mathcal{X}t}(s).$
It follows that 
\begin{equation} \label{eq:Yeps weakly converges}
Y_{X\mathcal{X}t}^{\epsilon}(s) \rightharpoonup \mathcal{Y}_{X\mathcal{X}t}(s)\mathcal{ = X} - \dfrac{1}{\lambda}\int_{t}^{s}\mathcal{Z}_{X\mathcal{X}t}(\tau)\dif \tau,\text{weakly in }L^{2}(\Omega,\mathcal{W}_{X,\mathcal{X},t}^{s},\bb{P};L_{m}^{2}(\bb{R}^n;\bb{R}^n)),\forall s.
\end{equation}

Define 
\begin{multline*}
J_{xt}^{\epsilon}(s) = \bb{E}\left[\int_{s}^{T}\del{\int_{0}^{1}D_{X}^{2}F((Y_{Xt}(\tau) + \theta\epsilon Y_{X\mathcal{X}t}^{\epsilon}(\tau)) \otimes  m)(Y_{X\mathcal{X}t}^{\epsilon}(\tau))\dif \theta - D_{X}^{2}F(Y_{Xt}(\tau) \otimes  m)(\mathcal{Y}_{X\mathcal{X}t}(\tau))}\dif \tau\right.\\
\left.\left. + \int_{0}^{1}D_{X}^{2}F_{T}((Y_{Xt}(T) + \theta\epsilon Y_{X\mathcal{X}t}^{\epsilon}(T)) \otimes  m)(Y_{X\mathcal{X}t}^{\epsilon}(T))\dif \theta - D_{X}^{2}F_{T}(Y_{Xt}(T) \otimes  m)(\mathcal{Y}_{X\mathcal{X}t}(T))\right|\mathcal{W}_{X,\mathcal{X},t}^{s}\right],
\end{multline*}
which is an element of $L^{2}(\Omega,\mathcal{W}_{X,\mathcal{X},t}^{s},\bb{P};L_{m}^{2}(\bb{R}^n;\bb{R}^n)).$
We are going to show that it converges weakly to $0.$ We write $J_{xt}^{\epsilon}(s) = I_{xt}^{\epsilon}(s) + II_{xt}^{\epsilon}(s)$
with 
\begin{multline*}
I_{xt}^{\epsilon}(s) := \bb{E}\left[\int_{s}^{T}\del{\int_{0}^{1}D_{X}^{2}F((Y_{Xt}(\tau) + \theta\epsilon Y_{X\mathcal{X}t}^{\epsilon}(\tau)) \otimes  m)(Y_{X\mathcal{X}t}^{\epsilon}(\tau))\dif \theta - D_{X}^{2}F(Y_{Xt}(\tau) \otimes  m)(Y_{X\mathcal{X}t}^{\epsilon}(\tau))}\dif \tau\right.\\
\left.\left. + \int_{0}^{1}D_{X}^{2}F_{T}((Y_{Xt}(T) + \theta\epsilon Y_{X\mathcal{X}t}^{\epsilon}(T)) \otimes  m)(Y_{X\mathcal{X}t}^{\epsilon}(T))\dif \theta - D_{X}^{2}F_{T}(Y_{Xt}(T) \otimes  m)(Y_{X\mathcal{X}t}^{\epsilon}(T))\right|\mathcal{W}_{X,\mathcal{X},t}^{s}\right]
\end{multline*}
and 
\begin{multline*}
II_{xt}^{\epsilon}(s) := \bb{E}\left[\int_{s}^{T}\del{D_{X}^{2}F(Y(\tau) \otimes  m)(Y_{X\mathcal{X}t}^{\epsilon}(\tau)) - D_{X}^{2}F(Y_{Xt}(\tau) \otimes  m)(\mathcal{Y}_{X\mathcal{X}t}(\tau))}\dif \tau\right.
\\
\left.\left. + D_{X}^{2}F_{T}(Y_{Xt}(T) \otimes  m)(Y_{X\mathcal{X}t}^{\epsilon}(T)) - D_{X}^{2}F_{T}(Y_{Xt}(T) \otimes  m)(\mathcal{Y}_{X\mathcal{X}t}(T))\right|\mathcal{W}_{X,\mathcal{X},t}^{s}\right].
\end{multline*}
Then by \eqref{eq:Yeps weakly converges} it follows that $II_{xt}^{\epsilon}(s)$ converges weakly to $0$ in $L^{2}(\Omega,\mathcal{W}_{X,\mathcal{X},t}^{s},\bb{P};L_{m}^{2}(\bb{R}^n;\bb{R}^n))$, since $D_{X}^{2}F(Y(\tau) \otimes  m)$ and $D_{X}^{2}F_T(Y(T) \otimes  m)$ are in $\s{L}(\s{H}_m,\s{H}_m)$. 
We turn our attention to $I_{xt}^{\epsilon}(s).$ We claim that 
\begin{equation}
\bb{E}\int_{\bb{R}^n}|I_{xt}^{\epsilon}(s)|\dif m(x)\rightarrow0,\:\forall s>t\label{eq:8-8}
\end{equation}
Indeed, 
\begin{multline*}
\bb{E}\int_{\bb{R}^n}|I_{xt}^{\epsilon}(s)|\dif m(x)\\
 \leq 
 \bb{E}\int_{\bb{R}^n}\left|\int_{s}^{T}\del{\int_{0}^{1}D_{X}^{2}F((Y_{Xt}(\tau) + \theta\epsilon Y_{X\mathcal{X}t}^{\epsilon}(\tau)) \otimes  m)(Y_{X\mathcal{X}t}^{\epsilon}(\tau))\dif \theta - D_{X}^{2}F(Y_{Xt}(\tau) \otimes  m)(Y_{X\mathcal{X}t}^{\epsilon}(\tau))}\dif \tau\right.
 \\
\left. + \int_{0}^{1}D_{X}^{2}F_{T}((Y_{Xt}(T) + \theta\epsilon Y_{X\mathcal{X}t}^{\epsilon}(T)) \otimes  m)(Y_{X\mathcal{X}t}^{\epsilon}(T))\dif \theta - D_{X}^{2}F_{T}(Y_{Xt}(T) \otimes  m)(Y_{X\mathcal{X}t}^{\epsilon}(T))\right|\dif m(x)
\\
 \leq \int_{s}^{T}\int_{0}^{1}\bb{E}\int_{\bb{R}^n}\left|D_{X}^{2}F((Y_{Xt}(\tau) + \theta\epsilon Y_{X\mathcal{X}t}^{\epsilon}(\tau)) \otimes  m)(Y_{X\mathcal{X}t}^{\epsilon}(\tau)) - D_{X}^{2}F(Y_{Xt}(\tau) \otimes  m)(Y_{X\mathcal{X}t}^{\epsilon}(\tau))\right|\dif m(x)\dif \theta \dif \tau + 
\\
\int_{0}^{1}\bb{E}\int_{\bb{R}^n}\left|D_{X}^{2}F_{T}((Y_{Xt}(T) + \theta\epsilon Y_{X\mathcal{X}t}^{\epsilon}(T)) \otimes  m)(Y_{X\mathcal{X}t}^{\epsilon}(T)) - D_{X}^{2}F_{T}(Y_{Xt}(T) \otimes  m)(Y_{X\mathcal{X}t}^{\epsilon}(T))\right|\dif m(x)\dif \theta .
\end{multline*}
From the assumption (\ref{eq:5-5}) we can assert that 
\[
\bb{E}\int_{\bb{R}^n}\left|D_{X}^{2}F((Y_{Xt}(\tau) + \theta\epsilon Y_{X\mathcal{X}t}^{\epsilon}(\tau)) \otimes  m)(Y_{X\mathcal{X}t}^{\epsilon}(\tau)) - D_{X}^{2}F(Y_{Xt}(\tau) \otimes  m)(Y_{X\mathcal{X}t}^{\epsilon}(\tau))\right|\dif m(x)\rightarrow0,\forall\theta,\tau
\]
 and this function of $\theta,\tau,\epsilon$ is bounded by a constant,
thanks to (\ref{eq:5-2}) and the uniform bound on $Y_{X\mathcal{X}t}^{\epsilon}(\tau)$
in $\mathcal{H}_{m}.$
Similar assertions apply to the term involving $F_{T}.$ 
Thus (\ref{eq:8-8}) follows from the bounded convergence theorem.
On the other hand, $I^{\epsilon}$ is bounded
in $L^{2}(\Omega,\mathcal{W}_{X,\mathcal{X},t}^{s},\bb{P};L_{m}^{2}(\bb{R}^n;\bb{R}^n)).$
From (\ref{eq:8-8}), $0$ is the unique weak limit point. Hence $I_{xt}^{\epsilon}(s)$
converges weakly to $0$ in $L^{2}(\Omega,\mathcal{W}_{X,\mathcal{X},t}^{s},\bb{P};L_{m}^{2}(\bb{R}^n;\bb{R}^n)).$
Therefore $J_{xt}^{\epsilon}(s)$ converges weakly to $0$ in $L^{2}(\Omega,\mathcal{W}_{X,\mathcal{X},t}^{s},\bb{P};L_{m}^{2}(\bb{R}^n;\bb{R}^n)).$
We deduce that 
\begin{multline} \label{eq:8-9}
Z_{X\mathcal{X}t}^{\epsilon}(s) = \bb{E}\left[\int_{0}^{1}\left(\int_{s}^{T}D_{X}^{2}F((Y_{Xt}(\tau) + \theta\epsilon Y_{X\mathcal{X}t}^{\epsilon}(\tau)) \otimes  m)(Y_{X\mathcal{X}t}^{\epsilon}(\tau))\dif \tau \right.\right.\\
\left.\left. + \left.D_{X}^{2}F_{T}((Y_{Xt}(T) + \theta\epsilon Y_{X\mathcal{X}t}^{\epsilon}(T)) \otimes  m)(Y_{X\mathcal{X}t}^{\epsilon}(T))\right)\dif \theta \right|\mathcal{W}_{X,\mathcal{X},t}^{s}\right]\\
\rightarrow
\left.\bb{E}\left[\int_{s}^{T}D_{X}^{2}F(Y_{Xt}(\tau) \otimes  m)(\mathcal{Y}_{X\mathcal{X}t}(\tau))\dif \tau + D_{X}^{2}F_{T}(Y_{Xt}(T) \otimes  m)(\mathcal{Y}_{X\mathcal{X}t}(T))\right|\mathcal{W}_{X,\mathcal{X},t}^{s}\right]
\\
\text{ weakly in }L^{2}(\Omega,\mathcal{W}_{X,\mathcal{X},t}^{s},\bb{P};L_{m}^{2}(\bb{R}^n;\bb{R}^n)), \ \forall s.
\end{multline}

Necessarily the weak limit $\mathcal{Z}_{X\mathcal{X}t}(s)$ of $Z_{X\mathcal{X}t}^{\epsilon}(s)$
in the Hilbert space $L_{\mathcal{W}_{X\mathcal{X}t}}^{2}(t,T;\mathcal{H}_{m})$
coincides with the right hand side of (\ref{eq:8-9}). But then the
weak limits $\mathcal{Y}_{X\mathcal{X}t}(s),\:\mathcal{Z}_{X\mathcal{X}t}(s)$
coincide with the solution of the system (\ref{eq:5-11}), which is
unique. Therefore the whole sequence converges weakly. We also obtain weak convergence for any $s$, by the same argument as in \eqref{eq:Yeps weakly converges}. Let us check that the convergence is strong. We first argue as in the proof of Proposition \ref{prop4-1} to derive the identities
\begin{multline} \label{eq:weak to strong0}
\frac{1}{2}\ip{Z_{X\mathcal{X}t}^{\epsilon}(t)}{\s{X}} = \dfrac{1}{2\lambda}\int_{t}^{T}||Z_{X\mathcal{X}t}^{\epsilon}(s)||^{2}\dif s \\
+ \dfrac{1}{2}\int_{t}^{T}\ip{\int_{0}^{1}D_{X}^{2}F((Y_{Xt}(s) + \theta\epsilon Y_{X\mathcal{X}t}^{\epsilon}(s)) \otimes  m)(Y_{X\mathcal{X}t}^{\epsilon}(s))\dif \theta}{Y_{X\mathcal{X}t}^{\epsilon}(s)}\dif s 
\\
 + \dfrac{1}{2}\ip{\int_{0}^{1}D_{X}^{2}F_{T}((Y_{Xt}(T) + \theta\epsilon Y_{X\mathcal{X}t}^{\epsilon}(T)) \otimes  m)(Y_{X\mathcal{X}t}^{\epsilon}(T))\dif \theta}{Y_{X\mathcal{X}t}^{\epsilon}(T)}.
\end{multline}
and
\begin{multline} \label{eq:weak to strong1}
\frac{1}{2}\ip{\mathcal{Z}_{X\mathcal{X}t}(t)}{\s{X}} = \dfrac{1}{2\lambda}\int_{t}^{T}||\mathcal{Z}_{X\mathcal{X}t}(s)||^{2}\dif s
\\ 
+ \dfrac{1}{2}\int_{t}^{T}\ip{D^{2}F(Y_{Xt}(s) \otimes  m)(\mathcal{Y}_{X\mathcal{X}t}(s))}{\mathcal{Y}_{X\mathcal{X}t}(s)}\dif s 
+ \dfrac{1}{2}\ip{D^{2}F_{T}(Y_{Xt}(T) \otimes  m)(\mathcal{V}_{X\mathcal{X}t}(T))}{\mathcal{V}_{X\mathcal{X}t}(T)}.
\end{multline}
By continuity of $D_X^2 F$ and $D_X^2 F_T$, we have
\begin{multline} \label{eq:weak to strong2}
\int_{t}^{T}\left\{ \ip{\int_{0}^{1}D_{X}^{2}F((Y_{Xt}(s) + \theta\epsilon Y_{X\mathcal{X}t}^{\epsilon}(s)) \otimes  m)(\mathcal{Y}_{X\mathcal{X}t}(s))\dif \theta}{\mathcal{Y}_{X\mathcal{X}t}(s)}\right.\\
\left. - \ip{D_{X}^{2}F(Y_{Xt}(s) \otimes  m)(\mathcal{Y}_{X\mathcal{X}t}(s))}{\mathcal{Y}_{X\mathcal{X}t}(s)} \right\} \dif s \\
+  \ip{\int_{0}^{1}D_{X}^{2}F_{T}((Y_{Xt}(T) + \theta\epsilon Y_{X\mathcal{X}t}^{\epsilon}(T)) \otimes  m)(\mathcal{Y}_{X\mathcal{X}t}(T))\dif \theta}{\mathcal{Y}_{X\mathcal{X}t}(T)}\\
- \ip{D_{X}^{2}F_{T}(Y_{Xt}(T) \otimes  m)(\mathcal{Y}_{X\mathcal{X}t}(T))}{\mathcal{Y}_{X\mathcal{X}t}(T)} \to 0
\end{multline}
and 
\begin{multline} \label{eq:weak to strong3}
\int_{t}^{T}\left\{ \ip{\int_{0}^{1}D_{X}^{2}F((Y_{Xt}(s) + \theta\epsilon Y_{X\mathcal{X}t}^{\epsilon}(s)) \otimes  m)(Y_{X\mathcal{X}t}^{\epsilon}(s))\dif \theta}{\mathcal{Y}_{X\mathcal{X}t}(s)}\right.\\
\left. - \ip{D_{X}^{2}F(Y_{Xt}(s) \otimes  m)(\mathcal{Y}_{X\mathcal{X}t}(s))}{\mathcal{Y}_{X\mathcal{X}t}(s)} \right\} \dif s \\
+  \ip{\int_{0}^{1}D_{X}^{2}F_{T}((Y_{Xt}(T) + \theta\epsilon Y_{X\mathcal{X}t}^{\epsilon}(T)) \otimes  m)(Y_{X\mathcal{X}t}^{\epsilon}(T))\dif \theta}{\mathcal{Y}_{X\mathcal{X}t}(T)}\\
- \ip{D_{X}^{2}F_{T}(Y_{Xt}(T) \otimes  m)(\mathcal{Y}_{X\mathcal{X}t}(T))}{\mathcal{Y}_{X\mathcal{X}t}(T)} \to 0
\end{multline}
Putting together \eqref{eq:weak to strong0}, \eqref{eq:weak to strong1}, \eqref{eq:weak to strong2}, and \eqref{eq:weak to strong3}, and using the fact that $Z_{X\s{X}t}(s) \rightharpoonup \s{Z}_{X\s{X}t}(s)$ weakly, we deduce
\begin{multline}
\dfrac{1}{2\lambda}\int_{t}^{T}\enVert{Z_{X\mathcal{X}t}^{\epsilon}(s) - \mathcal{Z}_{X\mathcal{X}t}(s)}^{2}\dif s \\
+ \frac{1}{2}\int_{t}^{T} \ip{\int_{0}^{1}D_{X}^{2}F((Y_{Xt}(s) + \theta\epsilon Y_{X\mathcal{X}t}^{\epsilon}(s)) \otimes  m)(Y_{X\s{X}t}^\epsilon(s) - \mathcal{Y}_{X\mathcal{X}t}(s))\dif \theta}{Y_{X\s{X}t}^\epsilon(s) - \mathcal{Y}_{X\mathcal{X}t}(s)}\dif s\\
+ \frac{1}{2}\ip{\int_{0}^{1}D_{X}^{2}F_T((Y_{Xt}(T) + \theta\epsilon Y_{X\mathcal{X}t}^{\epsilon}(T)) \otimes  m)(Y_{X\s{X}t}^\epsilon(T) - \mathcal{Y}_{X\mathcal{X}t}(T))\dif \theta}{Y_{X\s{X}t}^\epsilon(T) - \mathcal{Y}_{X\mathcal{X}t}(T)}\dif s \to 0
\end{multline}
From the assumption (\ref{eq:5-3}), it follows that
\[
\limsup_{\epsilon\to 0}\left(\dfrac{1}{\lambda}\int_{t}^{T}\enVert{Z_{X\mathcal{X}t}^{\epsilon}(s) - \mathcal{Z}_{X\mathcal{X}t}(s)}^{2}\dif s - c'\int_{t}^{T}\enVert{Y_{X\mathcal{X}t}^{\epsilon}(s) - \mathcal{Y}_{X\mathcal{X}t}(s)}^{2}\dif s - c'_{T}\enVert{Y_{X\mathcal{X}t}^{\epsilon}(T) - \mathcal{Y}_{X\mathcal{X}t}(T)}^{2}\right) \leq 0.
\]
Using $Y_{X\mathcal{X}t}^{\epsilon}(s) - \mathcal{Y}_{X\mathcal{X}t}(s) =  - \dfrac{1}{\lambda}\int_{t}^{s}(Z_{X\mathcal{X}t}^{\epsilon}(\tau) - \mathcal{Z}_{X\mathcal{X}t}(\tau))\dif \tau$
and the condition (\ref{eq:3-6}) on $\lambda,$ we get $$\int_{t}^{T}\enVert{Z_{X\mathcal{X}t}^{\epsilon}(s) - \mathcal{Z}_{X\mathcal{X}t}(s)}^{2}\dif s\rightarrow 0.$$
This implies immediately $Y_{X\mathcal{X}t}^{\epsilon}(s)\rightarrow\mathcal{Y}_{X\mathcal{X}t}(s)$
in $L^{2}(\Omega,\mathcal{W}_{X,\mathcal{X},t}^{s},\bb{P};L_{m}^{2}(\bb{R}^n;\bb{R}^n)),\:\forall s$,
and by arguments already used $Z_{X\mathcal{X}t}^{\epsilon}(s)\to \mathcal{Z}_{X\mathcal{X}t}(s)$
in $L^{2}(\Omega,\mathcal{W}_{X,\mathcal{X},t}^{s},\bb{P};L_{m}^{2}(\bb{R}^n;\bb{R}^n)),\:\forall s.$ 

Since we have $Z_{Xt}(t) = D_X V(X  \otimes  m,t)$ by Proposition \ref{prop4-2}, $Z_{X\mathcal{X}t}^{\epsilon}(t)\to \mathcal{Z}_{X\mathcal{X}t}(t)$ implies that $X \mapsto V(X  \otimes  m,t)$ is twice G\^ateaux differentiable and that the relation (\ref{eq:5-12}) holds.
From (\ref{eq:8-4}), we derive property (\ref{eq:5-1000}).
It remains to prove the continuity property (\ref{eq:5-10}).
By \eqref{eq:5-12}, this is equivalent to showing that if $t_k \downarrow t$ and $X_k, \s{X}_k \in \s{H}_{m,t_k}$ converge in $\s{H}_m$ to $X,\s{X}$ respectively, then $\s{Z}_{X_k\s{X}_kt_k}(t_k) \to \s{Z}_{X\s{X}t}(t)$ in $\s{H}_m$.
From (\ref{eq:8-4}) and the linearity of $\s{Y}_{\color{black}Xt}$ and $\s{Z}_{\color{black}Xt}$ with respect to $\mathcal{X}$, we have
$||\mathcal{Y}_{X_{k}\mathcal{X}_{k}t_{k}}(s) - \mathcal{Y}_{X_{k}\mathcal{X},t_{k}}(s)|| \leq  C_{T}||\mathcal{X}_{k} - \mathcal{X}||,\;$$||\mathcal{Z}_{X_{k}\mathcal{X}_{k}t_{k}}(s) - \mathcal{Z}_{X_{k}\mathcal{X},t_{k}}(s)|| \leq  C_{T}||\mathcal{X}_{k} - \mathcal{X}||$.
So we can assume, without loss of generality, that $\mathcal{X}_{k} = \mathcal{X}$, i.e.~it is enough to show that $\s{Z}_{X_k\s{X} t_k}(t_k) \to \s{Z}_{X\s{X}t}(t)$ in $\s{H}_m$.

By
definition,
\begin{equation}\label{eq:8-12}
\begin{aligned}
\mathcal{Y}_{X_{k}\mathcal{X},t_{k}}(s) &= \mathcal{X} - \dfrac{1}{\lambda}\int_{t_{k}}^{s}\mathcal{Z}_{X_{k}\mathcal{X},t_{k}}(\tau)\dif \tau,\:s>t_{k},\\
\mathcal{Z}_{X_{k}\mathcal{X},t_{k}}(s) &= \left.\bb{E}\left[\int_{s}^{T}D_{X}^{2}F(Y_{X_{k}t_{k}}(\tau) \otimes  m)(\mathcal{Y}_{X_{k}\mathcal{X},t_{k}}(\tau))\dif \tau + D_{X}^{2}F_{T}(Y_{X_{k}t_{k}}(T) \otimes  m)(\mathcal{Y}_{X_{k}\mathcal{X},t_{k}}(T))\right|\mathcal{W}_{X_{k}\mathcal{X},t_{k}}^{s}\right].
\end{aligned}
\end{equation}
 with 
\begin{equation}\label{eq:8-11}
\begin{aligned}
Y_{X_{k}t_{k}}(s) &= X_{k} - \dfrac{1}{\lambda}\int_{t_{k}}^{s}Z_{X_{k}t_{k}}(\tau)\dif \tau + \eta(w(s) - w(t_{k})),\\
Z_{X_{k}t_{k}}(s) &= \left.\bb{E}\left[\int_{s}^{T}D_{X}F(Y_{X_{k}t_{k}}(\tau) \otimes  m)\dif \tau + D_{X}F(Y_{X_{k}t_{k}}(T) \otimes  m)\right|\mathcal{W}_{X_{k}t_{k}}^{s}\right].
\end{aligned}
\end{equation}
We fix $s>t.$ We can assume that $s>t_{k}.$ From the proof of Proposition
\ref{prop4-5}, we obtain 
\begin{multline}
\enVert{Y_{X_{k}t_{k}}(s) - Y_{Xt}(s)} \leq \enVert{Y_{X_{k}t_{k}}(s) - Y_{Xt_{k}}(s)} + \enVert{Y_{Xt_{k}}(s) - Y_{Xt}(s)} \\
 \leq  C_{T}\enVert{X_{k} - X} + C_{T}\del{(t_k-t)^{\frac{1}{2}} + (t_k-t)}(1 + \enVert{X})\label{eq:8-110}
\end{multline}
 and a similar estimate for $||Z_{X_{k}t_{k}}(s) - Z_{Xt}(s)||$.  
We fix $s>t,$ and define $k(s) = \min\{k|\:t_{k}<s\}$.
The function
$k(s)$ is monotone decreasing. We define the $\sigma$-algebra 
\begin{equation}
\widetilde{\mathcal{W}}_{X\mathcal{X}t}^{s} = \vee_{j \geq  k(s)}(\mathcal{W}_{X_{j}\mathcal{X},t_{j}}^{s}\vee\mathcal{W}_{t}^{t_{j}})\label{eq:8-13}
\end{equation}
which is increasing in $s.$ Note that $X$ is $\widetilde{\mathcal{W}}_{X\mathcal{X}t}^{s}$
measurable. We call $\widetilde{\mathcal{W}}_{X\mathcal{X}t}$ the
filtration generated by the sequence $\widetilde{\mathcal{W}}_{X\mathcal{X}t}^{s}.$
For $k \geq  k(s),$ $\widetilde{\mathcal{W}}_{X\mathcal{X}t}^{s}$
is an extension of $\mathcal{W}_{X_{k}\mathcal{X},t_{k}}^{s},$ independent
of $\mathcal{W}_{s}.$ Therefore, according to Remark \ref{rem3-1}
, we can change the conditioning $\sigma$ algebra $\mathcal{W}_{X_{k}\mathcal{X},t_{k}}^{s}$
to $\widetilde{\mathcal{W}}_{X\mathcal{X}t}^{s}.$ 
Then \eqref{eq:8-12} becomes
\begin{equation}\label{eq:8-14}
\begin{aligned}
\mathcal{Y}_{X_{k}\mathcal{X},t_{k}}(s) &= \mathcal{X} - \dfrac{1}{\lambda}\int_{t_{k}}^{s}\mathcal{Z}_{X_{k}\mathcal{X},t_{k}}(\tau)\dif \tau,\:s>t_{k},\\
\mathcal{Z}_{X_{k}\mathcal{X},t_{k}}(s) &= \left.\bb{E}\left[\int_{s}^{T}D_{X}^{2}F(Y_{X_{k}t_{k}}(\tau) \otimes  m)(\mathcal{Y}_{X_{k}\mathcal{X},t_{k}}(\tau))\dif \tau + D_{X}^{2}F_{T}(Y_{X_{k}t_{k}}(T) \otimes  m)(\mathcal{Y}_{X_{k}\mathcal{X},t_{k}}(T))\right|\widetilde{\mathcal{W}}_{X\mathcal{X}t}^{s}\right].
\end{aligned}
\end{equation}
Note that $s>t_{k}$ implies $k \geq  k(s).$ To define the processes
for $t<s<t_{k}$ we set 
\begin{equation}\label{eq:8-15}
\begin{cases}
\mathcal{Y}_{X_{k}\mathcal{X},t_{k}}(s) = \mathcal{X},\\
\mathcal{Z}_{X_{k}\mathcal{X},t_{k}}(s) = \bb{E}[\mathcal{Z}_{X_{k}\mathcal{X},t_{k}}(t_{k})|\widetilde{\mathcal{W}}_{X\mathcal{X}t}^{s}], 
\end{cases}
\ \text{whenever} \ t<s<t_{k}.
\end{equation}
To simplify notation, we shall denote $\mathcal{Y}^{k}(s) = \mathcal{Y}_{X_{k}\mathcal{X},t_{k}}(s)$
and $\mathcal{Z}^{k}(s) = \mathcal{Z}_{X_{k}\mathcal{X},t_{k}}(s)$.
We use a similar argument as in the first part of the proof to show that these sequences converge strongly in $\s{H}_m$.

We first see that $\mathcal{Y}^{k}(s)$ and $\mathcal{Z}^{k}(s)$
remain bounded in $L_{\widetilde{\mathcal{W}}_{X\mathcal{X}t}}^{\infty}(t,T;\mathcal{H}_{m}).$
We pick a subsequence of $\mathcal{Z}^{k}(\cdot)$, which converges weakly
to $\widetilde{\mathcal{Z}}(\cdot)$ in $L_{\widetilde{\mathcal{W}}_{X\mathcal{X}t}}^{2}(t,T;\mathcal{H}_{m}).$
Then 
\[
\mathcal{Y}^{k}(s)\rightharpoonup \mathcal{X} - \dfrac{1}{\lambda}\int_{t}^{s}\widetilde{\mathcal{Z}}(\tau)\dif \tau = \mathcal{\widetilde{Y}}(s)\:\text{ weakly in}\:L^{2}(\Omega,\widetilde{\mathcal{W}}_{X\mathcal{X}t}^{s},\bb{P};L_{m}^{2}(\bb{R}^n,\bb{R}^n)),\:\forall s\in(t,T].
\]
As above, we write 
\begin{multline*}
J_{xt}^{k}(s) = \bb{E}\left[\int_{s}^{T}(D_{X}^{2}F(Y_{X_{k}t_{k}}(\tau) \otimes  m)(\mathcal{Y}^{k}(\tau)) - D_{X}^{2}F(Y_{Xt}(\tau) \otimes  m)(\mathcal{\widetilde{Y}}(\tau)))\dif \tau\right.\\
\left. + D_{X}^{2}F_{T}(Y_{X_{k}t_{k}}(T)) \otimes  m)(\mathcal{Y}^{k}(T)) - D_{X}^{2}F_{T}(Y_{Xt}(T) \otimes  m)(\widetilde{\mathcal{Y}}(T))|\widetilde{\mathcal{W}}_{X\mathcal{X}t}^{s}\right],
\end{multline*}
which is an element of $L^{2}(\Omega,\widetilde{\mathcal{W}}_{X\mathcal{X}t}^{s},\bb{P};L_{m}^{2}(\bb{R}^n;\bb{R}^n)).$
We show that it converges weakly to $0.$ We write $J_{xt}^{k}(s) = I_{xt}^{k}(s) + II_{xt}^{k}(s)$
with 
\begin{multline*}
I_{xt}^{k}(s) = \bb{E}\left[\int_{s}^{T}(D_{X}^{2}F(Y_{X_{k}t_{k}}(\tau) \otimes  m)(\mathcal{Y}^{k}(\tau)) - D_{X}^{2}F(Y_{Xt}(\tau) \otimes  m)(\mathcal{Y}^{k}(\tau)))\dif \tau\right.
\\
\left. + D_{X}^{2}F_{T}(Y_{X_{k}t_{k}}(T)) \otimes  m)(\mathcal{Y}^{k}(T)) - D_{X}^{2}F_{T}(Y_{Xt}(T) \otimes  m)(\mathcal{Y}^{k}(T))|\widetilde{\mathcal{W}}_{X\mathcal{X}t}^{s}\right]
\end{multline*}
and 
\begin{multline*}
II_{xt}^{k}(s) = \bb{E}\left[\int_{s}^{T}(D_{X}^{2}F(Y_{Xt}(\tau) \otimes  m)(\mathcal{Y}^{k}(\tau)) - D_{X}^{2}F(Y_{Xt}(\tau) \otimes  m)(\mathcal{\widetilde{Y}}(\tau)))\dif \tau\right.
\\
\left.\left. + D_{X}^{2}F_{T}(Y_{Xt}(T) \otimes  m)(\mathcal{Y}^{k}(T)) - D_{X}^{2}F_{T}(Y_{Xt}(T) \otimes  m)(\mathcal{\widetilde{Y}}(T))\right|\mathcal{W}_{X,\mathcal{X},t}^{s}\right].
\end{multline*}
Once again, it is easy to see that $II_{xt}^{k}(s)\rightharpoonup 0$ weakly in $L^{2}(\Omega,\widetilde{\mathcal{W}}_{X\mathcal{X}t}^{s},\bb{P};L_{m}^{2}(\bb{R}^n;\bb{R}^n))$, for any $s.$ Using (\ref{eq:8-110}) and similar reasoning as
for $I_{xt}^{\epsilon}(s)$, we have 
\[
\bb{E}\int_{\bb{R}^n}|I_{xt}^{k}(s)|\dif m(x)\rightarrow0,\,\forall s
\]
and as above it follows that $I_{xt}^{k}(s)\rightharpoonup 0$ weakly in $L^{2}(\Omega,\widetilde{\mathcal{W}}_{X\mathcal{X}t}^{s},\bb{P};L_{m}^{2}(\bb{R}^n;\bb{R}^n))$
, for any $s>t.$ It follows that 
\[
\left.\mathcal{Z}^{k}(s)\rightharpoonup \bb{E}\left[\int_{s}^{T}D_{X}^{2}F(Y_{Xt}(\tau) \otimes  m(\mathcal{\widetilde{Y}}(\tau))\dif \tau + D_{X}^{2}F_{T}(Y_{Xt}(T) \otimes  m)(\mathcal{\widetilde{Y}}(T))\right|\mathcal{W}_{X,\mathcal{X},t}^{s}\right]
\]
 weakly in $L^{2}(\Omega,\widetilde{\mathcal{W}}_{X\mathcal{X}t}^{s},\bb{P};L_{m}^{2}(\bb{R}^n;\bb{R}^n)),$
for any $s>t.$
Then the weak limits $\mathcal{\widetilde{Y}}(s)$
and $\mathcal{\widetilde{Z}}(s)$ satisfy 
\begin{equation}\label{eq:8-16}
\begin{aligned}
\mathcal{\widetilde{Y}}(s) &= \mathcal{X} - \dfrac{1}{\lambda}\int_{t_{k}}^{s}\mathcal{\widetilde{Z}}(\tau)\dif \tau,\:s>t,\\
\mathcal{\widetilde{Z}}(s) &= \bb{E}\left[\int_{s}^{T}D_{X}^{2}F(Y_{Xt}(\tau) \otimes  m)(\mathcal{\widetilde{Y}}(\tau))\dif \tau + D_{X}^{2}F_{T}(Y_{Xt}(T) \otimes  m)(\mathcal{\widetilde{Y}}(T))|\widetilde{\mathcal{W}}_{X\mathcal{X}t}^{s}\right].
\end{aligned}
\end{equation}
By uniqueness of solutions to the forward-backward system, $\mathcal{\widetilde{Y}}(s) = \mathcal{Y}_{X\mathcal{X}t}(s),\:\mathcal{\widetilde{Z}}(s) = \mathcal{Z}_{X\mathcal{X}t}(s).$
From the uniqueness of the limit, the whole sequence converges weakly,
for any s. The convergence is strong, by a reasoning identical to
that for $Y^{\epsilon},Z^{\epsilon}$ above. The continuity (\ref{eq:5-12})
is obtained and the proof of Proposition \ref{prop5-1} is completed.
$\blacksquare$ 

\subsection{PROOF OF THEOREM \ref{theo5-1}}

We begin a lemma reducing the Brownian increments appearing in certain inner products to generic Gaussian random variables.
\begin{lem}
\label{lem8-1}
Let $V:\s{P}_2(\bb{R}^n) \times [0,T] \to \bb{R}$ be such that $D_X^2 V(X  \otimes  m,t)$ is continuous.
Let $X \in \s{H}_{m,t}$, i.e.~let $X \in \s{H}_m$ be independent of $\s{W}_t$, and
let $N$ be any standard Gaussian $N$ in $\bb{R}^n$ that is independent of both $X$ and the filtration $\mathcal{W}_{t}$.
Then we have
\begin{equation}
\ip{D_{X}^{2}V(X \otimes  m,t + h)\left(\eta\dfrac{w(t + h) - w(t)}{\sqrt{h}}\right)}{\eta\dfrac{w(t + h) - w(t)}{\sqrt{h}}} = \ip{D_{X}^{2}V(X \otimes  m,t + h)(\eta N)}{\eta N}.\label{eq:8-17}
\end{equation}
\end{lem}

\begin{proof}
Using the definition of the second derivative as a limit and the representation
of the first derivative, see (\ref{eq:4-11}) we have 
\begingroup
\allowdisplaybreaks
\begin{multline*}
\ip{D_{X}^{2}V(X \otimes  m,t + h)\del{\eta\dfrac{w(t + h) - w(t)}{\sqrt{h}}}}{\eta\dfrac{w(t + h) - w(t)}{\sqrt{h}}}\\
 = \lim_{\epsilon\to 0} \dfrac{1}{\epsilon}\left(\ip{D\dod{V}{m}\del{\del{X + \epsilon\eta\dfrac{w(t + h) - w(t)}{\sqrt{h}}} \otimes  m,t + h}\left(X + \epsilon\eta\dfrac{w(t + h) - w(t)}{\sqrt{h}}\right)}{\eta\dfrac{w(t + h) - w(t)}{\sqrt{h}}}\right.\\
 \left. - \ip{D\dod{V}{m}(X \otimes  m,t + h)(X)}{\eta\dfrac{w(t + h) - w(t)}{\sqrt{h}}}\right).
\end{multline*}\endgroup
Now $X$ and $\dfrac{w(t + h) - w(t)}{\sqrt{h}}$ are independent, as are $X$ and $N$, so $X + \epsilon\eta\dfrac{w(t + h) - w(t)}{\sqrt{h}}$ and $X + \epsilon\eta N$ have the same law.
In like manner, the probability measure $\del{X + \epsilon\eta\dfrac{w(t + h) - w(t)}{\sqrt{h}}} \otimes  m$
depends only on the marginals, and thus it is equal to $(X + \epsilon\eta N) \otimes  m.$
The right-hand side is therefore equal to
\begin{equation*}
\lim_{\epsilon\to 0} \dfrac{1}{\epsilon}\left(\ip{D\dod{V}{m}\del{\del{X + \epsilon\eta N} \otimes  m,t + h}(X + \epsilon\eta N)}{\eta N}  - \ip{D\dod{V}{m}(X \otimes  m,t + h)(X)}{\eta N}\right),
\end{equation*}
which is equal to the right-hand side of \eqref{eq:8-17}. $\blacksquare$ 
\end{proof}
We next obtain a variant of the continuity property (\ref{eq:5-10}) from Proposition \ref{prop5-1}.
\begin{lem}
\label{lem8-2} Let $t_{k}\downarrow t$, and {\color{black}$X_{k},\mathcal{X}_{k} \in \s{H}_{\color{black}m,t_k}$ so that $X_{k}\rightarrow X$
in $\mathcal{H}_{m}$ and $\displaystyle\sup_k\|\mathcal{X}_{k}\|_{\mathcal{H}_{m}}<\infty$.}
Then
\begin{equation}
\bb{E}\int_{\bb{R}^n}\abs{D_{X}^{2}V(X_{k} \otimes  m,t_{k})(\mathcal{X}_{k}) - D_{X}^{2}V(X \otimes  m,t_{k})(\mathcal{X}_{k})}\dif m(x)\;\rightarrow0\label{eq:8-170}
\end{equation}
\end{lem}

\begin{proof}
We consider the system
\begin{equation}\label{eq:8-10}
\begin{aligned}
\mathcal{Y}_{X_{k}\mathcal{X}_{k}t_{k}}(s) &= \mathcal{X}_{k} - \dfrac{1}{\lambda}\int_{t_{k}}^{s}\mathcal{Z}_{X_{k}\mathcal{X}_{k}t_{k}}(\tau)\dif \tau,\:s>t_{k},\\
\mathcal{Z}_{X_{k}\mathcal{X}_{k}t_{k}}(s) &= \left.\bb{E}\left[\int_{s}^{T}D_{X}^{2}F(Y_{X_{k}t_{k}}(\tau) \otimes  m)(\mathcal{Y}_{X_{k}\mathcal{X}_{k}t_{k}}(\tau))\dif \tau + D_{X}^{2}F_{T}(Y_{X_{k}t_{k}}(T) \otimes  m)(\mathcal{Y}_{X_{k}\mathcal{X}_{k}t_{k}}(T))\right|\mathcal{W}_{X_{k}\mathcal{X}_{k}t_{k}}^{s}\right]
\end{aligned}
\end{equation}
and the equivalent for $\mathcal{Y}_{X\mathcal{X}_{k}t_{k}}(s),\mathcal{Z}_{X\mathcal{X}_{k}t_{k}}(s).$
By Remark \ref{rem3-1}, we can condition on the common $\sigma$-algebra $\mathcal{W}_{X_{k}X\,\mathcal{X}_{k}t_{k}}^{s} = \mathcal{W}_{X_{k}\,\mathcal{X}_{k}t_{k}}^{s} \vee \mathcal{W}_{X\,\mathcal{X}_{k}t_{k}}^{s}$. So we can write 
\begin{equation}\label{eq:8-171}
\begin{aligned}
\mathcal{Y}_{X_{k}\mathcal{X}_{k}t_{k}}(s) &= \mathcal{X}_{k} - \dfrac{1}{\lambda}\int_{t_{k}}^{s}\mathcal{Z}_{X_{k}\mathcal{X}_{k}t_{k}}(\tau)\dif \tau,\:s>t_{k}\\
\mathcal{Z}_{X_{k}\mathcal{X}_{k}t_{k}}(s) &= \left.\bb{E}\left[\int_{s}^{T}D_{X}^{2}F(Y_{X_{k}t_{k}}(\tau) \otimes  m)(\mathcal{Y}_{X_{k}\mathcal{X}_{k}t_{k}}(\tau))\dif \tau + D_{X}^{2}F_{T}(Y_{X_{k}t_{k}}(T) \otimes  m)(\mathcal{Y}_{X_{k}\mathcal{X}_{k}t_{k}}(T))\right|\mathcal{W}_{X_{k}X\mathcal{X}_{k}t_{k}}^{s}\right]
\end{aligned}
\end{equation}
and 
\begin{equation}\label{eq:8-172}
\begin{aligned}
\mathcal{Y}_{X\mathcal{X}_{k}t_{k}}(s) &= \mathcal{X}_{k} - \dfrac{1}{\lambda}\int_{t_{k}}^{s}\mathcal{Z}_{X\mathcal{X}_{k}t_{k}}(\tau)\dif \tau,\:s>t_{k}\\
\mathcal{Z}_{X\mathcal{X}_{k}t_{k}}(s) &= \left.\bb{E}\left[\int_{s}^{T}D_{X}^{2}F(Y_{Xt_{k}}(\tau) \otimes  m)(\mathcal{Y}_{X\mathcal{X}_{k}t_{k}}(\tau))\dif \tau + D_{X}^{2}F_{T}(Y_{Xt_{k}}(T) \otimes  m)(\mathcal{Y}_{X\mathcal{X}_{k}t_{k}}(T))\right|\mathcal{W}_{X_{k}X\mathcal{X}_{k}t_{k}}^{s}\right].
\end{aligned}
\end{equation}
By Proposition \ref{prop5-1}, we have 
\begin{equation}
\mathcal{Z}_{X_{k}\mathcal{X}_{k}t_{k}}(t_{k}) = D_{X}^{2}V(X_{k} \otimes  m,t_{k})(\mathcal{X}_{k}),\;\mathcal{Z}_{X\mathcal{X}_{k}t_{k}}(t_{k}) = D_{X}^{2}V(X \otimes  m,t_{k})(\mathcal{X}_{k}).\label{eq:8-173}
\end{equation}
Therefore the result $(\ref{eq:8-170})$ will be a consequence of
\begin{equation}
\sup_{s\in[t_{k},T]}\bb{E}\int_{\bb{R}^n}|\mathcal{Y}_{X_{k}\mathcal{X}_{k}t_{k}}(s) - \mathcal{Y}_{X\mathcal{X}_{k}t_{k}}(s)|\dif m(x),\;\sup_{s\in[t_{k},T]}\bb{E}\int_{\bb{R}^n}|\mathcal{Z}_{X_{k}\mathcal{X}_{k}t_{k}}(s) - \mathcal{Z}_{X\mathcal{X}_{k}t_{k}}(s)|\dif m(x)\rightarrow0,\:\text{as}\:k\rightarrow + \infty\label{eq:8-174}
\end{equation}
We define 
\[
\widetilde{\mathcal{Y}}_{X_{k}X\mathcal{X}_{k}t_{k}}(s) = \mathcal{Y}_{X_{k}\mathcal{X}_{k}t_{k}}(s) - \mathcal{Y}_{X\mathcal{X}_{k}t_{k}}(s),\;\widetilde{\mathcal{Z}}_{X_{k}X\mathcal{X}_{k}t_{k}}(s) = \mathcal{Z}_{X_{k}\mathcal{X}_{k}t_{k}}(s) - \mathcal{Z}_{X\mathcal{X}_{k}t_{k}}(s)
\]
Then the pair $\del{\widetilde{\mathcal{Y}}_{X_{k}X\mathcal{X}_{k}t_{k}}(s), \widetilde{\mathcal{Z}}_{X_{k}X\mathcal{X}_{k}t_{k}}(s)}$
is the solution of the system 
\begin{equation}\label{eq:8-175}
\begin{aligned}
\widetilde{\mathcal{Y}}_{X_{k}X\mathcal{X}_{k}t_{k}}(s) &=  - \dfrac{1}{\lambda}\int_{t_{k}}^{s}\widetilde{\mathcal{Z}}_{X_{k}X\mathcal{X}_{k}t_{k}}(\tau)\dif \tau,\\
\widetilde{\mathcal{Z}}_{X_{k}X\mathcal{X}_{k}t_{k}}(s) &= \bb{E}\left[\int_{s}^{T}D_{X}^{2}F(Y_{Xt_{k}}(\tau) \otimes  m)(\widetilde{\mathcal{Y}}_{X_{k}X\mathcal{X}_{k}t_{k}}(\tau))\dif \tau \right. \\
&\quad \quad \left.\left.+ D_{X}^{2}F_{T}(Y_{Xt_{k}}(T) \otimes  m)(\widetilde{\mathcal{Y}}_{X_{k}X\mathcal{X}_{k}t_{k}}(T))\right|\mathcal{W}_{X_{k}X\mathcal{X}_{k}t_{k}}^{s}\right] + I_{X_{k}X\mathcal{X}_{k}t_{k}}(s),
\end{aligned}
\end{equation}
where
\begin{multline}
I_{X_{k}X\mathcal{X}_{k}t_{k}}(s) := \bb{E}\left[\int_{s}^{T}(D_{X}^{2}F(Y_{X_{k}t_{k}}(\tau) \otimes  m) - D_{X}^{2}F(Y_{Xt_{k}}(\tau) \otimes  m))(\mathcal{Y}_{X_{k}\mathcal{X}_{k}t_{k}}(\tau))\dif \tau\right.\label{eq:8-177}\\
\left.\left. + (D_{X}^{2}F_{T}(Y_{X_{k}t_{k}}(T) \otimes  m) - D_{X}^{2}F_{T}(Y_{Xt_{k}}(T) \otimes  m))(\mathcal{Y}_{X_{k}\mathcal{X}_{k}t_{k}}(T))\right|\mathcal{W}_{X_{k}X\mathcal{X}_{k}t_{k}}^{s}\right].
\end{multline}
We have 
\begin{multline}
\bb{E}\int_{\bb{R}^n}|I_{X_{k}X\mathcal{X}_{k}t_{k}}(s)|\dif m(x) \leq \label{eq:8-178}\\
\int_{s}^{T}\bb{E}\int_{\bb{R}^n}\abs{(D_{X}^{2}F(Y_{X_{k}t_{k}}(\tau) \otimes  m) - D_{X}^{2}F(Y_{Xt_{k}}(\tau) \otimes  m))(\mathcal{Y}_{X_{k}\mathcal{X}_{k}t_{k}}(\tau))}\dif m(x)\dif \tau \\
 + \bb{E}\int_{\bb{R}^n}\abs{(D_{X}^{2}F_{T}(Y_{X_{k}t_{k}}(T) \otimes  m) - D_{X}^{2}F_{T}(Y_{Xt_{k}}(T) \otimes  m))(\mathcal{Y}_{X_{k}\mathcal{X}_{k}t_{k}}(T))}\dif m(x).
\end{multline}
By the inequality (\ref{eq:8-4}) and the fact that $\mathcal{X}_{k}$
is bounded, we have 
\[
\sup_{s\in[t_{k},T]}||\mathcal{Y}_{X_{k}\mathcal{X}_{k}t_{k}}(s)|| \leq  C_{T}
\]
Since both
$Y_{X_{k}t_{k}}(s)$ and $Y_{Xt_{k}}(s)$ converge to $Y_{Xt}(s)$
in $\mathcal{H}_{m}$ for any $s$, we can then use the continuity property (\ref{eq:5-5}) and the bounds to deduce that
\[
l_{k}(s) := \bb{E}\int_{\bb{R}^n}|I_{X_{k}X\mathcal{X}_{k}t_{k}}(s)|\dif m(x)\to 0 \ \forall s \in \intoc{t,T}
\]
and that $l_{k}(s)$ is bounded.

 Now, we make use of the assumption (\ref{eq:5-50}) to write
\[
\bb{E}\int_{\bb{R}^n}|\widetilde{\mathcal{Z}}_{X_{k}X\mathcal{X}_{k}t_{k}}(s)|\dif m(x) \leq  c\int_{s}^{T}\bb{E}\int_{\bb{R}^n}|\widetilde{\mathcal{Y}}_{X_{k}X\mathcal{X}_{k}t_{k}}(\tau)|\dif m(x)\dif \tau + c_{T}\bb{E}\int_{\bb{R}^n}|\widetilde{\mathcal{Y}}_{X_{k}X\mathcal{X}_{k}t_{k}}(T)|\dif m(x)
\]
and, from the definition (\ref{eq:8-175}),
\[
|\widetilde{\mathcal{Y}}_{X_{k}X\mathcal{X}_{k}t_{k}}(\tau)| \leq \dfrac{1}{\lambda}\int_{t_{k}}^{\tau}|\widetilde{\mathcal{Z}}_{X_{k}X\mathcal{X}_{k}t_{k}}(\theta)|\dif \theta,\:|\widetilde{\mathcal{Y}}_{X_{k}X\mathcal{X}_{k}t_{k}}(T)| \leq \dfrac{1}{\lambda}\int_{t_{k}}^{T}|\widetilde{\mathcal{Z}}_{X_{k}X\mathcal{X}_{k}t_{k}}(\theta)|\dif \theta
\]
Combining the two previous inequalities yields 
\[
\del{1 - \dfrac{1}{\lambda}(c\dfrac{T^{2}}{2} + c_{T}T)}\int_{t_{k}}^{T}\bb{E}\int_{\bb{R}^n}|\widetilde{\mathcal{Z}}_{X_{k}X\mathcal{X}_{k}t_{k}}(s)|\dif m(x)\dif s \leq \int_{t_{k}}^{T}l_{k}(s)\dif s\rightarrow0.
\]
Thanks to (\ref{eq:5-120}) we obtain 
\[
\sup_{s\in[t_{k},T]}\bb{E}\int_{\bb{R}^n}|\widetilde{\mathcal{Y}}_{X_{k}X\mathcal{X}_{k}t_{k}}(s)|\dif m(x),\:\bb{E}\int_{\bb{R}^n}|\widetilde{\mathcal{Z}}_{X_{k}X\mathcal{X}_{k}t_{k}}(s)|\dif m(x)\rightarrow0
\]
 and thus (\ref{eq:8-174}) is proven. The proof of the Lemma is complete.
\end{proof}

Our final lemma proves that, under the regularity properties that $V(X,t)$ satisfies, we have a formula that ``lifts'' the usual It\^o formula to our Hilbert space setting.
\begin{lem} \label{lem:V(X(t+h),t+h)expansion}
	Suppose $V:\s{P}_2(\bb{R}^n) \times [0,T] \to \bb{R}$ is any function such that
	\begin{itemize}
		\item $V$ is continuous and satisfies the estimates \eqref{eq:4-2} and \eqref{eq:4-26};
		\item $D_X V(X  \otimes  m,t)$ exists, and {\color{black} for each $m\in\mathcal{P}_2(\mathbb{R}^n)$, it is separately (i.e. marginally only but not jointly) continuous in $X$ and in $t$ in the sense of \eqref{eq:4-4} and \eqref{eq:4-27}};
		\item $D_X^2 V(X  \otimes  m,t)$ exists, and {\color{black}for each $m\in\mathcal{P}_2(\mathbb{R}^n)$, it is sequentially continuous in the order pair $(X,t)$ in the sense of (\ref{eq:5-10}), and $D_X^2 V(X  \otimes  m,t)$ also satisfies the property \eqref{eq:5-1000}};
		\item the continuity property \eqref{eq:8-170} is satisfied.
	\end{itemize}
Let $X \in \s{H}_{m,t}$, let $v \in L^2_{\s{W}_{Xt}}(t,T;\s{H}_m) \cap \s{C}([t,T];\s{H}_m)$, and let $X(s) = X_{Xt}(s)$ be given by the SDE \eqref{eq:SDE}.
Then for $h > 0$ small enough, we have
\begin{multline} \label{eq:V(X(t+h),t+h)expansion}
V(X_{Xt}(t + h) \otimes  m,t + h) = V(X \otimes  m,t + h)   + \ip{D_{X}V(X \otimes  m,t)}{\int_{t}^{t + h}v(s)\dif s}\\
+ \dfrac{h}{2}\ip{D_{X}^{2}V(X \otimes  m,t)\del{\eta N}}{\eta N}
+ {\color{black}R_h,}
\end{multline}{\color{black}
where $R_h:=R_h^1+R_h^2+R_h^3$ such that $\dfrac{1}{h} R_h\to 0$ as $h \to 0$; here
\begin{align}
R_h^1:=&\dfrac{1}{h}\ip{D_{X}V(X \otimes  m,t + h) - D_{X}V(X \otimes  m,t)}{\int_{t}^{t + h}v(s)\dif s},\\\nonumber
R_h^2:=&\dfrac{1}{h}\ipt{\int_{0}^{1}\int_{0}^{1}\theta D_{X}^{2}V\del{(X + \theta\mu\del{\int_{t}^{t + h}v(s)\dif s + \eta(w(t + h) - w(t))}} \otimes  m,t + h)\dif \theta d\mu}\\
&\del{\int_{t}^{t + h}v(s)\dif s + 2\eta(w(t + h) - w(t))}\ipbc{\int_{t}^{t + h}v(s)\dif s}
\end{align}
\small
\begin{align}\nonumber
R_h^3:=&\ipt{\int_{0}^{1}\int_{0}^{1}\theta D_{X}^{2}V\left(\left(X + \theta\mu\del{\int_{t}^{t + h}v(s)\dif s + \eta(w(t + h) - w(t))}\right) \otimes  m,t + h\right)\dif \theta d\mu}\\
&\del{\eta\dfrac{w(t + h) - w(t)}{\sqrt{h}}}\ipbc{\eta\dfrac{w(t + h) - w(t)}{\sqrt{h}}}- \dfrac{1}{2}\ip{D_{X}^{2}V(X \otimes  m,t + h)\del{\eta\dfrac{w(t + h) - w(t)}{\sqrt{h}}}}{\eta\dfrac{w(t + h) - w(t)}{\sqrt{h}}}.
\end{align}
\normalsize
}
\end{lem}

\begin{proof}	
	Since $V(X \otimes  m,t)$
	has a second derivative with respect to $X,$ we can begin with the following expansion: 
\begingroup
\allowdisplaybreaks
\begin{multline}
	V(X_{Xt}(t + h) \otimes  m,t + h) = V(X \otimes  m,t + h)  \label{eq:8-18}\\
	+ \ip{D_{X}V(X \otimes  m,t + h)}{\int_{t}^{t + h}v(s)\dif s + \eta(w(t + h) - w(t))} \\
	+ \ipt{\int_{0}^{1}\int_{0}^{1}\theta D_{X}^{2}V\del{X + \theta\mu\del{\int_{t}^{t + h}v(s)\dif s + \eta(w(t + h) - w(t))}} \otimes  m,t + h)\dif \theta d\mu}\\
	\del{\int_{t}^{t + h}v(s)\dif s + \eta(w(t + h) - w(t))}\ipbc{\int_{t}^{t + h}v(s)\dif s + \eta(w(t + h) - w(t))}.
	\end{multline}\endgroup
Since $D_X V(X  \otimes  m,t+h)$ is $\sigma(X)$-measurable and $X$ is independent of $\s{W}_t$, we have
	\begin{equation} \label{eq:8-18-1}
	\ip{D_{X}V(X \otimes  m,t + h)}{\eta(w(t + h) - w(t))} = 0.
	\end{equation}
	From the H\"older-in-time property (\ref{eq:4-27}) and the fact that $v(s)$ is bounded in $\s{H}_m$, we deduce
	\begin{equation} \label{eq:8-18-2}
	{\color{black}R_h^1=}\dfrac{1}{h}\ip{D_{X}V(X \otimes  m,t + h) - D_{X}V(X \otimes  m,t)}{\int_{t}^{t + h}v(s)\dif s} \to 0,\:\text{as}\:h\rightarrow0.
	\end{equation}
	Also, by the estimate \eqref{eq:5-1000} and again using the bound on $v(s)$, we deduce that
	\begin{multline} \label{eq:8-18-3}
	{\color{black}R_h^2=}\frac{1}{h}\ipt{\int_{0}^{1}\int_{0}^{1}\theta D_{X}^{2}V\left(\del{X + \theta\mu\del{\int_{t}^{t + h}v(s)\dif s + \eta(w(t + h) - w(t))}} \otimes  m,t + h\right)\dif \theta d\mu}\\
	\del{\int_{t}^{t + h}v(s)\dif s + {\color{black}2\eta}(w(t + h) - w(t))}\ipbc{\int_{t}^{t + h}v(s)\dif s} \to 0, \ h \to 0,
	\end{multline}
{\color{black}since $\lim_{h\rightarrow0}\frac{1}{h}\int_{t}^{t + h}v(s)\dif s<\infty$ and $\lim_{h\rightarrow0} \left(\int_{t}^{t + h}v(s)\dif s + {\color{black}2\eta}(w(t + h) - w(t))\right)\to 0$ and so a simple application of dominated convergence theorem warrants this convergence to zero.}
	We next prove that 
\small{\color{black}
	\begin{align}\nonumber
	R_h^3=&\ipt{\int_{0}^{1}\int_{0}^{1}\theta D_{X}^{2}V\left(\del{X + \theta\mu\del{\int_{t}^{t + h}v(s)\dif s + \eta(w(t + h) - w(t))}} \otimes  m,t + h\right)\dif \theta d\mu}\\\nonumber
	&\del{\eta\dfrac{w(t + h) - w(t)}{\sqrt{h}}}\ipbc{\eta\dfrac{w(t + h) - w(t)}{\sqrt{h}}}- \dfrac{1}{2}\ip{D_{X}^{2}V(X \otimes  m,t + h)\del{\eta\dfrac{w(t + h) - w(t)}{\sqrt{h}}}}{\eta\dfrac{w(t + h) - w(t)}{\sqrt{h}}}\\\label{eq:8-19}
 \to& 0, \ h \to 0.
	\end{align}}
\normalsize
	Set 
	\[
	X_{h}(\theta,\mu) = X + \theta\mu\del{\int_{t}^{t + h}v(s)\dif s + \eta(w(t + h) - w(t))},\;\mathcal{X}_{h} = \eta\dfrac{w(t + h) - w(t)}{\sqrt{h}}.
	\]
	Then the expression (\ref{eq:8-19}) is equivalent to
	\begin{equation}
	\int_{0}^{1}\int_{0}^{1}\theta L_h(\theta,\mu)\dif \theta \dif \mu\rightarrow 0.\label{eq:8-21}
	\end{equation}
	where
	\begin{equation*}
	L_h(\theta,\mu) := \bb{E}\int_{\bb{R}^n}(D_{X}^{2}V(X_{h}(\theta,\mu) \otimes  m,t + h)(\mathcal{X}_{h}) - D_{X}^{2}V(X \otimes  m,t + h)(\mathcal{X}_{h}))\cdot \mathcal{X}_{h}\dif m(x).
	\end{equation*}
	By estimate \eqref{eq:5-1000} we have
	\[
	|L_{h}(\theta,\mu)| \leq  C_{T}\enVert{\s{X}_h}^2 = C_T\text{tr }\eta\eta^{*},
	\]
	i.e.~$L_h(\theta,\mu)$ is bounded.
	So to prove \eqref{eq:8-21} it is enough to show that $L_h(\theta,\mu) \to 0$ pointwise as $h \to 0$.
	Now from the continuity estimate (\ref{eq:8-170}) we have
	\begin{equation}
	\bb{E}\int_{\bb{R}^n}|D_{X}^{2}V(X_{h}(\theta,\mu) \otimes  m,t + h)(\mathcal{X}_{h}) - D_{X}^{2}V(X \otimes  m,t + h)(\mathcal{X}_{h})|\dif m(x)\rightarrow0.\label{eq:8-20}
	\end{equation}
	Define
	\[
	L_{h\epsilon}(\theta,\mu) := \bb{E}\int_{\bb{R}^n}(D_{X}^{2}V(X_{h}(\theta,\mu) \otimes  m,t + h)(\mathcal{X}_{h}) - D_{X}^{2}V(X \otimes  m,t + h)(\mathcal{X}_{h}))\cdot \dfrac{\mathcal{X}_{h}}{1 + \epsilon|\mathcal{X}_{h}|}\dif m(x).
	\]
	Then by \eqref{eq:8-20} we have $L_{h\epsilon}(\theta,\mu) \to 0$ as $h \to 0$, for fixed $\epsilon > 0, \theta, \mu \in [0,1]$.
	Notice that
	\[
	L_{h}(\theta,\mu) - L_{h\epsilon}(\theta,\mu) := \epsilon \bb{E}\int_{\bb{R}^n}(D_{X}^{2}V(X_{h}(\theta,\mu) \otimes  m,t + h)(\mathcal{X}_{h}) - D_{X}^{2}V(X \otimes  m,t + h)(\mathcal{X}_{h}))\cdot \dfrac{\mathcal{X}_{h}|\mathcal{X}_{h}|}{1 + \epsilon|\mathcal{X}_{h}|}\dif m(x).
	\]
	By estimate \eqref{eq:5-1000} in Proposition \ref{prop5-1} we have 
	\[
	|L_{h}(\theta,\mu) - L_{h\epsilon}(\theta,\mu)| \leq \epsilon C_{T}\enVert{\s{X}_h}\enVert{\dfrac{\mathcal{X}_{h}|\mathcal{X}_{h}|}{1 + \epsilon|\mathcal{X}_{h}|}}.
	\]
	Taking the fourth moment of a Gaussian random variable, we have
	$\enVert{\dfrac{\mathcal{X}_{h}|\mathcal{X}_{h}|}{1 + \epsilon|\mathcal{X}_{h}|}} \leq \bb{E}\intcc{\abs{\s{X}_h}^4}^{1/2} \leq \sqrt{3}\enVert{\eta}^{2},$ and we conclude $|L_{h}(\theta,\mu) - L_{h\epsilon}(\theta,\mu)| \leq  C'_{T}\epsilon.$ Since $\epsilon > 0$ is arbitrary, we deduce that $L_h(\theta,\mu) \to 0$ pointwise, and property (\ref{eq:8-21}) follows. 
	
	Combining \eqref{eq:8-18-1}, \eqref{eq:8-18-2}, \eqref{eq:8-18-3}, and \eqref{eq:8-19} with \eqref{eq:8-18}, we obtain
	\begin{multline*}
	V(X_{Xt}(t + h) \otimes  m,t + h) = V(X \otimes  m,t + h)   + \ip{D_{X}V(X \otimes  m,t)}{\int_{t}^{t + h}v(s)\dif s}\\
	+ \dfrac{1}{2}\ip{D_{X}^{2}V(X \otimes  m,t + h)\del{\eta(w(t + h) - w(t))}}{\eta(w(t + h) - w(t))}
	+ o(h).
	\end{multline*}
	Using Lemma \ref{lem8-1} to rewrite the last term and applying the continuity of $D_{X}^{2}V(X \otimes  m,t)$ with respect to $t$, we deduce \eqref{eq:V(X(t+h),t+h)expansion}.
\end{proof}

We can now proceed with the proof of Theorem \ref{theo5-1}.
To prove that $V(X  \otimes  m,t)$ solves the Bellman equation \eqref{eq:Bellman}, first note that $V$ satisfies the hypotheses of Lemma \ref{lem:V(X(t+h),t+h)expansion} by Propositions \ref{prop4-1}, \ref{prop4-2}, \ref{prop4-5}, and \ref{prop5-1} as well as Lemma \ref{lem8-2}.
So we take \eqref{eq:V(X(t+h),t+h)expansion} with $X_{Xt}(s) = Y_{Xt}(s)$ (the optimal trajectory) and $v(s) = -\frac{1}{\lambda}Z_{Xt}(s)$ (the optimal control, which is continuous by \eqref{eq:ZXt(t+h)-ZXt(t)}), and we combine it with the optimality principle (\ref{eq:3-16}) to get
\begin{multline*}
V(X \otimes  m,t) - V(X \otimes  m,t + h) = \dfrac{1}{2\lambda}\int_{t}^{t + h}||Z_{Xt}(s)||^{2}\dif s + \int_{t}^{t + h}F(Y_{Xt}(s) \otimes  m)\dif s\\
 - \ip{D_{X}V(X \otimes  m,t)}{\dfrac{1}{\lambda}\int_{t}^{t + h}Z_{Xt}(s)\dif s} + \dfrac{h}{2}\ip{D_{X}^{2}V(X \otimes  m,t + h)(\eta N)}{\eta N} + o(h).
\end{multline*}
Letting $h \to 0$, we see that $V$ is right-differentiable and that the Bellman equation \eqref{eq:Bellman} is satisfied.

Conversely, suppose $V$ is any other classical solution to the Bellman equation \eqref{eq:Bellman}.
Note that it satisfies the regularity properties assumed in Lemma \ref{lem:V(X(t+h),t+h)expansion}.
we shall show that $V$ must be equal to the the value function defined by \eqref{eq:V(X,t) def}.
Take any $v \in L^2_{\s{W}_{Xt}}(t,T;\s{H}_m) \cap \s{C}([t,T];\s{H}_m)$, and let $X(s) = X_{Xt}(s)$ be given by the SDE \eqref{eq:SDE}.
Then by taking $t = s$ and $X = X_{Xt}(s)$ in \eqref{eq:V(X(t+h),t+h)expansion}, we get
\begin{multline*}
V(X_{Xt}(s + h) \otimes  m,s + h) = V(X_{Xt}(s) \otimes  m,s + h)   + \ip{D_{X}V(X_{Xt}(s) \otimes  m,t)}{\int_{s}^{s + h}v(\tau)\dif \tau}\\
+ \dfrac{h}{2}\ip{D_{X}^{2}V(X_{Xt}(s) \otimes  m,s)\del{\eta N}}{\eta N}
+ {\color{black}R_h}.
\end{multline*}
Subtract $V(X_{Xt}(s) \otimes  m,s)$ from both sides and divide by $h$, then send $h \to 0$.
Using the fact that $V(X  \otimes  m,t)$ is right-differentiable with respect to $t$, we see that $V(X_{Xt}(s) \otimes  m,s)$ is differentiable and
\begin{multline} \label{eq:verification}
\dod{}{s}\del{V(X_{Xt}(s) \otimes  m,s)} = \dpd{V}{t}(X_{Xt}(s) \otimes  m,s)
+ \ip{D_{X}V(X_{Xt}(s) \otimes  m,s)}{v(s)}\\
+ \frac{1}{2}\ip{D_{X}^{2}V(X_{Xt}(s) \otimes  m,s)\del{\eta N}}{\eta N}
= \ip{D_{X}V(X_{Xt}(s) \otimes  m,s)}{v(s)} \\
+ \frac{1}{2\lambda}\enVert{D_{X}V(X_{Xt}(s) \otimes  m,s)}^2 - F(X_{Xt}(s) \otimes  m)
\geq -\frac{\lambda}{2}\enVert{v(s)}^2 - F(X_{Xt}(s) \otimes  m).
\end{multline}
Integrating from $t$ to $T$ reveals $V(X,t) \leq J_{Xt}(v(\cdot)),$ where $J_{Xt}$ is the objective functional defined in \eqref{eq:J_X}.
In particular, we can take $v$ to be the optimal control $\hat v$, and thus $V(X,t) \leq J_{Xt}(\hat v(\cdot))$.
On the other hand, we can first solve the SDE
\begin{equation*}
X(s) = X - \frac{1}{\lambda}\int_t^s D_X(X(s)  \otimes  m,s)\dif s + \eta(w(s)-w(t))
\end{equation*}
and then take as a candidate control $v(s) = - \frac{1}{\lambda}D_X(X(s)  \otimes  m,s)$.
Then all the inequalities in \eqref{eq:verification} become equalities, and we see that $V(X,t) = J_{Xt}(v(\cdot))$.
It follows that $v(\cdot)$ must in fact be optimal and then $V$ is the value function.

This completes the proof of Theorem $\ref{theo5-1}.$ $\blacksquare$

\section{PROOFS FROM SECTION 6} \label{ap:D}

\subsection{PROOF OF THE CLAIM IN REMARK \ref{rem6-4}. }

Step $1$. We first aim to show, for each fixed $m\in\s{P}_2(\bb{R}^n)$ and a $\xi\in\bb{R}^n$, the continuity of $(\bar{Y}_{mt}(s,\xi,x),\bar{Z}_{mt}(s,\xi,x))$ in $x\in\bb{R}^n$  by making use of \eqref{convexF}-\eqref{convexFT} with $\lambda$ being only greater than $T(c_T+cT)$, a relatively smaller value. As before, we have \eqref{eq6-41}-\eqref{eq6-44}.
By using \eqref{convexF} to the second and third terms in the following third equality and \eqref{eq6-43} to forth and fifth terms in the same equality, one has
\begin{align*}
&\bb{E}\int_{\bb{R}^n} \left(\bar{Y}_{mt}(T,\xi,x')-\bar{Y}_{mt}(T,\xi,x)\right)\cdot \left(\bar{Z}_{mt}(T,\xi,x')-\bar{Z}_{mt}(T,\xi,x)\right) dm(\xi)\\
&-\bb{E}\int_{\bb{R}^n} \left(\bar{Y}_{mt}(t,\xi,x')-\bar{Y}_{mt}(t,\xi,x)\right)\cdot \left(\bar{Z}_{mt}(t,\xi,x')-\bar{Z}_{mt}(t,\xi,x)\right) dm(\xi)\\
=&\int_t^T d_s\left(\bb{E}\int_{\bb{R}^n} \left(\bar{Y}_{mt}(s,\xi,x')-\bar{Y}_{mt}(s,\xi,x)\right)\cdot \left(\bar{Z}_{mt}(s,\xi,x')-\bar{Z}_{mt}(s,\xi,x)\right) dm(\xi)\right)\\
=&\int_t^T\bb{E}\int_{\bb{R}^n} d_s\left(\bar{Y}_{mt}(s,\xi,x')-\bar{Y}_{mt}(s,\xi,x)\right)\cdot \left(\bar{Z}_{mt}(s,\xi,x')-\bar{Z}_{mt}(s,\xi,x)\right) dm(\xi)\\
&+\bb{E}\int_{\bb{R}^n} \left(\bar{Y}_{mt}(s,\xi,x')-\bar{Y}_{mt}(s,\xi,x)\right)\cdot d_s\left(\bar{Z}_{mt}(s,\xi,x')-\bar{Z}_{mt}(s,\xi,x)\right) dm(\xi)\\
=&\int_t^T\bigg[-\frac{1}{\lambda} \bb{E}\int_{\bb{R}^n} \left|\bar{Z}_{mt}(s,\xi,x')-\bar{Z}_{mt}(s,\xi,x)\right|^2 dm(\xi)\\
&-\bb{E}\int_{\bb{R}^n}\left(\bar{Y}_{mt}(s,\xi,x')-\bar{Y}_{mt}(s,\xi,x)\right)\cdot\left(D^{2}\dod{F}{m}(Y_{\cdot m t}(s) \otimes  m)(Y_{\xi mt}(s))\left(\bar{Y}_{mt}(s,\xi,x')-\bar{Y}_{mt}(s,\xi,x)\right)\right.\\
&+\left.\widetilde{\bb{E}}\int_{\bb{R}^n}D_2D_1\dod[2]{F}{m}(Y_{\cdot mt}(s) \otimes  m)(Y_{\xi mt}(s),\widetilde{Y}_{\zeta mt}(s))\left(\widetilde{\bar{Y}}_{mt}(s,\zeta,x')-\widetilde{\bar{Y}}_{mt}(s,\zeta,x)\right)\dif m(\zeta)\right)dm(\xi)\\
&+\bb{E}\int_{\bb{R}^n}\left(\bar{Y}_{mt}(s,\xi,x')-\bar{Y}_{mt}(s,\xi,x)\right)\cdot \left( \widetilde{\bb{E}}D_1\dod[2]{F}{m}(Y_{\cdot mt}(s) \otimes  m)(Y_{\xi mt}(s),\widetilde{Y}_{x'mt}(s))\right.\\
& \left.\left.-\widetilde{\bb{E}}D_1\dod[2]{F}{m}(Y_{\cdot mt}(s) \otimes  m)(Y_{\xi mt}(s),\widetilde{Y}_{xmt}(s))\right) dm(\xi)\right]ds\\
\leq &\int_t^T\left[-\frac{1}{\lambda} \bb{E}\int_{\bb{R}^n}\left|\bar{Z}_{mt}(s,\xi,x')-\bar{Z}_{mt}(s,\xi,x)\right|^2 dm(\xi)-\lambda_F\bb{E}\int_{\bb{R}^n}\left|\bar{Y}_{mt}(s,\xi,x')-\bar{Y}_{mt}(s,\xi,x)\right|^2 dm(\xi)\right.\\
&+\frac{c \lambda}{\lambda-T(c_T+cT)}\cdot \bb{E}\left(\int_{\bb{R}^n} \left|\bar{Y}_{mt}(s,\xi,x')-\bar{Y}_{mt}(s,\xi,x)\right|dm(\xi) \right)\cdot\left|x'-x\right|\bigg]ds.
\end{align*}
Thus, by using \eqref{convexFT} and \eqref{eq6-44} for the terminal cross term, also noting that, as a linear functional derivative with respect to $m$, $\bar{Y}_{mt}(t,\xi,x')=\bar{Y}_{mt}(t,\xi,x)=0$, we then obtain: \\
\begingroup
\allowdisplaybreaks\begin{align*}
&\int_t^T \bb{E}\int_{\bb{R}^n}\left|\bar{Z}_{mt}(s,\xi,x')-\bar{Z}_{mt}(s,\xi,x)\right|^2 dm(\xi)ds+\int_t^T \bb{E}\int_{\bb{R}^n}\left|\bar{Y}_{mt}(s,\xi,x')-\bar{Y}_{mt}(s,\xi,x)\right|^2 dm(\xi)ds\\
&+\bb{E}\int_{\bb{R}^n}\left|\bar{Y}_{mt}(T,\xi,x')-\bar{Y}_{mt}(T,\xi,x)\right|^2 dm(\xi)\\
\leq &C_4(c,c_T,\lambda,\lambda_F,T)\cdot\left(\int_t^T \bb{E}\int_{\bb{R}^n} \left|\bar{Y}_{mt}(s,\xi,x')-\bar{Y}_{mt}(s,\xi,x)\right| dm(\xi)ds\right.\\
&\left.+\bb{E}\int_{\bb{R}^n} \left|\bar{Y}_{mt}(T,\xi,x')-\bar{Y}_{mt}(T,\xi,x)\right| dm(\xi)\right)\cdot|x'-x|
\end{align*}\endgroup
\begin{align*}
\leq &\sqrt{2}\max\{\sqrt{T},1\}\cdot C_4(c,c_T,\lambda,\lambda_F,T)\cdot\left(\int_t^T \bb{E}\int_{\bb{R}^n} \left|\bar{Y}_{mt}(s,\xi,x')-\bar{Y}_{mt}(s,\xi,x)\right|^2 dm(\xi)ds\right.\\
&\left.+\bb{E}\int_{\bb{R}^n} \left|\bar{Y}_{mt}(T,\xi,x')-\bar{Y}_{mt}(T,\xi,x)\right|^2 dm(\xi)\right)^{1/2}\cdot|x'-x|,
\end{align*}
where $C_4(c,c_T,\lambda,\lambda_F,T):=\max\{\frac{1}{\lambda_F},\lambda\}\cdot\max\{c,c_T\}\cdot\frac{\lambda}{\lambda-T(c_T+cT)}$.
In particular, one has
\begin{align*}
&\int_t^T \bb{E}\int_{\bb{R}^n} \left|\bar{Y}_{mt}(s,\xi,x')-\bar{Y}_{mt}(s,\xi,x)\right|^2 dm(\xi)ds+\bb{E}\int_{\bb{R}^n} \left|\bar{Y}_{mt}(T,\xi,x')-\bar{Y}_{mt}(T,\xi,x)\right|^2 dm(\xi)\\
\leq& C_5(c,c_T,\lambda,\lambda_F,T)\cdot|x'-x|^{2},
\end{align*}
where $C_5(c,c_T,\lambda,\lambda_F,T):=2\max\{T,1\}\cdot C_4^2(c,c_T,\lambda,\lambda_F,T)$; besides,
\begin{align}\label{ZHolder}
&\int_t^T \bb{E}\int_{\bb{R}^n}\left|\bar{Z}_{mt}(s,\xi,x')-\bar{Z}_{mt}(s,\xi,x)\right|^2 dm(\xi)ds\leq C_5(c,c_T,\lambda,\lambda_F,T)\cdot|x'-x|^{2}.
\end{align}
Specifically, we can conclude that, for one-dimensional case $n=1$, $x\in\bb{R}$,
by the celebrated Kolmogorov continuity theorem, $\bar{Y}_{mt}(s,\xi,x)$ and $\bar{Z}_{mt}(s,\xi,x)$ possess a continuous version, denoted by $\bar{Y}_{mt}^{(1)}(s,\xi,x)$ and $\bar{Z}_{mt}^{(1)}(s,\xi,x)$ respectively, such that, for each $x\in\bb{R}$, the product measure of $\bb{P}\otimes\s{L}^1_{([t,T])}\otimes m\{(\omega,s,\xi)\in\Omega\times[t,T]\times\bb{R}:\bar{Y}_{mt}(\omega,s,\xi,x)\neq \bar{Y}_{mt}^{(1)}(\omega,s,\xi,x)\text{ or }\bar{Z}_{mt}(\omega,s,\xi,x)\neq \bar{Z}_{mt}^{(1)}(\omega,s,\xi,x)\}=0$ and, \textbf{for each} $(\omega,s,\xi)\in\Omega\times[t,T]\times\bb{R}$, $\bar{Y}_{mt}^{(1)}(\omega,s,\xi,x)$ and $\bar{Z}_{mt}^{(1)}(\omega,s,\xi,x)$ are locally $\gamma$-H\"{o}lder continuous in $x\in\bb{R}$ for every $0<\gamma<1/2$.\\
Step $2$. We now establish the continuity of $(\bar{Y}_{mt}(s,\xi,x),\bar{Z}_{mt}(s,\xi,x))$ in $m\in\s{P}_2(\bb{R}^n)$ for each fixed $\xi\in\bb{R}^n$ and $x\in\bb{R}^n$ by making use of \eqref{convexF}-\eqref{convexFT} with $\lambda$ being only greater than $T(c_T+cT)$. As before, we have 
\begin{align}\nonumber
&\bar{Y}_{m't}(s,\xi,x)-\bar{Y}_{mt}(s,\xi,x) \\
=&-\dfrac{1}{\lambda}\int_{t}^{s}\left(\bar{Z}_{m't}(\tau,\xi,x)-\bar{Z}_{mt}(\tau,\xi,x)\right)\dif \tau,\\
\nonumber
&\bar{Z}_{m' t}(s,\xi,x)-\bar{Z}_{mt}(s,\xi,x) \\\nonumber
=&\bb{E}\left[ \int_{s}^{T} \left(D^{2}\dod{F}{m}(Y_{\cdot m' t}(\tau) \otimes  m')(Y_{\xi m't}(\tau))\bar{Y}_{m't}(\tau,\xi,x)-D^{2}\dod{F}{m}(Y_{\cdot mt}(\tau) \otimes  m)(Y_{\xi mt}(\tau))\bar{Y}_{mt}(\tau,\xi,x)\right)\dif \tau \right.\\
\nonumber
&\ \ \ \ + D^{2}\dod{}{m}F_{T}(Y_{\cdot m't}(T) \otimes  m')(Y_{\xi m't}(T))\bar{Y}_{m't}(T,\xi,x)-D^{2}\dod{}{m}F_{T}(Y_{\cdot mt}(T) \otimes  m)(Y_{\xi mt}(T))\bar{Y}_{mt}(T,\xi,x)\\
\nonumber
&\ \ \ \ +\widetilde{\bb{E}}\int_{s}^{T}\int_{\bb{R}^n}D_2D_1\dod[2]{F}{m}(Y_{\cdot m't}(\tau) \otimes  m')(Y_{\xi m't}(\tau),\widetilde{Y}_{\zeta m't}(\tau))\widetilde{\bar{Y}}_{m't}(\tau,\zeta,x)\dif m'(\zeta) \dif \tau\\
\nonumber
&\ \ \ \ -\widetilde{\bb{E}}\int_{s}^{T}\int_{\bb{R}^n}D_2D_1\dod[2]{F}{m}(Y_{\cdot mt}(\tau) \otimes  m)(Y_{\xi mt}(\tau),\widetilde{Y}_{\zeta mt}(\tau))\widetilde{\bar{Y}}_{mt}(\tau,\zeta,x)\dif m(\zeta) \dif \tau\\
\nonumber
&\ \ \ \ +\widetilde{\bb{E}}\int_{\bb{R}^n}D_2D_1\dod[2]{F_{T}}{m}(Y_{\cdot m't}(T) \otimes  m')(Y_{\xi m't}(T),\widetilde{Y}_{\zeta m't}(T))\widetilde{\bar{Y}}_{m't}(T,\eta,x)\dif m'(\zeta)\\
\nonumber
&\ \ \ \ -\widetilde{\bb{E}}\int_{\bb{R}^n}D_2D_1\dod[2]{F_{T}}{m}(Y_{\cdot mt}(T) \otimes  m)(Y_{\xi mt}(T),\widetilde{Y}_{\zeta mt}(T))\widetilde{\bar{Y}}_{mt}(T,\zeta,x)\dif m(\zeta)
\end{align}
\begin{align}\nonumber
&+\int_{s}^{T}\widetilde{\bb{E}}D_1\dod[2]{F}{m}(Y_{\cdot m't}(\tau) \otimes  m')(Y_{\xi m't}(\tau),\widetilde{Y}_{xm't}(\tau))\dif \tau+\widetilde{\bb{E}}D_1\dod[2]{F_T}{m}(Y_{\cdot m't}(T) \otimes  m')(Y_{\xi m't}(T),\widetilde{Y}_{xm't}(T))\\
&- \left.\left.\int_{s}^{T}\widetilde{\bb{E}}D_1\dod[2]{F}{m}(Y_{\cdot mt}(\tau) \otimes  m)(Y_{\xi mt}(\tau),\widetilde{Y}_{xmt}(\tau))\dif \tau-\widetilde{\bb{E}}D_1\dod[2]{F_T}{m}(Y_{\cdot mt}(T) \otimes  m)(Y_{\xi mt}(T),\widetilde{Y}_{xmt}(T))\right|\mathcal{W}_{t}^{s}\right].
\end{align}
Before we proceed further, we note that, 
\begin{align*}
&D^{2}\dod{F}{m}(Y_{\cdot m' t}(\tau) \otimes  m')(Y_{\xi m't}(\tau))\bar{Y}_{m't}(\tau,\xi,x)-D^{2}\dod{F}{m}(Y_{\cdot mt}(\tau) \otimes  m)(Y_{\xi mt}(\tau))\bar{Y}_{mt}(\tau,\xi,x)\\
=&D^{2}\dod{F}{m}(Y_{\cdot m' t}(\tau) \otimes  m')(Y_{\xi m't}(\tau))\bar{Y}_{m't}(\tau,\xi,x)-D^{2}\dod{F}{m}(Y_{\cdot m't}(\tau) \otimes  m')(Y_{\xi m't}(\tau))\bar{Y}_{mt}(\tau,\xi,x)\\
&+D^{2}\dod{F}{m}(Y_{\cdot m' t}(\tau) \otimes  m')(Y_{\xi m't}(\tau))\bar{Y}_{mt}(\tau,\xi,x)-D^{2}\dod{F}{m}(Y_{\cdot mt}(\tau) \otimes  m)(Y_{\xi mt}(\tau))\bar{Y}_{mt}(\tau,\xi,x)
\end{align*}
and
\begin{align*}
&\int_{\bb{R}^n}D_2D_1\dod[2]{F}{m}(Y_{\cdot m't}(\tau) \otimes  m')(Y_{\xi m't}(\tau),\widetilde{Y}_{\zeta m't}(\tau))\widetilde{\bar{Y}}_{m't}(\tau,\zeta,x)\dif m'(\zeta)\\
&-\int_{\bb{R}^n}D_2D_1\dod[2]{F}{m}(Y_{\cdot mt}(\tau) \otimes  m)(Y_{\xi mt}(\tau),\widetilde{Y}_{\zeta mt}(\tau))\widetilde{\bar{Y}}_{mt}(\tau,\zeta,x)\dif m(\zeta)\\
=&\int_{\bb{R}^n}D_2D_1\dod[2]{F}{m}(Y_{\cdot m't}(\tau) \otimes  m')(Y_{\xi m't}(\tau),\widetilde{Y}_{\zeta m't}(\tau))\widetilde{\bar{Y}}_{m't}(\tau,\zeta,x)\dif m'(\zeta)\\
&-\int_{\bb{R}^n}D_2D_1\dod[2]{F}{m}(Y_{\cdot m't}(\tau) \otimes  m')(Y_{\xi m't}(\tau),\widetilde{Y}_{\zeta m't}(\tau))\widetilde{\bar{Y}}_{m't}(\tau,\zeta,x)\dif m(\zeta)\\
&+\int_{\bb{R}^n}D_2D_1\dod[2]{F}{m}(Y_{\cdot m't}(\tau) \otimes  m')(Y_{\xi m't}(\tau),\widetilde{Y}_{\zeta m't}(\tau))\widetilde{\bar{Y}}_{m't}(\tau,\zeta,x)\dif m(\zeta)\\
&-\int_{\bb{R}^n}D_2D_1\dod[2]{F}{m}(Y_{\cdot m't}(\tau) \otimes  m')(Y_{\xi m't}(\tau),\widetilde{Y}_{\zeta m't}(\tau))\widetilde{\bar{Y}}_{mt}(\tau,\zeta,x)\dif m(\zeta)\\
&+\int_{\bb{R}^n}D_2D_1\dod[2]{F}{m}(Y_{\cdot m't}(\tau) \otimes  m')(Y_{\xi m't}(\tau),\widetilde{Y}_{\zeta m't}(\tau))\widetilde{\bar{Y}}_{mt}(\tau,\zeta,x)\dif m(\zeta)\\
&-\int_{\bb{R}^n}D_2D_1\dod[2]{F}{m}(Y_{\cdot mt}(\tau) \otimes  m)(Y_{\xi mt}(\tau),\widetilde{Y}_{\zeta mt}(\tau))\widetilde{\bar{Y}}_{mt}(\tau,\zeta,x)\dif m(\zeta).
\end{align*}
By Assumption $(i)$ and $(ii)$,
\begin{align*}
&\left|D^{2}\dod{F}{m}(Y_{\cdot m' t}(\tau) \otimes  m')(Y_{\xi m't}(\tau)) -D^{2}\dod{F}{m}(Y_{\cdot mt}(\tau) \otimes  m)(Y_{\xi mt}(\tau)) \right|\\
\leq& C\left(W_2(Y_{\cdot m' t}(\tau) \otimes  m',Y_{\cdot m' t}(\tau) \otimes  m')+\left|Y_{\xi m't}(\tau)-Y_{\xi mt}(\tau)\right|\right), 
\end{align*}
and
\begin{align*}
&\left|D_2D_1\dod[2]{F}{m}(Y_{\cdot m't}(\tau) \otimes  m')(Y_{\xi m't}(\tau),\widetilde{Y}_{\zeta m't}(\tau))-D_2D_1\dod[2]{F}{m}(Y_{\cdot mt}(\tau) \otimes  m)(Y_{\xi mt}(\tau),\widetilde{Y}_{\zeta mt}(\tau))\right|\\
\leq& C\left(W_2(Y_{\cdot m' t}(\tau) \otimes  m',Y_{\cdot m' t}(\tau) \otimes  m')+\left|Y_{\xi m't}(\tau)-Y_{\xi mt}(\tau)\right|+\left|\widetilde{Y}_{\zeta m't}(\tau)-\widetilde{Y}_{\zeta mt}(\tau)\right|\right),
\end{align*}
where $C$ is a constant such that, without the cause of much confusion, those other $C$'s in the rest of this proof portion may differ from each other line by line in order to save more spaces.
By \eqref{eq:a priori1}, \eqref{convexF}, \eqref{Lipsm}, we also have \small
\begin{align*}
&\bb{E}\int_{\bb{R}^n}\left(\bar{Y}_{m't}(s,\xi,x)-\bar{Y}_{mt}(s,\xi,x)\right)\\
&\cdot\bigg(D^{2}\dod{F}{m}(Y_{\cdot m' t}(s) \otimes  m')(Y_{\xi m't}(s))\bar{Y}_{m't}(s,\xi,x)-D^{2}\dod{F}{m}(Y_{\cdot mt}(s) \otimes  m)(Y_{\xi mt}(s))\bar{Y}_{mt}(s,\xi,x)\\
&\ \ \ \ +\widetilde{\bb{E}} \int_{\bb{R}^n}D_2D_1\dod[2]{F}{m}(Y_{\cdot m't}(s) \otimes  m')(Y_{\xi m't}(s),\widetilde{Y}_{\zeta m't}(s))\widetilde{\bar{Y}}_{m't}(s,\zeta,x)\dif m'(\zeta)  \\
\nonumber
&\ \ \ \ -\widetilde{\bb{E}} \int_{\bb{R}^n}D_2D_1\dod[2]{F}{m}(Y_{\cdot mt}(s) \otimes  m)(Y_{\xi mt}(s),\widetilde{Y}_{\zeta mt}(s))\widetilde{\bar{Y}}_{mt}(s,\zeta,x)\dif m(\zeta) \\
&\ \ \ \ +\widetilde{\bb{E}}D_1\dod[2]{F}{m}(Y_{\cdot m't}(s) \otimes  m')(Y_{\xi m't}(s),\widetilde{Y}_{xm't}(s))-\widetilde{\bb{E}}D_1\dod[2]{F}{m}(Y_{\cdot mt}(s) \otimes  m)(Y_{\xi mt}(s),\widetilde{Y}_{xmt}(s))\bigg)dm(\xi)\\
\geq&\bb{E}\int_{\bb{R}^n}\left(\left(\bar{Y}_{m't}(s,\xi,x)-\bar{Y}_{mt}(s,\xi,x)\right)\cdot\widetilde{\bb{E}} \int_{\bb{R}^n}D_2D_1\dod[2]{F}{m}(Y_{\cdot m't}(s) \otimes  m')(Y_{\xi m't}(s),\widetilde{Y}_{\zeta m't}(s))\widetilde{\bar{Y}}_{m't}(s,\zeta,x)\dif \h{1pt}(m'-m)(\zeta)\right)dm(\xi)\\
&+\lambda_F\bb{E}\int_{\bb{R}^n}\left|\bar{Y}_{m't}(s,\xi,x)-\bar{Y}_{mt}(s,\xi,x)\right|^2dm(\xi)\\
&-C\sqrt{C_T(1+|x|^2)}\left(\bb{E}\int_{\bb{R}^n}\left|\bar{Y}_{m't}(s,\xi,x)-\bar{Y}_{mt}(s,\xi,x)\right|^2dm(\xi)\right)^{1/2}\bigg(\left(\bb{E}W_2^2(Y_{\cdot m' t}(s) \otimes  m',Y_{\cdot m' t}(s) \otimes  m')\right)^{1/2}\\
&\ \ \ \ \ \ \ \ + \left(\bb{E}\int_{\bb{R}^n}\left|Y_{\xi m't}(s)-Y_{\xi mt}(s)\right|^2dm(\xi)\right)^{1/2}\bigg). 
\end{align*}\normalsize
By \eqref{eq:6-15}, \eqref{eq:a priori1}, \eqref{xLipCon} and \eqref{LipsYxi}, \small
\begin{align*}
&\left|\widetilde{\bb{E}} \int_{\bb{R}^n}D_2D_1\dod[2]{F}{m}(Y_{\cdot m't}(s) \otimes  m')(Y_{\xi m't}(s),\widetilde{Y}_{\zeta m't}(s))\widetilde{\bar{Y}}_{m't}(s,\zeta,x)\dif \h{1pt}(m'-m)(\zeta)\right|^2\\
=&\left|\widetilde{\bb{E}} \int_{\widehat{\Omega}}D_2D_1\dod[2]{F}{m}(Y_{\cdot m't}(s) \otimes  m')(Y_{\xi m't}(s),\widetilde{Y}_{\hat{X}_{m'}(\widehat{\omega}) m't}(s))\widetilde{\bar{Y}}_{m't}(s,\hat{X}_{m'}(\widehat{\omega}),x)\dif \widehat{\bb{P}}(\widehat{\omega})\right.\\
&\left.-\widetilde{\bb{E}} \int_{\widehat{\Omega}}D_2D_1\dod[2]{F}{m}(Y_{\cdot m't}(s) \otimes  m')(Y_{\xi m't}(s),\widetilde{Y}_{\hat{X}_{m}(\widehat{\omega}) m't}(s))\widetilde{\bar{Y}}_{m't}(s,\hat{X}_{m}(\widehat{\omega}),x)\dif \widehat{\bb{P}}(\widehat{\omega})\right|^2\\
\leq& 2c^2\widetilde{\bb{E}}\int_{\widehat{\Omega}}\left|\widetilde{\bar{Y}}_{m't}(s,\hat{X}_{m'}(\widehat{\omega}),x)-\widetilde{\bar{Y}}_{m't}(s,\hat{X}_{m}(\widehat{\omega}),x)\right|^2d\widehat{\bb{P}}(\widehat{\omega})+2c^2 C_T(1+|x|^2) \cdot\widetilde{\bb{E}}\int_{\widehat{\Omega}}\left|\widetilde{Y}_{\hat{X}_{m'}(\widehat{\omega}) m't}(s)-\widetilde{Y}_{\hat{X}_{m}(\widehat{\omega}) m't}(s)\right|^2d\widehat{\bb{P}}(\widehat{\omega})\\
\leq& C(c,c_T,\lambda,T) C_T(1+|x|^2) \cdot\int_{\widehat{\Omega}}\left|\hat{X}_{m'}(\widehat{\omega})-\hat{X}_{m}(\widehat{\omega})\right|^2d\widehat{\bb{P}}(\widehat{\omega})\\
=& C(c,c_T,\lambda,T) C_T(1+|x|^2) W_2^2(m,m'),
\end{align*}\normalsize
where $\hat{X}_{m}$ and $\hat{X}_{m'}$ are random variables in $L^{2}(\widehat{\Omega},\widehat{\mathcal{A}},\widehat{\bb{P}};\bb{R}^{n})$ such that 
\begin{equation}
W_{2}^{2}(m,m')=\bb{E}[|\hat{X}_{m}-\hat{X}_{m'}|^{2}].
\end{equation}
Therefore, for the running cost term,
\small
\begin{align}\nonumber
&\bb{E}\int_{\bb{R}^n}\left(\bar{Y}_{m't}(s,\xi,x)-\bar{Y}_{mt}(s,\xi,x)\right)\\\nonumber
&\cdot\bigg(D^{2}\dod{F}{m}(Y_{\cdot m' t}(s) \otimes  m')(Y_{\xi m't}(s))\bar{Y}_{m't}(s,\xi,x)-D^{2}\dod{F}{m}(Y_{\cdot mt}(s) \otimes  m)(Y_{\xi mt}(s))\bar{Y}_{mt}(s,\xi,x)\\\nonumber
&\ \ \ \ +\widetilde{\bb{E}} \int_{\bb{R}^n}D_2D_1\dod[2]{F}{m}(Y_{\cdot m't}(s) \otimes  m')(Y_{\xi m't}(s),\widetilde{Y}_{\zeta m't}(s))\widetilde{\bar{Y}}_{m't}(s,\zeta,x)\dif m'(\zeta)  \\\nonumber
&\ \ \ \ -\widetilde{\bb{E}} \int_{\bb{R}^n}D_2D_1\dod[2]{F}{m}(Y_{\cdot mt}(s) \otimes  m)(Y_{\xi mt}(s),\widetilde{Y}_{\zeta mt}(s))\widetilde{\bar{Y}}_{mt}(s,\zeta,x)\dif m(\zeta) \\\nonumber
&\ \ \ \ +\widetilde{\bb{E}}D_1\dod[2]{F}{m}(Y_{\cdot m't}(s) \otimes  m')(Y_{\xi m't}(s),\widetilde{Y}_{xm't}(s))-\widetilde{\bb{E}}D_1\dod[2]{F}{m}(Y_{\cdot mt}(s) \otimes  m)(Y_{\xi mt}(s),\widetilde{Y}_{xmt}(s))\bigg)dm(\xi)
\end{align}
\begin{align}\nonumber
\geq&\lambda_F\bb{E}\int_{\bb{R}^n}\left|\bar{Y}_{m't}(s,\xi,x)-\bar{Y}_{mt}(s,\xi,x)\right|^2dm(\xi)\\\nonumber
&-C(c,c_T,\lambda,T)\sqrt{C_T(1+|x|^2)}\left(\bb{E}\int_{\bb{R}^n}\left|\bar{Y}_{m't}(s,\xi,x)-\bar{Y}_{mt}(s,\xi,x)\right|^2dm(\xi)\right)^{1/2}\cdot\bigg(\left(\bb{E}W_2^2(Y_{\cdot m' t}(s) \otimes  m',Y_{\cdot m' t}(s) \otimes  m')\right)^{1/2}\\\label{eq6-45}
&\ \ \ \ \ \ \ \ + \left(\bb{E}\int_{\bb{R}^n}\left|Y_{\xi m't}(s)-Y_{\xi mt}(s)\right|^2dm(\xi)\right)^{1/2}+W_2(m,m')\bigg). 
\end{align}\normalsize
By the same argument, one can also obtain for the terminal term,
\small
\begin{align}\nonumber
&\bb{E}\int_{\bb{R}^n}\left(\bar{Y}_{m't}(T,\xi,x)-\bar{Y}_{mt}(T,\xi,x)\right)\\\nonumber
&\cdot\bigg(D^{2}\dod{F_T}{m}(Y_{\cdot m' t}(T) \otimes  m')(Y_{\xi m't}(T))\bar{Y}_{m't}(T,\xi,x)-D^{2}\dod{F_T}{m}(Y_{\cdot mt}(T) \otimes  m)(Y_{\xi mt}(T))\bar{Y}_{mt}(T,\xi,x)\\\nonumber
&\ \ \ \ +\widetilde{\bb{E}} \int_{\bb{R}^n}D_2D_1\dod[2]{F_T}{m}(Y_{\cdot m't}(T) \otimes  m')(Y_{\xi m't}(T),\widetilde{Y}_{\zeta m't}(T))\widetilde{\bar{Y}}_{m't}(T,\eta,x)\dif m'(\zeta)  \\
\nonumber
&\ \ \ \ -\widetilde{\bb{E}} \int_{\bb{R}^n}D_2D_1\dod[2]{F_T}{m}(Y_{\cdot mt}(T) \otimes  m)(Y_{\xi mt}(T),\widetilde{Y}_{\zeta mt}(T))\widetilde{\bar{Y}}_{mt}(T,\zeta,x)\dif m(\zeta) \\\nonumber
&\ \ \ \ +\widetilde{\bb{E}}D_1\dod[2]{F_T}{m}(Y_{\cdot m't}(T) \otimes  m')(Y_{\xi m't}(T),\widetilde{Y}_{xm't}(T))-\widetilde{\bb{E}}D_1\dod[2]{F_T}{m}(Y_{\cdot mt}(T) \otimes  m)(Y_{\xi mt}(T),\widetilde{Y}_{xmt}(T))\bigg)dm(\xi)\\\nonumber
\geq&\lambda_F\bb{E}\int_{\bb{R}^n}\left|\bar{Y}_{m't}(T,\xi,x)-\bar{Y}_{mt}(T,\xi,x)\right|^2dm(\xi)\\\nonumber
&-C(c,c_T,\lambda,T)\sqrt{C_T(1+|x|^2)}\left(\bb{E}\int_{\bb{R}^n}\left|\bar{Y}_{m't}(T,\xi,x)-\bar{Y}_{mt}(T,\xi,x)\right|^2dm(\xi)\right)^{1/2}\cdot\bigg(\left(\bb{E}W_2^2(Y_{\cdot m' t}(T) \otimes  m',Y_{\cdot m' t}(T) \otimes  m')\right)^{1/2}\\\label{eq6-46}
&\ \ \ \ \ \ \ \ + \left(\bb{E}\int_{\bb{R}^n}\left|Y_{\xi m't}(T)-Y_{\xi mt}(T)\right|^2dm(\xi)\right)^{1/2}+W_2(m,m')\bigg). 
\end{align}\normalsize
By \eqref{eq6-45},\small
\begin{align*}
&\bb{E}\int_{\bb{R}^n}\left(\bar{Y}_{m't}(T,\xi,x)-\bar{Y}_{mt}(T,\xi,x)\right)\cdot \left(\bar{Z}_{m't}(T,\xi,x)-\bar{Z}_{mt}(T,\xi,x)\right) dm(\xi)\\
&-\bb{E}\int_{\bb{R}^n}\left(\bar{Y}_{m't}(t,\xi,x)-\bar{Y}_{mt}(t,\xi,x)\right)\cdot \left(\bar{Z}_{m't}(t,\xi,x)-\bar{Z}_{mt}(t,\xi,x)\right) dm(\xi)\\
=&\int_t^T d_s\left(\bb{E}\int_{\bb{R}^n}\left(\bar{Y}_{m't}(s,\xi,x)-\bar{Y}_{mt}(s,\xi,x)\right)\cdot \left(\bar{Z}_{m't}(s,\xi,x)-\bar{Z}_{mt}(s,\xi,x)\right) dm(\xi)\right)\\
=&\int_t^T \bb{E}\int_{\bb{R}^n}d_s\left(\bar{Y}_{m't}(s,\xi,x)-\bar{Y}_{mt}(s,\xi,x)\right)\cdot \left(\bar{Z}_{m't}(s,\xi,x)-\bar{Z}_{mt}(s,\xi,x)\right) dm(\xi)\\
&+\int_t^T \bb{E}\int_{\bb{R}^n}\left(\bar{Y}_{m't}(s,\xi,x)-\bar{Y}_{mt}(s,\xi,x)\right)\cdot d_s\left(\bar{Z}_{m't}(s,\xi,x)-\bar{Z}_{mt}(s,\xi,x)\right) dm(\xi)\\
\leq &\int_t^T \bigg[-\frac{1}{\lambda} \bb{E}\int_{\bb{R}^n} \left|\bar{Z}_{m't}(s,\xi,x)-\bar{Z}_{mt}(s,\xi,x)\right|^2 dm(\xi)-\lambda_F\bb{E} \int_{\bb{R}^n}\left|\bar{Y}_{m't}(s,\xi,x)-\bar{Y}_{mt}(s,\xi,x)\right|^2dm(\xi)\\
&\ \ +C(c,c_T,\lambda,T)\sqrt{C_T(1+|x|^2)}\cdot\left(\bb{E}\int_{\bb{R}^n}\left|\bar{Y}_{m't}(s,\xi,x)-\bar{Y}_{mt}(s,\xi,x)\right|^2dm(\xi)\right)^{1/2}\cdot\bigg(\left(\bb{E}W_2^2(Y_{\cdot m' t}(s) \otimes  m',Y_{\cdot m' t}(s) \otimes  m')\right)^{1/2}\\
&\ \ + \left(\bb{E}\int_{\bb{R}^n}\left|Y_{\xi m't}(s)-Y_{\xi mt}(s)\right|^2dm(\xi)\right)^{1/2}+W_2(m,m')\bigg)\bigg]ds. 
\end{align*}\normalsize
Thus, by using \eqref{Lipsm} and \eqref{LipsYm} for the last term on the right hand side, and then using \eqref{eq6-46} for the terminal cross term on the left hand side by also noting that $\bar{Y}_{m't}(t,\xi,x)=\bar{Y}_{mt}(t,\xi,x)=0$, we then obtain: \\
\begingroup
\allowdisplaybreaks
\small
\begin{align*}
&\frac{1}{\lambda}\int_t^T \bb{E}\int_{\bb{R}^n} \left|\bar{Z}_{m't}(s,\xi,x)-\bar{Z}_{mt}(s,\xi,x)\right|^2 dm(\xi) ds+\lambda_F\int_t^T \bb{E}\int_{\bb{R}^n} \left|\bar{Y}_{m't}(s,\xi,x)-\bar{Y}_{mt}(s,\xi,x)\right|^2 dm(\xi) ds\\
&+\lambda_F\bb{E}\int_{\bb{R}^n} \left|\bar{Y}_{m't}(T,\xi,x)-\bar{Y}_{mt}(T,\xi,x)\right|^2 dm(\xi)\\
\leq &C(c,c_T,\lambda,T)\sqrt{C_T(1+|x|^2)}\left(\int_t^T\bb{E}\int_{\bb{R}^n}\left|\bar{Y}_{m't}(s,\xi,x)-\bar{Y}_{mt}(s,\xi,x)\right|^2dm(\xi)ds\right)^{1/2}\\
&\cdot\Bigg(\left(\int_t^T \bb{E}W_2^2(Y_{\cdot m' t}(s) \otimes  m',Y_{\cdot m' t}(s) \otimes  m')ds\right)^{1/2}+ \left(\int_t^T \bb{E}\int_{\bb{R}^n}\left|Y_{\xi m't}(s)-Y_{\xi mt}(s)\right|^2dm(\xi)ds\right)^{1/2}+\sqrt{T}W_2(m,m')\Bigg)\\
&+C(c,c_T,\lambda,T)\sqrt{C_T(1+|x|^2)}\left(\bb{E}\int_{\bb{R}^n}\left|\bar{Y}_{m't}(T,\xi,x)-\bar{Y}_{mt}(T,\xi,x)\right|^2dm(\xi)\right)^{1/2}\cdot\bigg(\left(\bb{E}W_2^2(Y_{\cdot m' t}(T) \otimes  m',Y_{\cdot m' t}(T) \otimes  m')\right)^{1/2}\\
&\ \ \ \ \ \ \ \ + \left(\bb{E}\int_{\bb{R}^n}\left|Y_{\xi m't}(T)-Y_{\xi mt}(T)\right|^2dm(\xi)\right)^{1/2}+W_2(m,m')\bigg).\\
\leq&C(c,c_T,\lambda,\lambda_F,T)\sqrt{C_T(1+|x|^2)}W_2(m,m')\cdot\left(\int_t^T\bb{E}\int_{\bb{R}^n}\left|\bar{Y}_{m't}(s,\xi,x)-\bar{Y}_{mt}(s,\xi,x)\right|^2dm(\xi)ds\right.\\
&\left.+\bb{E}\int_{\bb{R}^n}\left|\bar{Y}_{m't}(T,\xi,x)-\bar{Y}_{mt}(T,\xi,x)\right|^2dm(\xi)\right)^{1/2}
\end{align*}\normalsize\endgroup
Therefore, we finally arrive with
\begin{align}\nonumber
&\int_t^T \bb{E}\int_{\bb{R}^n} \left|\bar{Z}_{m't}(s,\xi,x)-\bar{Z}_{mt}(s,\xi,x)\right|^2 dm(\xi) ds+\int_t^T \bb{E}\int_{\bb{R}^n} \left|\bar{Y}_{m't}(s,\xi,x)-\bar{Y}_{mt}(s,\xi,x)\right|^2 dm(\xi) ds\\\nonumber
&+\bb{E}\int_{\bb{R}^n} \left|\bar{Y}_{m't}(T,\xi,x)-\bar{Y}_{mt}(T,\xi,x)\right|^2 dm(\xi)\\
\leq &C(c,c_T,\lambda,\lambda_F,T)\sqrt{C_T(1+|x|^2)}W_2^2(m,m'),
\end{align}
for some positive constant $C$ depending only on $c$, $c_T$, $\lambda$, $\lambda_F$ and $T$. $\blacksquare$

\subsection{PROOF OF PROPOSITION \ref{prop6-1}. }
First we establish (\ref{eq:6-16}) and (\ref{eq:6-17}) as a priori estimates.
We begin with (\ref{eq:6-11})-(\ref{eq:6-12}).
Note that (\ref{eq:4-90}) holds for sufficiently large $\lambda$.
Using \eqref{eq:4-90} and the estimates \eqref{eq:6-15} applied to both $F$ and $F_T$ in (\ref{eq:6-12}), we deduce
\begin{multline*}
\bb{E}\int_{\bb{R}^n}|\bar{Z}_{mt}(s,\xi,x)|^{2}\dif m(\xi)\\
\leq C_{T}\left(\int_{s}^{T}\bb{E}\int_{\bb{R}^n}|\bar{Y}_{mt}(\tau,\xi,x)|^{2}\dif m(\xi)\dif \tau+\bb{E}\int_{\bb{R}^n}|\bar{Y}_{mt}(T,\xi,x)|^{2}\dif m(\xi)\right)
+C_{T}(1+|x|^{2}).
\end{multline*}
On the other hand, by (\ref{eq:6-11}) it follows from the Cauchy-Schwartz inequality that
\[
\bb{E}\int_{\bb{R}^n}\abs{\bar{Y}_{mt}(s,\xi,x)}^{2}\dif m(\xi)\leq\dfrac{s-t}{\lambda^{2}}\int_{t}^{s}\bb{E}\int_{\bb{R}^n}\abs{\bar{Z}_{mt}(\tau,\xi,x)}^{2}\dif m(\xi).
\]
Combining these two inequalities and taking $\lambda$ sufficiently
large, we obtain the estimates (\ref{eq:6-16}). For (\ref{eq:6-17}),
we use the estimates \eqref{eq:6-15} for both $F$ and $F_T$ in \eqref{eq:calYxmt}-\eqref{eq:calZxmt} to see that the solution $\del{\mathcal{Y}_{\xi mt}(s), \mathcal{Z}_{\xi mt}(s)}$ satisfies 
\begin{equation}
\bb{E}|\mathcal{Y}_{xmt}(s)|^{2}\leq C_{T},\;\bb{E}|\mathcal{Z}_{xmt}(s)|^{2}\leq C_{T}.\label{eq:6-19}
\end{equation}
Using these estimates, as well as \eqref{eq:4-90}
and (\ref{eq:6-16}), in the system (\ref{eq:6-13})- (\ref{eq:6-14}), we argue similarly
as for (\ref{eq:6-11})-(\ref{eq:6-12}) to see that (\ref{eq:6-17}) holds
for $\lambda$ sufficiently large.

We have thus obtained a priori
estimates for the solutions of (\ref{eq:6-11})-(\ref{eq:6-12})
and (\ref{eq:6-13})-(\ref{eq:6-14}). Since these systems are linear,
the existence and uniqueness of the solutions are obtained by {\color{black}the following standard Banach
fixed point argument. }

{\color{black}
First, denote $\s{D}:=C([t,T]\times\mathbb{R}^n\times\mathbb{R}^n;\mathbb{R}^n)$, the space of $\mathbb{R}^n$-valued continuous functions on $[t,T]$, and we equip $\s{D}$ with following norm:
$$\|A\|_{\s{D}}:=\sup_{s\in[t,T]}\sup_{x\in\mathbb{R}^n}\frac{\mathbb{E}\int_{\mathbb{R}^n}|A(s,\xi,x)|^2dm(\xi)}{1+|x|^2}.$$
For a fixed $t\in[0,T]$ and $m\in\mathcal{P}_2(\mathbb{R}^n)$, also define the iteration metric space,
\begin{align*}
I_{mt}:=\bigg\{&\del{\bar{Y}_{mt}(s,\xi,x),\:\bar{Z}_{mt}(s,\xi,x)}:\bar{Y},\bar{Z}\in\s{D},\ \left\|\bar{Y}_{mt}(s,\xi,x)\right\|_{\s{D}}\leq C_Y,\ \left\|\bar{Z}_{mt}(s,\xi,x)\right\|_{\s{D}}\leq C_Z\bigg\},
\end{align*}
where 
$C_Y$ and $C_Z$ are positive constants to be set later. Next, we define the iteration map $$\del{\bar{Y}_{mt}^{(i)}(s,\xi,x),\:\bar{Z}_{mt}^{(i)}(s,\xi,x)}\in I_{mt}\mapsto \del{\bar{Y}_{mt}^{(i+1)}(s,\xi,x),\:\bar{Z}_{mt}^{(i+1)}(s,\xi,x)}\in \s{D}$$ such that\small
\begin{align}\label{6-18}
\bar{Y}_{mt}^{(i+1)}(s,\xi,x) =&-\dfrac{1}{\lambda}\int_{t}^{s}\bar{Z}_{mt}^{(i)}(\tau,\xi,x)\dif \tau,\\\label{6-19}
\bar{Z}_{mt}^{(i+1)}(s,\xi,x) =&\bb{E}\left[ \int_{s}^{T} D^{2}\dod{F}{m}(Y_{\cdot mt}(\tau) \otimes  m)(Y_{\xi mt}(\tau))\bar{Y}_{mt}^{(i+1)}(\tau,\xi,x)\dif \tau \right.\\
\nonumber
&
+ D^{2}\dod{}{m}F_{T}(Y_{\cdot mt}(T) \otimes  m)(Y_{\xi mt}(T))\bar{Y}_{mt}^{(i+1)}(T,\xi,x)\\
\nonumber
&
+\widetilde{\bb{E}}\int_{s}^{T}\int_{\bb{R}^n}D_2D_1\dod[2]{F}{m}(Y_{\cdot mt}(\tau) \otimes  m)(Y_{\xi mt}(\tau),\widetilde{Y}_{\zeta mt}(\tau))\widetilde{\bar{Y}}_{mt}^{(i+1)}(\tau,\eta,x)\dif m(\zeta) \dif \tau\\
\nonumber
&
+\widetilde{\bb{E}}\int_{\bb{R}^n}D_2D_1\dod[2]{F_{T}}{m}(Y_{\cdot mt}(T) \otimes  m)(Y_{\xi mt}(T),\widetilde{Y}_{\zeta mt}(T))\widetilde{\bar{Y}}_{mt}^{(i+1)}(T,\eta,x)\dif m(\zeta)\\
\nonumber
&
+ \left.\left.\int_{s}^{T}\widetilde{\bb{E}}D_1\dod[2]{F}{m}(Y_{\cdot mt}(\tau) \otimes  m)(Y_{\xi mt}(\tau),\widetilde{Y}_{xmt}(\tau))\dif \tau+\widetilde{\bb{E}}D_1\dod[2]{F_T}{m}(Y_{\cdot mt}(T) \otimes  m)(Y_{\xi mt}(T),\widetilde{Y}_{xmt}(T))\right|\mathcal{W}_{t}^{s}\right],
\end{align}\normalsize
in accordance with the equations \eqref{eq:6-11}-\eqref{eq:6-12}. To facilitate the use the Banach Fixed Point Theorem, we shall next show that the iteration map is self-map and contractive as follows.

By \eqref{6-18}, the Cauchy-Schwartz inequality, \eqref{eq:4-90} and the estimates \eqref{eq:6-15} applied to both $F$ and $F_T$ in \eqref{6-19}, one has
\begin{align}\label{6-20}
\begin{cases}
\|\bar{Y}_{mt}^{(i+1)}\|_{\s{D}}\leq \frac{|T-t|^2}{\lambda^2} \|\bar{Z}_{mt}^{(i)}\|_{\s{D}}\leq \frac{|T|^2}{\lambda^2} C_Z,\\
\|\bar{Z}_{mt}^{(i+1)}\|_{\s{D}}
\leq C_{T}(1+\|\bar{Y}_{mt}^{(i+1)}\|_{\s{D}})\leq C_{T}\left(1+\frac{|T|^2}{\lambda^2} C_Z\right),
\end{cases}
\end{align}
where $C_T$ is a positive constant depending on $c,c_T,$ and $T$, but it does not depend on $C_Y$ and $C_Z$; indeed, by \eqref{eq:6-15} and \eqref{eq:4-90}, $\big|\widetilde{\bb{E}}D_1\dod[2]{F}{m}(Y_{\cdot mt}(\tau) \otimes  m)(Y_{\xi mt}(\tau),\widetilde{Y}_{xmt}(\tau))\big|\leq c \widetilde{\bb{E}}(1+|\widetilde{Y}_{xmt}(\tau)|)\leq c C_T(1+|x|^2)$, and so, for instance, the second last term of \eqref{6-19} can be bounded by:
\begin{align*}
\left\|\bb{E}\left[\int_{s}^{T}\widetilde{\bb{E}}D_1\dod[2]{F}{m}(Y_{\cdot mt}(\tau) \otimes  m)(Y_{\xi mt}(\tau),\widetilde{Y}_{xmt}(\tau))\dif \tau\big|\mathcal{W}_{t}^{s}\right]\right\|_{\s{D}}\leq c C_T T,
\end{align*}
where $C_T$ is the constant used in the estimate \eqref{eq:4-90}. 
If we take $C_Y\geq \frac{|T|^2}{\lambda^2} C_Z$, $C_Z\geq \frac{C_T}{1-C_T T^2/\lambda^2}$ and $\lambda>T\sqrt{C_T}$, then $\del{\bar{Y}_{mt}^{(i+1)}(s,\xi,x),\:\bar{Z}_{mt}^{(i+1)}(s,\xi,x)}\in I_{mt}$ and so the iteration map defined by \eqref{6-18}-\eqref{6-19} is a self-map. Next, we shall show that the iteration mapping is also contractive. Particularly,
\begingroup
\allowdisplaybreaks
\small
\begin{align}\nonumber
&\bar{Y}_{mt}^{(i+2)}(s,\xi,x)-\bar{Y}_{mt}^{(i+1)}(s,\xi,x) =-\dfrac{1}{\lambda}\int_{t}^{s}\left(\bar{Z}_{mt}^{(i+1)}(\tau,\xi,x)-\bar{Z}_{mt}^{(i)}(\tau,\xi,x)\right)\dif \tau,\\\nonumber
&\bar{Z}_{mt}^{(i+2)}-\bar{Z}_{mt}^{(i+1)}(s,\xi,x) =\bb{E}\left[ \int_{s}^{T} D^{2}\dod{F}{m}(Y_{\cdot mt}(\tau) \otimes  m)(Y_{\xi mt}(\tau))\left(\bar{Y}_{mt}^{(i+2)}(\tau,\xi,x)-\bar{Y}_{mt}^{(i+1)}(\tau,\xi,x)\right)\dif \tau \right.\\
\nonumber
&\hspace{3.5cm}
+ D^{2}\dod{}{m}F_{T}(Y_{\cdot mt}(T) \otimes  m)(Y_{\xi mt}(T))\left(\bar{Y}_{mt}^{(i+2)}(T,\xi,x)-\bar{Y}_{mt}^{(i+1)}(T,\xi,x)\right)\\
\nonumber
&\hspace{3.5cm}+\widetilde{\bb{E}}\int_{s}^{T}\int_{\bb{R}^n}D_2D_1\dod[2]{F}{m}(Y_{\cdot mt}(\tau) \otimes  m)(Y_{\xi mt}(\tau),\widetilde{Y}_{\zeta mt}(\tau))\left(\widetilde{\bar{Y}}_{mt}^{(i+2)}(\tau,\eta,x)-\widetilde{\bar{Y}}_{mt}^{(i+1)}(\tau,\eta,x)\right)\dif m(\zeta) \dif \tau\\
\nonumber
&\hspace{3.5cm} \left.\left.+\widetilde{\bb{E}}\int_{\bb{R}^n}D_2D_1\dod[2]{F_{T}}{m}(Y_{\cdot mt}(T) \otimes  m)(Y_{\xi mt}(T),\widetilde{Y}_{\zeta mt}(T))\left(\widetilde{\bar{Y}}_{mt}^{(i+2)}(T,\eta,x)-\widetilde{\bar{Y}}_{mt}^{(i+1)}(T,\eta,x)\right)\dif m(\zeta)\right|\mathcal{W}_{t}^{s}\right],
\end{align}
\normalsize
\endgroup
by \eqref{6-18}, the Cauchy-Schwartz inequality, \eqref{eq:4-90} and the estimates \eqref{eq:6-15} applied to both $F$ and $F_T$ in \eqref{6-19}, similar as \eqref{6-20}, one can obtain
\begin{align*}\begin{cases}
\|\bar{Y}_{mt}^{(i+2)}-\bar{Y}_{mt}^{(i+1)}\|_{\s{D}}\leq& \frac{|T-t|^2}{\lambda^2} \|\bar{Z}_{mt}^{(i+1)}-\bar{Z}_{mt}^{(i)}\|_{\s{D}}\leq \frac{|T|^2}{\lambda^2} \|\bar{Z}_{mt}^{(i+1)}-\bar{Z}_{mt}^{(i)}\|_{\s{D}},\\
\|\bar{Z}_{mt}^{(i+2)}-\bar{Z}_{mt}^{(i+1)}\|_{\s{D}}
\leq& C_{T}\|\bar{Y}_{mt}^{(i+2)}-\bar{Y}_{mt}^{(i+1)}\|_{\s{D}}\leq C_T\frac{|T|^2}{\lambda^2} \|\bar{Z}_{mt}^{(i+1)}-\bar{Z}_{mt}^{(i)}\|_{\s{D}},
\end{cases}\end{align*}
and therefore, combing these two estimates,
\begin{align*}
\|\bar{Y}_{mt}^{(i+2)}-\bar{Y}_{mt}^{(i+1)}\|_{\s{D}}\leq& \frac{|T|^2}{\lambda^2} \|\bar{Z}_{mt}^{(i+1)}-\bar{Z}_{mt}^{(i)}\|_{\s{D}}\leq C_T\frac{|T|^2}{\lambda^2} \|\bar{Y}_{mt}^{(i+1)}-\bar{Y}_{mt}^{(i)}\|_{\s{D}}.
\end{align*}
By choosing $\lambda>T\sqrt{2C_T}$, then $\|\bar{Y}_{mt}^{(i+2)}-\bar{Y}_{mt}^{(i+1)}\|_{\s{D}}\leq \frac12 \|\bar{Y}_{mt}^{(i+1)}-\bar{Y}_{mt}^{(i)}\|_{\s{D}}$ and $\|\bar{Z}_{mt}^{(i+2)}-\bar{Z}_{mt}^{(i+1)}\|_{\s{D}}\leq \frac{1}{2}\|\bar{Z}_{mt}^{(i+1)}-\bar{Z}_{mt}^{(i)}\|_{\s{D}}$. Therefore, by Banach Fixed Point Theorem, we can obtain the existence and uniqueness.
}
 $\blacksquare$

\subsection{PROOF OF LEMMA \ref{lem 6.9}. }
$(i)$ The pair $\del{Y_{xmt}(s),Z_{xmt}(s)}$ is the unique solution of
\begin{align}
Y_{xmt}(s) &=x-\dfrac{1}{\lambda}\int_{t}^{s}Z_{xmt}(\tau)\dif \tau+\eta(w(s)-w(t)),\\
Z_{xmt}(s) &=\bb{E}\left[\left.\int_{s}^{T} D \dod{F}{m}(Y_{\cdot mt}(\tau) \otimes  m)(Y_{xmt}(\tau))\dif \tau+D \dod{F_T}{m}(Y_{\cdot mt}(T) \otimes  m)(Y_{xmt}(T))\right|\mathcal{W}_{t}^{s}\right].
\end{align}
Then, for $s\in[t,T]$,
\begin{align*}
Y_{x'mt}(s)-Y_{xmt}(s) =&x'-x-\dfrac{1}{\lambda}\int_{t}^{s}Z_{x'mt}(\tau)-Z_{xmt}(\tau)\dif \tau,\\
Z_{x'mt}(s)-Z_{xmt}(s) =&\bb{E}\left[\left.\int_{s}^{T} D \dod{F}{m}(Y_{\cdot mt}(\tau) \otimes  m)(Y_{x'mt}(\tau))-D \dod{F}{m}(Y_{\cdot mt}(\tau) \otimes  m)(Y_{xmt}(\tau))\dif \tau\right.\right.\\
&\left.\left.\ \ +D \dod{F_T}{m}(Y_{\cdot mt}(T) \otimes  m)(Y_{x'mt}(T))-D \dod{F_T}{m}(Y_{\cdot mt}(T) \otimes  m)(Y_{xmt}(T))\right|\mathcal{W}_{t}^{s}\right]\\
=&\bb{E}\left[\left.\int_{s}^{T} \int_0^1 D^2 \dod{F}{m}(Y_{\cdot mt}(\tau) \otimes  m)(\theta Y_{x'mt}(\tau)+(1-\theta)Y_{xmt}(\tau))(Y_{x'mt}(\tau)-Y_{xmt}(\tau))d\theta \dif \tau\right.\right.\\
&\left.\left.\ \ +\int_0^1 D^2 \dod{F_T}{m}(Y_{\cdot mt}(T) \otimes  m)(\theta Y_{x'mt}(T)+(1-\theta)Y_{xmt}(T))(Y_{x'mt}(T)-Y_{xmt}(T))d\theta\right|\mathcal{W}_{t}^{s}\right].
\end{align*}
In fact, here $Y_{\cdot mt}(\tau) \otimes  m$ depends solely on the initial distribution $m$, but it is totally independent of the two positions $x$ and $x'$, see Lemma 3.1 of \cite{BLPR}.
By \eqref{eq:6-15}, we have
\begin{align*}
\sup_{s\in[t,T]} \left|  Y_{x'mt}(s)-Y_{xmt}(s)\right| \leq &|x'-x|+\dfrac{1}{\lambda}|T-t|\cdot\sup_{s\in[t,T]}\left|Z_{x'mt}(s)-Z_{xmt}(s)\right|,\\
\sup_{s\in[t,T]} \left|Z_{x'mt}(s)-Z_{xmt}(s)\right| \leq &(c_T+c|T-t|)\cdot\sup_{s\in[t,T]}\left|Y_{x'mt}(s)-Y_{xmt}(s)\right|.
\end{align*}
Thus, by simply combining these two estimates, whenever $\lambda>T(c_T+cT)$, we arrive with:
\begin{align}\label{xLipCon}
\begin{cases}
\sup_{s\in[t,T]} \left|  Y_{x'mt}(s)-Y_{xmt}(s)\right| \leq &\frac{\lambda}{\lambda-|T-t|(c_T+c|T-t|)}|x'-x|\leq \frac{\lambda}{\lambda-T(c_T+cT)}|x'-x|,\\
\sup_{s\in[t,T]} \left|Z_{x'mt}(s)-Z_{xmt}(s)\right| \leq &\frac{\lambda (c_T+c|T-t|)}{\lambda-|T-t|(c_T+c|T-t|)}|x'-x|\leq \frac{\lambda (c_T+cT)}{\lambda-T(c_T+cT)}|x'-x|.
\end{cases}
\end{align}
Therefore, $Y_{xmt}(s)$ and $Z_{xmt}(s)$ are both Lipschitz continuous in $x$, particularly being uniformly in $m\in\s{P}_2(\bb{R}^n)$, with the respective Lipschitz constants $\frac{\lambda}{\lambda-T(c_T+cT)}$ and $\frac{\lambda (c_T+cT)}{\lambda-T(c_T+cT)}$. 

$(ii)$ First, we have
\begin{align*}
Y_{xm't}(s)-Y_{xmt}(s) =&-\dfrac{1}{\lambda}\int_{t}^{s}Z_{xm't}(\tau)-Z_{xmt}(\tau)\dif \tau,\\
Z_{xm't}(s)-Z_{xmt}(s) =&\bb{E}\left[\left.\int_{s}^{T} D \dod{F}{m}(Y_{\cdot m't}(\tau) \otimes  m')(Y_{xm't}(\tau))-D \dod{F}{m}(Y_{\cdot mt}(\tau) \otimes  m)(Y_{xmt}(\tau))\dif \tau\right.\right.\\
&\left.\left.\ \ +D \dod{F_T}{m}(Y_{\cdot m't}(T) \otimes  m')(Y_{xm't}(T))-D \dod{F_T}{m}(Y_{\cdot mt}(T) \otimes  m)(Y_{xmt}(T))\right|\mathcal{W}_{t}^{s}\right].
\end{align*}
Since $D\dod{F}{m}(m)(x)$ and $D\dod{F_T}{m}(m)(x)$ are jointly Lipschitz continuous in $(m,x)\in \mathcal{P}_2(\mathbb{R}^n)\times\mathbb{R}^n$, there exists a positive constant $C_{DdF}$ such that
\begin{align}\nonumber
&\bb{E}\left|D \dod{F}{m}(Y_{\cdot m't}(s) \otimes  m')(Y_{xm't}(s))-D \dod{F}{m}(Y_{\cdot mt}(s) \otimes  m)(Y_{xmt}(s))\right|^2\\\label{CDdF}
\leq & C_{DdF} \left(\bb{E}W_2^2(Y_{\cdot m't}(s) \otimes  m',Y_{\cdot mt}(s) \otimes  m)+\bb{E}\left|Y_{xm't}(s)-Y_{xmt}(s)\right|^2\right)
\end{align}
and 
\begin{align}\nonumber
&\bb{E}\left|D \dod{F_T}{m}(Y_{\cdot m't}(T) \otimes  m')(Y_{xm't}(T))-D \dod{F}{m}(Y_{\cdot mt}(T) \otimes  m)(Y_{xmt}(T))\right|^2\\
\leq & C_{DdF} \left(\bb{E}W_2^2(Y_{\cdot m't}(T) \otimes  m',Y_{\cdot mt}(T) \otimes  m)+\bb{E}\left|Y_{xm't}(T)-Y_{xmt}(T)\right|^2\right),
\end{align}
based on which we have\small
\begin{align*}
\sup_{s\in[t,T]}\bb{E}\left|Y_{xm't}(s)-Y_{xmt}(s)\right|^2 \leq & \dfrac{|T-t|^2}{\lambda^2}\bb{E}\sup_{s\in[t,T]}\left|Z_{xm't}(s)-Z_{xmt}(s)\right|^2 ,\\
\sup_{s\in[t,T]}\bb{E}\left|Z_{xm't}(s)-Z_{xmt}(s)\right|^2\leq &C_{DdF}(1+|T-t|^2) \sup_{s\in[t,T]}\left(\bb{E}W_2^2(Y_{\cdot m't}(s) \otimes  m',Y_{\cdot mt}(s) \otimes  m)+\bb{E}\left|Y_{xm't}(s)-Y_{xmt}(s)\right|^2\right).
\end{align*}\normalsize
Thus, by simply combining these two estimates, whenever $\lambda>T\sqrt{C_{DdF}(1+T^2)}$, we arrive with:
\begin{align}\label{iieq1}
\sup_{s\in[t,T]}\bb{E}\left|Y_{xm't}(s)-Y_{xmt}(s)\right|^2 \leq & \dfrac{T^2C_{DdF}(1+T^2)}{\lambda^2-T^2C_{DdF}(1+T^2)}\sup_{s\in[t,T]}\bb{E}W_2^2(Y_{\cdot m't}(s) \otimes  m',Y_{\cdot mt}(s) \otimes  m),\\\label{iieq2}
\sup_{s\in[t,T]}\bb{E}\left|Z_{xm't}(s)-Z_{xmt}(s)\right|^2\leq &\dfrac{\lambda^2C_{DdF}(1+T^2)}{\lambda^2-T^2C_{DdF}(1+T^2)}\sup_{s\in[t,T]}\bb{E}W_2^2(Y_{\cdot m't}(s) \otimes  m',Y_{\cdot mt}(s) \otimes  m).
\end{align}
By the definition of Wasserstein metric, and then by \eqref{xLipCon} and \eqref{iieq1}, we derive:
\begin{align}\nonumber
&\bb{E}W_2^2(Y_{\cdot m't}(s) \otimes  m',Y_{\cdot mt}(s) \otimes  m)\\\nonumber
\leq&\bb{E}  \int_{\widehat{\Omega}} \left|Y_{\hat{X}_{m'}(\widehat{\omega})m't}(s)-Y_{\hat{X}_{m}(\widehat{\omega})mt}(s)\right|^2 d\widehat{\bb{P}}(\widehat{\omega})\\\nonumber
\leq&2 \bb{E}  \int_{\widehat{\Omega}} \left( \left|Y_{\hat{X}_{m'}(\widehat{\omega})m't}(s)-Y_{\hat{X}_{m}(\widehat{\omega})m't}(s)\right|^2+\left|Y_{\hat{X}_{m}(\widehat{\omega})m't}(s)-Y_{\hat{X}_{m}(\widehat{\omega})mt}(s)\right|^2 \right)d\widehat{\bb{P}}(\widehat{\omega})  \\\nonumber
\leq&2 \left(\frac{\lambda}{\lambda-T(c_T+cT)}\right)^2\cdot W_2^2(m,m')+ \dfrac{2T^2C_{DdF}(1+T^2)}{\lambda^2-T^2C_{DdF}(1+T^2)}\sup_{s\in[t,T]}\bb{E}W_2^2(Y_{\cdot m't}(s) \otimes  m',Y_{\cdot mt}(s) \otimes  m),
\end{align}
where, as before, $\hat{X}_{m}$, following $m$ marginally, and $\hat{X}_{m'}$, following $m'$ marginally, are a couple of random variables in another independent probability space $L^{2}(\widehat{\Omega},\widehat{\mathcal{A}},\widehat{\bb{P}};\bb{R}^{n})$ such that 
\begin{equation}
W_{2}^{2}(m,m')=\bb{E}^{\widehat{\bb{P}}}[|\hat{X}_{m}-\hat{X}_{m'}|^{2}].
\end{equation}
Thus, whenever $\lambda>T\sqrt{3C_{DdF}(1+T^2)}$,
\begin{align}\label{W2Y}
\sup_{s\in[t,T]}\bb{E}W_2^2(Y_{\cdot m't}(s) \otimes  m',Y_{\cdot mt}(s) \otimes  m)
\leq2 \left(\frac{\lambda}{\lambda-T(c_T+cT)}\right)^2\cdot\dfrac{\lambda^2-T^2C_{DdF}(1+T^2)}{\lambda^2-3T^2C_{DdF}(1+T^2)}\cdot W_2^2(m,m')
\end{align}
and then we further have
\begin{align*}
\sup_{s\in[t,T]}\bb{E}\left|Y_{xm't}(s)-Y_{xmt}(s)\right|^2 \leq & \dfrac{2T^2C_{DdF}(1+T^2)}{\lambda^2-3T^2C_{DdF}(1+T^2)}  \left(\frac{\lambda}{\lambda-T(c_T+cT)}\right)^2 \cdot W_2^2(m,m'),\\
\sup_{s\in[t,T]}\bb{E}\left|Z_{xm't}(s)-Z_{xmt}(s)\right|^2\leq &\dfrac{2\lambda^2C_{DdF}(1+T^2)}{\lambda^2-3T^2C_{DdF}(1+T^2)}  \left(\frac{\lambda}{\lambda-T(c_T+cT)}\right)^2 \cdot W_2^2(m,m').
\end{align*}
$\blacksquare$
\subsection{PROOF OF THE CLAIM IN REMARK \ref{rem6-8}. }
Under the additional assumptions \eqref{convexF}-\eqref{convexFT}, we can have an alternative approach of deducing the continuity of $(Y_{xmt}(s),Z_{xmt}(s))$ in $m\in\s{P}_2(\bb{R}^n)$ without invoking $\lambda>T\sqrt{3C_{DdF}(1+T^2)}$.
\begin{align*}
Y_{xm't}(s)-Y_{xmt}(s) =&-\dfrac{1}{\lambda}\int_{t}^{s}Z_{xm't}(\tau)-Z_{xmt}(\tau)\dif \tau,\\
Z_{xm't}(s)-Z_{xmt}(s) =&\bb{E}\left[\left.\int_{s}^{T} D \dod{F}{m}(Y_{\cdot m't}(\tau) \otimes  m')(Y_{xm't}(\tau))-D \dod{F}{m}(Y_{\cdot mt}(\tau) \otimes  m)(Y_{xmt}(\tau))\dif \tau\right.\right.\\
&\left.\left.\ \ +D \dod{F_T}{m}(Y_{\cdot m't}(T) \otimes  m')(Y_{xm't}(T))-D \dod{F_T}{m}(Y_{\cdot mt}(T) \otimes  m)(Y_{xmt}(T))\right|\mathcal{W}_{t}^{s}\right].
\end{align*}
Before we proceed, we first split the integrand: 
\begin{align*}
&D \dod{F}{m}(Y_{\cdot m't}(s) \otimes  m')(Y_{xm't}(s))-D \dod{F}{m}(Y_{\cdot mt}(s) \otimes  m)(Y_{xmt}(s))\\
=&D \dod{F}{m}(Y_{\cdot m't}(s) \otimes  m')(Y_{xm't}(s))-D \dod{F}{m}(Y_{\cdot m't}(s) \otimes  m)(Y_{xm't}(s))\\
&+D \dod{F}{m}(Y_{\cdot m't}(s) \otimes  m)(Y_{xm't}(s))-D \dod{F}{m}(Y_{\cdot mt}(s) \otimes  m)(Y_{xmt}(s)),
\end{align*}
based on which we next consider\small
\begin{align*}
&\bb{E}\int_{\bb{R}^n}\left(Y_{xm't}(s)-Y_{xmt}(s)\right)\cdot\bigg(D \dod{F}{m}(Y_{\cdot m't}(s) \otimes  m')(Y_{xm't}(s))-D \dod{F}{m}(Y_{\cdot mt}(s) \otimes  m)(Y_{xmt}(s))\bigg)dm(x)\\
=&\bb{E}\int_{\bb{R}^n}\left(Y_{xm't}(s)-Y_{xmt}(s)\right)\cdot\int_0^1 \widetilde{\bb{E}}\int_{\bb{R}^n}D_1 \dod[2]{F}{m}(Y^\theta_\otimes(s))(Y_{xm't}(s),\widetilde{Y}_{\zeta m't}(s))d(m'-m)(\zeta)d\theta dm(x)\\
&+\bb{E}\int_{\bb{R}^n}\left(Y_{xm't}(s)-Y_{xmt}(s)\right)\cdot\int_0^1 D^2 \dod{F}{m}(Y^\theta_\cdot(s) \otimes  m)(Y^\theta_x(s))d\theta\left(Y_{xm't}(s)-Y_{xmt}(s)\right)dm(x)\\
&+\bb{E}\int_{\bb{R}^n}\left(Y_{xm't}(s)-Y_{xmt}(s)\right)\cdot\int_0^1 \widetilde{\bb{E}}\int_{\bb{R}^n} D_2 D_1 \dod[2]{F}{m}(Y^\theta_\cdot(s) \otimes  m)(Y^\theta_x(s),\widetilde{Y}^\theta_\zeta(s))\left(\widetilde{Y}_{\zeta m't}(s)-\widetilde{Y}_{\zeta mt}(s)\right)dm(\zeta)d\theta dm(x)\\
\geq&\bb{E}\int_{\bb{R}^n}\left(Y_{xm't}(s)-Y_{xmt}(s)\right)\cdot\int_0^1\widetilde{\bb{E}}\int_{\bb{R}^n}D_1 \dod[2]{F}{m}(Y^\theta_\otimes(s))(Y_{xm't}(s),\widetilde{Y}_{\zeta m't}(s))d(m'-m)(\zeta)d\theta dm(x)\\
&+\lambda_F\bb{E}\int_{\bb{R}^n}\left|Y_{xm't}(s)-Y_{xmt}(s)\right|^2dm(x)
\end{align*}\normalsize
where $Y^\theta_\otimes(s) := Y_{\cdot m't}(s) \otimes  \left(\theta m'+(1-\theta)m\right)$ and $Y^\theta_x(s):=\theta Y_{xm't}(s)+(1-\theta)Y_{xmt}(s)$; and the last line of inequality follows by using \eqref{convexF} for the last two terms of the sum in the second equality.
By \eqref{eq:6-15} together with mean value theorem, and then \eqref{xLipCon}, 
\begin{align*}
&\left|\int_0^1\widetilde{\bb{E}}\int_{\bb{R}^n}D_1 \dod[2]{F}{m}(Y^\theta_\otimes(s))(Y_{xm't}(s),\widetilde{Y}_{\zeta m't}(s))d(m'-m)(\zeta)d\theta\right|\\
=&\left|\int_0^1\widetilde{\bb{E}}\int_{\widehat{\Omega}}D_1 \dod[2]{F}{m}(Y^\theta_\otimes(s))(Y_{xm't}(s),\widetilde{Y}_{\hat{X}_{m'}(\widehat{\omega}) m't}(s))d\widehat{\bb{P}}(\widehat{\omega})d\theta\right.\\
&\left.-\int_0^1\widetilde{\bb{E}}\int_{\widehat{\Omega}}D_1 \dod[2]{F}{m}(Y^\theta_\otimes(s))(Y_{xm't}(s),\widetilde{Y}_{\hat{X}_{m}(\widehat{\omega}) m't}(s))d\widehat{\bb{P}}(\widehat{\omega})d\theta\right|\\
\leq& c\,\widetilde{\bb{E}}\int_{\widehat{\Omega}}\left|\widetilde{Y}_{\hat{X}_{m'}(\widehat{\omega}) m't}(s)-\widetilde{Y}_{\hat{X}_{m}(\widehat{\omega}) m't}(s)\right|d\widehat{\bb{P}}(\widehat{\omega})\\
\leq& \frac{c \lambda}{\lambda-T(c_T+cT)}\int_{\widehat{\Omega}}\left|\hat{X}_{m'}(\widehat{\omega})-\hat{X}_{m}(\widehat{\omega})\right|d\widehat{\bb{P}}(\widehat{\omega})=\frac{c \lambda}{\lambda-T(c_T+cT)}W_1(m,m')\\
\leq& \frac{c \lambda}{\lambda-T(c_T+cT)}W_2(m,m'),
\end{align*}
where $\hat{X}_{m}$, following $m$ marginally, and $\hat{X}_{m'}$, following $m'$ marginally, are a couple of random variables in another independent probability space $L^{2}(\widehat{\Omega},\widehat{\mathcal{A}},\widehat{\bb{P}};\bb{R}^{n})$ such that 
\begin{equation}
W_{2}^{2}(m,m')=\bb{E}^{\widehat{\bb{P}}}[|\hat{X}_{m}-\hat{X}_{m'}|^{2}].
\end{equation}
Therefore, 
\begin{align}\nonumber
&\bb{E}\int_{\bb{R}^n}\left(Y_{xm't}(s)-Y_{xmt}(s)\right)\cdot\bigg(D \dod{F}{m}(Y_{\cdot m't}(s) \otimes  m')(Y_{xm't}(s))-D \dod{F}{m}(Y_{\cdot mt}(s) \otimes  m)(Y_{xmt}(s))\bigg)dm(x)\\\nonumber
\geq&-\frac{c \lambda}{\lambda-T(c_T+cT)}W_2(m,m')\cdot \bb{E}\int_{\bb{R}^n}\left|Y_{xm't}(s)-Y_{xmt}(s)\right|dm(x)\\\label{mds1}
&+\lambda_F\bb{E}\int_{\bb{R}^n}\left|Y_{xm't}(s)-Y_{xmt}(s)\right|^2dm(x).
\end{align}
By the same argument and using \eqref{convexFT}, one also has
\begin{align}\nonumber
&\bb{E}\int_{\bb{R}^n}\left(Y_{xm't}(T)-Y_{xmt}(T)\right)\cdot\bigg(D \dod{F_T}{m}(Y_{\cdot m't}(T) \otimes  m')(Y_{xm't}(T))-D \dod{F_T}{m}(Y_{\cdot mt}(T) \otimes  m)(Y_{xmt}(T))\bigg)dm(x)\\\nonumber
\geq&-\frac{c_T \lambda}{\lambda-T(c_T+cT)}W_2(m,m')\cdot \bb{E}\int_{\bb{R}^n}\left|Y_{xm't}(T)-Y_{xmt}(T)\right|dm(x)\\\label{mds2}
&+\lambda_F\bb{E}\int_{\bb{R}^n}\left|Y_{xm't}(T)-Y_{xmt}(T)\right|^2dm(x).
\end{align}
By \eqref{mds1},
\begingroup
\allowdisplaybreaks
\begin{align}\nonumber
&\bb{E}\int_{\bb{R}^n}\left(Y_{xm't}(T)-Y_{xmt}(T)\right)\cdot \left(Z_{xm't}(T)-Z_{xmt}(T)\right) dm(x)\\\nonumber
&-\bb{E}\int_{\bb{R}^n}\left(Y_{xm't}(t)-Y_{xmt}(t)\right)\cdot \left(Z_{xm't}(t)-Z_{xmt}(t)\right) dm(x)\\\nonumber
=&\int_t^T d_s\left(\bb{E}\int_{\bb{R}^n}\left(Y_{xm't}(s)-Y_{xmt}(s)\right)\cdot \left(Z_{xm't}(s)-Z_{xmt}(s)\right) dm(x)\right)\\\nonumber
=&\int_t^T\left(\bb{E}\int_{\bb{R}^n}d_s\left(Y_{xm't}(s)-Y_{xmt}(s)\right)\cdot \left(Z_{xm't}(s)-Z_{xmt}(s)\right) dm(x)\right.\\\nonumber
&\left.+\bb{E}\int_{\bb{R}^n}\left(Y_{xm't}(s)-Y_{xmt}(s)\right)\cdot d_s\left(Z_{xm't}(s)-Z_{xmt}(s)\right) dm(x)\right)\\\nonumber
=&\int_t^T \left[-\frac{1}{\lambda} \bb{E}\int_{\bb{R}^n} \left|Z_{xm't}(s)-Z_{xmt}(s)\right|^2 dm(x)\right.\\\nonumber
&\left.-\bb{E}\int_{\bb{R}^n}\left(Y_{xm't}(s)-Y_{xmt}(s)\right)\cdot\left(D \dod{F}{m}(Y_{\cdot m't}(s) \otimes  m')(Y_{xm't}(s))-D \dod{F}{m}(Y_{\cdot mt}(s) \otimes  m)(Y_{xmt}(s))\right)dm(x)\right]ds\\\nonumber
\leq &\int_t^T \left[-\frac{1}{\lambda} \bb{E}\int_{\bb{R}^n} \left|Z_{xm't}(s)-Z_{xmt}(s)\right|^2 dm(x)-\lambda_F\bb{E} \int_{\bb{R}^n}\left|Y_{xm't}(s)-Y_{xmt}(s)\right|^2dm(x)\right.\\\label{eq6-37}
&\left.+\frac{c \lambda}{\lambda-T(c_T+cT)}W_2(m,m')\cdot\bb{E}\int_{\bb{R}^n}\left|Y_{xm't}(s)-Y_{xmt}(s)\right|dm(x)\right]ds
\end{align}\endgroup
Thus, by using \eqref{mds2} for the terminal cross term, also noting that $Y_{xm't}(t)=Y_{xmt}(t)=x$, we obtain:
\begin{align*}
&\frac{1}{\lambda}\int_t^T \bb{E} \int_{\bb{R}^n}\left|Z_{xm't}(s)-Z_{xmt}(s)\right|^2 dm(x) ds+\lambda_F\int_t^T \bb{E}\int_{\bb{R}^n} \left|Y_{xm't}(s)-Y_{xmt}(s)\right|^2 dm(x) ds\\
&+\lambda_F\bb{E}\int_{\bb{R}^n}\left|Y_{xm't}(T)-Y_{xmt}(T)\right|^2dm(x)\\
\leq &\frac{c \lambda}{\lambda-T(c_T+cT)}W_2(m,m')\int_t^T \bb{E}\int_{\bb{R}^n}\left|Y_{xm't}(s)-Y_{xmt}(s)\right|dm(x)ds\\
&+\frac{c_T \lambda}{\lambda-T(c_T+cT)}W_2(m,m')\cdot \bb{E}\int_{\bb{R}^n}\left|Y_{xm't}(T)-Y_{xmt}(T)\right|dm(x).
\end{align*}
Therefore, by simple algebra and Cauchy-Schwarz inequality, we have
\begin{align}\nonumber
&\int_t^T \bb{E} \int_{\bb{R}^n}\left|Z_{xm't}(s)-Z_{xmt}(s)\right|^2 dm(x) ds+\int_t^T \bb{E}\int_{\bb{R}^n} \left|Y_{xm't}(s)-Y_{xmt}(s)\right|^2 dm(x) ds\\\nonumber
&+\bb{E}\int_{\bb{R}^n}\left|Y_{xm't}(T)-Y_{xmt}(T)\right|^2dm(x)\\\label{Lipinm}
\leq &C_1(c,c_T,\lambda,\lambda_F,T)W_2^2(m,m'),
\end{align}
where $C_1(c,c_T,\lambda,\lambda_F,T):=\frac{2}{\lambda_F}\left(\lambda+\frac{1}{\lambda_F}\right)\left(\frac{\max\{c T,c_T\}\lambda}{\lambda-T(c_T+cT)}\right)^2$, which is a constant solely depending on $c$, $c_T$, $\lambda$, $\lambda_F$ and $T$.
Recall $\hat{X}_{m}$ and $\hat{X}_{m'}$
the random variables in $L^{2}(\widehat{\Omega},\widehat{\mathcal{A}},\widehat{\bb{P}};\bb{R}^{n})$ as defined above, then, by \eqref{xLipCon},
\begin{align}\label{eq6-35}
\sup_{s\in[t,T]} \left|  Y_{\hat{X}_{m'}(\widehat{\omega})mt}(s)-Y_{\hat{X}_{m}(\widehat{\omega})mt}(s)\right|+\sup_{s\in[t,T]} \left|  Z_{\hat{X}_{m'}(\widehat{\omega})mt}(s)-Z_{\hat{X}_{m}(\widehat{\omega})mt}(s)\right|  \leq C_2(c,c_T,\lambda,T)|\hat{X}_{m'}(\widehat{\omega})-\hat{X}_{m}(\widehat{\omega})|,
\end{align}
where $C_2(c,c_T,\lambda,T):=\frac{\lambda (1+c_T+cT)}{\lambda-T(c_T+cT)}$.
By definition of Wasserstein metric, and then by \eqref{Lipinm} and \eqref{eq6-35},
\begin{align}\nonumber
&\int_t^T \bb{E}W_2^2(Z_{\cdot m't}(s) \otimes  m',Z_{\cdot mt}(s) \otimes  m)ds\\\nonumber
&+\int_t^T \bb{E}W_2^2(Y_{\cdot m't}(s) \otimes  m',Y_{\cdot mt}(s) \otimes  m)ds+\bb{E}W_2^2(Y_{\cdot m't}(T) \otimes  m',Y_{\cdot mt}(T) \otimes  m)\\\nonumber
\leq&\int_t^T \bb{E}  \int_{\widehat{\Omega}} \left|Z_{\hat{X}_{m'}(\widehat{\omega})m't}(s)-Z_{\hat{X}_{m}(\widehat{\omega})mt}(s)\right|^2 d\widehat{\bb{P}}(\widehat{\omega})  ds+\int_t^T \bb{E}  \int_{\widehat{\Omega}} \left|Y_{\hat{X}_{m'}(\widehat{\omega})m't}(s)-Y_{\hat{X}_{m}(\widehat{\omega})mt}(s)\right|^2 d\widehat{\bb{P}}(\widehat{\omega})  ds\\\nonumber
&+\bb{E} \int_{\widehat{\Omega}}\left|Y_{\hat{X}_{m'}(\widehat{\omega})m't}(T)-Y_{\hat{X}_{m}(\widehat{\omega})mt}(T)\right|^2d\widehat{\bb{P}}(\widehat{\omega}) \\\nonumber
\leq&2\int_t^T \bb{E}  \int_{\widehat{\Omega}} \left|Z_{\hat{X}_{m'}(\widehat{\omega})m't}(s)-Z_{\hat{X}_{m}(\widehat{\omega})m't}(s)\right|^2+\left|Z_{\hat{X}_{m}(\widehat{\omega})m't}(s)-Z_{\hat{X}_{m}(\widehat{\omega})mt}(s)\right|^2 d\widehat{\bb{P}}(\widehat{\omega})  ds\\\nonumber
&+2\int_t^T \bb{E}  \int_{\widehat{\Omega}} \left|Y_{\hat{X}_{m'}(\widehat{\omega})m't}(s)-Y_{\hat{X}_{m}(\widehat{\omega})m't}(s)\right|^2+\left|Y_{\hat{X}_{m}(\widehat{\omega})m't}(s)-Y_{\hat{X}_{m}(\widehat{\omega})mt}(s)\right|^2 d\widehat{\bb{P}}(\widehat{\omega})  ds\\\nonumber
&+2\bb{E} \int_{\widehat{\Omega}}\left|Y_{\hat{X}_{m'}(\widehat{\omega})m't}(T)-Y_{\hat{X}_{m}(\widehat{\omega})m't}(T)\right|^2+\left|Y_{\hat{X}_{m}(\widehat{\omega})m't}(T)-Y_{\hat{X}_{m}(\widehat{\omega})mt}(T)\right|^2d\widehat{\bb{P}}(\widehat{\omega}) \\\nonumber
\leq&2\left((1+T)C_2(c,c_T,\lambda,T)^2+C_1(c,c_T,\lambda,\lambda_F,T)\right)\cdot W_2^2(m,m')\\\label{LipinW2}
=&C_3(c,c_T,\lambda,\lambda_F,T)\cdot W_2^2(m,m'),
\end{align}
where $C_3(c,c_T,\lambda,\lambda_F,T):=\frac{2(1+T)\lambda(1+c_T+cT)}{\lambda-T(c_T+cT)}+\frac{4}{\lambda_F}\left(\lambda+\frac{1}{\lambda_F}\right)\left(\frac{\max\{c T,c_T\}\lambda}{\lambda-T(c_T+cT)}\right)^2$, and we telescoped in the second inequality. $\blacksquare$
\end{document}